\documentclass[12pt,leqno]{amsart}
\usepackage{amssymb}
\usepackage{mathtools}
\usepackage[mathscr]{euscript}
\usepackage{stmaryrd}
\usepackage{tikz}
\usetikzlibrary{matrix,arrows}
\usepackage{xparse}
\usepackage{array}
\usepackage{caption}
\usepackage{pdflscape}
\usepackage{color}
\usepackage[all]{xy}

\usepackage{enumitem}
\setenumerate[1]{label=(\alph*)}
\setenumerate[2]{label=(\roman*)}

\usepackage{rotating}
\usepackage{extarrows}
\usepackage{accents}

\usepackage{leftidx}
\usepackage{upgreek}

\usepackage{xr-hyper}
\usepackage{hyperref}
\hypersetup{colorlinks=true, anchorcolor=red, citecolor=red,
filecolor=magenta, linkcolor=red, menucolor=red, urlcolor=blue}

\renewcommand{\epsilon}{\varepsilon}
\renewcommand{\phi}{\varphi}

\newtheorem{theorem}{Theorem}[section]
\newtheorem{proposition}[theorem]{Proposition}
\newtheorem{lemma}[theorem]{Lemma}
\newtheorem{corollary}[theorem]{Corollary}

\theoremstyle{definition}
\newtheorem{definition}[theorem]{Definition}
\newtheorem{example}[theorem]{Example}

\theoremstyle{remark}
\newtheorem{remark}[theorem]{Remark}

\newcommand{\oldul}[1]{{\underline{#1}}}
\newcommand{\ul}[1]{{\underline{\smash{#1}}\vphantom{#1}}}

\newcommand{\ud}[1]{{%
    \text{%
      \tikz[baseline=(todotted.base)]{
        \node[inner ysep=1.8pt,inner xsep=0pt,outer sep=0pt]
        (todotted) {\smash{${#1}$}};
        \draw[-,densely dotted, thick] (todotted.south west) --
        (todotted.south east);
      }%
    }%
}}%

\newcommand{\oldud}[1]{{%
    \text{%
      \tikz[baseline=(todotted.base)]{
        \node[inner ysep=1.8pt,inner xsep=0pt,outer sep=0pt]
        (todotted) {{${#1}$}};
        \draw[-,densely dotted, thick] (todotted.south west) --
        (todotted.south east);
      }%
    }%
}}%

\DeclareMathAccent{\wtilde}{\mathord}{largesymbols}{"65}

\newcommand{\Yo}{{\op{Yo}}}

\newcommand{\Ho}{{\op{Ho}}}
\newcommand{\ohe}{{\op{ohe}}}

\newcommand{\iCone}{\op{iCone}}

\newcommand{\fussnote}[1]{}

\newcommand{\II}{{\mathbb{I}}}
\newcommand{\ii}{{\mathbf{i}}}
\newcommand{\iotaii}{{\upiota}}

\newcommand{\FF}{{\mathbb{F}}}

\newcommand{\EE}{{\mathbb{E}}}
\newcommand{\ee}{{\mathbf{e}}}
\newcommand{\epsilonee}{{\upepsilon}}

\DeclareMathOperator{\ccof}{\mathbf{cof}}
\newcommand{\gammacc}{{\upgamma}}
\DeclareMathOperator{\ffib}{\mathbf{fib}}
\newcommand{\phiff}{{\upphi}}

\newcommand{\R}{{\mathsf{R}}}
\newcommand{\RR}{{\mathscr{R}}}
    
\newcommand{\kk}{{\mathsf{k}}}
\newcommand{\KK}{{\mathscr{K}}}

\newcommand{\MMod}{{\mathbb{M}\op{od}}}

\newcommand{\IIMod}{{\mathbb{IM}\op{od}}}
\newcommand{\PPMod}{{\mathbb{PM}\op{od}}}
\newcommand{\pp}{{\mathbf{p}}}

\newcommand{\PP}{{\mathbb{P}}}

\newcommand{\ms}[1]{{\mathscr{#1}}}
\newcommand{\ulms}[1]{{\ul{\ms{#1}}}}

\newcommand{\sptag}[1]{\href{http://stacks.math.columbia.edu/tag/#1}{#1}}
\newcommand{\citestacks}[1]{\cite[\sptag{#1}]{stacks-project}}

\newcommand{\op}[1]{{\operatorname{#1}}}
\newcommand{\Sets}{{\op{Sets}}}

\newcommand{\id}{{\op{id}}}

\newcommand{\subtimes}[1]{\underset{{#1}}{\times}}
\newcommand{\subsqcup}[1]{\underset{{#1}}{\bigsqcup}}

\newcommand{\ra}{\rightarrow}
\newcommand{\la}{\leftarrow}
\newcommand{\sra}{\twoheadrightarrow}
\newcommand{\hra}{\hookrightarrow}

\newcommand{\xra}[2][]{\xrightarrow[{#1}]{#2}}
\newcommand{\xla}[2][]{\xleftarrow[{#1}]{#2}}

\newcommand{\sira}{\xra{\sim}}
\newcommand{\xsira}[1]{\xrightarrow[\sim]{#1}}

\newcommand{\sila}{\xla{\sim}}

\newcommand{\xlongra}[2][]{\xlongrightarrow[{#1}]{#2}}
\newcommand{\xlongla}[2][]{\xlongleftarrow[{#1}]{#2}}

\newcommand{\xlongsira}[1]{\xlongrightarrow[\sim]{#1}}

\newcommand{\xsra}[1]{\overset{#1}{\twoheadrightarrow}}
\newcommand{\xhra}[1]{{\xhookrightarrow{#1}}}

\NewDocumentCommand{\xrightleftarrows}{ O{}O{} }{%
\mathrel{%
\vcenter{\hbox{%
\begin{tikzpicture}
  \node[minimum width=1cm,minimum height=1ex,anchor=south,align=center] (a){\text{\vphantom{hg}#1}\\[0.5ex] \vphantom{hg}#2};
  \draw[->] ([yshift=0.35ex]a.west) -- ([yshift=0.35ex]a.east);
  \draw[<-] ([yshift=-0.35ex]a.west) -- ([yshift=-0.35ex]a.east);
\end{tikzpicture}
}}%
}%
}

\newcommand{\dR}{\mathbf{R}}
\newcommand{\dL}{\mathbf{L}}

\newcommand{\bZ}{\mathbb{Z}}
\newcommand{\bN}{\mathbb{N}}

\DeclareMathOperator{\Mor}{Mor}
\newcommand{\Hom}{{\op{Hom}}}

\newcommand{\Sh}{{\op{Sh}}}

\newcommand{\Mod}{{\op{Mod}}}

\newcommand{\C}{{\op{C}}}

\newcommand{\D}{{\op{D}}}
\newcommand{\qc}{{\op{qc}}}

\renewcommand{\L}{{\op{L}}}

\newcommand{\ac}{{\op{ac}}}

\newcommand{\sheafHom}{{\mathcal{H}{om}}}

\newcommand{\End}{\op{End}}

\newcommand{\ol}[1]{{\overline{#1}}}

\newcommand{\Kokern}{\op{cok}}

\newcommand{\inv}{^{-1}}

\newcommand{\Ab}{\op{Ab}}

\newcommand{\cof}{{\op{cof}}}
\newcommand{\fib}{{\op{fib}}}
\newcommand{\bifib}{{\op{bifib}}}

\newcommand{\tp}{{\op{t}}}

\DeclareMathOperator{\Obj}{Obj}

\newcommand{\cat}{\op{cat}}
\newcommand{\trcat}{\op{trcat}}
\newcommand{\dgcat}{\op{dgcat}}
\newcommand{\enh}{\op{enh}}
\newcommand{\fml}{\op{fml}}

\newcommand{\CAT}{\op{CAT}}
\newcommand{\TRCAT}{\op{TRCAT}}
\newcommand{\DGCAT}{\op{DGCAT}}
\newcommand{\ENH}{\op{ENH}}
\newcommand{\FML}{\op{FML}}

\newcommand{\card}{\op{card}}

\newcommand{\pt}{{\op{pt}}}

\newcommand{\tildew}[1]{\widetilde{#1}}

\newcommand{\colim}{\operatorname{colim}}

\newcommand{\Fun}{\op{Fun}}

\newcommand{\opp}{{\op{op}}}
\newcommand{\co}{{\op{co}}}

\newcommand{\define}[1]{{\textbf{#1}}}

\newcommand{\hinj}{{\op{hinj}}}
\newcommand{\whinj}{{\op{whinj}}}

\newcommand{\hflat}{{\op{hflat}}}

\newcommand{\Cone}{\op{Cone}}

\newcommand{\pretr}{{\op{pre-tr}}}

\makeatletter
\newcommand{\leqnomode}{\tagsleft@true\let\veqno\@@leqno}
\newcommand{\reqnomode}{\tagsleft@false\let\veqno\@@eqno}
\makeatother

\newcounter{tablerow}

\newcommand{\labelrow}[2]{%
  \begin{minipage}[b][1ex]{12mm}
    \reqnomode
    \stepcounter{tablerow}
    \begin{equation}
      \label{#2} \tag{T#1.\thetablerow}
    \end{equation}
    \leqnomode
  \end{minipage}
}

\numberwithin{equation}{section}

\title[Six operations on dg enhancements]{Six operations on dg
  enhancements\\
  of derived categories of sheaves}  
\author{Olaf M.~Schn{\"u}rer}

\address{
  Mathematisches Institut\\ 
  Universit{\"a}t Bonn\\
  Endenicher Allee 60\\
  53115 Bonn\\
  Germany
}

\email{olaf.schnuerer@math.uni-bonn.de}

\begin{document}

\begin{abstract}
  We lift Grothendieck--Verdier--Spaltenstein's six functor formalism for derived
  categories of sheaves on ringed spaces over a field 
  to 
  differential graded enhancements.
  Our main tools come from enriched model category theory.
\end{abstract}

\maketitle

\setcounter{tocdepth}{1}

\tableofcontents

\section{Introduction}
\label{sec:introduction}

\subsection{}

Grothendieck--Verdier--Spaltenstein's six functor formalism in the topological setting
concerns the six functors $\otimes^\dL$, $\dR\sheafHom$,
$\dL \alpha^*$, $\dR \alpha_*$, $\dR \alpha_!$, $\alpha^!$
between derived categories 
of sheaves on ringed spaces
and their relations
\cite{verdier-dualite-bourbaki,spaltenstein,KS,wolfgang-olaf-locallyproper}.
Nowadays, triangulated categories
are often replaced by suitable differential graded (dg)
enhancements 
because some useful constructions can be performed
with dg categories but not with triangulated categories
\cite{bondal-kapranov-enhancements,keller-on-dg-categories-ICM}. 
Therefore it is natural to ask 
whether Grothendieck--Verdier--Spaltenstein's six functor formalism lifts to the level
of dg enhancements.
We give an affirmative answer to this question
if we fix a field
$\kk$ and work with $\kk$-ringed spaces, i.\,e.\ pairs $(X,
\mathcal{O})$ consisting of a topological space $X$ and
a sheaf $\mathcal{O}$ of commutative $\kk$-algebras on $X$. 
Here is our main result.

\begin{theorem}
  \label{t:main}
  Let $\kk$ be a field. Then
  Grothendieck--Verdier--Spaltenstein's six functor formalism for \mbox{$\kk$-ringed} spaces
  lifts to dg $\kk$-enhancements.
  The formalism involving the first four functors
  lifts 
  more generally
  for $\kk$-ringed topoi.
\end{theorem}

Our main tools to lift the six functors 
come from
enriched model category theory. These tools are not applicable if
we work over the integers. More details and the precise 
meaning of this theorem are explained later on in this
introduction.

Our original interest in this lifting problem arose when
we studied the relation between geometric and
homological smoothness and translated Fourier-Mukai
functors to the dg world
\cite{lunts-categorical-resolution,valery-olaf-new-enhancements}.
For the arguments given there and in many other instances it is
necessary to lift certain 
isomorphisms from the triangulated level to the dg level. 
Our dg enhanced six functor formalism provides
a solution to this lifting problem, see section~\ref{sec:lifting-actions-dg}.

The need to address lifting
problems as mentioned above is well-known to the experts. For
example, Drinfeld writes in the introduction to
\cite{drinfeld-dg-quotients-arxiv-newer-than-published}:
``Hopefully the part of homological algebra most relevant for
algebraic geometry will be rewritten using dg categories or
rather the more flexible notion of $A_\infty$-category.'' In this
article we mostly work in the topological setting. However, our
techniques can be extended to many other contexts. For example,
they can be used to lift the formalism involving the various
derived categories associated to schemes (or algebraic stacks) over
a field $\kk$ to dg $\kk$-enhancements. Presumably, more details will
appear
in forthcoming articles. Some preparatory results concerning
modules over dg $\kk$-categories and, in
particular, modules over $\kk$-algebras are
already given here.

On the higher categorical level, six functor formalisms have
been considered in the context of stable $\infty$-categories by
Liu and Zheng \cite{liu-zheng-enhanced-six-operations} and 
in the context of derivators
by H\"ormann 
\cite{hoermann-fibered-multiderivators-and-six-functors}.
The notion of a derived functor in the dg setting is discussed by
Drinfeld
\cite{drinfeld-dg-quotients-arxiv-newer-than-published}. To the
best of our knowledge, however, a six functor formalism in the
context of dg categories has not been established elsewhere in
the literature.  Based on Weibel's results \cite[App.]{weibel-htpy-alg-K-theory},
one out of the six functors, the derived inverse image, is lifted
to dg enhancements by Cirone \cite{cirone-strictly-functorial}.  There
are several ad hoc constructions lifting certain morphisms
from the triangulated level to the dg level, see e.\,g.\
\cite{guillermou-dg-methods-microlocalization,
  kuznetsov-height,
  valery-olaf-new-enhancements,
  polishchuk-vandenbergh-sod-equiv-coh-sheaves-reflection-gps}.
Let us also mention the claim without proof in
\cite[2.2]{nadler-microlocal-branes} that the six functor
formalism can be lifted to dg enhancements.

\subsection{}
\label{sec:short-overview}

Grothendieck--Verdier--Spaltenstein's six functor formalism takes place in the
$\kk$-linear 2-multicategory $\TRCAT_\kk$ of triangulated
$\kk$-categories 
(the prefix ``multi'' takes care of functors with several inputs
like $\otimes^\dL$ and $\dR\sheafHom$; $\kk$-linearity refers to
the $\kk$-vector spaces of natural transformations). The relevant objects are
the derived categories $\D(X)$ of sheaves of
$\mathcal{O}$-modules on $X$. The six functors $\otimes^\dL$,
$\dR\sheafHom$, $\dL \alpha^*$, $\dR \alpha_*$, $\dR \alpha_!$,
$\alpha^!$ are 1-morphisms between these objects.  The relations
between these functors are encoded in two ways: first, by
2-morphisms like $\id \ra \dR\alpha_*\dL\alpha^*$ and
2-isomorphisms like
$\dR\alpha_*\dR\sheafHom(-, \alpha^!(-)) \sira
\dR\sheafHom(\dR\alpha_!(-),-)$;
second, by commutative diagrams constructed from these
2-morphisms (i.\,e.\ by equalities between compositions of
2-morphisms): they encode for example that 
$(\dL \alpha^*, \dR \alpha_*)$ is a pair of adjoint functors or
that $(\D(X), \otimes^\dL)$ is a symmetric monoidal category.

Ideally,
one would hope to find a
corresponding formalism in some subcategory of the
$\kk$-linear 2-multicategory $\DGCAT_\kk$ of dg $\kk$-categories.
We achieve this goal up to inverting some 2-morphisms as
follows. 

We define a $\kk$-linear 2-multicategory $\ENH_\kk$ of dg
$\kk$-enhancements. 
Its objects are additive pretriangulated dg
$\kk$-categories, its 1-morphisms coincide with the
corresponding 1-morphisms in $\DGCAT_\kk$, i.\,e.\ with dg
$\kk$-functors, and its 2-morphisms 
can be represented
by zig-zags of 2-morphisms in $\DGCAT_\kk$, i.\,e.\
by zig-zags of honest
dg $\kk$-natural transformations,
where the arrows pointing in the
wrong direction 
are objectwise homotopy equivalences
(see section~\ref{sec:2-mult}).
More precisely, the morphism categories of $\ENH_\kk$ are
obtained from the corresponding morphism categories of
$\DGCAT_\kk$ by inverting the objectwise homotopy equivalences. 
Mapping a dg $\kk$-category $\ulms{M}$ to its homotopy category
$[\ulms{M}]$ defines a functor
\begin{equation}
  \label{eq:intro:[-]-ENH-TRCAT}
  [-] \colon \ENH_\kk \ra \TRCAT_\kk. 
\end{equation}
We establish a six functor
formalism in $\ENH_\kk$ and show that it corresponds to
Gro\-then\-dieck--Verdier--Spaltenstein's six 
functor formalism under the functor
\eqref{eq:intro:[-]-ENH-TRCAT}.
The objects of $\ENH_\kk$ that enhance the derived categories $\D(X)$
are the dg $\kk$-categories $\ul\II(X)$ of h-injective complexes
of injective sheaves on $X$: the homotopy category $[\ul\II(X)]$ of
$\ul\II(X)$ is equivalent to $\D(X)$.  The 1-morphisms in $\ENH_\kk$
that enhance the six functors $\otimes^\dL$, $\dR\sheafHom$,
$\dL \alpha^*$, $\dR \alpha_*$, $\dR \alpha_!$, $\alpha^!$ are dg
$\kk$-functors $\ul\otimes$, $\ul\sheafHom$, $\ul\alpha^*$,
$\ul\alpha_*$, $\ul\alpha_!$, $\ul\alpha^!$. Their definition
relies on the fact that $\kk$ is a field. 
We lift all the standard 2-(iso)morphisms in $\TRCAT_\kk$ in the
right columns of tables~\ref{tab:main} and \ref{tab:subsequent}
on pages \pageref{tab:main} and \pageref{tab:subsequent} to
2-(iso)morphisms in $\ENH_\kk$ in the middle columns of these
tables.
All these lifts are defined by explicit zig-zags of 2-morphisms
in $\DGCAT_\kk$; the lifts marked as 2-isomorphisms come from
zig-zags of objectwise homotopy equivalences (with one exception,
namely the 2-isomorphism 
$\ul{\beta}'_*\ul\alpha'^! \sira \ul\alpha^!\ul\beta_*$
in row \eqref{eq:tab:proper-base-change-upper-shriek-ENH} of
table~\ref{tab:subsequent}).
We prove commutativity of certain
diagrams constructed from these lifts (see
section~\ref{sec:some-comm-diagr} for some of these diagrams);
for example we show that 
$(\ul\alpha^*, \ul\alpha_*)$ forms a pair of adjoint 1-morphisms
and 
that $(\ul\II(X), \ul\otimes)$ is a symmetric
monoidal object of
$\ENH_\kk$. 
\fussnote{ 
  is it a symmetric closed monoidal object?
} 

We encourage the reader to look at
section~\ref{sec:description-the-main} where 
we give a very precise description of our results. 
In 
section~\ref{sec:users-guide-dg} we give a user's guide
to the dg 
$\kk$-enhanced six functor formalism. We also describe the
$\kk$-linear 2-multicategory $\FML_\kk$ of formulas there 
(see section~\ref{sec:2-mult-form}).
It is our main formal tool
to describe all relations between functors between
derived categories of sheaves that we can lift to dg 
$\kk$-enhancements. We use it to explain the precise
meaning of Theorem~\ref{t:main}
(see Theorem~\ref{t:intro:interprete-compare-FML-ENH}
and
section~\ref{sec:mean-theorem-main}).
In section \ref{sec:lifting-actions-dg} we explain how to lift certain
isomorphisms from the triangulated level to the dg level.
We then comment on the
ingredients from enriched model category theory in
section~\ref{sec:model-categories-dg}. Let us also mention
Lemma~\ref{l:no-additive-injective-repl-for-Ab},
which explains why we need to work over a field.

\subsection{}
\label{sec:appendices}

Our article contains three appendices. The first two appendices
contain foundational statements of independent
interest. 
In appendix~\ref{sec:spalt-results-ring} we extend 
some results of Spaltenstein \cite{spaltenstein} from ringed
spaces to ringed topoi.
In appendix~\ref{sec:change-universe} we explain that passing to
a higher Grothendieck universe yields fully faithful embeddings
on derived categories of sheaves; more precisely, we
show that the change of Grothendieck 
universe functor preserves h-injective complexes of
sheaves. Surprisingly, these results seem not
to be available elsewhere in the literature.

\subsection{Plan of the article}
\label{sec:plan-the-article}

We consider section~\ref{sec:description-the-main} described above
as an extended introduction.  The aim of
section~\ref{sec:dg-enrich-funct} is to formulate and prove
Theorem~\ref{t:existence-enriched-functorial-fact}. It provides
criteria to ensure that a model structure on a dg $\kk$-category
admits dg $\kk$-enriched functorial factorizations.  This result
is applied in section~\ref{sec:applications} where we consider
various model structures on categories of complexes of
sheaves 
and on categories of dg modules and show that they are compatible
with the dg
$\kk$-enrichments of these categories. 
The $\kk$-linear
2-multicategory $\ENH_\kk$ of dg $\kk$-enhancements is defined in
section~\ref{sec:some-2-mult}.  We then lift the six functor
formalism: the four functors $\otimes^\dL$, $\dR\sheafHom$,
$\dL \alpha^*$, $\dR \alpha_*$ associated to $\kk$-ringed topoi
and their morphisms are treated in
section~\ref{sec:four-operations}; the two functors
$\dR \alpha_!$, $\alpha^!$ associated to suitable morphisms of
$\kk$-ringed spaces (cf.\ \cite{wolfgang-olaf-locallyproper}) are
included into the picture in section~\ref{sec:two-operations}.
The $\kk$-linear 2-multicategory of formulas $\FML_\kk$ is
defined and used in section~\ref{sec:2-multicat-formulas}
to state the summarizing
Theorem~\ref{t:compare-interpretations-FML}. 
In section~\ref{sec:some-remarks-2-morphisms-ENH} we explain some
useful facts concerning 2-morphisms in $\ENH_\kk$; for example,
we provide some instances where such 2-morphisms can be
represented by 
roofs of 2-morphisms in $\DGCAT_\kk$ where the arrow pointing in
the wrong direction is an objectwise homotopy equivalence.
In three appendices we prove some facts we could not locate in
the literature. The content of 
the appendices~\ref{sec:spalt-results-ring} and
\ref{sec:change-universe} was already discussed in
section~\ref{sec:appendices}. 
Some basic model categorical facts are proved in
appendix~\ref{sec:some-results-model}. 

The results of the two appendices~\ref{sec:change-universe} and
\ref{sec:some-results-model} 
are only used in
section~\ref{sec:some-remarks-2-morphisms-ENH}.
This section is not used elsewhere in this article, and the
reader is advised to skip it on a first reading.
This also has the advantage that Grothendieck universes
and related set-theoretical issues can be safely ignored.

\subsection{Acknowledgements}
\label{sec:acknowledgements}

We thank Valery Lunts for many inspiring discussions.
He was hoping very much that a theory as presented in
this work should exist.
We thank Michael Mandell for discussions
and Emily Riehl and Michael Shulman for useful correspondence
concerning model categories. 
We thank Timothy Logvinenko, Hanno Becker, Alexander Efimov, 
James Gillespie, Greg Stevenson, Pierre-Yves \mbox{Gaillard},
Lorenzo Ramero and Amnon Neeman for useful discussions.
Hanno Becker and Jan Weidner shared
an observation
which led to
Lemma~\ref{l:no-additive-injective-repl-for-Ab}.  
Fr\'ed\'eric D\'eglise answered a question concerning
Theorem~\ref{t:E-model-structure}. 
We thank the referee for very detailed comments, in particular for
drawing our attention to set-theoretical problems concerning
functor categories, and for suggesting a more intrinsic 
definition of the 2-multicategory $\ENH_\kk$ of dg enhancements.

\fussnote{Barcelona reference has appeared}


The author was supported by a postdoctoral
fellowship of the DFG,
and by SPP 1388 and SFB/TR 45 of the DFG.

\subsection{Conventions}
\label{sec:conventions}

All rings considered are assumed to be associative and unital. 
The symbol $\R$ always denotes a commutative ring, and 
$\kk$ always denotes a field. 
We often write $\otimes$ instead of $\otimes_\R$ or
$\otimes_\kk$ or $\otimes_{\mathcal{O}_X}$ and hope that its meaning
is clear from the context.

Starting from \ref{sec:notation}
we use the notation
$\RR$ for the category
of complexes of $\R$-modules (in some Grothendieck universe)
and $\ul\RR$ for the dg $\R$-category of such complexes.
Similarly, $\KK$ denotes the category of complexes of
$\kk$-vector spaces and $\ul\KK$ the dg $\kk$-category of
complexes of $\kk$-vector spaces.

Starting from \ref{sec:diff-grad-categ}
we will mostly use the term
$\RR$-category for a category
enriched in $\RR$. This is just another name for dg
$\R$-category. So $\ul\RR$ is an $\RR$-category.
Similarly, we will mostly say $\KK$-category
instead of dg $\kk$-category. So $\ul\KK$ is a $\KK$-category.

\subsubsection{Set-theoretical conventions}
\label{sec:set-theor-conv}

We use (Grothendieck) universes in order to handle set-theoretical
issues \cite{SGA4-1, KS-cat-sh}.
These issues are most relevant in
section~\ref{sec:some-remarks-2-morphisms-ENH}; the reader can
safely ignore them before reading this section.
To be more precise, we work in a
model of Tarski-Grothendieck set theory. 
All 
universes we consider are assumed to contain the set $\bN$ of
natural numbers, the ring $\R$ 
and the field $\kk$ as elements.


Given a universe $\mathscr{U}$, we use the following terminology
(which follows
\cite{gabber-ramero-found-almost-v12} but 
differs from \cite{SGA4-1}): a set (or an algebraic
structure like a module over some ring) is called
\define{$\mathscr{U}$-small} if it is an element of 
$\mathscr{U}$; a category $\mathcal{C}$ 
\define{has $\mathscr{U}$-small $\Hom$-sets} 
if 
$\mathcal{C}(A,B)$ is $\mathscr{U}$-small for all objects $A, B
\in \mathcal{C}$; a category $\mathcal{C}$ is 
\define{$\mathscr{U}$-small} if  
it has $\mathscr{U}$-small $\Hom$-sets and if 
its set $\Obj \mathcal{C}$ of objects is $\mathscr{U}$-small;
a category $\mathcal{C}$ \define{has objects in $\mathscr{U}$} if
$\Obj \mathcal{C}$ is a subset of $\mathscr{U}$.

We use the convention of 
\cite[Rem.~1.1.12.(iii)]{gabber-ramero-found-almost-v12}   
that a map 
$f \colon S \ra S'$ of sets is given by its graph (and not by the
triple consisting of $S$, $S'$ and its graph); this has the small
technical advantage that 
Remark~\ref{rem:functor-categories-size} is true.
Similarly, we do not assume that the $\Hom$-sets of a category
$\mathcal{C}$ 
are disjoint; nevertheless we refer to 
$\Mor(\mathcal{C}):= \{(A,B,f) \mid A,B \in \Obj \mathcal{C}, f
\in \mathcal{C}(A,B)\}$ as the set of morphisms in $\mathcal{C}$ 
(with source and target remembered).

We try to mention size issues whenever relevant.
This increase in correctness hopefully makes up for some 
decrease in readability.


 
\newpage

\thispagestyle{empty}

\begin{landscape}
  \tiny
  \setcounter{tablerow}{0}
  \begin{table}
    \caption{Interpretation of 2-morphisms in $\FML_\kk$ in
      $\ENH_\kk$ and $\TRCAT_\kk$}
    \label{tab:main}
    \begin{tabular}{%
      r@{}|>{$}r<{$}@{\;}>{$}c<{$}@{\;}>{$}l<{$}|%
      >{$}r<{$}@{\;}>{$}c<{$}@{\;}>{$}l<{$} c|%
      >{$}r<{$}@{\;}>{$}c<{$}@{\;}>{$}l<{$}|}
      \cline{2-11}
      & \multicolumn{3}{|>{$}c<{$}|}{\FML_\kk}
      & \multicolumn{3}{|>{$}c<{$}}{\ENH_\kk}
      & {}
      & \multicolumn{3}{|>{$}c<{$}|}{\TRCAT_\kk}
      \\
      \cline{2-11}
      \cline{2-11}
      & \multicolumn{10}{|l|}{\text{Let $(\mathcal{X},
        \mathcal{O})$ be 
      a $\kk$-ringed site.}}
      \\
      \cline{2-11}
      \labelrow{\ref{tab:main}}{eq:tab:left}
      & (\ud{\mathcal{O}} \;\ud\otimes -)
      & \sira 
      & \id
      &(\ul{\mathcal{O}} \;\ul\otimes -)
      & \sira 
      & \id
      & \eqref{eq:left}
      &(\mathcal{O} \otimes^\dL -)
      & \sira 
      & \id
      \\
      \labelrow{\ref{tab:main}}{eq:tab:right}
      & (- \ud\otimes \; \ud{\mathcal{O}})
      & \sira 
      & \id 
      & (- \ul\otimes \; \ul{\mathcal{O}})
      & \sira 
      & \id 
      & \eqref{eq:right}
      & (- \otimes^\dL \; \mathcal{O})
      & \sira 
      & \id 
      \\
      \labelrow{\ref{tab:main}}{eq:tab:ass-and-symm}
      & ((- \ud\otimes -) \ud\otimes -)
      & \sira
      &  (- \ud\otimes (- \ud\otimes -))
      & ((- \ul\otimes -) \ul\otimes -)
      & \sira
      &  (- \ul\otimes (- \ul\otimes -))
      & \eqref{eq:ass-and-symm}
      & ((- \otimes^\dL -) \otimes^\dL -)
      & \sira
      &  (- \otimes^\dL (- \otimes^\dL -))
      \\
      \labelrow{\ref{tab:main}}{eq:tab:ulotimes-swap}
      & (- \ud\otimes ?) 
      & \sira
      & (? \ud\otimes -)
      &  (- \ul\otimes ?) 
      & \sira
      & (? \ul\otimes -)
      & \eqref{eq:ulotimes-swap}
      &  (- \otimes^\dL ?) 
      & \sira
      & (? \otimes^\dL -)
      \\
      \labelrow{\ref{tab:main}}{eq:tab:48}
      & \ud\sheafHom(- \ud\otimes -, -)
      & \sira
      &  \ud\sheafHom(-, \ud\sheafHom(-,-))
      & \ul\sheafHom(- \ul\otimes -, -)
      & \sira
      &  \ul\sheafHom(-, \ul\sheafHom(-,-))
      & \eqref{eq:48}
      & \dR\sheafHom(- \otimes^\dL -, -)
      & \sira
      &  \dR\sheafHom(-, \dR\sheafHom(-,-))
      \\
      \labelrow{\ref{tab:main}}{eq:tab:46}
      & \ud\Hom(-,-) 
      & \sira
      & \ud\Gamma \; \ud\sheafHom(-,-) 
      & \ul\Hom(-,-) 
      & \sira
      & \ul\Gamma \; \ul\sheafHom(-,-) 
      & \eqref{eq:46}
      & \dR\Hom(-,-) 
      & \sira
      & \dR\Gamma \; \dR\sheafHom(-,-) 
      \\
      \cline{2-11}
      & \multicolumn{10}{|l|}{\text{Let
      $(\Sh(\mathcal{Y}), \mathcal{O}_\mathcal{Y}) \xra{\alpha}
      (\Sh(\mathcal{X}), \mathcal{O}_\mathcal{X})$
      be a morphism of $\kk$-ringed topoi.}}
      \\
      \cline{2-11}
      \labelrow{\ref{tab:main}}{eq:tab:aa-id}
      & \ud\alpha^* \ud\alpha_* 
      & \ra 
      & \id  
      & \ul\alpha^* \ul\alpha_* 
      & \ra 
      & \id  
      & \eqref{eq:aa-id}
      & \dL\alpha^* \dR\alpha_* 
      & \ra 
      & \id  
      \\
      \labelrow{\ref{tab:main}}{eq:tab:id-aa}
      & \id 
      & \ra 
      & \ud\alpha_* \ud\alpha^* 
      & \id 
      & \ra 
      & \ul\alpha_* \ul\alpha^* 
      & \eqref{eq:id-aa}
      & \id 
      & \ra 
      & \dR\alpha_* \dL\alpha^* 
      \\
      \labelrow{\ref{tab:main}}{eq:tab:ul-alpha^*-otimes}
      & \ud\alpha^* (- \ud\otimes -)
      & \sira
      &  (\ud\alpha^* -) \ud\otimes (\ud\alpha^* -) 
      & \ul\alpha^* (- \ul\otimes -)
      & \sira
      &  (\ul\alpha^* -) \ul\otimes (\ul\alpha^* -) 
      & \eqref{eq:ul-alpha^*-otimes}
      & \dL\alpha^* (- \otimes^\dL -)
      & \sira
      &  (\dL\alpha^* -) \otimes^\dL (\dL\alpha^* -) 
      \\
      \labelrow{\ref{tab:main}}{eq:tab:pull-push-sheafHom-ENH}
      & \ud\alpha_* \ud\sheafHom(\ud\alpha^* -, -)
      & \sira
      & \ud\sheafHom(-, \ud\alpha_*-)
      & \ul\alpha_* \ul\sheafHom(\ul\alpha^* -, -)
      & \sira
      & \ul\sheafHom(-, \ul\alpha_*-)
      & \eqref{eq:pull-push-sheafHom-ENH}
      & \dR\alpha_* \dR\sheafHom(\dL\alpha^* -, -)
      & \sira
      & \dR\sheafHom(-, \dR\alpha_*-)
      \\
      \cline{2-11}
      & \multicolumn{10}{|l|}{\text{Let $(\Sh(\mathcal{Z}),
      \mathcal{O}_\mathcal{Z}) \xra{\beta} 
      (\Sh(\mathcal{Y}), \mathcal{O}_\mathcal{Y}) \xra{\alpha}
      (\Sh(\mathcal{X}), \mathcal{O}_\mathcal{X})$ be morphisms of
      $\kk$-ringed 
      topoi.}}
      \\
      \cline{2-11}
      \labelrow{\ref{tab:main}}{eq:tab:id_*-ENH}
      & \ud\id_* 
      & \sira 
      & \id 
      & \ul\id_* 
      & \sira 
      & \id
      & \eqref{eq:id_*-ENH}
      & \dR\id_* 
      & \sira 
      & \id 
      \\
      \labelrow{\ref{tab:main}}{eq:tab:alphabeta_*-ENH}
      & \ud{(\alpha \beta)}_* 
      & \sira 
      & \ud\alpha_* \ud\beta_*,
      & \ul{(\alpha \beta)}_* 
      & \sira 
      & \ul\alpha_* \ul\beta_*,
      & \eqref{eq:alphabeta_*-ENH}
      & \dR{(\alpha \beta)}_* 
      & \sira 
      & \dR\alpha_* \dR\beta_*,
      \\
      \labelrow{\ref{tab:main}}{eq:tab:id^*-ENH}
      & \ud\id^*
      & \sila 
      & \id
      & \ul\id^*
      & \sila 
      & \id
      & \eqref{eq:id^*-ENH}
      & \dL\id^*
      & \sila 
      & \id
      \\
      \labelrow{\ref{tab:main}}{eq:tab:alphabeta^*-ENH}
      & \ud{(\alpha\beta)}^*
      & \sila 
      & \ud\beta^*\ud\alpha^*
      & \ul{(\alpha\beta)}^*
      & \sila 
      & \ul\beta^*\ul\alpha^*
      & \eqref{eq:alphabeta^*-ENH}
      & \dL{(\alpha\beta)}^*
      & \sila 
      & \dL\beta^*\dL\alpha^*
      \\
      \cline{2-11}
      & \multicolumn{10}{|l|}{\text{Let $A$ be a $\kk$-algebra and
      $Y \xra{\alpha} X$ a morphism 
      of topological spaces.}}
      \\
      \cline{2-11}
      \labelrow{\ref{tab:main}}{eq:tab:ulalpha*-inv}
      & \ud\alpha^* 
      & \sira 
      & \ud\alpha\inv
      & \ul\alpha^* 
      & \sira 
      & \ul\alpha\inv
      & \eqref{eq:ulalpha*-inv}
      & \dL\alpha^* 
      & \sira 
      & \dL\alpha\inv
      \\
      \labelrow{\ref{tab:main}}{eq:tab:alpha_!-to-alpha_*-not-proper}
      & \text{If $\alpha$ is not proper:}\quad
      \ud\alpha_! 
      & \ra 
      & \ud\alpha_*
      & \ul\alpha_! 
      & \ra 
      & \ul\alpha_*
      & \eqref{eq:alpha_!-to-alpha_*}
      & \dR\alpha_! 
      & \ra 
      & \dR\alpha_*
      \\
      \labelrow{\ref{tab:main}}{eq:tab:alpha_!-to-alpha_*-proper}
      & \text{If $\alpha$ is proper:}\quad
      \ud\alpha_! 
      & \sira 
      & \ud\alpha_*
      & \ul\alpha_! 
      & \sira 
      & \ul\alpha_*
      & \eqref{eq:alpha_!-to-alpha_*}
      & \dR\alpha_! 
      & \sira 
      & \dR\alpha_*
      \\
      \cline{2-11}
      & \multicolumn{10}{|l|}{\text{Let $A$ be a
      $\kk$-algebra and
      $Y \xra{\alpha} X$ a separated, locally proper
      morphism of 
      top.\ spaces with
      $\alpha_! \colon 
      \Mod(Y_A) \ra \Mod(X_A)$ 
      of finite cohomological dimension.}}
      \\
      \cline{2-11}
      \labelrow{\ref{tab:main}}{eq:tab:a_!a^!-id}
      & \ud\alpha_! \ud\alpha^! 
      & \ra 
      & \id
      & \ul\alpha_! \ul\alpha^! 
      & \ra 
      & \id
      & \eqref{eq:a_!a^!-id}
      & (\dR\alpha_!) \alpha^! 
      & \ra 
      & \id
      \\
      \labelrow{\ref{tab:main}}{eq:tab:id-a^!a_!}
      & \id 
      & \ra 
      & \ud\alpha^! \ud\alpha_! 
      & \id 
      & \ra 
      & \ul\alpha^! \ul\alpha_! 
      & \eqref{eq:id-a^!a_!}
      & \id 
      & \ra 
      & \alpha^! \dR\alpha_! 
      \\
      \labelrow{\ref{tab:main}}{eq:tab:projection-fml-ENH}
      & \ud{\alpha}_!(-) \ud\otimes (-)
      & \sira
      & \ud\alpha_!((-) \ud\otimes \;\ud\alpha\inv(-))
      & \ul{\alpha}_!(-) \ul\otimes (-)
      & \sira
      & \ul\alpha_!((-) \ul\otimes \;\ul\alpha\inv(-))
      & \eqref{eq:projection-fml-ENH}
      & \dR{\alpha}_!(-) \otimes^\dL (-)
      & \sira
      & \dR\alpha_!((-) \otimes^\dL \dL\alpha\inv(-))
      \\
      \labelrow{\ref{tab:main}}{eq:tab:!-adjunction-sheafHom-ENH}
      & \ud\alpha_*\ud\sheafHom(-, \ud\alpha^!(-)) 
      & \sira
      & \ud\sheafHom(\ud\alpha_!(-),-) 
      & \ul\alpha_*\ul\sheafHom(-, \ul\alpha^!(-)) 
      & \sira
      & \ul\sheafHom(\ul\alpha_!(-),-) 
      & \eqref{eq:!-adjunction-sheafHom-ENH}
      & \dR\alpha_*\dR\sheafHom(-, \alpha^!(-)) 
      & \sira
      & \dR\sheafHom(\dR\alpha_!(-),-) 
      \\
      \labelrow{\ref{tab:main}}{eq:tab:upper-!-sheafHom-ENH}
      & \ud\sheafHom(\ud\alpha\inv(-), \ud\alpha^!(-)) 
      & \sira
      & \ud\alpha^!\ud\sheafHom(-,-) 
      & \ul\sheafHom(\ul\alpha\inv(-), \ul\alpha^!(-)) 
      & \sira
      & \ul\alpha^!\ul\sheafHom(-,-) 
      & \eqref{eq:upper-!-sheafHom-ENH}
      & \dR\sheafHom(\dL\alpha\inv(-), \alpha^!(-)) 
      & \sira
      & \alpha^!\dR\sheafHom(-,-) 
      \\ 
      \cline{2-11}
      & \multicolumn{10}{|l|}{\text{Let $A$ be a
      $\kk$-algebra,
      $Z \xra{\beta} Y \xra{\alpha} X$ sep., loc.\ proper
      morphisms of 
      top.\ spaces with
      $\alpha_! \colon 
      \Mod(Y_A) \ra \Mod(X_A)$, $\beta_! \colon 
      \Mod(Z_A) \ra \Mod(Y_A)$ 
      of finite cohom.\ dimension.}}
      \\
      \cline{2-11}
      \labelrow{\ref{tab:main}}{eq:tab:id_!-ENH}
      & \ud\id_! 
      & \sila
      & \id
      & \ul\id_! 
      & \sila
      & \id
      & \eqref{eq:id_!-ENH}
      & \dR\id_! 
      & \sila
      & \id
      \\
      \labelrow{\ref{tab:main}}{eq:tab:alphabeta_!-ENH} 
      & \ud{(\alpha \beta)}_! 
      & \sila 
      & \ud\alpha_! \ud\beta_!
      & \ul{(\alpha \beta)}_! 
      & \sila 
      & \ul\alpha_! \ul\beta_!
      & \eqref{eq:alphabeta_!-ENH}
      & \dR{(\alpha \beta)}_! 
      & \sila 
      & \dR\alpha_! \dR\beta_!
      \\
      \labelrow{\ref{tab:main}}{eq:tab:id^!-ENH}
      & \ud\id^!
      & \sira 
      & \id
      & \ul\id^!
      & \sira 
      & \id
      & \eqref{eq:id^!-ENH}
      & \id^!
      & \sira 
      & \id
      \\
      \labelrow{\ref{tab:main}}{eq:tab:alphabeta^!-ENH} 
      & \ud{(\alpha\beta)}^!
      & \sira 
      & \ud\beta^!\ud\alpha^!
      & \ul{(\alpha\beta)}^!
      & \sira 
      & \ul\beta^!\ul\alpha^!
      & \eqref{eq:alphabeta^!-ENH}
      & (\alpha\beta)^!
      & \sira 
      & \beta^!\alpha^!
      \\ 
      \cline{2-11}
      & \multicolumn{10}{|l|}{\text{Let  $A$ be a
      $\kk$-algebra and
      $\xymatrix@=1pc{
      {Y'} \ar[r]^-{\beta'} \ar[d]^-{\alpha'} 
      & {Y} \ar[d]^-{\alpha} \\
      {X'} \ar[r]^-{\beta} 
      & {X}}$
        a cartesian diagram of top.\ spaces with $\alpha$
        separated, loc. proper, and $\alpha_! \colon
        \Mod(Y_A) \ra \Mod(X_A)$ of finite cohom.\ dimension.}}
      \\
      \cline{2-11}
      \labelrow{\ref{tab:main}}{eq:tab:proper-base-change-ENH} 
      & \ud{\beta}\inv \ud\alpha_! 
      & \sira 
      & \ud\alpha'_! \ud\beta'^{-1}
      & \ul{\beta}\inv \ul\alpha_! 
      & \sira 
      & \ul\alpha'_! \ul\beta'^{-1}
      & \eqref{eq:proper-base-change-ENH}
      & \dL{\beta}\inv \dR\alpha_! 
      & \sira 
      & \dR\alpha'_! \dL\beta'^{-1}
      \\
      \cline{2-11}
    \end{tabular}
  \end{table}
\end{landscape}

\newpage

\thispagestyle{empty}
\begin{landscape}
  \tiny{
    \setcounter{tablerow}{0}
    \begin{table}
      \caption{Some subsequently constructed 2-(iso)morphisms and
        their interpretation}
      \label{tab:subsequent}
      \begin{tabular}{%
        r@{}|>{$}r<{$}@{\;}>{$}c<{$}@{\;}>{$}l<{$}|%
        >{$}r<{$}@{\;}>{$}c<{$}@{\;}>{$}l<{$} c|%
        >{$}r<{$}@{\;}>{$}c<{$}@{\;}>{$}l<{$}|}
        \cline{2-11}
        & \multicolumn{3}{|>{$}c<{$}|}{\FML_\kk}
        & \multicolumn{3}{|>{$}c<{$}}{\ENH_\kk}
        & {}
        & \multicolumn{3}{|>{$}c<{$}|}{\TRCAT_\kk}
        \\
        \cline{2-11}
        \cline{2-11}
        & \multicolumn{10}{|l|}{\text{Let $(\mathcal{X},
        \mathcal{O})$ be 
        a $\kk$-ringed site.}}
        \\
        \cline{2-11}
        \labelrow{\ref{tab:subsequent}}{eq:tab:28}
        & \ud\Hom(- \ud\otimes -, -)
        & \sira
        & \ud\Hom(-, \ud\sheafHom(-,-)) 
        & \ul\Hom(- \ul\otimes -, -)
        & \sira
        & \ul\Hom(-, \ul\sheafHom(-,-)) 
        & \eqref{eq:28}
        & \dR\Hom(- \otimes^\dL -, -)
        & \sira
        & \dR\Hom(-, \dR\sheafHom(-,-))
        \\ 
        \cline{2-11}
        & \multicolumn{10}{|l|}{\text{Let
        $(\Sh(\mathcal{Y}), \mathcal{O}_\mathcal{Y}) \xra{\alpha}
        (\Sh(\mathcal{X}), \mathcal{O}_\mathcal{X})$
        be a morphism of $\kk$-ringed topoi.}}
        \\
        \cline{2-11}
        \labelrow{\ref{tab:subsequent}}{eq:tab:enhanced-pull-push-ulHom}
        & \ud\Hom(\ud\alpha^* -, -)
        & \sira
        &  \ud\Hom(-, \ul\alpha_*-)
        & \ul\Hom(\ul\alpha^* -, -)
        & \sira
        &  \ul\Hom(-, \ul\alpha_*-)
        & \eqref{eq:enhanced-pull-push-ulHom}
        & \dR\Hom(\dL\alpha^* -, -)
        & \sira
        & \dR\Hom(-, \dR\alpha_*-)
        \\
        \labelrow{\ref{tab:subsequent}}{eq:tab:push-sheafHom-ENH}
        & \ud\alpha_* \ud\sheafHom(-, -)
        & \ra
        & \ud\sheafHom(\ud\alpha_*(-), \ud\alpha_*(-))
        & \ul\alpha_* \ul\sheafHom(-, -)
        & \ra
        & \ul\sheafHom(\ul\alpha_*(-), \ul\alpha_*(-))
        & \eqref{eq:push-sheafHom-ENH}
        & \dR\alpha_* \dR\sheafHom(-, -)
        & \ra
        & \dR\sheafHom(\dR\alpha_*(-), \dR\alpha_*(-))
        \\
        \cline{2-11}
        & \multicolumn{10}{|l|}{\text{Let $A$ be a $\kk$-algebra and
        $Y \xra{\alpha} X$ a morphism 
        of topological spaces.}}
        \\
        \cline{2-11}
        \labelrow{\ref{tab:subsequent}}{eq:tab:ainva-id}
        & \ud\alpha\inv \ud\alpha_* 
        & \ra 
        & \id
        & \ul\alpha\inv \ul\alpha_* 
        & \ra 
        & \id
        & \eqref{eq:ainva-id}
        & \dL\alpha\inv \dR\alpha_* 
        & \ra 
        & \id
        \\
        \labelrow{\ref{tab:subsequent}}{eq:tab:id-aainv}
        & \id 
        & \ra 
        & \ud\alpha_* \ud\alpha\inv
        & \id 
        & \ra 
        & \ul\alpha_* \ul\alpha\inv
        & \eqref{eq:id-aainv}
        & \id 
        & \ra 
        & \dR\alpha_* \dL\alpha\inv
        \\
        \cline{2-11}
        & \multicolumn{10}{|l|}{\text{Let $A$ be a
        $\kk$-algebra and
        $Y \xra{\alpha} X$ a separated, locally proper
        morphism of 
        top.\ spaces with
        $\alpha_! \colon 
        \Mod(Y_A) \ra \Mod(X_A)$ 
        of finite cohomological dimension.}}
        \\
        \cline{2-11}
        \labelrow{\ref{tab:subsequent}}{eq:tab:!-adjunction-Hom-ENH}
        & \ud\Hom(-, \ud\alpha^!-)
        & \sira
        & \ud\Hom(\ud\alpha_!-,-) 
        & \ul\Hom(-, \ul\alpha^!-)
        & \sira
        & \ul\Hom(\ul\alpha_!-,-) 
        & \eqref{eq:!-adjunction-Hom-ENH}
        & \dR\Hom(-, \alpha^!-)
        & \sira
        & \dR\Hom(\dR\alpha_!-,-) 
        \\
        \cline{2-11}
        & \multicolumn{10}{|l|}{\text{Let  $A$ be a
        $\kk$-algebra and
        $\xymatrix@=1pc{
        {Y'} \ar[r]^-{\beta'} \ar[d]^-{\alpha'} 
        & {Y} \ar[d]^-{\alpha} \\
        {X'} \ar[r]^-{\beta} 
        & {X}}$
          a cartesian diagram of top.\ spaces with $\alpha$
          separated, loc. proper, and $\alpha_! \colon
          \Mod(Y_A) \ra \Mod(X_A)$ of finite cohom.\ dimension.}}
        \\
        \cline{2-11}
        \labelrow{\ref{tab:subsequent}}{eq:tab:alpha_!beta'_*-beta_*alpha'_!-ENH}
        & \ud\alpha_!\ud\beta'_* 
        & \ra
        & \ud\beta_*\ud\alpha'_!
        & \ul\alpha_!\ul\beta'_* 
        & \ra
        & \ul\beta_*\ul\alpha'_!
        & \eqref{eq:alpha_!beta'_*-beta_*alpha'_!-ENH}
        & \dR\alpha_!\dR\beta'_* 
        & \ra
        & \dR\beta_*\dR\alpha'_!
        \\
        \labelrow{\ref{tab:subsequent}}{eq:tab:proper-base-change-upper-shriek-ENH}
        & \ud{\beta}'_*\ud\alpha'^! 
        & \ra
        & \ud\alpha^!\ud\beta_* 
        & \ul{\beta}'_*\ul\alpha'^! 
        & \sira
        & \ul\alpha^!\ul\beta_* 
        & \eqref{eq:proper-base-change-upper-shriek-ENH}
        & (\dR{\beta}'_*)\alpha'^! 
        & \sira
        & \alpha^!\dR\beta_* 
        \\
        \cline{2-11}
      \end{tabular}
    \end{table}
  }
\end{landscape}

\subsection{Some commutative diagrams in
  \texorpdfstring{$\ENH_\kk$}{ENHk}} 
\label{sec:some-comm-diagr}

\subsubsection{}
As mentioned before, some of the relations between the six
functors are encoded by commutative diagrams. For example, 
given a morphism
$(\Sh(\mathcal{Y}), \mathcal{O}_\mathcal{Y}) \xra{\alpha}
(\Sh(\mathcal{X}), \mathcal{O}_\mathcal{X})$
of $\kk$-ringed topoi,
the
fact that  
$(\dL \alpha^*, \dR \alpha_*)$ is a pair of adjoint functors is
more precisely encoded by a quadruple
$(\dL \alpha^*, \dR \alpha_*, \id \xra{\eta} \dR\alpha_* \dL\alpha^*,
\dL\alpha^* \dR\alpha_* \xra{\theta} \id)$ where $\eta$ and
$\theta$ are unit and counit of the adjunction such that the two
diagrams 
\begin{equation}
  \label{eq:intro:alpha*-triangles-derived}
  \xymatrix@C-2mm{
    {\dR\alpha_*}
    \ar@/^4ex/[rrrr]^-{\id}
    \ar[rr]_-{\eta\dR\alpha_*}
    &&
    {\dR\alpha_*\dL\alpha^*\dR\alpha_*}
    \ar[rr]_-{\dR\alpha_*\theta}
    &&
    {\dR\alpha_*,}
  }
  \qquad
  \xymatrix@C-2mm{
    {\dL\alpha^*}
    \ar@/^4ex/[rrrr]^-{\id}
    \ar[rr]_-{\dL\alpha^*\eta}
    &&
    {\dL\alpha^*\dR\alpha_*\dL\alpha^*}
    \ar[rr]_-{\theta\dL\alpha^*}
    &&
    {\dL\alpha^*}
  }
\end{equation}
are commutative (these are the triangle identities). 
The unit and counit 2-morphisms in $\TRCAT_\kk$ 
appear as the rightmost entries
of rows \eqref{eq:tab:id-aa} and
\eqref{eq:tab:aa-id} 
in table~\ref{tab:main}.
They are lifted to the corresponding middle entries 
$\id \xra{\eqref{eq:tab:id-aa}} \ul\alpha_* \ul\alpha^*$ and
$\ul\alpha^* \ul\alpha_* \xra{\eqref{eq:tab:aa-id}} \id$ of
these rows 
which are 2-morphisms in
$\ENH_\kk$. Therefore it makes sense to ask whether the datum $(\ul\alpha^*,
\ul\alpha_*, \eqref{eq:tab:id-aa},
\eqref{eq:tab:aa-id})$ defines an adjunction in $\ENH_\kk$. This
is indeed the case,
the
two diagrams
\begin{equation}
  \label{eq:intro:alpha*-triangles}
  \xymatrix@C-2mm{
    {\ul\alpha_*}
    \ar@/^4ex/[rrrr]^-{\id}
    \ar[rr]_-{\eqref{eq:tab:id-aa}\ul\alpha_*}
    &&
    {\ul\alpha_*\ul\alpha^*\ul\alpha_*}
    \ar[rr]_-{\ul\alpha_*\eqref{eq:tab:aa-id}}
    &&
    {\ul\alpha_*,}
  }
  \qquad
  \xymatrix@C-2mm{
    {\ul\alpha^*}
    \ar@/^4ex/[rrrr]^-{\id}
    \ar[rr]_-{\ul\alpha^*\eqref{eq:tab:id-aa}}
    &&
    {\ul\alpha^*\ul\alpha_*\ul\alpha^*}
    \ar[rr]_-{\eqref{eq:tab:aa-id}\ul\alpha^*}
    &&
    {\ul\alpha^*}
  }
\end{equation}
in $\ENH_\kk$ are commutative (see
Proposition~\ref{p:adjunction-*-pull-push-in-ENH}).

\subsubsection{}
In the rest of this section~\ref{sec:some-comm-diagr} 
we give further similar examples of commutative diagrams
in $\ENH_\kk$ but do not display the corresponding commutative
diagrams in $\TRCAT_\kk$. As above, when referring to a row of
table~\ref{tab:main}, we often just mean the entry in the obvious
column. 

\subsubsection{}
If $(\mathcal{X},
\mathcal{O})$ is 
a $\kk$-ringed site, 
$(\ul\II(\mathcal{X}), \ul\otimes, \ul{\mathcal{O}},
\eqref{eq:tab:ass-and-symm}, \eqref{eq:tab:left},
\eqref{eq:tab:right}, 
\eqref{eq:tab:ulotimes-swap})$ 
is a symmetric monoidal object in 
$\ENH_\kk$, see Lemma~\ref{l:monoidal-enhancement}. This means
that the following diagrams in $\ENH_\kk$ commute.
\begin{equation}
  \label{eq:intro:ulotimes-ass}
  \xymatrix{
    {(((- \ul\otimes -) \ul\otimes -) \ul\otimes -)}
    \ar[rr]^-{\eqref{eq:tab:ass-and-symm} (\ul\otimes, \id,\id)}_-{\sim}
    \ar[d]_-{\eqref{eq:tab:ass-and-symm} \ul\otimes \id}^-{\sim}
    &&
    {(- \ul\otimes -) \ul\otimes (- \ul\otimes -)}
    \ar[rr]^-{\eqref{eq:tab:ass-and-symm}(\id,\id, \ul\otimes)}_-{\sim}
    &&
    {(- \ul\otimes (- \ul\otimes (- \ul\otimes -)))}
    \\
    {((- \ul\otimes (- \ul\otimes -)) \ul\otimes -)}
    \ar[rrrr]^-{\eqref{eq:tab:ass-and-symm} (\id, \ul\otimes, \id)}_-{\sim}
    &&&&
    {(- \ul\otimes ((- \ul\otimes -) \ul\otimes -))}
    \ar[u]_-{\id \ul\otimes \eqref{eq:tab:ass-and-symm}}^-{\sim}
  }
\end{equation}
\begin{equation}
  \label{eq:intro:ulotimes-unital}
  \xymatrix{
    {((- \ul\otimes \; \ul{\mathcal{O}}) \ul\otimes -)}
    \ar@/^4ex/[rrrr]^-{\eqref{eq:tab:ass-and-symm}(\id, \ul{\mathcal{O}},\id)}_-{\sim}
    \ar[rr]_-{\eqref{eq:tab:right} \otimes \id}^-{\sim}
    &&
    {(- \ul\otimes -)}
    &&
    {(- \ul\otimes (\ul{\mathcal{O}} \; \ul\otimes -))}
    \ar[ll]^-{\id \otimes \eqref{eq:tab:left}}_-{\sim}
  }
\end{equation}
\begin{equation}
  \label{eq:intro:ulotimes-symm}
  \quad
  \xymatrix{
    {(- \ul\otimes ?)}
    \ar@/^4ex/[rrrr]^-{\id}
    \ar[rr]_-{\eqref{eq:tab:ulotimes-swap}}^-{\sim}
    &&
    {(? \ul\otimes -)}
    \ar[rr]_-{\eqref{eq:tab:ulotimes-swap}}^-{\sim}
    &&
    {(- \ul\otimes ?)}
  }
\end{equation}

\subsubsection{}
If
$(\Sh(\mathcal{Z}), \mathcal{O}_\mathcal{Z}) \xra{\beta}
(\Sh(\mathcal{Y}), \mathcal{O}_\mathcal{Y}) \xra{\alpha}
(\Sh(\mathcal{X}), \mathcal{O}_\mathcal{X})$
are morphisms of $\kk$-ringed topoi, the 2-isomorphisms
$\ul{(\alpha \beta)}_* \xsira{\eqref{eq:tab:alphabeta_*-ENH}}
\ul\alpha_* \ul\beta_*$
and
$\ul\beta^*\ul\alpha^* \xsira{\eqref{eq:tab:alphabeta^*-ENH}}
\ul{(\alpha\beta)}^*$ 
are 
conjugate (with respect to the adjunctions
$(\ul{(\alpha\beta)}^*, \ul{(\alpha\beta)}_*)$ and 
$(\ul\beta^*\ul\alpha^*, \ul\alpha_*\ul\beta_*)$), see
Proposition~\ref{p:pullback-comp-adjoint-pushforward-comp}.
This means
that the following diagram in $\ENH_\kk$ commutes.
\begin{equation}
  \label{eq:intro:pullback-comp-adjoint-pushforward-comp}
  \xymatrix{
    {\id} 
    \ar[r]^-{\eqref{eq:tab:id-aa}}
    \ar[d]_-{\eqref{eq:tab:id-aa}}
    &
    {\ul\alpha_*\ul\alpha^*}
    \ar[rr]^-{\ul\alpha_*\eqref{eq:tab:id-aa}\ul\alpha^*}
    &&
    {\ul\alpha_*\ul\beta_*\ul\beta^*\ul\alpha^*}
    \ar[d]^-{\ul\alpha_*\ul\beta_*\eqref{eq:tab:alphabeta^*-ENH}}_-{\sim} 
    \\
    {\ul{(\alpha\beta)}_*\ul{(\alpha\beta)}^*}
    \ar[rrr]^-{\eqref{eq:tab:alphabeta_*-ENH}\ul{(\alpha\beta)}^*}_-{\sim}  
    &&&
    {\ul\alpha_*\ul\beta_*\ul{(\alpha\beta)}^*}
  }
\end{equation}

\subsubsection{}
If $(\mathcal{X},
\mathcal{O})$ is 
a $\kk$-ringed site, 
the 2-isomorphisms
$\ul\id_* \xsira{\eqref{eq:tab:id_*-ENH}} \id$ and $\id
\xsira{\eqref{eq:tab:id^*-ENH}} \ul\id^*$ 
are conjugate
(with respect to the adjunctions $(\ul\id^*,
\ul\id_*)$ and $(\id,\id)$), see
Proposition~\ref{p:pullback-comp-adjoint-pushforward-comp}.
This means
that the following diagram in $\ENH_\kk$ commutes.
\begin{equation}
  \label{eq:intro:pullback-id-adjoint-pushforward-id}
  \xymatrix{
    {\id} 
    \ar[rr]^-{\id}
    \ar[d]_-{\eqref{eq:tab:id-aa}}
    &&
    {\id}
    \ar[d]^-{\eqref{eq:tab:id^*-ENH}}_-{\sim} 
    \\
    {\ul\id_*\ul\id^*}
    \ar[rr]^-{\eqref{eq:tab:id_*-ENH}\ul\id^*}_-{\sim}  
    &&
    {\ul\id^*}
  }
\end{equation}

\subsubsection{}
If $A$ is a commutative
$\kk$-algebra and
$\alpha \colon Y \ra X$ is a separated, locally proper
morphism of 
topological spaces with
$\alpha_! \colon 
\Mod(Y_A) \ra \Mod(X_A)$ 
of finite cohomological dimension,
then the datum $(\ul\alpha_!, \ul\alpha^!,
\eqref{eq:tab:id-a^!a_!},
\eqref{eq:tab:a_!a^!-id})$
is an adjunction in $\ENH_\kk$, see
Proposition~\ref{p:adjunction-!-push-pull-in-ENH}.
This means
that the following diagrams in $\ENH_\kk$ commute.
\begin{equation}
  \label{eq:intro:alpha!-triangles}
  \xymatrix@C-2mm{
    {\ul\alpha^!}
    \ar@/^4ex/[rrrr]^-{\id}
    \ar[rr]_-{\eqref{eq:tab:id-a^!a_!}\ul\alpha^!}
    &&
    {\ul\alpha^!\ul\alpha_!\ul\alpha^!}
    \ar[rr]_-{\ul\alpha^!\eqref{eq:tab:a_!a^!-id}}
    &&
    {\ul\alpha^!,}
  }
  \qquad
  \xymatrix@C-2mm{
    {\ul\alpha_!}
    \ar@/^4ex/[rrrr]^-{\id}
    \ar[rr]_-{\ul\alpha_!\eqref{eq:tab:id-a^!a_!}}
    &&
    {\ul\alpha_!\ul\alpha^!\ul\alpha_!}
    \ar[rr]_-{\eqref{eq:tab:a_!a^!-id}\ul\alpha_!}
    &&
    {\ul\alpha_!}
  }
\end{equation}

\section{Description of the main results}
\label{sec:description-the-main}

Our aim in \ref{sec:users-guide-dg} is to enable the reader to use the
dg $\kk$-enhanced six functor formalism without having to digest
all technical details of this long article.
In \ref{sec:lifting-actions-dg} we explain how certain
isomorphisms lift from the triangulated level to the dg level.
The ingredients from enriched model category theory we use are
then explained in \ref{sec:model-categories-dg}. Remember that
$\kk$ always denotes a field.

\subsection{User's guide to the dg \texorpdfstring{$\kk$}{k}-enhanced six functor
  formalism}
\label{sec:users-guide-dg}

\subsubsection{Dg \texorpdfstring{$\kk$}{k}-enhancements
  considered -- lifts of derived categories of sheaves}
\label{sec:enhanc-cons-1}

Let $X$ be a $\kk$-ringed space.  Let $\ul\C(X)$ be the dg
$\kk$-category of complexes of sheaves on $X$ and $\ul\II(X)$ its
full dg $\kk$-subcategory of h-injective complexes of injective
sheaves. Then $\ul\II(X)$ is a strongly pretriangulated dg
$\kk$-category. Its homotopy category
$[\ul\II(X)]$ is a triangulated $\kk$-category and the obvious
functor is an equivalence $[\ul\II(X)] \sira \D(X)$ of triangulated
$\kk$-categories. This means that $\ul\II(X)$ is a dg
$\kk$-enhancement of $\D(X)$.  In
\ref{sec:dg-texorpdfstr-enric-1} and \ref{sec:definitions} we
construct an equivalence
\begin{equation}
  \label{eq:intro:DX-[ulIIX]}
  \ol{[\ii]} \colon \D(X) \sira [\ul\II(X)]  
\end{equation}
of triangulated $\kk$-categories in the other direction.
Fixing this equivalence 
allows a very precise formulation of our results. 

\subsubsection{Lifts of the six functors}
\label{sec:lifts-six-functors}

We provide dg
$\kk$-functors
\begin{align}
  \label{eq:intro:def-lifts-six-functors}
  (-)\ul\otimes(-) \colon \ul\II(X) \otimes \ul\II(X) 
  & \ra \ul\II(X),
  & \ul\sheafHom(-,-) \colon \ul\II(X)^\opp \otimes \ul\II(X) 
  & \ra \ul\II(X),
  \\
  \notag
  \ul\alpha^* \colon \ul\II(X) 
  & \ra \ul\II(Y),
  & \ul\alpha_* \colon \ul\II(Y) 
  & \ra \ul\II(X),
  \\
  \notag
  \ul\alpha_! \colon \ul\II(Y) 
  & \ra \ul\II(X),
  & \ul\alpha^! \colon \ul\II(X) 
  & \ra \ul\II(Y)
\end{align}
where $\alpha \colon Y \ra X$ is a morphism of $\kk$-ringed
spaces
which is assumed to satisfy some
additional properties whenever !-functors are involved, cf.\
\ref{sec:some-explanations-1}. 
The construction of these dg $\kk$-functors is 
explained in \ref{sec:dg-texorpdfstr-enric-1} and
\ref{sec:definition-lifts-six}. 
They
lift the triangulated $\kk$-functors
$\otimes^\dL$, $\dR\sheafHom$, $\dL \alpha^*$, $\dR \alpha_*$,
$\dR \alpha_!$, $\alpha^!$ in the following sense.
The functors $\ul\alpha^*$ and
$\ul\alpha_*$ induce, by passing to homotopy categories, functors  
$[\ul\alpha^*] \colon [\ul\II(X)] \ra [\ul\II(Y)]$ and
$[\ul\alpha_*] \colon [\ul\II(Y)] \ra [\ul\II(X)]$ 
of triangulated $\kk$-categories, and
the diagrams
\begin{equation}
  \label{eq:intro:omega-alpha^*_*}
  \xymatrix{
    {\D(X)} \ar[rr]^-{\dL \alpha^*}
    \ar[d]^-{\sim}_-{\ol{[\ii]}}
    &&
    {\D(Y)}
    \ar[d]^-{\sim}_-{\ol{[\ii]}}
    \\
    [\ul\II(X)] 
    \ar[rr]_-{[\ul\alpha^*]} 
    &&
    {[\ul\II(Y)],} 
  }
  \quad \quad \quad
  \xymatrix{
    {\D(Y)} 
    \ar[rr]^-{\dR \alpha_*} 
    \ar[d]^-{\sim}_-{\ol{[\ii]}}
    &&
    {\D(X)} 
    \ar[d]^-{\sim}_-{\ol{[\ii]}}
    \\
    {[\ul\II(Y)]}
    \ar[rr]_-{[\ul\alpha_*]}
    &&
    {[\ul\II(X)]}
  }
\end{equation}
in $\TRCAT_\kk$ commute up to canonical
2-isomorphisms; there are similar canonically 2-commutative
diagrams relating 
the dg $\kk$-functors $\ul\otimes$, $\ul\sheafHom$,
$\ul\alpha_!$, $\ul\alpha^!$ to the 
functors
$\otimes^\dL$, $\dR\sheafHom$, $\dR \alpha_!$, $\alpha^!$,
see \ref{sec:lifts-deriv-funct}, \ref{sec:lifts-deriv-funct-2}.

\subsubsection{The 2-multicategory
  \texorpdfstring{$\ENH_\kk$}{ENHk}
  of enhancements}
\label{sec:2-mult}

In order to lift the relations between the six functors
we introduce the $\kk$-linear 2-multicategory $\ENH_\kk$ of enhancements
(see \ref{sec:2-mult-enhanc}). Its objects are additive
pretriangulated dg $\kk$-categories. The
objects relevant here are the dg $\kk$-categories
$\ul\II(X)$ 
introduced above. To define the morphism
$\kk$-categories of $\ENH_\kk$ we need a definition.

Given objects 
$\ul\II(X)$ and $\ul\II(Y)$
let $\tau \colon F' \ra 
F$ be a 
morphism in 
$\DGCAT_\kk(\ul\II(X); \ul\II(Y))$ as illustrated by the picture
\begin{equation}
  \xymatrix{
    {\ul\II(Y)}
    && 
    {\ul\II(X).}
    \ar@/^1pc/[ll]^-{F'}_-{}="f"
    \ar@/_1pc/[ll]_-{F}^-{}="g"
    \ar@{=>}_-{\tau}"f";"g"
  }
\end{equation}
We say that $\tau$ is an objectwise homotopy equivalence if 
$\tau_I \colon F'(I) \ra F(I)$ is a homotopy equivalence (or,
equivalently, a 
quasi-isomorphism) for all $I \in \ul\II(X)$. An equivalent
condition is that the induced 
morphism $[\tau] \colon [F'] \ra [F]$ in 
$\TRCAT_\kk([\ul\II(X)]; [\ul\II(Y)])$ 
is an isomorphism. 

We define the morphism $\kk$-category $\ENH_\kk(\ul\II(X);
\ul\II(Y))$ as the target of 
the additive $\kk$-lo\-cal\-iza\-tion 
(see Definition~\ref{d:R-localization})
of $\DGCAT_\kk(\ul\II(X); \ul\II(Y))$ with
respect to the set
of objectwise homotopy equivalences; a similar
definition is used when several source objects are involved.
Note that the set of 1-morphisms in $\ENH_\kk$ with fixed sources and
target coincides with the corresponding set of 1-morphisms in
$\DGCAT_\kk$. The six dg $\kk$-functors 
$\ul\otimes$, $\ul\sheafHom$, $\ul\alpha^*$, $\ul\alpha_*$,
$\ul\alpha_!$, $\ul\alpha^!$
are 
1-morphisms in $\ENH_\kk$.
 
Clearly, taking homotopy
categories defines a functor
$[-] \colon \ENH_\kk \ra \TRCAT_\kk$. This is the functor   
\eqref{eq:intro:[-]-ENH-TRCAT} mentioned above.


\subsubsection{Lifts of relations}
\label{sec:lifts-relations}

Explicit zig-zags of dg $\kk$-natural transformations define a
2-morphism $\id \ra \ul\alpha_* \ul\alpha^*$ and a 
2-isomorphism
$\ul\alpha_*\ul\sheafHom(-,\ul\alpha^!(-)) \sira
\ul\sheafHom(\ul\alpha_!(-),-)$
in $\ENH_\kk$ whose images under \eqref{eq:intro:[-]-ENH-TRCAT}
coincide modulo the equivalences \eqref{eq:intro:DX-[ulIIX]} and
the canonical 2-isomorphisms 
(cf.\ \eqref{eq:intro:omega-alpha^*_*}) with the 2-morphism
$\id \ra \dR\alpha_*\dL\alpha^*$ and the Verdier duality
2-isomorphism 
$\dR\alpha_*\dR\sheafHom(-, \alpha^!(-)) \sira
\dR\sheafHom(\dR\alpha_!(-),-)$, respectively.
Similarly, we define all the 2-(iso)morphisms in the middle
column of table~\ref{tab:main} on page \pageref{tab:main} and
show that they lift the corresponding 2-(iso)morphisms in the
right column.
Combining these
2-(iso)morphisms
provides many other lifts. We have
assembled some of them in table~\ref{tab:subsequent} on page
\pageref{tab:subsequent}. The equation
numbers in the middle columns of these tables refer to the
corresponding 2-(iso)morphisms in the main body of this article.
See \ref{sec:some-explanations-1} for more explanations
concerning these tables.

Moreover, we establish that $(\ul\alpha^*, \ul\alpha_*)$ together
with the 2-morphisms 
$\id \ra \ul\alpha_*\ul\alpha^*$ 
and 
$\ul\alpha^*\ul\alpha_* \ra \id$ 
(in rows 
\eqref{eq:tab:id-aa} and \eqref{eq:tab:aa-id}) form an adjunction
in $\ENH_\kk$. We also show that 
$(\ul\II(X), \ul\otimes)$ together with 
$\ul{\mathcal{O}}$ and the middle entries of rows
\eqref{eq:tab:ass-and-symm}, \eqref{eq:tab:left},
\eqref{eq:tab:right}, 
\eqref{eq:tab:ulotimes-swap}
is a symmetric monoidal object in $\ENH_\kk$. 
These two facts are encoded in commutativity of the
five diagrams in \eqref{eq:intro:alpha*-triangles},
\eqref{eq:intro:ulotimes-ass},
\eqref{eq:intro:ulotimes-unital},
\eqref{eq:intro:ulotimes-symm}
in 
\ref{sec:some-comm-diagr}.
We also prove commutativity of the other
diagrams in
\eqref{eq:intro:pullback-comp-adjoint-pushforward-comp},
\eqref{eq:intro:pullback-id-adjoint-pushforward-id},
\eqref{eq:intro:alpha!-triangles}
there. They lift standard relations between some of the six
functors. There are many other commutative diagrams relating the
six functors. We strongly believe that their counterparts in
$\ENH_\kk$ can be shown to commute using our techniques.
We have restricted ourselves to the
list in \ref{sec:some-comm-diagr} for reasons of space and time.

\subsubsection{2-Morphisms in \texorpdfstring{$\ENH_\kk$}{ENHk}}
\label{sec:2-morph-ENH}

All 2-morphisms in $\ENH_\kk$ can be represented 
by zig-zags of 2-morphisms in $\DGCAT_\kk$
where the arrows pointing in the
wrong direction 
are objectwise homotopy equivalences.
As already mentioned, all 2-morphisms in the middle columns of
tables~\ref{tab:main} and \ref{tab:subsequent} 
are explicitly
defined in this way; the 2-morphisms marked as 2-isomorphisms
there are
defined by zig-zags of objectwise homotopy
equivalences, with one exception, namely the 2-isomorphism
$\ul{\beta}'_*\ul\alpha'^! \sira \ul\alpha^!\ul\beta_*$
in \eqref{eq:tab:proper-base-change-upper-shriek-ENH}.

In general, we do not know whether any invertible 2-morphism in
$\ENH_\kk$ can be represented by a zig-zag (or even a roof) 
of objectwise homotopy equivalences. 
We also would like to know whether the functor
$[-] \colon \ENH_\kk \ra \TRCAT_\kk$ reflects
2-isomorphisms since an affirmative answer to this question 
would give a
useful method to show that certain 2-morphisms in $\ENH_\kk$ are
invertible.
These problems seem to be of set-theoretical
origin and have partial solutions explained in
\ref{sec:some-remarks-2-morphisms-ENH}; a more precise discussion
of properties we would like to have is also given there.

Therefore, it might be useful to remember the definition
of certain 2-(iso)morphisms in $\ENH_\kk$ in applications.
Nevertheless, the definition of the 2-category $\ENH_\kk$ seems
very natural to us; for example, it would be very
annoying to formulate the neat statement that
$(\ul\alpha^*, \ul\alpha_*)$ is an adjunction in $\ENH_\kk$
in terms of zig-zags: this would amount to the big commutative
diagram~\eqref{eq:eta-alpha_*-alpha_*-theta-big-diagram} (plus
the argument that blue and green arrow there coincide plus the same
amount of work for the other triangle identity). 

\subsubsection{\texorpdfstring{$\kk$}{k}-ringed topoi}
\label{sec:texorpdfstr-ring-top}

More generally, all our constructions involving the four functors
$\otimes^\dL$, 
$\dR\sheafHom$, $\dL \alpha^*$, $\dR \alpha_*$ work for
$\kk$-ringed sites $(X, \mathcal{O}_X)$ 
and for morphisms
$\alpha \colon (\Sh(Y), \mathcal{O}_Y) \ra
(\Sh(X), \mathcal{O}_X)$
of $\kk$-ringed topoi. In the rest of this section we use
this language and extend our notation in the obvious
way from $\kk$-ringed spaces to $\kk$-ringed sites.

\subsubsection{Some explanations concerning the
  tables
  and \texorpdfstring{$!$}{!}-functors}
\label{sec:some-explanations-1}

The appearance of $\kk$-ringed sites and morphisms of
$\kk$-ringed topoi in tables~\ref{tab:main} 
and \ref{tab:subsequent} is justified by
\ref{sec:texorpdfstr-ring-top}. All $\kk$-algebras in these
tables are assumed to be commutative.

The symbol $\ul{\mathcal{O}}$ 
is an h-injective complex of injective sheaves
lifting the structure
sheaf $\mathcal{O}$ (see \ref{sec:definition-lifts-six} for the
definition, and also 
Remark~\ref{rem:ENH-emptysource}).  
The 1-morphisms $\ul\Gamma$ and $\ul\Hom$ in $\ENH_\kk$ 
enhancing
$\dR\Gamma$ and $\dR\Hom$ are defined in
\ref{sec:definition-lifts-six}.

We also need to explain the symbol $\ul\alpha\inv$.
Whenever a commutative $\kk$-algebra $A$ and a morphism 
$\alpha \colon
Y \ra X$ 
of topological spaces are given we view $\alpha$ as
a morphism $X_A \ra Y_A$ of $\kk$-ringed spaces in the obvious
way where $X_A$ and $Y_A$ denote 
$X$ and $Y$ equipped with the constant sheaf of $\kk$-algebras
with stalk 
$A$. Then $\alpha^*=\alpha\inv$ is exact (and $\alpha$ is
flat). In this case we define a 
1-morphism $\ul\alpha\inv$ (see \ref{sec:definition-lifts-six})
lifting $\dL \alpha\inv$
and a 2-isomorphism $\ul\alpha^* \sira 
\ul\alpha\inv$ 
lifting $\dL\alpha^* \sira
\dL\alpha\inv$
(see row \eqref{eq:tab:ulalpha*-inv}).

When working with $!$-functors we follow the approach of
\cite{wolfgang-olaf-locallyproper} to impose conditions on
maps rather than spaces (which are traditionally assumed to
be locally compact), cf.\ \ref{sec:proper-direct-image}. We then 
only consider morphisms of $\kk$-ringed 
spaces coming from a morphism of topological spaces and a
$\kk$-algebra as explained above; the only reason for this
restriction is that a more general theory is not yet documented
in the literature.

\subsubsection{The 2-multicategory of formulas}
\label{sec:2-mult-form}

The $\kk$-linear 2-multicategory $\FML_\kk$ of formulas
is our formal tool
to describe all relations between functors between derived
categories 
of sheaves that we can lift to dg
$\kk$-enhancements. 
We explain how to summarize our lifting results using this
multicategory. 

Essentially, $\FML_\kk$ is the free
$\kk$-linear 2-multicategory with relations whose objects are symbols
$\ud X$ (and $\ud X^\opp$),
for each $\kk$-ringed site $X$, whose
generating 
1-morphisms are symbols
$\ud{\otimes}$, $\ud\sheafHom$, $\ud\alpha^*$, $\ud\alpha_*$,
$\ud\alpha_!$, $\ud\alpha^!$ (and $\ud{\mathcal{O}}$,
$\ud\alpha\inv$, $\ud\Gamma$, $\ud\Hom$)
and whose generating 2-morphisms are 
the ``formulas'' given by the entries of the left column of
table~\ref{tab:main};
the relations we impose just say that all the
generating 2-morphisms labeled $\sim$ are invertible.
We refer the reader to \ref{sec:2-multicat-formulas} for a
precise definition of $\FML_\kk$ (there are a few more generating
2-morphisms and relations that are not important and ignored here).

There is an obvious interpretation functor 
\begin{equation}
  \label{eq:intro:interprete-TRCAT}
  \FML_\kk \ra
  \TRCAT_\kk   
\end{equation}
of $\kk$-linear 2-multicategories
mapping $\ud{X}$ to $\D(X)$, $\ud{X}^\opp$ to 
$\D(X)^\opp$, mapping the six generating 1-morphisms to the six
functors
$\otimes^\dL$, $\dR\sheafHom$,
$\dL \alpha^*$, $\dR \alpha_*$, $\dR \alpha_!$, $\alpha^!$, and
mapping the generating 2-morphisms to the corresponding
2-morphisms in the right column in table~\ref{tab:main}.

\begin{theorem}
  [{cf.\ Theorems~\ref{t:interprete-FML-ENH},
    \ref{t:compare-interpretations-FML},
    Remark~\ref{rem:comm-diagrams}}]
  \label{t:intro:interprete-compare-FML-ENH}
  There is an interpretation functor
  \begin{equation}
    \label{eq:intro:interprete-ENH}
    \FML_\kk \ra
    \ENH_\kk   
  \end{equation}
  of $\kk$-linear 2-multicategories mapping $\ud{X}$ to
  $\ul\II(X)$, 
  mapping the six generating 1-morphisms to the six
  dg $\kk$-functors
  $\ul\otimes$, $\ul\sheafHom$,
  $\ul\alpha^*$, $\ul\alpha_*$, $\ul\alpha_!$, $\ul\alpha^!$, and
  mapping the generating 2-morphisms to the corresponding
  2-morphisms in the middle column of table~\ref{tab:main}.
  If we view all
  diagrams in \ref{sec:some-comm-diagr} 
  (except diagram \eqref{eq:intro:alpha*-triangles-derived})
  as diagrams in $\FML_\kk$
  by replacing underlines by underdots, the interpretation
  functor
  \eqref{eq:intro:interprete-ENH}
  maps these diagrams to commutative diagrams (namely to the
  commutative diagrams  
  in \ref{sec:some-comm-diagr}), and similarly for the
  commutative diagrams in 
  Lemmas~\ref{l:ul-object-4-functors} and
  \ref{l:ul-object-2-functors}.
  Moreover, 
  there is a
  pseudo-natural transformation $\omega$ as illustrated by the
  diagram
  \begin{equation}
    \label{eq:intro:pseudonatural-trafo-omega}
    \begin{tikzpicture}[baseline=(current bounding box.center),
      description/.style={fill=white,inner sep=2pt},
      natural/.style={double,double equal sign distance,-implies} ]
      
      \matrix (m) [matrix of math nodes, row sep=4em,
      column sep=2em, text height=1.5ex, text depth=0.25ex]
      { && \ENH_\kk \\
        \FML_\kk && \TRCAT_\kk\\};
      \path[->,font=\scriptsize]
      (m-2-1) edge node[above left] {\eqref{eq:intro:interprete-ENH}} (m-1-3)
      (m-2-1) edge node[below] (mitte) {\eqref{eq:intro:interprete-TRCAT}} (m-2-3)
      (m-1-3) edge node[right] {$[-]$} (m-2-3);

      \path[->,font=\scriptsize]
      (mitte) edge[natural,shorten >=1.5em,shorten <=1.0em]
      node[below right] {$\omega$} (m-1-3);
    \end{tikzpicture}
  \end{equation}
  %
  %
  that maps an object $X$
  to the 1-morphism 
  \eqref{eq:intro:DX-[ulIIX]}
  (which is an equivalence of triangulated $\kk$-categories) 
  and that maps the generating 1-morphisms $\ud\alpha^*$,
  $\ud\alpha_*$
  to the 
  canonical 2-isomorphisms making the diagrams in
  \eqref{eq:intro:omega-alpha^*_*} commutative, respectively,
  and similarly for the other generating 1-morphisms.
\end{theorem}

Theorem~\ref{t:intro:interprete-compare-FML-ENH}
follows immediately from our 
lifting results 
explained in  
\ref{sec:enhanc-cons-1},
\ref{sec:lifts-six-functors},
\ref{sec:lifts-relations}.

For the sake of completeness let us mention that the entries of
the left column in table~\ref{tab:subsequent} are defined in the
obvious way as
compositions of suitable generating 2-morphisms 
such that they are mapped to the entries in the
other two columns under our interpretation functors.

\subsubsection{The meaning of
  Theorem~\ref{t:main}}
\label{sec:mean-theorem-main}

The precise meaning of 
Theorem~\ref{t:main} consists of
\begin{enumerate}
\item 
  the definition of the $\kk$-linear 2-multicategory $\ENH_\kk$;
\item 
  the definition of the 1-morphisms
  $\ul\otimes$, $\ul\sheafHom$,
  $\ul\alpha^*$, $\ul\alpha_*$, $\ul\alpha_!$, $\ul\alpha^!$ in
  $\ENH_\kk$;
\item 
  the definition of all the 2-(iso)morphisms in $\ENH_\kk$ in the
  middle column of   
  table~\ref{tab:main} by explicit zig-zags;
\item 
  the definition of the $\kk$-linear 2-multicategory $\FML_\kk$;
\item 
  the statement of
  Theorem~\ref{t:intro:interprete-compare-FML-ENH} 
  saying in a very precise way that the previous data lift the
  part of  
  Grothendieck--Verdier--Spaltenstein's six functor 
  formalism considered in this article to the $\kk$-linear
  2-multicategory $\ENH_\kk$. 
\end{enumerate}

\subsection{Lifting actions and (iso)morphisms 
  from the derived level 
  to the dg level} 
\label{sec:lifting-actions-dg}

It may be helpful to look at the example in \ref{sec:an-example-1}
first.

\subsubsection{Lifting actions}
\label{sec:lifting-actions}

Let 
\begin{equation}
  \label{eq:M-as-1-morphism}
  M \colon (\D(X_1)^{\epsilon_1}, \dots, \D(X_n)^{\epsilon_n}) \ra \D(Y)
\end{equation}
be a 1-morphism in $\TRCAT_\kk$ where $X_1, \dots, X_n, Y$ are
$\kk$-ringed sites and $\epsilon_i \in \{\emptyset, \opp\}$.
Then $M$ gives rise to a $\kk$-functor
\begin{equation}
  \label{eq:M-as-module}
  M \colon \D(X_1)^{\epsilon_1} \otimes \dots \otimes
  \D(X_n)^{\epsilon_n} \ra \D(Y),
\end{equation}
i.\,e.\ a ``left $\D(X_1)^{\epsilon_1} \otimes \dots \otimes
  \D(X_n)^{\epsilon_n}$-module with values in $\D(Y)$''.
If $M$ is 
a composition of 1-morphisms which can be lifted to $\ENH_\kk$,
i.\,e.\ if $M$ is in the image of the interpretation functor \eqref{eq:intro:interprete-TRCAT},
there is a dg $\kk$-functor
\begin{equation}
  \ul{M} \colon \ul\II(X_1)^{\epsilon_1} \otimes \dots \otimes
  \ul\II(X_n)^{\epsilon_n} \ra \ul\II(Y), 
\end{equation}
i.\,e.\ a ``left dg $\ul\II(X_1)^{\epsilon_1} \otimes \dots \otimes
\ul\II(X_n)^{\epsilon_n}$-module with values in $\ul\II(Y)$'',
such that
\begin{equation}
  [\ul{M}] \colon [\ul\II(X_1)^{\epsilon_1}] \otimes \dots \otimes
  [\ul\II(X_n)^{\epsilon_n}] \ra [\ul\II(Y)] 
\end{equation}
is canonically identified with \eqref{eq:M-as-module}.
In the special case that $Y=(\pt,\kk)$ we have $\ul\II(Y)=\ul\C(\kk)$
and $\ul{M}$ is  
a left dg $\ul\II(X_1)^{\epsilon_1} \otimes \dots \otimes
\ul\II(X_n)^{\epsilon_n}$-module.

\subsubsection{Lifting (iso)morphisms}
\label{sec:lifting-morphisms}

Let 
$\rho \colon M_1 \ra M_2$
be a 2-morphism in $\TRCAT_\kk$ where $M_1$ and $M_2$ are as in
\eqref{eq:M-as-1-morphism}. In particular, we can view $\rho$ as a
morphism of left $\D(X_1)^{\epsilon_1} \otimes \dots \otimes
\D(X_n)^{\epsilon_n}$-module with values in $\D(Y)$.
If $\rho$ and $M_1$ and $M_2$ are in the image of the
interpretation
functor
\eqref{eq:intro:interprete-TRCAT},
there is a 2-morphism 
$\ul{\rho} \colon \ul{M}_1 \ra \ul{M}_2$
in $\ENH_\kk$ such that $[\ul{\rho}] \colon [\ul{M}_1] \ra
[\ul{M}_2]$ is canonically identified with $\rho$.

If we know that $\ul{\rho}$ can be represented by a zig-zag of
objectwise homotopy equivalences, then
$\ul{M}_1$ and
$\ul{M}_2$ are isomorphic in the ''derived category of left dg $\ul\II(X_1)^{\epsilon_1} \otimes \dots \otimes
\ul\II(X_n)^{\epsilon_n}$-modules with values in $\ul\II(Y)$''.
In the special case that $Y=(\pt,\kk)$ they are isomorphic
in the more familiar derived category of left dg
$\ul\II(X_1)^{\epsilon_1} \otimes \dots \otimes
\ul\II(X_n)^{\epsilon_n}$-modules.

\subsubsection{An example}
\label{sec:an-example-1}

Let $\alpha \circ \beta' =
\beta \circ \alpha'$ be a cartesian diagram as in 
the row above \eqref{eq:tab:proper-base-change-ENH} 
in table~\ref{tab:main} on page
\pageref{tab:main}; assume
that $\beta$ is separated and locally proper with $\beta_! \colon
\Mod(X'_A) \ra \Mod(X_A)$ of finite cohomological dimension. Then
there is a
2-isomorphism
\begin{equation}
  \dR\Hom(\dR\alpha_! (-), \dR\sheafHom(\dR\beta_! (-), -)), 
  \cong
  \dR\Hom(\dR\alpha'_!\dL \beta'^*(-) \otimes^\dL (-), \beta^!(-)) 
\end{equation}
between
triangulated 
$\kk$-trilinear functors
from $\D(Y)^\opp \times \D(X')^\opp \times \D(X)$ to $\D(\pt)$,
by
\eqref{eq:tab:!-adjunction-sheafHom-ENH}, 
\eqref{eq:tab:enhanced-pull-push-ulHom},
\eqref{eq:tab:ulalpha*-inv},
\eqref{eq:tab:proper-base-change-ENH},
\eqref{eq:tab:28}.
The above discussion shows that it lifts to a 2-isomorphism
\begin{equation}
  \ul\Hom(\ul\alpha_! (-), \ul\sheafHom(\ul\beta_! (-), -))
  \cong
  \ul\Hom(\ul\alpha'_!\ul\beta'^*(-) \ul\otimes (-), \ul\beta^!(-)) 
\end{equation}
in $\ENH_\kk$ which can be represented by a zig-zag of objectwise
homotopy equivalences; hence it gives rise to an isomorphism
in the derived category of
left dg $\ul\II(Y)^\opp \otimes \ul\II(X')^\opp \otimes
\ul\II(X)$-modules.

\subsection{Model categories and dg
  \texorpdfstring{$\kk$}{k}-categories}
\label{sec:model-categories-dg}

The key ingredients for our results come from the interplay
between model category theory and enriched category
theory. Recall that a dg $\kk$-category is a category
enriched in the monoidal category of complexes of $\kk$-vector
spaces. We usually denote a dg $\kk$-category by an underlined
symbol and omit the underline when referring to the
underlying category.

\subsubsection{Dg \texorpdfstring{$\kk$}{k}-enriched functorial
  factorizations}
\label{sec:dg-texorpdfstr-enric}

Any model category $\ms{M}$ admits functorial factorizations and in
particular functorial cofibrant and fibrant resolutions.
If $\ms{M}$ is the underlying category of a
dg $\kk$-category $\ulms{M}$ it is natural to ask whether it admits dg
$\kk$-enriched functorial factorizations.
We provide some general criteria to ensure a positive answer to this
question in
Theorem~\ref{t:existence-enriched-functorial-fact}. 
This result is the adaptation to dg $\kk$-categories of general
results from enriched model category theory 
\cite{shulman-enriched-homotopy-theory-arxiv,riehl-categorical-homotopy-theory}.
As a
consequence we obtain:

\begin{theorem}
  [{see Theorems~\ref{t:enriched-functorial-fact-k-ringed-site},
    \ref{t:enriched-functorial-fact-modules-K-cat},
    \ref{t:enriched-functorial-fact-flat}}]
  \label{t:intro:examples-admit-functorial-fact}
  Let $\kk$ be a field.
  Consider the following list of 
  dg $\kk$-categories and model structures on
  their underlying categories turning them into model categories.
  \begin{enumerate}
  \item 
    \label{enum:ulCX-inj-extenstion-by-zero}
    $\ul\C(X)$ with $\C(X)$ carrying the
    injective model structure (see
    \ref{sec:inject-model-struct-1}) 
    or the
    extension-by-zero model
    structure 
    (see \ref{sec:extension-zero-model}), for
    a $\kk$-ringed site $(X, \mathcal{O})$;
  \item the dg $\kk$-category $\ul\MMod(\ulms{C})$ of (right) dg
    modules over a 
    dg $\kk$-category $\ulms{C}$,
    with $\MMod(\ulms{C})$ carrying the
    injective or the projective model
    structure 
    (see \ref{sec:model-structures});
  \item $\ul\C(X)$ with $\C(X)$ carrying the 
    flat model structure 
    (see \ref{sec:flat-model-structure})
    or the injective model structure, for a $\kk$-ringed
    space $(X, \mathcal{O})$;
  \end{enumerate}
  Then all the model structures in this list admit 
  dg $\kk$-enriched functorial factorizations.
\end{theorem}

\subsubsection{Dg \texorpdfstring{$\kk$}{k}-enriched injective and
  flat resolutions}
\label{sec:dg-texorpdfstr-enric-1}

Given a $\kk$-ringed site $(X, \mathcal{O})$, the fibrant objects
of the injective model structure on $\C(X)$ are
precisely the objects of $\II(X)$ and the cofibrant
objects of the extension-by-zero model structure on
$\C(X)$ are h-flat and have flat components (see
Lemma~\ref{l:E-cofibrant-(pullback)-h-flat}).
Therefore
Theorem~\ref{t:intro:examples-admit-functorial-fact}.\ref{enum:ulCX-inj-extenstion-by-zero}
yields the following 
key ingredient to lift the six functors to dg
$\kk$-functors.

\begin{corollary}
  \label{c:intro:enriched-inj-flat-resolutions-k-ringed-site}
  Let $\kk$ be a field.
  Let $(X, \mathcal{O})$ be a $\kk$-ringed
  site and let $\ul\C(X)_{\hflat, \op{cwflat}}$ denote
  the full dg $\kk$-subcategory of $\ul\C(X)$ of h-flat
  and componentwise flat objects.
  Then there are dg $\kk$-functors
  \begin{align}
    \ii 
    & \colon \ul\C(X) \ra \ul\II(X),\\
    \ee 
    & \colon \ul\C(X) \ra
          \ul\C(X)_{\hflat, \op{cwflat}}
  \end{align}
  together with dg $\kk$-natural transformations 
  \begin{align}
    \iotaii \colon \id \ra \ii 
    & \colon \ul\C(X) \ra \ul\C(X),\\ 
    \epsilonee \colon \ee \ra \id 
    & \colon \ul\C(X) \ra \ul\C(X)
  \end{align}
  whose evaluations $\iotaii_G \colon G \ra \ii G$ and
  $\epsilonee_G \colon \ee G \ra G$ at each object
  $G \in \ul\C(X)$ are quasi-isomorphisms.
\end{corollary}

The assumption that $\kk$ is a field is essential 
for this key 
ingredient. We give a counterexample in
Lemma~\ref{l:no-additive-injective-repl-for-Ab}.  
We fix such ``injective'' and ``flat'' resolution functors for
each $\kk$-ringed site. 

\subsubsection{Definition of the equivalence \eqref{eq:intro:DX-[ulIIX]}}
\label{sec:definitions}

The dg $\kk$-functor $\ii$ induces a functor 
\begin{equation}
  [\ii] \colon [\ul\C(X)] \ra
  [\ul\II(X)]  
\end{equation}
of triangulated $\kk$-categories which
maps acyclic objects to zero. It factors to an equivalence
$\ol{[\ii]} \colon \D(X) \ra [\ul\II(X)]$ of triangulated
$\kk$-categories. This defines
the equivalence \eqref{eq:intro:DX-[ulIIX]}.

\subsubsection{Definition of the lifts of the six functors}
\label{sec:definition-lifts-six}

The first five of the dg $\kk$-functors
in \eqref{eq:intro:def-lifts-six-functors} are defined by
\begin{align}
  (-)\ul\otimes(-) 
  & := \ii(\ee(-) \otimes \ee(-)),
  & \ul\sheafHom(-,-) 
  & := \ii\sheafHom(-,-),
  \\
  \ul\alpha^* 
  & := \ii\alpha^*\ee,
  & \ul\alpha_* 
  & :=\ii\alpha_*,
  \\
  \ul\alpha_! & := \ii\alpha_!.
\end{align}
The definition of the sixth dg $\kk$-functor $\ul\alpha^!$ 
uses the explicit description of a right adjoint functor
$\alpha^!$ to $\dR\alpha_!$. We refer the reader to
\ref{sec:lifts-deriv-funct-2}. The object $\ul{\mathcal{O}}$ appearing
in table~\ref{tab:main} is defined by $\ul{\mathcal{O}}:=\ii
\mathcal{O}$. Further definitions are 
$\ul\alpha\inv:=\ii \alpha\inv$, $\ul\Gamma:=\ul\sigma_*$
(where $\sigma \colon (\Sh(\mathcal{X}),
\mathcal{O}_\mathcal{X}) \ra (\Sh(\pt), \kk)$ is the structure
morphism, see \eqref{eq:56}), and $\ul\Hom:=\ii\Gamma\sheafHom$.





\section{Dg \texorpdfstring{$\kk$}{k}-enriched functorial factorizations}
\label{sec:dg-enrich-funct}


We provide 
criteria to ensure that a model structure on a dg $\kk$-category
admits dg $\kk$-enriched functorial factorizations, see
Theorem~\ref{t:existence-enriched-functorial-fact}.  
We use the language of enriched categories, in particular of dg
categories, and of model categories. All basic facts we need can
be found in \cite{kelly-enriched, keller-on-dg-categories-ICM,
  hirschhorn-model, hovey-model-categories,
  may-ponto-more-concise-alg-topo,
  riehl-categorical-homotopy-theory}.
We fix a Grothendieck universe $\mathscr{U}$. 

\subsection{Categories of complexes}
\label{sec:categories-complexes}

Recall that $\R$ is a commutative ring.
We say $\R$-category instead of $\R$-linear category.
When size issues seem important we will be more precise and
speak for example about $\R$-categories with $\mathscr{U}$-small
$\Hom$-sets.

Let $\mathcal{A}$ be an additive
$\R$-category. We write $\mathcal{A}^\bZ$ for the additive
$\R$-category of $\bZ$-graded objects in $\mathcal{A}$.  The
additive $\R$-category of (cochain) complexes in $\mathcal{A}$
with cochain maps as morphisms is denoted $\C(\mathcal{A})$. We
sometimes view $\mathcal{A}^\bZ$ as the full subcategory of
$\C(\mathcal{A})$ consisting of complexes with vanishing
differential. We also have the $\R$-functor $\C(\mathcal{A}) \ra
\mathcal{A}^\bZ$ that forgets the differentials. The two
categories $\mathcal{A}^\bZ$ and 
$\C(\mathcal{A})$ are abelian if $\mathcal{A}$ is abelian.  They
come endowed with shift functors $[n]$, for $n \in \bZ$. If
$f \colon M \ra N$ is a morphism in $\C(\mathcal{A})$, we denote
its cone by $\Cone(f)$, i.\,e.\ $\Cone(f)=N \oplus [1]M$ 
with differential $d_{\Cone(f)}= \big(
\begin{smallmatrix}
  d_N & [1]f\\
  0 & d_{[1]M}
\end{smallmatrix}
\big)
$. 
As usual, we view $\mathcal{A}$ as a full subcategory of
$\mathcal{A}^\bZ$ and $\C(\mathcal{A})$.
In particular, if $M$ is an object of $\mathcal{A}$, the
shift $[n]M$ for $n \in \bZ$ is the complex in $\mathcal{A}$
concentrated in 
degree $-n$ with $([n]M)^{-n}=M$. 
There is a canonical monomorphism $[n]M \ra \Cone(\id_{[n]M})$ in
$\C(\mathcal{A})$ which
splits in 
$\mathcal{A}^\bZ$.

We write $\ul\C(\mathcal{A})$ for the dg
$\R$-category of complexes in $\mathcal{A}$. 
Given complexes $M, N$ in $\mathcal{A}$ we usually use the
notation
\begin{align}
  \label{eq:78}
  \C_\mathcal{A}(M,N) & =\Hom_{\C(\mathcal{A})}(M,N),\\
  \ul\C_\mathcal{A}(M,N) & =\Hom_{\ul\C(\mathcal{A})}(M,N).
\end{align}
The $\R$-module $\C_\mathcal{A}(M,N)$ consists of the closed
degree zero morphisms in the dg $\R$-module
$\ul\C_\mathcal{A}(M,N)$, i.\,e.\
$\C_\mathcal{A}(M,N)= Z^0(\ul\C_\mathcal{A}(M,N))$.

As usual, we write $[\ul\C(\mathcal{A})]$ for the homotopy
$\R$-category of 
complexes in $\mathcal{A}$. Its objects are complexes in
$\mathcal{A}$, and given two such complexes $M,N$ we have
\begin{equation}
  \label{eq:79}
  [\ul\C_\mathcal{A}](M,N):=
  \Hom_{[\ul\C(\mathcal{A})]}(M,N):=H^0(\ul\C_\mathcal{A}(M,N)).
\end{equation}



 
If $\mathcal{A}$ is an abelian category, we 
write $\C_\ac(\mathcal{A})$ for the full subcategory of
$\C(\mathcal{A})$ of acyclic objects. The Verdier
quotient
$\D(\mathcal{A})=[\C(\mathcal{A})]/[\C_\ac(\mathcal{A})]$ is the
derived category of $\mathcal{A}$ and comes with the Verdier
localization functor $[\C(\mathcal{A})] \ra
\D(\mathcal{A})$. We usually assume that this localization
functor is the identity on objects.
We abbreviate $\D_\mathcal{A}(M,N)=\Hom_{\D(\mathcal{A})}(M,N)$.


\subsubsection{Notation}
\label{sec:notation}

Let $\Mod(\R)=\Mod(\R, \mathscr{U})$ be the abelian $\R$-category of
$\mathscr{U}$-small $\R$-modules.
The following notation will be used in the rest of this
article.

Let 
$\RR:=\C(\R):=\C(\Mod(\R))$ be the category of complexes in
$\Mod(\R)$ 
and $\ul\RR:=\ul\C(\R):=\ul\C(\Mod(\R))$ the dg $\R$-category 
of complexes in
$\Mod(\R)$. 
Similarly, starting from the field $\kk$, we write
$\KK:=\C(\kk):=\C(\Mod(\kk))$ and
$\ul\KK:=\ul\C(\kk):=\ul\C(\Mod(\kk))$. 
If relevant, we add the universe as an index and write for example
$\RR_\mathscr{U}$ instead of $\RR$.

\subsubsection{Complexes of vector spaces}
\label{sec:compl-vect-spac}


\begin{lemma}
  \label{l:objects-C(k)}
  Let $\kk$ be a field.
  Any object of $\KK=\C(\kk)$ is isomorphic to a coproduct of
  shifts of the two 
  objects $\kk$ and $\Cone(\id_\kk)$.
\end{lemma}

\begin{proof}
  Obvious; see Lemma~\ref{l:monos-C(k)} for a stronger result.
\end{proof}

\subsection{Differential graded categories}
\label{sec:diff-grad-categ}

Note that $\RR$ is a closed symmetric monoidal category in the
obvious way: the monoidal structure is given by the tensor
product $\otimes=\otimes_\R$.
In particular, we can speak about $\RR$-categories (= categories
enriched in $\RR$), see \cite{kelly-enriched}.
Note that $\RR$-category is just
another name for dg $\R$-category with $\mathscr{U}$-small
$\Hom$-sets. 
We transfer
the usual terminology and notation from dg $\R$-categories
to $\RR$-categories. 

We usually denote an $\RR$-category by an underlined symbol
(e.\,g.\ $\ulms{M}$) and its underlying ordinary category by a
plain symbol (e.\,g.\ $\ms{M}$). We have already used this
convention above: 
we have the $\RR$-category 
$\ul\C(\mathcal{A})$ for an additive $\R$-category $\mathcal{A}$
with $\mathscr{U}$-small $\Hom$-sets
(whose underlying category $\C(\mathcal{A})$ was even considered
as an $\R$-category),
the $\RR$-category $\ul\RR$, and the $\KK$-category
$\ul\KK$.

If $\ulms{M}$ is an $\RR$-category, we denote its homotopy
category by $[\ulms{M}]$. Note that $[\ulms{M}]$ is an
$\R$-category. 

%

\subsubsection{Dg-enriched arrow categories}
\label{sec:enrich-arrow-categ}

For the convenience of the reader we recall the definition of the
arrow category in the $\RR$-enriched setting.
For $n \in \bN$, let $[n]$ denote the 
free category on the graph $0 \ra 1 \ra \dots \ra n$.

Let $\ulms{M}$ be an $\RR$-category.
The $\RR$-category $\ulms{M}^{[1]}$ of arrows in $\ms{M}$
has
morphisms $A \xra{f} B$ in $\ms{M}$ as objects and morphism
objects
\begin{equation}
  \label{eq:10}
  \Hom_{\ulms{M}^{[1]}}((A \xra{f} B), (A' \xra{f'} B'))
  =
  \ulms{M}(A,A') \subtimes{\ulms{M}(A,B')}
  \ulms{M}(B,B'). 
\end{equation}

For readers more familiar with dg
$\R$-categories let us describe 
this category in the dg language.
The dg $\R$-category $\ulms{M}^{[1]}$ has closed degree zero
morphisms 
in $\ulms{M}$ as objects, and given two such objects
$A \xra{f} B$ and $A' \xra{f'} B'$, the dg $\R$-module of
morphisms 
from $f$ to $f'$ is the dg
submodule of 
$\ulms{M}(A,A') \times \ulms{M}(B,B')$ consisting of pairs $(a,b)$
making the diagram
\begin{equation}
  \label{eq:11}
  \xymatrix{
    {A} \ar[r]^-{f} \ar[d]^-{a} & {B} \ar[d]^-{b}\\
    {A'} \ar[r]^-{f'} & {B'}
  }
\end{equation}
commutative.

The $\RR$-category $\ulms{M}^{[2]}$ of two composable arrows
is defined similarly.
Composition of arrows defines an $\RR$-functor
\begin{equation}
  \label{eq:12}
  \ulms{M}^{[2]} \ra \ulms{M}^{[1]}.
\end{equation}

If $\mathcal{A}$ is an additive $\R$-category, taking
cones as defined in 
\ref{sec:categories-complexes}
extends naturally to an $\RR$-functor
$\Cone \colon \ul\C(\mathcal{A})^{[1]} \ra \ul\C(\mathcal{A})$.
Note also that the shift functors $[n]$ come from $\RR$-functors
$[n] \colon \ul\C(\mathcal{A}) \ra \ul\C(\mathcal{A})$. 

\subsubsection{Categories of dg modules over a dg category}
\label{sec:categ-dg-modul}

Let $\ulms{M}$ be a $\mathscr{U}$-small $\RR$-category. A
(right) $\ulms{M}$-module 
(= a dg $\ulms{M}$-module) is an $\RR$-functor
$\ulms{M}^\opp \ra \ul\RR$. Let $\MMod(\ulms{M})$ be
the
category of $\ulms{M}$-modules and $\ul\MMod(\ulms{M})$ the
$\RR$-category of $\ulms{M}$-modules (note that these categories 
have $\mathscr{U}$-small $\Hom$-sets and objects in
$\mathscr{U}$, by Remark~\ref{rem:functor-categories-size} below).
Given two
$\ulms{M}$-modules $M$ and $N$ we write
$\MMod_{\ulms{M}}(M,N)=\Hom_{\MMod(\ulms{M})}(M,N)$ and
$\ul\MMod_{\ulms{M}}(M,N)=\Hom_{\ul\MMod(\ulms{M})}(M,N)$.  There
are shift $\RR$-functors
$[n] \colon \ul\MMod(\ulms{M}) \ra \ul\MMod(\ulms{M})$ and
a cone $\RR$-functor
$\Cone \colon \ul\MMod(\ulms{M})^{[1]} \ra \ul\MMod(\ulms{M})$
defined in the obvious way.

The (enriched) Yoneda functor is the fully faithful $\RR$-functor
\begin{align}
  \label{eq:RR-Yoneda}
  \Yo \colon \ulms{M} & \ra \ul\MMod(\ulms{M}),\\
  \notag
  M & \mapsto \Yo(M)=\ulms{M}(-,M).
\end{align}
Its underlying functor $\ms{M} \ra
\MMod(\ulms{M})$ is denoted by the same symbol.
The Yoneda functor induces a fully faithful $\RR$-functor
$\Yo^{[1]} \colon \ulms{M}^{[1]} \ra
\ul\MMod(\ulms{M})^{[1]}$
mapping $(A \xra{f} B)$ to $(\Yo(A) \xra{\Yo(f)} \Yo(B))$.


\begin{remark}
  \label{rem:functor-categories-size} 
  Let $\mathcal{C}$ and $\mathcal{D}$ be categories
  and let
  $\Fun(\mathcal{C}, \mathcal{D})$ be the category of functors
  from $\mathcal{C}$ to $\mathcal{D}$.
  Then the following statements are true (see
  \cite[Rem.~1.1.12,
  Cor.~1.4.1.(i)]{gabber-ramero-found-almost-v12}).   
  \begin{enumerate}
  \item 
    \label{enum:U-klein-neu}
    If $\mathcal{C}$ and $\mathcal{D}$ are $\mathscr{U}$-small, so
    is $\Fun(\mathcal{C}, \mathcal{D})$.
  \item 
    \label{enum:U-klein-nach-lokal-U-klein-neu}
    If $\mathcal{C}$ is $\mathscr{U}$-small and $\mathcal{D}$ has
    $\mathscr{U}$-small $\Hom$-sets, 
    then 
    $\Fun(\mathcal{C}, \mathcal{D})$ 
    has $\mathscr{U}$-small $\Hom$-sets.
    If, in addition, $\mathcal{D}$ has objects in $\mathscr{U}$,
    then 
    $\Fun(\mathcal{C}, \mathcal{D})$ has objects in
    $\mathscr{U}$.
  \end{enumerate}
\end{remark}

\subsubsection{Strongly pretriangulated dg categories}
\label{sec:strongly-pretr-categ}

Let $\ulms{M}$ be a $\mathscr{U}$-small $\RR$-category.
 
We say that $\ulms{M}$ is \define{strongly
pretriangulated} (resp.\ \define{pretriangulated}) if
the following objects are in the essential image of
$\Yo \colon \ms{M} \ra \MMod(\ulms{M})$
(resp.\ $[\Yo] \colon [\ulms{M}] \ra [\ul\MMod(\ulms{M})]$):
\begin{itemize}
\item 
  the zero module;
\item 
  the object
  $[n]\Yo(M)$, for all objects $M \in \ms{M}$ and all $n \in
  \bZ$;
\item 
  the object
  $\Cone(\Yo^{[1]}(f))$,
  for all objects $f \in \ms{M}^{[1]}$.
\end{itemize}
We have added the first condition to the definition from
\cite[4.3]{bondal-larsen-lunts-grothendieck-ring}; it implies
that $\ms{M}$ (resp.\ $[\ulms{M}]$) has a zero object.

Let $\ulms{M}$ be a strongly pretriangulated $\mathscr{U}$-small $\RR$-category. Then
there are
$\RR$-functors ``(formal) shift'' 
$[1] \colon \ulms{M} \ra \ulms{M}$ 
and ``(formal) cone'' 
$\Cone \colon \ulms{M}^{[1]} \ra \ulms{M}$ 
such that the diagrams
\begin{equation}
  \label{eq:formal-shift-and-cone}
  \xymatrix{
    {\ulms{M}} \ar[r]^-{\Yo} \ar[d]_-{[1]} &
    {\ul\MMod(\ulms{M})} \ar[d]^-{[1]}\\
    {\ulms{M}} \ar[r]^-{\Yo} &
    {\ul\MMod(\ulms{M})}
  }
  \quad \text{and} \quad
  \xymatrix{
    {\ulms{M}^{[1]}} \ar[r]^-{\Yo^{[1]}} \ar[d]_-{\Cone} &
    {\ul\MMod(\ulms{M})^{[1]}} \ar[d]^-{\Cone}\\
    {\ulms{M}} \ar[r]^-{\Yo} &
    {\ul\MMod(\ulms{M})}
  }
\end{equation}
commute up to isomorphisms of $\RR$-functors. We fix such
isomorphisms so that shift $[1]$ and cone $\Cone$
are uniquely determined in the obvious way.
We define $[-1]$ similarly, put $[n]=[1]^n$ and $[-n]=[-1]^n$,
for $n \in \bN$, and call all these $\RR$-functors shift functors.
Cone and shift functors commute 
up to unique isomorphism of $\RR$-functors
with any $\RR$-functor 
between strongly pretriangulated $\RR$-categories.

\begin{example}
  \label{exam:strongly-pretriang-dgcat}
  The categories
  $\ul\C(\mathcal{A})$, for an additive
  $\R$-category $\mathcal{A}$, or
  $\ul\MMod(\ulms{C})$, for an $\RR$-category
  $\ulms{C}$, 
  are 
  strongly pretriangulated $\RR$-categories, up to size issues:
  the category $\mathcal{A}$ should be $\mathscr{U}$-small, and
  $\ulms{C}$ 
  should be $\mathscr{U}'$-small, where $\mathscr{U}'$ is a
  universe that is an element of $\mathscr{U}$.
  We can and will assume that their formal shift and cone
  functors are the usual functors introduced previously.
  We make the same assumption for strongly pretriangulated
  $\RR$-subcategories of these categories that are stable under
  shifts and cones.
\end{example}

Consider the $\RR$-functor $\ulms{M} \ra \ulms{M}^{[1]}$
mapping an object $M$ to $M \xra{\id_M} M$ and a morphism $f
\colon M \ra N$ to $(f,f)$.
We define the
$\RR$-functor $\iCone$, called \define{identity cone}, to be the composition of this functor with
the cone functor, i.\,e.\
\begin{equation}
  \label{eq:iCone}
  \iCone
  \colon 
  \ulms{M} \ra
  \ulms{M}^{[1]} \xra{\Cone}
  \ulms{M},
\end{equation}
and denote its underlying functor 
$\iCone \colon \ms{M} \ra \ms{M}$
by the same symbol.
If $f \colon M \ra N$ is a morphism in $\ms{M}$ we have
\begin{equation}
  \iCone(f)=\Cone(f,f) \colon \iCone(M)=\Cone(\id_M) \ra
  \iCone(N)=\Cone(\id_N). 
\end{equation}

\begin{remark}
  \label{rem:htpy-cat-of-pretriangulated-is-triang}
  Given a pretriangulated $\RR$-category $\ulms{M}$ we 
  can
  similarly fix a shift $\R$-functor
  $[1] \colon [\ulms{M}] \ra [\ulms{M}]$ together
  with an 
  isomorphism $[1] [\Yo] \sira [\Yo][1]$.
  Then the homotopy
  category $[\ulms{M}]$ becomes a triangulated $\R$-category 
  in the obvious way. The choice when fixing its
  shift functor $[1]$ is inessential.
  Hence the homotopy category of a pretriangulated $\RR$-category
  is canonically triangulated.
  In fact, we obtain a functor from the category of
  pretriangulated $\RR$-categories to the category of
  triangulated $\R$-categories.
\end{remark}

\subsubsection{Dg enhancements}
\label{sec:dg-enhancements}

We recall the notion of an $\RR$-enhancement (= dg
$\R$-en\-hance\-ment) from
\cite{bondal-kapranov-enhancements,
  bondal-larsen-lunts-grothendieck-ring,lunts-orlov-enhancement}.

\begin{definition}
  \label{d:dg-enhancement}
  Let $\mathcal{T}$ be a $\mathscr{U}$-small triangulated
  $\R$-category.  
  An \define{$\RR$-enhancement} of $\mathcal{T}$ is a
  pretriangulated $\mathscr{U}$-small $\RR$-category $\ulms{E}$
  together with an 
  equivalence $\epsilon \colon [\ulms{E}] \sira \mathcal{T}$ of
  triangulated $\R$-categories (where $[\ulms{E}]$ is considered
  as a triangulated $\R$-category as explained in 
  Remark~\ref{rem:htpy-cat-of-pretriangulated-is-triang}).
\end{definition}

\subsubsection{Tensors, cotensors, and completeness}
\label{sec:tens-cotens-compl}

An $\RR$-category $\ulms{M}$ is \define{tensored}
if there is a bifunctor $\odot \colon \ms{M}
\times \RR \ra \ms{M}$ together with  
natural isomorphisms
\begin{equation}
  \label{eq:tensored}
  \ul\RR(R,\ulms{M}(M,N)) \cong
  \ulms{M}(M \odot R, N)
\end{equation}
in $R \in \RR$, $M,N \in \ms{M}$ (see
\cite[3.7]{riehl-categorical-homotopy-theory}); then the functor
$\odot$ is unique up to unique isomorphism. 
Similarly, $\ulms{M}$ is
\define{cotensored} if there is a bifunctor
$\Psi \colon \RR^\opp \times \ms{M} \ra \ms{M}$ together with
natural isomorphisms
\begin{equation}
  \label{eq:cotensored}
  \ul\RR(R,\ulms{M}(M,N)) 
  \cong \ulms{M}(M, \Psi(R,N)).
\end{equation}
in $R \in \RR$, $M,N \in \ms{M}$; then the functor $\Psi$ is
unique up to unique isomorphism. 
Note that $\odot$ and $\Psi$ come from $\RR$-bifunctors
(see 
\cite[3.7.4, 7.3.1]{riehl-categorical-homotopy-theory}).

Recall from \cite[7.6]{riehl-categorical-homotopy-theory}
that 
an $\RR$-category $\ulms{M}$ is $\RR$-bicomplete
if and only if it is
tensored and cotensored and the underlying category $\ms{M}$ is
bicomplete (i.\,e.\ complete and cocomplete); more precisely,
(co)completeness means that all $\mathscr{U}$-small (co)limits
exist.

If an $\RR$-category $\ulms{M}$ is tensored
and strongly
pretriangulated, we
deduce from \eqref{eq:tensored} 
natural isomorphisms $M \odot 0 \cong 0$,
$([m]M) \odot
([r]R) \cong [m+r](M \odot R)$ for $M \in \ms{M}$, $R
\in \RR$, $m,r \in \bZ$. 
Moreover, we have 
$M \odot \R \cong M$ and 
$M \odot \iCone(\R) \cong \iCone(M)$ naturally in $M \in
\ms{M}$.
Also note that $(- \odot -)$ preserves coproducts in each
argument. 

\subsection{Model categories and dg enriched functorial
  factorizations}
\label{sec:model-categ-enrich}


Let $\ulms{M}$ be an $\RR$-category. Recall from 
\cite[12, 13]{riehl-categorical-homotopy-theory}
that an
\define{$\RR$-enriched functorial factorization} on $\ulms{M}$
is a section of the composition functor
\eqref{eq:12}, i.\,e.\ an $\RR$-functor 
$\ulms{M}^{[1]} \ra \ulms{M}^{[2]}$ whose composition with 
\eqref{eq:12} is the identity.
It is convenient to denote an $\RR$-enriched functorial
factorization by a triple $(L,M,R)$ encoding its effect
on objects and morphisms:
\begin{align}
  \label{eq:13}
  (A \xra{f} B) & \mapsto (A \xra{Lf} M(f) \xra{Rf} B),\\
  \notag
  (a,b) & \mapsto (a, M(a,b), b).
\end{align}

If the underlying category $\ms{M}$ of $\ulms{M}$ carries a 
model structure, we say that
this model structure \define{admits  
$\RR$-enriched 
functorial 
factorizations} if
there exist 
$\RR$-enriched functorial factorizations
$(L,M,R)$ and $(L', M', R')$ such that 
$Lf$ is a cofibration, $Rf$ is a trivial
fibration,
$L'f$ is a trivial cofibration, and $R'f$ is fibration,
for all objects $f$
of $\ulms{M}^{[1]}$. Then, in particular,
there exist $\RR$-functors
\begin{equation}
  \label{eq:14}
  \ccof \colon \ulms{M} \ra \ulms{M}
  \quad \text{and} \quad
  \ffib \colon \ulms{M} \ra \ulms{M}
\end{equation}
together with $\RR$-natural transformations
$\ccof \xra{\gamma} \id \xra{\phiff} \ffib$
such that for each $M \in \ulms{M}$ the object $\ccof(M)$ is
cofibrant and the morphism
$\gammacc_M \colon \ccof(M) \ra M$ is a trivial
fibration, and $\ffib(M)$ is fibrant and $\phiff_M \colon M \ra \ffib(M)$ is a
trivial cofibration.
We call any such pair $(\ccof, \gammacc)$ (resp.\
$(\ffib, \phiff)$) an \define{$\RR$-enriched cofibrant}
(resp.\ \define{fibrant}) \define{resolution functor} and call
$\ccof$ (resp.\ $\ffib$) an \define{$\RR$-enriched 
  cofibrant} (resp.\ \define{fibrant})
\define{replacement functor}.

Note that the notion of a model category depends on a universe: a
model category with respect to the universe $\mathscr{U}$ has
$\mathscr{U}$-small Hom-sets and all $\mathscr{U}$-small
(co)limits.


\begin{theorem}
  \label{t:existence-enriched-functorial-fact}
  Let $\kk$ be a field.
  Let $\ulms{M}$ be a strongly pretriangulated
  $\KK$-bicomplete $\KK$-category.
  Assume that the underlying category $\ms{M}$
  is equipped with a model structure
  turning it into a cofibrantly generated model category
  (with respect to the universe $\mathscr{U}$) (so it admits the
  small object argument) 
  such that 
  the sets of cofibrations and of trivial cofibrations
  are stable 
  \begin{enumerate}
  \item under taking shifts,
    i.\,e.\ 
    mapping $i \colon E \ra F$ to $[n]i \colon [n]E \ra
    [n]F$, for all $n \in \bZ$, and
  \item 
    under applying the functor $\iCone \colon \ms{M} \ra \ms{M}$,
    i.\,e.\  
    under mapping $i \colon E \ra F$ to $\iCone(i) \colon
    \iCone(E) \ra 
    \iCone(F)$ (see
    \eqref{eq:iCone}).
  \end{enumerate}
  Then this model structure admits $\KK$-enriched functorial
  factorizations.
\end{theorem}

The assumption that $\kk$ is a field is essential as we explain 
in \ref{sec:counterexample}.

\begin{proof}
  This follows from
  \cite[Cor.~13.2.4]{riehl-categorical-homotopy-theory} and the
  following Lemma~\ref{l:tensor-K-left-Quillen}.  Note that
  $\KK$-bicompleteness of $\ulms{M}$ implies that
  $\ulms{M}$ is tensored (see
  \ref{sec:tens-cotens-compl}).
\end{proof}

\begin{lemma}
  \label{l:tensor-K-left-Quillen}
  Let $\kk$ be a field.  Let $\ulms{M}$ be a strongly
  pretriangulated tensored $\KK$-category, with tensored
  structure given by the bifunctor $\odot$.   
  Assume that the underlying category $\ms{M}$ is equipped with a
  model structure turning it into a model category such that the
  set of cofibrations (resp.\ of 
  trivial cofibrations) is stable under all shifts and under the
  functor $\iCone(-)$.  Then, for each $K \in \KK$, the functor
  $(- \odot K) \colon \ms{M} \ra \ms{M}$ preserves cofibrations
  (resp.\ trivial cofibrations).
\end{lemma}

\begin{proof}
  Lemma~\ref{l:objects-C(k)}
  tells us that any object of $\KK$ is a coproduct of 
  shifts of $\kk$ and $\Cone(\id_\kk)=\iCone(\kk)$.
  Since $(- \odot -)$ commutes with coproducts 
  and shifts in each argument
  and 
  since the set of (trivial) cofibrations in $\ms{M}$ is closed
  under 
  coproducts and shifts
  it is enough to prove the
  claim in the two cases $K=\kk$ and $K=\iCone(\kk)$.
  The claim for $K=\kk$ is trivial because $(-\odot \kk) \cong
  \id$, and the claim for $K=\iCone(\kk)$ is clear by assumption
  because $(- \odot \iCone(\kk)) \cong
  \iCone(-)$. 
\end{proof}

\begin{remark}
  \label{rem:existence-enriched-functorial-fact}
  In most of our applications of 
  Theorem~\ref{t:existence-enriched-functorial-fact}, 
  $\ulms{M}$ is some category $\ul\C(\mathcal{A})$
  of complexes in an abelian $\kk$-category $\mathcal{A}$ with
  $\mathscr{U}$-small $\Hom$-sets, the 
  weak equivalences
  of the considered model structure on $\ms{M}=\C(\mathcal{A})$
  are the quasi-isomorphisms, and the cofibrations and trivial
  cofibrations are obviously stable under all shifts. 
  In this case the functor $\iCone$ maps any morphism to a
  quasi-isomorphism because it maps any object to a contractible
  and hence acyclic object.
  Therefore it is enough to check that the cofibrations are stable
  under the functor $\iCone(-)$.
  In Lemma~\ref{l:cone-ii-trivial-cofibration} we give a
  criterion ensuring this condition.
\end{remark}


Note that the notion of a Grothendieck abelian category depends
on a universe: a 
Grothendieck abelian category with respect to the universe
$\mathscr{U}$ has 
$\mathscr{U}$-small Hom-sets, has $\mathscr{U}$-small coproducts,
and all $\mathscr{U}$-small filtered colimits are exact.

\begin{lemma}
  \label{l:cone-ii-trivial-cofibration}
  Let $\mathcal{A}$ be a Grothendieck abelian
  category (with respect to the universe $\mathscr{U}$)
  and let a model
  structure on $\C(\mathcal{A})$ be 
  given turning it into a model category whose weak
  equivalences are the quasi-isomorphisms. 
  Assume that the cofibrant objects are stable under all shifts
  and that
  a morphism
  is a cofibration if and only if it is a monomorphism with
  cofibrant cokernel (resp.\ a degreewise split monomorphism with
  cofibrant 
  cokernel). 
  Then $\iCone$ maps cofibrations to trivial cofibrations.
 \end{lemma}

\begin{proof}
  Let $i \colon A \ra B$ be a
  cofibration. 
  The obvious commutative diagram
  \begin{equation}
    \label{eq:50a}
    \xymatrix{
      {A} \ar[r] \ar[d]^-{i} & 
      {\iCone(A)} \ar[d]^-{\iCone(i)}\\
      {B} \ar[r] & 
      {\iCone(B)}
    }
  \end{equation}
  induces a morphism
  \begin{equation}
    \label{eq:41a}
    B \;\subsqcup{A}\; \iCone(A) \ra \iCone(B).
  \end{equation}
  Since $i$ is a monomorphism we have $B \;\subsqcup{A}\;
  \iCone(A) = \Cone(i)$. We obtain the factorization
  \begin{equation}
    \label{eq:87a}
    \iCone(i) \colon \Cone(\id_A)
    \xrightarrow[{=i \oplus \id}]{\Cone(\id_A,i)} 
    \Cone(i)
    \xrightarrow[{=\id \oplus [1]i}]{\Cone(i, \id_B)}
    \Cone(\id_B).
  \end{equation}
  The morphism $\Cone(\id_A, i)$ is a pushout of $i$ and hence a
  cofibration. The morphism $\Cone(i, \id_B)$ is a cofibration
  because it is a monomorphism (resp.\ degreewise split
  monomorphism) 
  whose cokernel is the cofibrant object $[1]\Kokern(i)$.
  This implies that $\iCone(i)$ is a cofibration. It is trivial
  because source and target are contractible.
\end{proof}

\begin{corollary}
  \label{c:cone-ii-trivial-cofibration}
  Under the assumptions of
  Lemma~\ref{l:cone-ii-trivial-cofibration},
  given a (trivial) cofibration $i \colon A \ra B$, the induced
  morphism  
  $B \;\subsqcup{A}\; \iCone(A) \ra \iCone(B)$ is a (trivial)
  cofibration. 
\end{corollary}

\begin{proof}
  Clear from the proof of 
  Lemma~\ref{l:cone-ii-trivial-cofibration}. Note that
  $B \;\subsqcup{A}\;
  \iCone(A) = \Cone(i)$ is acyclic if $i$ is trivial.
\end{proof}

\section{Applications}
\label{sec:applications}

We recall some model structures and provide several examples
where 
Theorem~\ref{t:existence-enriched-functorial-fact} is
applicable. 
The examples concerning
modules over dg $\kk$-categories in \ref{sec:modules-over-dg}
and
sheaves on $\kk$-ringed spaces 
in \ref{sec:ringed-spaces}
are not relevant for the rest of this article. 
Let $\mathscr{U} \in \mathscr{V}$ be universes.

\subsection{Grothendieck abelian categories}
\label{sec:groth-abel-categ}

\fussnote{
  Could work over $\R$ and $\RR$ here instead of starting with
  $\bZ$ and $\C(\bZ)$...
}

Let $\mathcal{A}$ be a Grothendieck abelian category with respect
to the universe $\mathscr{U}$. 
Then
$\C(\mathcal{A})$ is a Grothendieck abelian category; in
particular, it is bicomplete. 

\subsubsection{Injective model structure}
\label{sec:inject-model-struct-1}


Recall that a complex $I$ in $\mathcal{A}$ is called
h-injective
if $[\ul\C_\mathcal{A}](A,I)=0$ for any acyclic complex $A$ in
$\mathcal{A}$. 
The \define{injective model structure} or 
\define{$\II$-model structure} on $\C(\mathcal{A})$
(considered as an ordinary category)
is
the following triple 
of sets of morphisms in
$C(\mathcal{A})$
(see
\cite[Prop.~3.13]{beke-sheafifiable-homotopy-model-categories}
or
\cite[Cor.~7.1]{gillespie-kaplansky-classes}).
\begin{itemize}
\item the weak equivalences are the quasi-isomorphisms;
\item the cofibrations are the monomorphisms;
\item the fibrations are the epimorphisms with
  h-injective and componentwise injective kernel.
\end{itemize}
This model structure turns $\C(\mathcal{A})$ into a cofibrantly
generated model category.  We call the
(co)fibrations and (co)fibrant objects of this model structure
$\II$-(co)fibrations and $\II$-(co)fibrant objects, respectively.
The $\II$-fibrant objects are precisely the h-injective and
componentwise injective objects of $\C(\mathcal{A})$; they are
sometimes called dg-injective. 
Let 
$\ul\II(\mathcal{A}) \subset \ul\C(\mathcal{A})$ denote the full
subcategory of these objects. It is a strongly pretriangulated
$\C(\bZ)$-category. 
All objects of $\C(\mathcal{A})$
are $\II$-cofibrant.

\begin{remark}
  \label{rem:I-fibrant-for-right-derived}
  Right derived functors can be computed using $\II$-fibrant
  resolutions.
\end{remark}

\begin{lemma}
  \label{l:cone-II-model}
  Let $\mathcal{A}$ be a Grothendieck abelian category.
  Then $\iCone$ maps $\II$-co\-fi\-bra\-tions to trivial
  $\II$-cofibrations.
\end{lemma}

\begin{proof}
  Immediate from Lemma~\ref{l:cone-ii-trivial-cofibration}.
\end{proof}

\subsubsection{Injective enhancements}
\label{sec:inject-enhanc}

Let $\R$ be a commutative ring and $\mathcal{A}$ a Grothendieck
abelian $\R$-category. The
strongly pretriangulated
$\RR$-category $\ul\II(\mathcal{A})$
together with the obvious 
equivalence
\begin{equation}
  \label{eq:24}
  [\ul\II(\mathcal{A})] \sira \D(\mathcal{A})
\end{equation}
of triangulated
$\R$-categories
is an $\RR$-enhancement of $\D(\mathcal{A})$. 
We call it the
\define{injective enhancement} or \define{$\II$-en\-hance\-ment}
of $\D(\mathcal{A})$. 

\subsubsection{Dg
\texorpdfstring{$\kk$}{k}-enriched functorial factorizations for
  Grothendieck abelian \texorpdfstring{$\kk$}{k}-categories}
\label{sec:enrich-funct-fact-5}


\begin{theorem}
  \label{t:enriched-functorial-fact-grothendieck-abelian}
  Let $\mathcal{A}$ be a Grothendieck abelian $\kk$-category
  where $\kk$ is a field. 
  Assume that the $\KK$-category $\ul\C(\mathcal{A})$ is
  $\KK$-bicomplete.  
  Then the $\II$-model structure 
  on $\C(\mathcal{A})$
  admits $\KK$-enriched functorial factorizations.
\end{theorem}

\begin{proof}
  Obviously, $\ul\C(\mathcal{A})$ is strongly
  pretriangulated. 
  The $\II$-model structure on $\C(\mathcal{A})$ is
  cofibrantly generated and its (trivial) cofibrations are stable
  under all shifts. Together with Lemma~\ref{l:cone-II-model}
  this shows that all assumptions of 
  Theorem~\ref{t:existence-enriched-functorial-fact} are
  satisfied. 
\end{proof}

\subsubsection{A counterexample}
\label{sec:counterexample}

Let $\Mod(\bZ)$ be the category of $\mathscr{U}$-small abelian
groups. We consider 
the injective model structure on $\C(\bZ)=\C(\Mod(\bZ))$. Let
$\II(\bZ) \subset \C(\bZ)$ be the subcategory of $\II$-fibrant
objects. 
The categories $\C(\Ab)$ and $\II(\Ab)$ are additive.

Note that $\ul\C(\bZ)$ is $\ul\C(\bZ)$-bicomplete (e.\,g.\ by
Lemma~\ref{l:ModC-tensored-and-cotensored-K-bicomplete} below).
Therefore the following result shows that
Theorems~\ref{t:existence-enriched-functorial-fact}
and \ref{t:enriched-functorial-fact-grothendieck-abelian}
with the field $\kk$ replaced 
by the integers $\bZ$ are not true.

\begin{lemma}
  \label{l:no-additive-injective-repl-for-Ab}
  There is no pair $(\ii, \iotaii)$ 
  where $\ii$ is an \textbf{additive} functor $\C(\bZ) \ra
  \II(\bZ)$  
  and $\iotaii$ is a natural transformation $\id \ra \ii$ of
  functors such
  that for each $M \in \C(\bZ)$ the morphism $\iotaii_M \colon M
  \ra \ii(M)$ is a 
  quasi-isomorphism (cf.\ \ref{sec:model-categ-enrich}).
 
  In particular, the $\II$-model structure on $\C(\bZ)$ does not
  admit $\C(\bZ)$-enriched functorial factorizations, and there
  is no $\C(\bZ)$-enriched fibrant replacement functor.
\end{lemma}

\begin{proof}
  Assume that such a pair $(\ii, \iotaii)$ exists.
  Consider $T:=\bZ/2\bZ$ as an object of $\C(\bZ)$ sitting in
  degree zero. Then
  \begin{equation}
    \label{eq:80}
    2\id_{\ii(T)}=2\ii(\id_T)=\ii(2\id_T)=\ii(0)=0,
  \end{equation}
  so multiplication by $2$ is the zero map on $\ii(T)$.  Since
  all components of $\ii(T)$ are injective and therefore
  divisible, multiplication by $2$ is surjective on $\ii(T)$.
  Hence $\ii(T)=0$ contradicting the assumption that $\iotaii_T$
  is a quasi-isomorphism.
\end{proof}

\begin{remark}
  \label{rem:no-additive-injective-repl-for-Ab}
  The above proof essentially boils down to the fact that there
  is no additive functor $\Mod(\bZ) \ra \Mod(\bZ)^{[1]}$ mapping
  an abelian group $A$ to a monomorphism $A \hra I_A$ into an
  injective abelian group $I_A$ (see \cite[Lemma
  3]{olaf-barcelona}). 
\end{remark}

\subsubsection{Drinfeld quotients}
\label{sec:drinfeld-quotients}

We use the name
\define{Drinfeld quotient} for Drinfeld's 
notion of a quotient in
\cite[1.2]{drinfeld-dg-quotients-arxiv-newer-than-published}. 

\begin{lemma}
  \label{l:strong-drinfeld-quotient}
  Let 
  $\ulms{A}$
  be a pretriangulated $\mathscr{U}$-small
  $\RR$-category 
  and
  $\ulms{B}$ a full pretriangulated $\RR$-subcategory.
  Let 
  $F \colon \ulms{A} \ra \ulms{C}$ 
  be an $\RR$-functor 
  to a pretriangulated $\mathscr{U}$-small $\RR$-category $\ulms{C}$.
  Then the following statements are equivalent:
  \begin{enumerate}
  \item the functor $[F] \colon [\ulms{A}] \ra [\ulms{C}]$
    of triangulated $\R$-categories
    (see Remark~\ref{rem:htpy-cat-of-pretriangulated-is-triang})
    vanishes on $[\ulms{B}]$ and the induced functor
    $[\ulms{A}]/[\ulms{B}] \ra [\ulms{C}]$ is an equivalence of
    triangulated 
    $\R$-categories;
  \item the functor $F \colon \ulms{A} \ra \ulms{C}$ (more
    precisely, the diagram $\ulms{A} \xla{\id} \ulms{A} \xra{F}
    \ulms{C}$) is a Drinfeld quotient of $\ulms{A}$ by $\ulms{B}$.
  \end{enumerate}
\end{lemma}

\begin{proof}
  Clear from the definition of a Drinfeld quotient.
\end{proof}

Recall the universe $\mathscr{V}$ that contains $\mathscr{U}$ as
an element.

\begin{lemma}
  \label{l:ii-Drinfeld-quotient}
  Let $\kk$ be a field and
  let $\mathcal{A}$ be a Grothendieck abelian $\kk$-category
  (with respect to the universe $\mathscr{U}$)
  with objects in $\mathscr{U}$, and assume that
  $\ul\C(\mathcal{A})$ is 
  $\KK$-bicomplete. Let $\ii \colon \ul\C(\mathcal{A}) \ra
  \ul\II(\mathcal{A})$ be a $\KK$-enriched $\II$-fibrant replacement
  functor (which exists by 
  Theorem~\ref{t:enriched-functorial-fact-grothendieck-abelian}).
  Then $\ii$ 
  is a
  Drinfeld quotient of the pretriangulated $\mathscr{V}$-small
  $\KK$-category 
  $\ul\C(\mathcal{A})$ by its 
  full $\KK$-subcategory $\ul\C_\ac(\mathcal{A})$ of acyclic
  objects. In particular, $\ii$ 
  induces a unique equivalence
  \begin{equation}
    \label{eq:Drinfeld-equiv}
    \ol{[\ii]} \colon \D(\mathcal{A}) \sira [\ul\II(\mathcal{A})]
  \end{equation}
  of triangulated $\kk$-categories such that the diagram 
  \begin{equation}
    \label{eq:64}
    \xymatrix{
      {[\ul\C(\mathcal{A})]} \ar[r]^-{[\ii]} \ar[d] 
      &
      {[\ul\II(\mathcal{A})]} 
      \\
      {\D(\mathcal{A})} \ar[ru]^-{\ol{[\ii]}}_-{\sim} 
    }
  \end{equation}
  commutes.
  This equivalence has \eqref{eq:24} as a quasi-inverse.
\end{lemma}

\begin{proof}
  The $\KK$-categories $\ul\C(\mathcal{A})$,
  $\ul\C_\ac(\mathcal{A})$,
  $\ul\II(\mathcal{A})$ are clearly $\mathscr{V}$-small (since
  they have objects in $\mathscr{U}$) and
  pretriangulated. 
  Since an acyclic h-injective complex is contractible, $[\ii]
  \colon [\ul\C(\mathcal{A})] \ra 
  [\ul\II(\mathcal{A})]$ maps acyclic objects to zero objects,
  and therefore induces a 
  unique functor 
  $\ol{[\ii]} \colon \D(\mathcal{A}) = 
  [\ul\C(\mathcal{A})]/[\ul\C_\ac(\mathcal{A})]
  \ra [\ul\II(\mathcal{A})]$
  by the universal property of the Verdier quotient.
  The usual construction of a quasi-inverse to a
  fully faithful essentially surjective functor can be used to
  show that 
  $\ol{[\ii]}$ is quasi-inverse 
  to \eqref{eq:24}.
  Now apply Lemma~\ref{l:strong-drinfeld-quotient} using the
  universe $\mathscr{V}$.
\end{proof}

\subsection{Ringed sites}
\label{sec:ringed-sites}

Let $(\mathcal{X},
\mathcal{O}=\mathcal{O}_\mathcal{X})$ be a ringed site 
\citestacks{04KQ}. We assume that it is $\mathscr{U}$-small,
i.\,e.\ its underlying category is $\mathscr{U}$-small, all
coverings of objects are $\mathscr{U}$-sets, and the sheaf
$\mathcal{O}_\mathcal{X}$ takes values in the category of
$\mathscr{U}$-small commutative rings.
We write $\Sh(\mathcal{X})$ for the category of sheaves of
$\mathscr{U}$-small sets
on the site $\mathcal{X}$. 
The category $\Mod(\mathcal{X})$ of sheaves of
$\mathscr{U}$-small $\mathcal{O}$-modules 
has $\mathscr{U}$-small $\Hom$-sets and objects in $\mathscr{U}$ 
(by Remark~\ref{rem:functor-categories-size})
and
is a Grothendieck abelian category 
\citestacks{07A5} with respect to the universe $\mathscr{U}$. We
usually say $\mathcal{O}$-module instead of 
sheaf of $\mathscr{U}$-small $\mathcal{O}$-modules. 
We abbreviate $\Hom_\mathcal{X}=\Hom_{\Mod(\mathcal{X})}$ and
$\otimes=\otimes_{\mathcal{O}}$ and
$\sheafHom=\sheafHom_{\mathcal{O}}$, and
write $\C(\mathcal{X})$ instead of $\C(\Mod(\mathcal{X}))$, and
$\C_\mathcal{X}(M,N)=\C_{\Mod(\mathcal{X})}(M,N)$, and similarly
for $\ul\C(\mathcal{X})$, 
$[\ul\C(\mathcal{X})]$, and $\D(\mathcal{X})$.
Sometimes we write $\Mod(\mathcal{X}, \mathcal{O})$,
\dots{}
instead of
$\Mod(\mathcal{X})$, \dots{} in order to avoid ambiguity.
All sites and ringed sites in the rest of this subsection are
assumed to be $\mathscr{U}$-small. 

\subsubsection{Injective model structure}
\label{sec:inject-model-struct}
   
If $(\mathcal{X},
\mathcal{O})$ is a ringed site,
$\C(\mathcal{X})$ carries
the injective 
model structure from \ref{sec:inject-model-struct-1}.
We abbreviate 
$\ul\II(\mathcal{X})=\ul\II(\Mod(\mathcal{X}))$.

\subsubsection{Extension-by-zero model structure}
\label{sec:extension-zero-model}

Following
\cite{cisinski-deglise-local-and-stable-homological-algebra}
we introduce a model structure on $\C(\mathcal{X})$ which
is useful for computing derived tensor products and
derived pullbacks.

Let $(\mathcal{X}, \mathcal{O})$ be a ringed site.
For $U \in \mathcal{X}$ 
the continuous and cocontinuous functor $j_U \colon
\mathcal{X}/U \ra \mathcal{X}$ \citestacks{00XZ} gives rise to a
morphism of ringed topoi
\begin{equation}
  \label{eq:103}
  (j_U, j_U^\sharp:=\id) \colon (\Sh(\mathcal{X}/U),
  \mathcal{O}_U:=\mathcal{O}_{\mathcal{X}/U}:=j\inv 
  \mathcal{O}) \ra (\Sh(\mathcal{X}), \mathcal{O}). 
\end{equation}
The
extension by zero
$j_{U!}\mathcal{O}_U 
\in \Mod(\mathcal{X})$
is a flat $\mathcal{O}$-module \citestacks{03EV}.

\begin{theorem}
  \label{t:E-model-structure}
  Let $(\mathcal{X}, \mathcal{O})$ be a 
  ringed site. 
  There is a model structure on
  $\C(\mathcal{X})$ 
  turning it into a 
  cofibrantly generated model category 
  whose weak equivalences are the
  quasi-isomorphisms and whose cofibrations are 
  the smallest set of morphisms in $\C(\mathcal{X})$ 
  that contains all morphisms of the form
  \begin{equation}
    \label{eq:433}
    [n]j_{U!}\mathcal{O}_U  
    \ra 
    \iCone([n]j_{U!}\mathcal{O}_U),
    \quad \quad \text{for $U \in \mathcal{X}$ and $n \in \bZ$,}
  \end{equation}
  and is closed under
  pushouts, retracts, and transfinite compositions.
\end{theorem}

\begin{proof}
  This follows from \cite[Ex.~2.3,
  Thm.~2.5]{cisinski-deglise-local-and-stable-homological-algebra}.
  Note that the modern definition of a hypercover used there (see
  e.\,g.\
  \cite{dugger-hollander-isaksen-hypercovers-simpl-presheaves})
  does not require that the site has fiber products.
\end{proof}

We call the model structure of Theorem~\ref{t:E-model-structure}
the \define{extension-by-zero model structure} or
\define{$\EE$-model structure} and call its (co)fibrations and
(co)fibrant objects $\EE$-(co)fibrations and $\EE$-(co)fibrant
objects, respectively. 
Note that $j_{U!}\mathcal{O}_U$ and all
its shifts are $\EE$-cofibrant.
Let 
$\ul\EE(\mathcal{X}) \subset \ul\C(\mathcal{X})$ denote the full
subcategory of $\EE$-cofibrant objects.

\begin{lemma}
  \label{l:from-CD}
  Let $(\mathcal{X}, \mathcal{O})$ be a
  ringed site. 
  Then $\iCone$ maps $\EE$-cofibrations to trivial
  $\EE$-cofibrations.
\end{lemma}

\begin{proof}
  Let $i \colon A \ra B$ be an $\EE$-cofibration.
  Consider the factorization
  \begin{equation}
    \label{eq:87aa}
    \iCone(i) \colon \iCone(A) \ra B \;\subsqcup{A}\; \iCone(A)
    \ra \iCone(B) 
  \end{equation}
  obtained from the commutative diagram \eqref{eq:50a}.  As a
  pushout of an $\EE$-cofibration, the first morphism is a
  $\EE$-cofibration. The second morphism is an
  $\EE$-cofibration, by the proof of
  \cite[Lemma~2.7]{cisinski-deglise-local-and-stable-homological-algebra}.
  Therefore the composition is an
  $\EE$-cofibration; it has acyclic source and target and is
  therefore trivial.
\end{proof}


\begin{lemma}
  \label{l:characterize-E-cofibrations}
  Let $(\mathcal{X}, \mathcal{O})$ be a ringed site.
  A morphism $A \ra B$ in $\C(\mathcal{X})$ is an
  $\EE$-cofibration if and only if it is a degreewise split
  monomorphism with $\EE$-cofibrant cokernel.
\end{lemma}

\begin{proof}
  The set of all degreewise split monomorphisms in
  $\C(\mathcal{X})$ is closed under push\-outs, retracts, and
  transfinite compositions (by Grothendieck's AB5
  condition). Since it contains all morphisms
  in \eqref{eq:433} any $\EE$-cofibration is a degreewise split
  monomorphism. If $i \colon A \ra B$ is an $\EE$-cofibration,
  its pushout along $A \ra 0$ is the $\EE$-cofibration $0 \ra
  \Kokern(i)$. Hence $\Kokern(i)$ is $\EE$-cofibrant.

  Conversely, let $i \colon A \ra B$ be a degreewise split
  monomorphism with 
  $\EE$-cofibrant cokernel $\Kokern(i)$.
  The shift $C:=[-1] \Kokern(i)$ is also $\EE$-cofibrant.
  Since $i$ is degreewise split, 
  $i$ is isomorphic to  
  the obvious inclusion $A \hra A \oplus [1]C$ in
  $\Mod(\mathcal{X})^\bZ$, and
  there 
  is a morphism $m \colon C \ra A$ in $\C(\mathcal{X})$ 
  such that $i$ is isomorphic to the canonical morphism
  $A \ra \Cone(m)$.
  But 
  \begin{equation}
    \label{eq:22}
    \xymatrix{
      {C} \ar[d] \ar[r]^-{m} 
      & 
      {A} \ar[d]\\
      {\iCone(C)} \ar[r]^-{m \oplus \id} 
      &
      {\Cone(m)}
    }
  \end{equation}
  is a pushout diagram whose
  left vertical arrow is 
  an $\EE$-cofibration: in the proof of 
  Lemma~\ref{l:from-CD} we have seen that the morphism 
  $B' \;\subsqcup{A'}\; \iCone(A') \ra \iCone(B')$
  is an $\EE$-cofibration if $A' \ra B'$ is an $\EE$-cofibration;
  apply this to $0 \ra C$. This implies that $A \ra \Cone(m)$ is
  an $\EE$-cofibration, and so is $i$.
\end{proof}



\begin{lemma}
  \label{l:E-cofibrant-(pullback)-h-flat}
  Let $(\mathcal{X}, \mathcal{O})$ be a ringed site.
  If $E$ 
  is an $\EE$-cofibrant object, then
  $E$ is an h-flat complex of $\mathcal{O}$-modules with flat
  components.   
\end{lemma}

\begin{proof}
  All morphisms in \eqref{eq:433} are degreewise split
  monomorphisms with h-flat and componentwise flat cokernel
  $[n+1]j_{U!}\mathcal{O}_U$.
  Therefore the set of all degreewise split monomorphisms
  $A \xra{m} B$ in $\C(\mathcal{X})$ such that
  $\Kokern(m)$ is h-flat and componentwise flat
  contains
  all morphisms in \eqref{eq:433}. Moreover, it is closed under
  pushouts, 
  retracts, and transfinite compositions (by Grothendieck's AB5
  condition). Therefore it contains all $\EE$-cofibrations and in
  particular $0 \ra E$.
\end{proof}

\begin{remark}
  \label{rem:E-cofibrant-for-tensor}
  Lemma~\ref{l:E-cofibrant-(pullback)-h-flat}
  implies that derived tensor
  products and derived pullbacks can be computed using
  $\EE$-cofibrant resolutions, cf.\
  Propositions~\ref{p:spalt-6.1+5-sites}
  and \ref{p:spalt-6.7a-sites}.
\end{remark}


\subsubsection{The final topos}
\label{sec:final-topos}

Let $\pt$ be the punctual
site. 
Sheaves and presheaves of $\mathscr{U}$-small sets on $\pt$ are
just $\mathscr{U}$-small sets. 
Let $\mathcal{X}$ be a site.
Consider the morphism of topoi
\begin{equation}
  \label{eq:53a}
  \sigma = (\sigma\inv, \sigma_*) \colon \Sh(\mathcal{X}) \ra
  \Sh(\pt)=\Sets_\mathscr{U}, 
\end{equation}
where 
$\sigma\inv \colon \Sets_\mathscr{U}=\Sh(\pt) \ra
\Sh(\mathcal{X})$ is the 
exact and colimit preserving functor that maps a set $S$ to the constant sheaf
$\sigma\inv S$ with value $S$,
and where
$\sigma_* = \Gamma= \Gamma(\mathcal{X},-) \colon \Sh(\mathcal{X})
\ra 
\Sh(\pt)=\Sets_\mathscr{U}$ is the global section functor, see
\citestacks{06UN}. Up to unique isomorphism, $\sigma$ is the
unique morphism of topoi $\Sh(\mathcal{X}) \ra
\Sh(\pt)=\Sets_\mathscr{U}$, 
by \cite[Exp.~IV, 4.3]{SGA4-1}.

\subsubsection{Ringed sites and topoi over a ring}
\label{sec:ringed-sites-over}

Let $\R$ be a commutative ring.

\begin{definition}
  \label{d:k-ringed-site}
  An 
  \define{$\R$-ringed site} 
  or
  \define{ringed site over $\R$} 
  is a ringed site $(\mathcal{X}, \mathcal{O})$ such that
  $\mathcal{O}$ is a sheaf of commutative $\R$-algebras.
  The category of
  $\R$-ringed sites is defined in the obvious way, cf.\
  \citestacks{04KQ}. 
\end{definition}

For example, the ringed site $(\pt, \R)$ is an $\R$-ringed site.
A sheaf of $\mathscr{U}$-small $\R$-modules on $(\pt, \R)$ is a
$\mathscr{U}$-small $\R$-module, 
so
$\Mod(\pt, \R)=\Mod(\R)$ and $\C(\pt)=\RR=\C(\R)$.

If $(\mathcal{X}, \mathcal{O})$ is an $\R$-ringed site there is a canonical 
morphism 
$\sigma^\sharp \colon \sigma\inv \R \ra \mathcal{O}$
of sheaves of rings turning 
\eqref{eq:53a}
into a
morphism
\begin{equation}
  \label{eq:56}
  (\sigma, \sigma^\sharp) \colon (\Sh(\mathcal{X}), \mathcal{O})
  \ra 
  (\Sh(\pt), \R)
\end{equation}
of ringed topoi. 
Conversely, if $(\mathcal{X}, \mathcal{O})$ is a ringed
site and 
$\rho \colon \sigma\inv\R \ra \mathcal{O}$ is a 
morphism of
sheaves of rings, then $(\mathcal{X}, \mathcal{O})$ becomes
canonically an $\R$-ringed site. These two constructions are
obviously inverse to each other. 
\fussnote{
  More precisely, the category of $\R$-ringed sites is equivalent
  to the category whose objects are pairs consisting of a
  ringed site $(\mathcal{X},
  \mathcal{O})$ and a morphism $(\Sh(\mathcal{X}),
  \mathcal{O}) \ra (\Sh(\pt), \R)$ of ringed topoi, and whose
  morphisms are morphisms of ringed sites such that the obvious
  diagram of morphisms of ringed topoi commutes.
  
  Analog sollte gehen:
  The category of $\R$-ringed topoi is equivalent ...
}

\begin{definition}
  \label{d:k-ringed-topos}
  An \define{$\R$-ringed topos} or 
  \define{ringed topos over $\R$}
  is a pair $(\Sh(\mathcal{X}), \mathcal{O})$ where
  $(\mathcal{X}, \mathcal{O})$ is a
  $\R$-ringed site. 
  The category of
  $\R$-ringed topoi is defined in the obvious way, cf.\
  \citestacks{01D3}.  
\end{definition}

\subsubsection{Enriched completeness}
\label{sec:enrich-compl}


If $(\mathcal{X}, \mathcal{O})$ is a ringed site over 
$\R$ then 
$\ul\C(\mathcal{X})$ is an $\RR$-category with underlying
category $\C(\mathcal{X})$.  

\begin{lemma}
  \label{l:CX-tensored-and-cotensored-R-bicomplete}
  Let $(\mathcal{X}, \mathcal{O})$ be a ringed
  site over $\R$.
  Then $\ul\C(\mathcal{X})$ is tensored and cotensored, and 
  $\RR$-bicomplete. 
\end{lemma}

\begin{proof}
  The morphism \eqref{eq:56} of ringed topoi associated to
  $(\mathcal{X}, \mathcal{O})$ gives rise to an adjoint pair
  $(\sigma^*, \sigma_*)$ of $\RR$-functors
  $\sigma^* \colon 
  \ul\RR 
  \rightleftarrows
  \ul\C(\mathcal{X}) \colon \sigma_*$.
  Using this adjunction and the
  $(\otimes, \sheafHom)$-adjunction we obtain natural isomorphisms
  \begin{equation}
    \label{eq:9}
    \ul\RR(R, \sigma_* \sheafHom(M,N)) \cong
    \ul\C_\mathcal{X}(M \otimes \sigma^*R, N)
    \cong \ul\C_\mathcal{X}(M, \sheafHom(\sigma^*R,N))
  \end{equation}
  in  
  $R \in \RR$ and $M,N \in \C(\mathcal{X})$.
  Observe that
  \begin{equation}
    \label{eq:15}
    \ul\C_\mathcal{X}(M,N)=\sigma_* \sheafHom(M,N).
  \end{equation}
  Therefore the isomorphisms \eqref{eq:9} show that
  the $\RR$-category
  $\ul\C(\mathcal{X})$ is tensored and cotensored (cf.\ 
  \eqref{eq:tensored}, \eqref{eq:cotensored}).
  Note that $\C(\mathcal{X})$ is a Grothendieck abelian category
  and therefore bicomplete. This implies that $\ul\C(\mathcal{X})$
  is $\RR$-bicomplete, by 
  \ref{sec:tens-cotens-compl}.
\end{proof}

\begin{remark}
  The fact that $\ul\C(\mathcal{X})$ is tensored
  and cotensored can more formally be deduced as follows.
  The tensor product $\otimes$ turns both categories
  $\C(\mathcal{X})$ and $\RR$ into closed
  symmetric monoidal categories.
  In particular, these two categories are
  enriched, tensored and cotensored over themselves.
  The functor $\sigma^* \colon \RR \ra \C(\mathcal{X})$ is
  strong monoidal by \citestacks{03EL}.
  Therefore 
  $\C(\mathcal{X})$ canonically becomes a tensored and cotensored
  $\RR$-category $\ul\C(\mathcal{X})$, by
  \cite[Thm.~3.7.11]{riehl-categorical-homotopy-theory}.
\end{remark}

\subsubsection{Dg
\texorpdfstring{$\kk$}{k}-enriched functorial factorizations for
\texorpdfstring{$\kk$}{k}-ringed sites}
\label{sec:enrich-funct-resol-for-k-ringed-sites}

\begin{theorem}
  \label{t:enriched-functorial-fact-k-ringed-site}
  Let $(\mathcal{X}, \mathcal{O})$ be a $\mathscr{U}$-small ringed
  site over a field $\kk$.
  Then the $\II$-model structure and the $\EE$-model
  structure on 
  $\C(\mathcal{X})$
  admit $\KK$-enriched functorial factorizations.
\end{theorem}

\begin{proof}
  Obviously, $\ul\C(\mathcal{X})$ is strongly
  pretriangulated. It is $\KK$-bicomplete by 
  Lemma~\ref{l:CX-tensored-and-cotensored-R-bicomplete}.
  The claim for the $\II$-model structure follows from
  Theorem~\ref{t:enriched-functorial-fact-grothendieck-abelian},
  and the same line of argument works for the $\EE$-model
  structure: 
  it is
  cofibrantly generated, its (trivial) cofibrations are stable
  under all shifts, and Lemma \ref{l:from-CD} holds; therefore 
  Theorem~\ref{t:existence-enriched-functorial-fact} applies.
\end{proof}


\subsection{Modules over dg categories}
\label{sec:modules-over-dg}

Recall our notation for module categories over $\RR$-categories
from 
\ref{sec:categ-dg-modul}.

\subsubsection{Model structures}
\label{sec:model-structures}

Let $\ulms{C}$ be an $\RR$-category. We always assume that it is 
$\mathscr{U}$-small in the rest of this subsection.
The category $\MMod(\ulms{C})$ has $\mathscr{U}$-small
$\Hom$-sets, objects in $\mathscr{U}$, and all $\mathscr{U}$-small
colimits and limits. It has two well-known model
structures whose weak equivalences are the quasi-isomorphisms
(\cite[Thm.~3.2]{keller-on-dg-categories-ICM}, \cite[Thm.~2.2]{valery-olaf-smoothness-equivariant}):
\begin{enumerate}
\item The injective model structure whose cofibrations are the
  monomorphisms. We call this model structure the
  \define{$\II$-model structure} and its 
  (co)fibrations and
  (co)fibrant objects $\II$-(co)fibrations and $\II$-(co)fibrant
  objects, respectively.
\item The projective model structure whose fibrations are
  the epimorphisms. We call this model structure the
  \define{$\PP$-model structure} and its 
  (co)fibrations and
  (co)fibrant objects $\PP$-(co)fibrations and $\PP$-(co)fibrant
  objects, respectively. 
\end{enumerate}
Each of these model structures turns $\MMod(\ulms{C})$ into a
cofibrantly generated model category.
Any $\II$-fibrant object is h-injective,
and all objects are $\II$-cofibrant. 
Any $\PP$-cofibrant object is h-projective 
and h-flat
(\cite[Lemmas~2.6, 2.8]{valery-olaf-smoothness-equivariant}),
and all objects are $\PP$-fibrant.
We write $\ul\IIMod(\ulms{C})$ (resp.\
$\ul\PPMod(\ulms{C})$)
for the full $\RR$-subcategory of
$\ul\MMod(\ulms{C})$ consisting of $\II$-fibrant (resp.\
$\PP$-cofibrant) objects.

\begin{lemma}
  \label{l:cone-for-II-and-PP-model-for-dg-modules}
  Let $\ulms{C}$ be an $\RR$-category.  
  Then $\iCone$ maps $\II$-cofibrations (resp.\
  $\PP$-cofibrations) to trivial 
  $\II$-cofibrations (resp.\ trivial $\PP$-cofibrations).
\end{lemma}

\begin{proof}
  The $\II$-cofibrations are the monomorphisms; the
  $\PP$-cofibrations are the monomorphisms with 
  $\PP$-cofibrant 
  cokernel, as explained before
  \cite[Lemma~2.4]{valery-olaf-smoothness-equivariant}.
  Now it is easy to adapt the proof of
  Lemma~\ref{l:cone-ii-trivial-cofibration}.
\end{proof}

\subsubsection{Projective and injective enhancements}
\label{sec:proj-inject-enhanc}

Let $\ulms{C}$ be an $\RR$-category. The derived category of
$\ulms{C}$-modules is denoted $\D(\ulms{C})$. 
Its
\define{projective enhancement} or \define{$\PP$-en\-hance\-ment}
is the strongly pretriangulated (see
\cite[Lemma~2.4]{valery-olaf-smoothness-equivariant}) 
$\RR$-category $\ul\PPMod(\ulms{C})$ 
together with the obvious $\R$-linear triangulated equivalence
\begin{equation}
  \label{eq:24-modules}
  [\ul\PPMod(\ulms{C})] \sira \D(\ulms{C}).
\end{equation}
Similarly, $\ul\IIMod(\ulms{C})$ is the
\define{injective enhancement} or \define{$\II$-en\-hance\-ment}
of $\D(\ulms{C})$.

\subsubsection{Enriched completeness}
\label{sec:enrich-compl-2}

\begin{lemma}
  \label{l:ModC-tensored-and-cotensored-K-bicomplete}
  Let $\ulms{C}$ be an $\RR$-category.
  Then $\ul\MMod(\ulms{C})$ is tensored and cotensored, and 
  $\RR$-bicomplete. 
\end{lemma}

\begin{proof}
  Let $\ulms{D}$ be an $\RR$-category. Let $M$ be a
  $\ulms{C}$-module, $L$ a $\ulms{D}$-module and $N$ a $\ulms{C}
  \otimes \ulms{D}$-module. 
  Then there is an isomorphism
  \begin{equation}
    \label{eq:outer-tensor-product-right}
    \ul\MMod_{\ulms{C} \otimes \ulms{D}}(M \otimes L, N) 
    \cong 
    \ul\MMod_{\ulms{C}}(M,\mathsf{rHom}_{\ulms{D}}(L,N))   
  \end{equation}
  where $\mathsf{rHom}_{\ulms{D}}(L,N)$ is the $\ulms{C}$-module
  whose evaluation at $C \in \ulms{C}$ is
  $\ul\MMod_{\ulms{D}}(L, N(C,-))$.
  Similarly, there is an isomorphism
  \begin{equation}
    \label{eq:outer-tensor-product-left}
    \ul\MMod_{\ulms{C} \otimes \ulms{D}}(M \otimes L, N) 
    \cong 
    \ul\MMod_{\ulms{D}}(L,\mathsf{lHom}_{\ulms{C}}(M,N))   
  \end{equation}
  where $\mathsf{lHom}_{\ulms{C}}(M,N)$ is the $\ulms{D}$-module
  whose evaluation at $D \in \ulms{D}$ is
  $\ul\MMod_{\ulms{C}}(M, N(-,D))$.
  Both $\mathsf{rHom}$ and $\mathsf{lHom}$ are $\RR$-functors,
  and the above two isomorphisms
  are natural in $M$, $L$ and $N$.

  Let $\ul\R$ be the $\RR$-category with one object whose
  endomorphisms are $\R$. Then $\MMod(\ul\R)=\C(\R)=\RR$ and
  $\ul\MMod(\ul\R)=\ul\C(\R)=\ul\RR$. If we apply the above to
  $\ulms{D}=\ul\R$ then
  $\ulms{C} \otimes \ul\R = \ulms{C}$ and
  $\mathsf{lHom}_{\ulms{C}}(M,N)=\ul\MMod_{\ulms{C}}(M,N)$ and
  we obtain isomorphisms
  \begin{equation}
    \label{eq:119}
    \ul\MMod_{\ulms{C}}(M, \mathsf{rHom}_{\ul\R}(R,N)) 
    \cong 
    \ul\MMod_{\ulms{C}}(M \otimes R, N) 
    \cong 
    \ul\RR(R, \ul\MMod_{\ulms{C}}(M,N))
  \end{equation}
  natural in $R \in \RR$, $M,N \in \ms{M}$.
  They show that $\ul\MMod(\ulms{C})$ is tensored and
  cotensored (cf.\ \eqref{eq:tensored}, \eqref{eq:cotensored}).
  As a Grothendieck abelian category,
  $\MMod(\ulms{C})$ is bicomplete and hence
  $\ul\MMod(\ulms{C})$ is
  $\RR$-bicomplete, by
  \ref{sec:tens-cotens-compl}.
\end{proof}

\subsubsection{Dg
\texorpdfstring{$\kk$}{k}-enriched functorial factorizations for modules over dg
\texorpdfstring{$\kk$}{k}-categories}
\label{sec:enrich-funct-fact-4}

\begin{theorem}
  \label{t:enriched-functorial-fact-modules-K-cat}
  Let $\kk$ be a field and $\ulms{C}$ a $\mathscr{U}$-small
  $\KK$-category. 
  Then the $\II$-model structure and the $\PP$-model
  structure on 
  $\MMod(\ulms{C})$ admit $\KK$-enriched functorial
  factorizations. 
\end{theorem}

\begin{proof}
  Obviously, $\ul\MMod(\ulms{C})$ is strongly
  pretriangulated. It is $\KK$-bicomplete by 
  Lemma~\ref{l:ModC-tensored-and-cotensored-K-bicomplete}.
  Both $\II$-model structure and $\PP$-model structure are
  cofibrantly generated and their (trivial) cofibrations are stable
  under all shifts.
  Lemma~\ref{l:cone-for-II-and-PP-model-for-dg-modules}
  then shows that all assumptions of 
  Theorem~\ref{t:existence-enriched-functorial-fact} are
  satisfied. 
\end{proof}

\begin{remark}
  \label{rem:ii-pp-modules-Drinfeld-quotient}
  In the setting of
  Theorem~\ref{t:enriched-functorial-fact-modules-K-cat}, 
  any
  $\KK$-enriched $\II$-fibrant (resp.\ $\PP$-cofibrant) replacement
  functor $\ii \colon \ul\MMod(\ulms{C}) \ra \ul\IIMod(\ulms{C})$
  (resp.\ $\pp \colon \ul\MMod(\ulms{C}) \ra \ul\PPMod(\ulms{C})$)
  is a Drinfeld quotient of $\ul\MMod(\ulms{C})$ by its subcategory
  of acyclic objects. This is proved in the same way as
  Lemma~\ref{l:ii-Drinfeld-quotient}.
\end{remark}

\subsection{Ringed spaces}
\label{sec:ringed-spaces}

A $\mathscr{U}$-small ringed space is a pair $(X, \mathcal{O}_X)$
consisting of a $\mathscr{U}$-small
topological space $X$ and
a sheaf 
$\mathcal{O}_X$
of $\mathscr{U}$-small commutative rings on $X$. All ringed
spaces and sites in the rest of this subsection are assumed to be
$\mathscr{U}$-small. 
Any ringed space gives rise to a ringed site in the obvious way,
and the categories of sheaves (of modules) on the ringed
space and the associated ringed site coincide canonically. 
We adapt the notation from
\ref{sec:ringed-sites} in the obvious way to ringed
spaces; for example we write $\C(X)$ for the category of
complexes of $\mathcal{O}_X$-modules.
Note that all the previous results about ringed sites apply
to ringed spaces. 

\subsubsection{Flat model structure}
\label{sec:flat-model-structure}

Let $(X, \mathcal{O})$ be a ringed space.
An object $F \in \C(X)$ is called \define{dg-flat} if 
it is h-flat and all its components $F^n$ are flat
$\mathcal{O}$-modules. 
This terminology coincides with that of
\cite{gillespie-kaplansky-classes} by
\cite[Prop.~5.6]{gillespie-ChOX}. 
The \define{flat model structure} or 
\define{$\FF$-model structure} on $\C(X)$ is
given as follows
(see
\cite[Cor.~7.8]{gillespie-kaplansky-classes}):
\begin{itemize}
\item the weak equivalences are the quasi-isomorphisms;
\item the cofibrations are the 
  monomorphisms with dg-flat cokernel.
\end{itemize}
This model structure turns $\C(X)$ into a cofibrantly generated
model category.
The $\FF$-cofibrant objects are precisely the dg-flat
objects of $\C(X)$. Let $\ul\FF(X) \subset \ul\C(X)$ denote the
full subcategory of $\FF$-cofibrant objects.

\begin{lemma}
  \label{l:cone-for-flat-model}
  Let $(X, \mathcal{O})$ be a ringed space.
  Then $\iCone$ maps $\FF$-cofibrations to trivial
  $\FF$-cofibrations.
\end{lemma}

\begin{proof}
  Apply Lemma~\ref{l:cone-ii-trivial-cofibration}.
\end{proof}

\subsubsection{Ringed spaces over a ring}
\label{sec:ringed-spaces-over}

\begin{definition}
  An \define{$\R$-ringed space} or \define{ringed space over
    $\R$} is a pair $(X, \mathcal{O}_X)$
  consisting of a topological space $X$ and a sheaf
  $\mathcal{O}_X$ of
  commutative $\R$-algebras on $X$. The category of $\R$-ringed
  spaces is 
  defined in the obvious way.
\end{definition}

If we denote the topological space with a unique point by $\pt$,
then an $\R$-ringed space can equivalently be defined as a ringed
space over the ringed space $(\pt, \R)$. 
Any $\R$-ringed space gives rise to an $\R$-ringed site and to
an $\R$-ringed topos in the obvious way.
In particular, all previous results for $\R$-ringed sites can be
applied to $\R$-ringed spaces. 
For example,
Lemma~\ref{l:CX-tensored-and-cotensored-R-bicomplete} shows that 
$\ul\C(\mathcal{X})$ is $\R$-bicomplete if $(X, \mathcal{O})$ is
an $\R$-ringed space.

\subsubsection{Dg
\texorpdfstring{$\kk$}{k}-enriched functorial factorizations for
\texorpdfstring{$\kk$}{k}-ringed spaces} 
\label{sec:enrich-funct-fact-1}

Certainly
Theorem~\ref{t:enriched-functorial-fact-k-ringed-site} 
applies to ringed spaces over a field. 
There is a similar result for the
flat model structure:

\begin{theorem}
  \label{t:enriched-functorial-fact-flat}
  Let $(X, \mathcal{O}_X)$ be a ringed space over a field 
  $\kk$.
  Then the $\FF$-model structure on
  $\C(X)$ admits $\KK$-enriched functorial
  factorizations. 
\end{theorem}

\begin{proof}
  Mutatis mutandis 
  the proof of
  Theorem~\ref{t:enriched-functorial-fact-k-ringed-site},
  using Lemma~\ref{l:cone-for-flat-model}.
\end{proof}

\subsection{Modules over rings}
\label{sec:modules-over-rings}

All rings and modules in this subsection are assumed to be $\mathscr{U}$-small.
If $A$ is a ring,
$\Mod(A)$ denotes the category of $A$-modules, and
$\ul{A}$ denotes the $\C(\bZ)$-category with one object whose
endomorphisms are $A$.

\subsubsection{Model structures on complexes of modules}
\label{sec:model-struct-compl}

Let $A$ be a commutative ring. 
Note that $\ul\MMod(\ul A)=\ul\C(\pt,
A)
=\ul\C(\Mod(A))=:\ul\C(A)$.
We have introduced several model structures on 
$\MMod(\ul A)=\C(\pt, A)
=\C(A)$ whose weak
equivalences are the quasi-iso\-mor\-phisms.
The 
$\II$-model structure on $\MMod(\ul A)$
coincides with the $\II$-model structure on
$\C(\pt, A)$
and the injective model structure on $\C(A)$
(because the cofibrations are the
monomorphisms). 
The 
$\PP$-model structure on $\MMod(\ul A)$
coincides with the $\EE$-model structure on
$\C(\pt, A)$ (because the cofibrations can be generated by the same
generating cofibrations).


\subsubsection{Dg
\texorpdfstring{$\kk$}{k}-enriched functorial factorizations for modules
  over a \texorpdfstring{$\kk$}{k}-algebra}
\label{sec:modules-over-an-1}


Let $A$ be an algebra over a field $\kk$.
Then Theorem~\ref{t:enriched-functorial-fact-modules-K-cat} 
can be applied to the $\KK$-category 
$\ul A$ with one object whose
endomorphisms are $A$;
if $A$ is commutative, 
Theorems~\ref{t:enriched-functorial-fact-k-ringed-site}
and \ref{t:enriched-functorial-fact-flat}
can be applied to the $\kk$-ringed space
$(\pt, A)$.

\subsubsection{The field case}
\label{sec:field-case}

\begin{lemma}
  \label{l:model-str-for-k}
  If $\kk$ is a field, the following model structures on 
  $\KK=\MMod(\ul\kk)=\C(\pt, \kk)
  =\C(\kk)$ 
  coincide:
  the 
  $\II$-model structure on $\MMod(\ul\kk)$,
  the $\II$-model structure on
  $\C(\pt, \kk)$, 
  the injective model structure on  
  $\C(\kk)$,
  the 
  $\PP$-model structure on $\MMod(\ul\kk)$,
  the $\EE$-model structure on
  $\C(\pt, \kk)$, 
  the 
  $\FF$-model structure on $\C(\pt, \kk)$.
  All objects of $\KK$ are fibrant and cofibrant with respect to
  any of these model structures.
\end{lemma}

\begin{proof}
  We prove first 
  that $\II$-model
  structure, $\EE$-model 
  structure and $\FF$-model structure on $\C(\pt,\kk)$ coincide.
  
  $\II$=$\EE$: 
  By Lemma~\ref{l:characterize-E-cofibrations},
  any $\EE$-cofibration is a monomorphism, i.\,e.\ an
  $\II$-cofibration. 
  Conversely, 
  Lemma~\ref{l:monos-C(k)} below
  shows that any monomorphism is an $\EE$-cofibration:
  note that $0 \hra [1]\kk$ is a pushout of the $\EE$-cofibration
  $\kk 
  \ra \iCone(\kk)$.

  $\II$=$\FF$:
  The $\FF$-cofibrations are the monomorphisms with dg-flat
  kernel. But any object of $\KK$ is dg-flat:
  all vector spaces are flat, and any acyclic object of $\KK$ is
  contractible, by 
  Lemma~\ref{l:objects-C(k)}
  (or
  Lemma~\ref{l:C-semisimple-abelian-category-monomorphisms}.\ref{enum:Csemisimple-objects}
  below). 

  The first claim of the lemma follows now from
  \ref{sec:model-struct-compl}.
 
  Trivially, all objects of $\KK$ are $\II$-cofibrant. 
  Recall that the $\II$-fibrant objects of $\KK$ are the
  h-injective 
  complexes with injective components.
  But all vector spaces are injective, and 
  all objects of $\KK$ are h-injective because 
  any acyclic complex is contractible, by
  Lemma~\ref{l:objects-C(k)}. Hence all objects of $\KK$ are
  $\II$-fibrant. This proves the second claim.
\end{proof}

\begin{lemma}
  \label{l:monos-C(k)}
  Let $\kk$ be a field.
  Any monomorphism in
  $\KK=\C(\kk)$ is isomorphic to a 
  coproduct of 
  shifts of the five monomorphisms $\kk \xra{\id} \kk$,
  $\iCone(\kk) 
  \xra{\id} \iCone(\kk)$, $0 \hra \kk$, $0 \hra \iCone(\kk)$, 
  $\kk \xra{(\id,0)^\tp} \iCone(\kk)$.
\end{lemma}

\begin{proof}
  This follows from the slightly more general
  Lemma~\ref{l:C-semisimple-abelian-category-monomorphisms}.\ref{enum:Csemisimple-monos}
  below.
\end{proof}

\subsubsection{Complexes in a semisimple abelian category}
\label{sec:compl-semis-abel}



\fussnote{ 
  The following lemma seems to be true in the periodic
  case also (even one-periodic), so can also use methods for
  matrix factorizations.  
}

\begin{lemma}
  \label{l:C-semisimple-abelian-category-monomorphisms}
  Let $\mathcal{A}$ be an abelian category which is semisimple in
  the sense that any short exact sequence splits.
  Then:
  \begin{enumerate}
  \item 
    \label{enum:Csemisimple-objects}
    Any object $A \in \C(\mathcal{A})$ is isomorphic to
    $S \oplus \iCone(T)$ for objects $S, T \in \mathcal{A}^\bZ$;
    more precisely, $S \cong H(A)$ and $T \cong B(A)$ (where
    $B(A)$ denotes the boundaries).
  \item 
    \label{enum:Csemisimple-injectives}
    All
    objects
    $\iCone(T)$ for $T \in \mathcal{A}^\bZ$ are projective and
    injective in $\C(\mathcal{A})$;
  \item
    \label{enum:Csemisimple-monos}
    Any short exact sequence in $\C(\mathcal{A})$ is
    isomorphic 
    to a direct sum
    of short exact sequences of the following five types:
    $S \xhra{\id} S \sra 0$,
    $\iCone(T) \xhra{\id} \iCone(T) \sra 0$,
    $0 \hra U \xsra{\id} U$,
    $0 \hra \iCone(V) \xsra{\id} \iCone(V)$,
    $W \xhra{(\id,0)^\tp} \iCone(W) \xra{(0,\id)} [1]W$,
    where $S, T, U, V, W$ are objects of $\mathcal{A}^\bZ$.
  \end{enumerate}
\end{lemma}

\begin{proof}
  \ref{enum:Csemisimple-objects}
  Consider the short exact sequences $Z(A) \hra A \xsra{d}
  [1]B(A)$ and $B(A) \hra Z(A) \sra H(A)$ in
  $\C(\mathcal{A})$. All terms in these sequences except for $A$
  have vanishing differential. Since $\mathcal{A}$ is semisimple,
  the 
  second sequence splits and hence the first sequence is
  isomorphic to $H(A) \oplus B(A) \hra A \xsra{d} [1]B(A)$.
  This sequence splits in $\mathcal{A}^\bZ$ (i.\,e.\ if we forget
  the differentials) and we obtain
  $A \cong H(A) \oplus \iCone(B(A))$ in $\C(\mathcal{A})$.

  \ref{enum:Csemisimple-injectives}
  The exact functor ``forget the differential''
  $\C(\mathcal{A}) \ra \mathcal{A}^\bZ$ has $T \mapsto
  \iCone(T)$ as a right adjoint and $T \mapsto [-1]\iCone(T)$ as a
  left adjoint. Hence $\iCone \colon \mathcal{A}^\bZ \ra
  \C(\mathcal{A})$ preserves injective (resp.\ projective)
  objects.
  Since $\mathcal{A}$ is semisimple, all objects of
  $\mathcal{A}^\bZ$ are injective and projective.
  
  \ref{enum:Csemisimple-monos} 
  Let $E \hra F \sra G$ be a short
  exact sequence in $\C(\mathcal{A})$.
  By 
  \ref{enum:Csemisimple-objects}
  and
  \ref{enum:Csemisimple-injectives},
  our sequence is isomorphic to the direct sum of the three short
  exact sequences 
  $\iCone(T) \xhra{\id} \iCone(T) \sra 0$,
  $0  \hra \iCone(V) \xsra{\id} \iCone(V)$, and
  $E' \hra F' \sra G'$ where $T=B(E)$,
  $V=B(G)$, $E'=H(E)$, 
  and $G'=H(G)$ are objects of $\mathcal{A}^\bZ$ and $F'$ is a
  suitable object of $\C(\mathcal{A})$. 
  In $\mathcal{A}^\bZ$ our sequence
  $E' \hra F' \sra G'$ is isomorphic to the obvious short exact
  sequence $E' \hra E' \oplus G' \sra G'$, and the differential
  $d_{F'}$ of 
  $F'$ maps $E'$ to zero and $G'$ to $E'$. Let $W$ denote the
  image of $d_{F'}$. By semisimplicity we can decompose
  $E' \cong S \oplus W$ and $G' \cong [1]W \oplus U$ such that 
  $d_{F'}^n$ is given by the composition
  \begin{equation}
    \label{eq:74}
    F'^n 
    \cong
    S^n \oplus W^n \oplus W^{n+1} \oplus U^n
    \xra{\text{proj.}} W^{n+1}
    \xra{\text{incl.}}
    S^{n+1} \oplus W^{n+1} \oplus W^{n+2} \oplus U^{n+1}
    \cong
    F'^{n+1}. 
  \end{equation}
  The claim follows.
\end{proof}

\subsection{Dg \texorpdfstring{$\kk$}{k}-enriched model categories}
\label{sec:enrich-model-categ-1}

This subsection may be skipped on a first reading.
Our aim is to show Theorem~\ref{t:examples-enriched-model-cats}
for later use in \ref{sec:some-remarks-2-morphisms-ENH}.
All model categories in this subsection are to be considered with
respect to the
universe $\mathscr{U}$. 

Let $\kk$ be a field and $\ul\KK=\ul\C(\kk)$.
We view $\KK$ as a model category equipped with any of the
identical model
structures of Lemma~\ref{l:model-str-for-k}.
The following result is similar to
Lemma~\ref{l:tensor-K-left-Quillen}. 

\begin{lemma}
  \label{l:pushout-product-cofibrations}
  Let $\kk$ be a field.  Let $\ulms{M}$ be a strongly
  pretriangulated tensored $\KK$-category, with tensored
  structure given by the bifunctor $\odot$.   
  Assume that the underlying category $\ms{M}$ is equipped with a
  model structure turning it into a model category such that 
  \begin{enumerate}
  \item 
    the set of cofibrations and the set of
    trivial cofibrations are stable under all shifts;
  \item 
    the functor $\iCone(-)$ maps cofibrations to
    trivial cofibrations;
  \item 
    for each cofibration 
    $E \ra F$ the
    morphism
    \begin{equation}
      \label{eq:84}
      \iCone(E) \;\subsqcup{E}\; F \ra \iCone(F)
    \end{equation}
    (cf.\ \eqref{eq:41a}) is a cofibration.
  \end{enumerate}
  Then the following is true:
  if $i \colon E \ra F$ is a cofibration in $\ms{M}$ and 
  $j \colon K \ra L$ is a cofibration in $\KK$, then the
  pushout-product morphism 
  \begin{equation}
    \label{eq:65}
    i \hat{\odot} j \colon 
    E \odot L \subsqcup{E \odot K} F
    \odot K \ra F \odot L 
  \end{equation}
  induced by the commutative diagram
  \begin{equation}
    \label{eq:85}
    \xymatrix{
      {E \odot K} 
      \ar[r]^-{i \odot \id} 
      \ar[d]^-{\id \odot j} &
      {F \odot K} 
      \ar[d]^-{\id \odot j} \\
      {E \odot L} 
      \ar[r]^-{i \odot \id} &
      {F \odot L} 
    }
  \end{equation}
  is a cofibration which is trivial if $i$ or $j$ is a trivial
  cofibration. 
\end{lemma}

\begin{proof}
  Lemma~\ref{l:monos-C(k)}
  allows us to write $j$ as a coproduct of 
  shifts of monomorphisms $\kk \xra{\id} \kk$, $\iCone(\kk)
  \xra{\id} \iCone(\kk)$, $0 \hra \kk$, $0 \hra \iCone(\kk)$, 
  $\kk \xra{(\id,0)^\tp} \iCone(\kk)$.
  As in the proof of
  Lemma~\ref{l:tensor-K-left-Quillen}
  we see
  that it is enough to consider these basic five monomorphisms.
  Let $i \colon E \ra F$ be a cofibration.
  
  If $j$ is an isomorphism, then 
  $i \hat{\odot} j$ is an isomorphism and hence a trivial
  cofibration. 

  If $j$ is $0 \hra \kk$, then
  $i \hat{\odot} j$ is isomorphic
  to $i$. This uses $M \odot 0 \cong 0$ and
  $(- \odot \kk) \cong \id$.

  If $j$ is the trivial cofibration $0 \hra \iCone(\kk)$, then 
  $i \hat{\odot} j$ is isomorphic
  to the morphism $\iCone(i) \colon \iCone(E)
  \ra \iCone(F)$ which is a 
  trivial
  cofibration by assumption.
  This uses $M \odot \iCone(\kk) \cong \iCone(M)$ naturally in $M$,
  and $M \odot 0 \cong 0$.

  Now assume that $j$ is the cofibration 
  $\kk \xra{(\id,0)^\tp} \iCone(\kk)$.
  Note that the functor $M \mapsto (M \odot \kk \ra M \odot
  \iCone(\kk))$ is 
  isomorphic to $M \mapsto (M \ra \iCone(M))$.
  Therefore 
  $i \hat{\odot} j$ is isomorphic to the morphism
  \eqref{eq:84} and therefore a cofibration.
  If $i \colon E \ra F$ is a trivial cofibration, then the
  composition
  \begin{equation}
    \label{eq:86}
    \iCone(E) \ra
    \iCone(E) \;\subsqcup{E}\; F \ra 
    \iCone(F)
  \end{equation}
  is a trivial cofibration by assumption and so is the arrow on
  the left as a pushout of $i$. Therefore the cofibration on the
  right 
  is trivial by the 2-out-of-3 property.
\end{proof}


\begin{proposition}
  \label{p:KK-excellent-and-MM-KK-model-cat}
  Let $\kk$ be a field.
  \begin{enumerate}
  \item 
    \label{enum:KK-excellent}
    The model category $\KK$ is monoidal and excellent (see
    \cite[Def.~A.3.1.2, A.3.2.16]{lurie-higher-topos}). 
  \item 
    \label{enum:MM-KK-enriched-model}
    Additionally to the assumptions
    of Lemma~\ref{l:pushout-product-cofibrations}, assume that 
    $\ulms{M}$ is 
    cotensored over $\KK$.
    Then 
    $\ulms{M}$ is a $\KK$-enriched model category (see
    \cite[Def.~A.3.1.5]{lurie-higher-topos}). 
  \end{enumerate}
\end{proposition}

\begin{proof}
  \ref{enum:KK-excellent}
  Lemma~\ref{l:pushout-product-cofibrations}
  (using Lemma~\ref{l:cone-ii-trivial-cofibration}
  and Corollary~\ref{c:cone-ii-trivial-cofibration})
  and
  Lemma~\ref{l:model-str-for-k}
  show that $\KK$ is a monoidal model category.
  Certainly, $\KK$ is a combinatorial model category
  (cf.\ the proof of Theorem~\ref{t:examples-enriched-model-cats}
  below), and 
  the set of quasi-isomorphisms is stable under 
  $\mathscr{U}$-small filtered colimits: this is true for the
  category of 
  complexes in any
  abelian category satisfying Grothendieck's condition AB5. 
  Now use
  \cite[Thm.~1]{lawson-localization-of-enriched-cats}
  to see that $\KK$ is excellent.

  Claim~\ref{enum:MM-KK-enriched-model} follows from
  Lemma~\ref{l:pushout-product-cofibrations}.
\end{proof}

\fussnote{
  The condition in
  \cite[Def.~A.3.1.1]{lurie-higher-topos} 
  that $\odot$ preserves small colimits in each variable
  is satisfied since 
  (by \eqref{eq:tensored} 
  and
  \eqref{eq:cotensored})
  both 
  $M \odot -$ and $- \odot R$ are left adjoints.
  Of course we use that our category is tensored 
  (by the assumptions
  in Lemma~\ref{l:pushout-product-cofibrations})
  and
  cotensored.
}

\begin{theorem}
  \label{t:examples-enriched-model-cats}
  Let $\kk$ be a field.
  The following $\KK$-categories are combinatorial 
  (see \cite[Def.~A.2.6.1]{lurie-higher-topos}) $\KK$-enriched
  model 
  categories: 
  \begin{enumerate}
  \item $\ul\C(\mathcal{X})$ with $\C(\mathcal{X})$ carrying the
    $\II$- or $\EE$-model structure, for
    a $\mathscr{U}$-small $\kk$-ringed site $(\mathcal{X},
    \mathcal{O})$; 
  \item $\ul\MMod(\ulms{C})$ with $\MMod(\ulms{C})$ carrying the
    $\II$- or $\PP$-model structure, for a
    $\mathscr{U}$-small $\KK$-category $\ulms{C}$;
  \item $\ul\C(X)$ with $\C(X)$ carrying the $\FF$- or $\II$- or
    $\EE$-model structure, for a $\mathscr{U}$-small $\kk$-ringed
    space $(X, \mathcal{O})$;
  \item $\ul\C(A)$ with $\C(A)$ carrying any of the model
    structures from \ref{sec:model-struct-compl} where $A$ is a
    $\mathscr{U}$-small commutative $\kk$-algebra.
  \end{enumerate}
\end{theorem}

\begin{proof}
  All model categories considered are cofibrantly generated;
  they are combinatorial because any
  Grothendieck abelian category is locally presentable, by
  \cite[Prop.~3.10]{beke-sheafifiable-homotopy-model-categories},
  and hence presentable in Lurie's terminology.
  They are $\KK$-enriched model categories
  because 
  Proposition~\ref{p:KK-excellent-and-MM-KK-model-cat}.\ref{enum:MM-KK-enriched-model}
  is applicable:
  use the proven results of section~\ref{sec:applications},
  and Corollary~\ref{c:cone-ii-trivial-cofibration}
  (mutatis mutandis for $\MMod(\ulms{C})$).
\end{proof}

\begin{remark}
  \label{rem:model-str-on-KK-enriched-functor-cats}
  Theorem~\ref{t:examples-enriched-model-cats} and
  Proposition~\ref{p:KK-excellent-and-MM-KK-model-cat}.\ref{enum:KK-excellent}
  show that \cite[Prop.~A.3.3.2]{lurie-higher-topos} is
  applicable to all the combinatorial $\KK$-enriched model
  categories mentioned in
  Theorem~\ref{t:examples-enriched-model-cats}: the category of
  $\KK$-functors from a $\mathscr{U}$-small
  $\KK$-category to one of
  these $\KK$-categories carries the projective and the injective
  combinatorial model structure, each of them turning it into a
  model category. This will be used in 
  \ref{sec:some-remarks-2-morphisms-ENH}.
\end{remark}

\section{Some 2-multicategories}
\label{sec:some-2-mult}



The main goal of this section is to introduce the
$\R$-linear 2-multicategory $\ENH_\R$ of enhancements. 
This is a key definition which allows a concise formulation of
our lifting results later on.
We also introduce a multicategory of formulas.

For the rest of this article, we fix universes $\mathscr{U} \in
\mathscr{V}$. 


By an $\R$-category we mean an $\R$-category with
$\mathscr{V}$-small $\Hom$-sets.
By a $\mathscr{V}$-small $\RR$-category we mean 
a $\mathscr{V}$-small $\RR_\mathscr{V}$-category (i.\,e.\ the
set of objects is $\mathscr{V}$-small and the $\Hom$-sets are
$\mathscr{V}$-small dg $\R$-modules). 

\subsection{2-multicategories of 
   categories, 
   dg categories, and
  triangulated categories}
\label{sec:2-mult-categ}


Our reference for enriched multicategories is
\cite{leinster-higher-operads-categories}.
Recall that a 2-multicategory (with $\mathscr{V}$-small
$\Hom$-categories) is a multicategory
enriched in the symmetric monoidal category of
($\mathscr{V}$-small) categories.
By an $\R$-linear 2-multicategory we mean a 2-multicategory 
enriched in the symmetric monoidal category of $\mathscr{V}$-small
$\R$-categories, cf.\
\cite[Def.~2.5]{ganter-kapranov-rep-and-char-thy-in-2-cats}.
In such a multicategory the set of all 
2-morphisms with same source and
target is a $\mathscr{V}$-small $\R$-module.

\fussnote{
  maybe \cite{hermida-repres-multicats} is a good reference?
  Free multicategory.
}
\fussnote{
  Wenn ich den Begriff 2-Kategorie verwende, meine ich ``strict
  2-category''. Wohl analog bei 2-Multikategorie.
}
\fussnote{
  I guess all my multicats are in fact symmetric? see leinster.
}

\subsubsection{Categories and dg categories}
\label{sec:categ-dg-categ}

The monoidal category of $\mathscr{V}$-small $\R$-categories
gives rise to the 
multicategory $\cat_\R$, by
\cite[Ex.~2.1.3]{leinster-higher-operads-categories}.  Its
objects are the $\mathscr{V}$-small $\R$-categories,
and its morphism sets
\begin{equation}
  \cat_\R(\mathcal{C}_1, \dots, \mathcal{C}_n; \mathcal{D})
\end{equation}
are $\R$-functors
$\mathcal{C}_1 \otimes \dots \otimes \mathcal{C}_n \ra
\mathcal{D}$, or, equivalently, $\R$-multilinear
functors
$\mathcal{C}_1 \times \dots \times \mathcal{C}_n \ra
\mathcal{D}$.
We adopt this latter viewpoint because 
later on we will focus our attention on triangulated
$\R$-categories. Note that 
$\cat_\R(\emptyset;\mathcal{D})=\Obj \mathcal{D}$ where $\emptyset$ is
the empty sequence because the empty tensor product of
$\R$-categories is the
$\R$-category with one object whose endomorphisms are $\R$.  

The multicategory
$\cat_\R$ is the underlying multicategory of the $\R$-linear
2-mul\-ti\-cat\-e\-go\-ry 
$\CAT_\R$ of $\mathscr{V}$-small $\R$-categories whose objects are the
$\mathscr{V}$-small $\R$-categories and whose morphism $\R$-categories
\begin{equation}
  \CAT_\R(\mathcal{C}_1, \dots, \mathcal{C}_n; \mathcal{D})
\end{equation}
are the categories of $\R$-multilinear functors 
$F, G \colon \mathcal{C}_1 \times \dots \times \mathcal{C}_n \ra
\mathcal{D}$ and $\R$-natural transformations $\sigma, \tau
\colon F \ra G$ (note that these morphism categories (or
$\Hom$-categories) are indeed $\mathscr{V}$-small by 
the obvious variant of Remark~\ref{rem:functor-categories-size}.\ref{enum:U-klein-neu}). Since $\mathcal{D}$ is an $\R$-category, $r
\sigma+ r'\tau$ is again an 
$\R$-natural transformation $F \ra G$, for $r,r' \in \R$; this
explains that $\CAT_\R$ is $\R$-linear.
Note that
$\CAT_\R(\emptyset; \mathcal{D})=\mathcal{D}$.

The $\R$-linear 2-multicategory $\DGCAT_\R$ of
$\mathscr{V}$-small $\RR$-categories (= dg $\R$-categories) is
defined similarly: its 
objects are the $\mathscr{V}$-small
$\RR$-categories, and its morphism $\R$-categories
\begin{equation}
  \DGCAT_\R(\ulms{A}_1, \dots, \ulms{A}_n; \ulms{B})
\end{equation}
are the $\mathscr{V}$-small $\R$-categories of $\RR$-functors 
$\ulms{A}_1 \otimes \dots \otimes \ulms{A}_n \ra
\ulms{B}$ and $\RR$-natural transformations. The underlying
multicategory of $\DGCAT_\R$ is the 
multicategory $\dgcat_\R$ of $\RR$-categories and $\RR$-functors.
Note that $\DGCAT_\R(\emptyset;\ulms{B})=\ms{B}$.

Mapping a $\mathscr{V}$-small $\RR$-category $\ulms{A}$ to its
homotopy 
category $[\ulms{A}]$ induces a functor
\begin{equation}
  \label{eq:htpy-DGCAT-CAT}
  [-] \colon \DGCAT_\R \ra \CAT_\R
\end{equation}
of $\R$-linear 2-multicategories: a 1-morphism
$F \colon \ulms{A}_1 \otimes \dots \otimes \ulms{A}_n \ra
\ulms{B}$
is mapped to
the composition
\begin{equation}
\label{eq:[F]-def}
  [F] \colon [\ulms{A}_1] \times \dots \times [\ulms{A}_n] \ra
  [\ulms{A}_1 \otimes \dots \otimes \ulms{A}_n]
  \xra{[F]}
  [\ulms{B}],
\end{equation}
again denoted $[F]$, and a 2-morphism $\tau \colon F \ra G$ is
mapped to the induced 
2-morphism $[\tau] \colon [F] \ra [G]$.

\subsubsection{Triangulated categories and pretriangulated dg categories}
\label{sec:triang-categ-pretr}

We denote the $\R$-linear 2-multicategory of $\mathscr{V}$-small 
triangulated
$\R$-categories by $\TRCAT_\R$. Its objects are
$\mathscr{V}$-small triangulated
$\R$-categories, and its morphism $\R$-categories
\begin{equation}
  \TRCAT_\R(\mathcal{T}_1, \dots, \mathcal{T}_n; \mathcal{S})
\end{equation}
are the $\R$-categories of triangulated $\R$-multilinear functors 
$\mathcal{T}_1 \times \dots \times \mathcal{T}_n \ra
\mathcal{S}$ (i.\,e.\ $\R$-multilinear functors of $\R$-categories
with translation that send
triangles in each argument to triangles, cf.\
\cite[Def.~10.3.6]{KS-cat-sh}) and their transformations.
Since $\mathcal{S}$ is additive, these $\R$-categories are additive.
We use the convention $\TRCAT_\R(\emptyset; \mathcal{S})=\mathcal{S}$.
The underlying 
multicategory of $\TRCAT_\R$ is the
multicategory $\trcat_\R$ of $\mathscr{V}$-small triangulated
$\R$-categories and 
triangulated $\R$-functors.


If $\ulms{A}_1, \dots,
\ulms{A}_n, \ulms{B}$ are
$\mathscr{V}$-small pretriangulated $\RR$-categories
and
$F \colon \ulms{A}_1 \otimes \dots \otimes \ulms{A}_n \ra
\ulms{B}$ is an $\RR$-functor, then $[F] \colon
[\ulms{A}_1] \times \dots \times [\ulms{A}_n] \ra 
[\ulms{B}]$ is a triangulated $\RR$-multilinear functor
(by the obvious generalization of Remark~\ref{rem:htpy-cat-of-pretriangulated-is-triang});
if 
$G$ is another such $\RR$-functor and
$\tau \colon F \ra G$ is an $\RR$-natural transformation, then 
$[\tau]$ is a transformation of triangulated $\R$-multilinear
functors.

Let $\DGCAT_\R^{\pretr}$ 
be the full $\R$-linear 2-multisubcategory
of $\DGCAT_\R$
whose objects are the
pretriangulated 
$\RR$-categories.
The above
discussion shows that 
the functor \eqref{eq:htpy-DGCAT-CAT} induces 
a functor 
\begin{equation}
  \label{eq:htpy-DGCAT-pretr-shift-TRCAT}
  [-] \colon \DGCAT_\R^{\pretr} \ra \TRCAT_\R
\end{equation}
of $\R$-linear 2-multicategories
so that we obtain
a commutative diagram 
\begin{equation}
  \label{eq:htpy-DGCAT-pretr-shift-TRCAT-square}
  \xymatrix{
    {\DGCAT_\R^{\pretr}} \ar[r]^-{[-]} \ar[d] &
    {\TRCAT_\R} \ar[d]\\
    {\DGCAT_\R} \ar[r]^-{[-]} &
    {\CAT_\R}
  }
\end{equation}
of $\R$-linear 2-multicategories. 

\begin{remark}
  \label{rem:TRCAT-CAT-reflects-1-and-2-isos}
  The vertical functor on the right 
  in
  \eqref{eq:htpy-DGCAT-pretr-shift-TRCAT-square}
  trivially 
  reflects 1-isomorphisms
  and 2-isomorphisms.
\end{remark}

\fussnote{
  I guess a 1-isomorphism necessarily has only one source!
}  

\subsubsection{Objectwise homotopy equivalences}
\label{sec:objectw-homot-equiv-1}

\begin{definition}
  \label{d:objectw-homot-equiv-1}
  Let $\ulms{A}_1, \dots, \ulms{A}_n$, $\ulms{B} \in \DGCAT_\R$.
  We call a morphism 
  $\tau \colon F \ra G$ in
  $\DGCAT_\R(\ulms{A}_1, \dots, \ulms{A}_n; \ulms{B})$
  an \define{objectwise homotopy equivalence} if
  \begin{equation}
    \tau_{(A_1,\dots,A_n)} \colon F(A_1, \dots, A_n) \ra
    G(A_1, \dots, A_n)  
  \end{equation}
  is a homotopy equivalence in
  $\ulms{B}$, for all $(A_1, \dots, A_n) \in \ulms{A}_1 \otimes
  \dots \otimes \ulms{A}_n$.
  Our notation for an objectwise homotopy equivalence is
  $F \xra{[\sim]} G$. 
\end{definition}

Equivalently, an objectwise homotopy equivalence is a 2-morphism
in $\DGCAT_\R$ whose image under the functor 
\eqref{eq:htpy-DGCAT-CAT} is a 2-isomorphism.




\begin{remark}
  \label{rem:composition-ohe}
  The vertical (resp.\ horizontal) composition in $\DGCAT_\kk$ of
  objectwise homotopy equivalences is an objectwise homotopy
  equivalence.
\end{remark}

\subsubsection{Existence of additive linear localizations}
\label{sec:exist-line-local}

\begin{definition}
  [{cf.\ \cite[Def.~8.3.1]{hirschhorn-model}}]
  \label{d:R-localization}
  An \define{additive $\R$-localization} of an additive
  $\R$-category 
  $\mathcal{M}$ 
  with respect to some subset $\mathcal{W}$ of the set
  $\Mor(\mathcal{M})$ of morphisms in 
  $\mathcal{M}$ is an additive $\R$-category $\L \mathcal{M}$
  together with 
  an $\R$-functor $\gamma \colon
  \mathcal{M} \ra \L 
  \mathcal{M}$ (which is automatically additive)
  that maps all 
  morphisms in $\mathcal{W}$ to isomorphisms and is universal
  with this property: given any
  $\R$-functor $\phi \colon \mathcal{M} \ra \mathcal{N}$ to an
  additive $\R$-category $\mathcal{N}$ mapping
  all
  morphisms in $\mathcal{W}$ to isomorphisms there is a unique
  $\R$-functor $\delta \colon \L\mathcal{M} \ra \mathcal{N}$ such
  that $\phi=\delta \gamma$.
\end{definition}

\begin{proposition}
  \label{p:R-localization-additive-small}
  Let $\mathscr{V}$ be a universe,
  $\mathcal{M}$ a $\mathscr{V}$-small additive $\R$-category, and
  $\mathcal{W} \subset \Mor \mathcal{M}$ a subset.
  Assume that $\mathcal{W}$ is closed under direct sums, i.\,e.\
  for all  
  morphisms $w \colon A \ra B$ und $w' \colon A' \ra B'$ in
  $\mathcal{W}$ the morphism $w \oplus w' \colon A \oplus A' \ra
  B \oplus B'$ is again in $\mathcal{W}$.
  Then there is a 
  an additive $\R$-localization
  $\gamma \colon \mathcal{M} \ra \L \mathcal{M}$ of $\mathcal{M}$
  with respect to $\mathcal{W}$ 
  where $\L\mathcal{M}$ is $\mathcal{V}$-small.
  Moreover,
  the underlying functor is the
  ordinary localization of the category $\mathcal{M}$ with
  respect to $\mathcal{W}$. 
\end{proposition}

\begin{proof}
  The existence of the localization 
  $\gamma \colon \mathcal{M} \ra \L \mathcal{M}$
  as an ordinary category and the fact that $\L \mathcal{M}$ is
  $\mathscr{V}$-small are
  well-known, see e.\,g.\ 
  \cite[Thm.~1.6.9 together with its proof]{gabber-ramero-found-almost-v12}.
  Then \cite{cisinski-mathoverflow-localization-additive} 
  shows that $\L \mathcal{M}$ is additive and that $\gamma$ is an
  additive functor.
  Now it is clear how to multiply a morphism $A \ra B$ in
  $\L\mathcal{M}$ by an element $r \in R$: just precompose with
  $\gamma(r \id_A)$ (or postcompose with $\gamma(r \id_B)$).
  All claims of the proposition follow.
\end{proof}

\begin{corollary}
  \label{c:R-localization-additive-small}
  Let $\ulms{A}, \ulms{B}$ be $\mathscr{V}$-small
  $\RR$-categories, and assume that $\ulms{B}$ is additive. 
  Then the additive $\R$-localization 
  of $\DGCAT_\R(\ulms{A}, \ulms{B})$ with
  respect to the set of objectwise homotopy
  equivalences exists and is $\mathscr{V}$-small.
\end{corollary}

\begin{proof}
  Since $\ulms{B}$ is
  additive, the $\mathscr{V}$-small $\R$-category
  $\DGCAT_\R(\ulms{A}, \ulms{B})$ is additive.
  The direct sum of objectwise homotopy equivalences is again an
  objectwise homotopy equivalence. Hence we can use
  Proposition~\ref{p:R-localization-additive-small}.
\end{proof}

We can and will assume in the following that all considered additive
$\R$-localization functors are the identity maps on the sets of
objects.  

\subsection{The 2-multicategory of dg
  \texorpdfstring{$\R$}{R}-enhancements} 
\label{sec:2-mult-enhanc}

Let
$\tildew{\ENH}_\R$
be the full $\R$-linear 2-multisubcategory 
of $\DGCAT_\R$ consisting of all additive pretriangulated
$\mathscr{V}$-small $\RR$-categories. Any strongly
pretriangulated $\mathscr{V}$-small $\RR$-category is additive
and hence an object of $\tildew{\ENH}_\R$.
The $\R$-linear 2-multicategory $\ENH_\R$ of
$\RR$-enhancements is defined as follows. 
Its objects are the objects of $\tildew{\ENH}_\R$.
Given objects $\ulms{A}_1, \dots, \ulms{A}_n$ and $\ulms{B}$ of
$\ENH_\R$, the $\R$-category of morphisms
\begin{equation}
  \label{eq:def-ENH}
  \ENH_\R(\ulms{A}_1, \dots, \ulms{A}_n; \ulms{B})
\end{equation}
is defined to be the target of the additive $\R$-localization 
(see Definition~\ref{d:R-localization})
of the $\mathscr{V}$-small additive $\R$-category 
\begin{equation}
  \tildew{\ENH}_\R(\ulms{A}_1, \dots, \ulms{A}_n; \ulms{B})
  =\DGCAT_\R(\ulms{A}_1 \otimes \dots \otimes
  \ulms{A}_n, \ulms{B})
\end{equation}
with respect to the set of objectwise homotopy equivalences.
It exists and is a $\mathscr{V}$-small additive $\R$-category, by
Corollary~\ref{c:R-localization-additive-small}.
The definition of the identities in $\ENH_\R$ is obvious, and 
compositions are defined in the obvious way using 
Remark \ref{rem:composition-ohe} and the fact that
additive $\R$-localization of additive $\R$-categories commutes
with finite products. 

\begin{remark}
  \label{rem:ENH-objects-are-enhancements}
  Each object $\ulms{B}$ of $\ENH_\R$ (or $\tildew{\ENH}_\R$) is a
  pretriangulated $\RR$-category, so
  $[\ulms{B}]$ is a 
  triangulated $\R$-category and $\ulms{B}$ together with the
  identity functor $[\ulms{B}] \ra [\ulms{B}]$
  is an $\RR$-enhancement of $[\ulms{B}]$, see
  Definition~\ref{d:dg-enhancement}.
  This explains why $\ENH_\R$ is called the 2-category of
  enhancements. 
  For the purpose of this article, the most important objects 
  of $\ENH_\R$ are
  the strongly pretriangulated $\RR$-categories
  $\ul\II(\mathcal{X})$, where
  $(\mathcal{X}, \mathcal{O})$
  is a $\mathscr{U}$-small
  $\R$-ringed site.
  Then 
  $[\ul\II(\mathcal{X})] \xsira{\eqref{eq:24}}
  \D(\mathcal{X})$ is an equivalence and 
  $\ul\II(\mathcal{X})$ is usually viewed as an $\RR$-enhancement of
  $\D(\mathcal{X})$.
  Similarly, 
  $\ul\IIMod(\ulms{C})$ and $\ul\PPMod(\ulms{C})$ are
  important objects of $\ENH_\R$, if
  $\ulms{C}$ is a $\mathscr{U}$-small
  $\RR$-category. They are enhancements of $\D(\ulms{C})$, cf.\
  \eqref{eq:24-modules}.
  Note that $\ul\II(\mathcal{X})$, 
  $\ul\IIMod(\ulms{C})$ and $\ul\PPMod(\ulms{C})$
  have $\mathscr{U}$-small $\Hom$-sets and
  objects in $\mathscr{U}$, by the discussion
  in \ref{sec:ringed-sites} and
  \ref{sec:model-structures}, so they are in particular 
  $\mathscr{V}$-small.
\end{remark}

\fussnote{
  Begr\"undung zur folgenden Bemerkung siehe wohl (ausgelagert!)
  ref{sec:zu-multikategorien}. 
}

\begin{remark}
  \label{rem:ENH-emptysource}
  If $\ulms{B}$ is any strongly pretriangulated
  $\mathscr{V}$-small $\RR$-category,
  it is straightforward to see that 
  $\ms{B} \ra [\ulms{B}]$ is the additive $\R$-localization of 
  $\ms{B}$ 
  with respect to the set of homotopy equivalences. 
  In
  particular, $\ENH_\R(\emptyset, \ulms{B})=[\ulms{B}]$ as an
  $\R$-category because 
  it is defined as the additive $\R$-localization
  of $\tildew{\ENH}_\R(\emptyset, \ulms{B})=
  \DGCAT_\R(\emptyset, \ulms{B})=\ms{B}$ 
  with respect to the set of (objectwise) homotopy equivalences. 
  In particular, objects and
  morphisms in
  $\ms{B}$ give rise to 1-morphisms and 2-morphisms in $\ENH_\R$.
\end{remark}

The construction of $\ENH_\R$ yields a functor
\begin{equation}
  \label{eq:delta-tilde-ENH-to-ENH}
  \delta \colon \tildew{\ENH}_\R \ra \ENH_\R
\end{equation}
of $\R$-linear 2-multicategories.
We can and will assume that $\delta$ is the identity on
1-morphisms. 
By the universal property of
additive $\R$-localizations (and 
Remark~\ref{rem:TRCAT-CAT-reflects-1-and-2-isos})
there is a unique functor
\begin{equation}
  \label{eq:htpy-ENH-TRCAT}
  [-] \colon \ENH_\R \ra \TRCAT_\R
\end{equation}
of $\R$-linear 2-multicategories
such that the diagram
\begin{equation}
  \xymatrix{
    {\tildew{\ENH}_\R} 
    \ar[r]^-{\delta} 
    \ar[rd]_-{[-]}
    & 
    {\ENH_\R}
    \ar[d]^-{[-]}\\
    &
    {\TRCAT_\R}
  }
\end{equation}
of $\R$-linear
2-multicategories commutes where the diagonal arrow
is 
obtained from \eqref{eq:htpy-DGCAT-pretr-shift-TRCAT}.

Let $\enh_\R$ be the underlying multicategory of
$\tildew{\ENH}_\R$. 
It coincides with the underlying multicategory of 
$\ENH_\R$ and comes with the obvious functor 
\begin{equation}
  \label{eq:[]-enh}
  [-] \colon \enh_\R \ra \trcat_\R.
\end{equation}

\fussnote{
  Auszuarbeiten (und hoffentlich richtig):
  
  \cite[4.1, Prop~1]{toen-lectures-dg-cats}
  implies that there is a natural bijection between the set of
  isomorphism classes of objects in $\ENH_\kk(\ulms{A}_1,
  \dots, \ulms{A}_n; \ulms{B})$ and the set of morphisms from
  $\ulms{A}_1 \otimes 
  \dots \otimes \ulms{A}_n \ra \ulms{B}$ in the homotopy category
  of $\KK$-categories (which is the localization of $\dgcat_\kk$
  with respect to the set of quasi-isomorphisms).

  Auch Prop.~3  auf Seite 40 interessant. Aber ich glaube,
  diese Art Lokalisierung ist schw\"acher als das, was ich
  mache.
  
  Kann ich sinnvolles Diagram zeichnen?
}

\begin{remark}
  \label{rem:zig-zags-define-2-morph-in-ENH}
  Let
  \begin{equation}
    \label{eq:zig-zag-left-pointing-ohe}
    F \xla[{[\sim]}]{\alpha_1} E_1 \xra{\beta_1} 
    E_1' 
    \xla[{[\sim]}]{}
    \dots
    \ra
    E_{n-1}' \xla[{[\sim]}]{\alpha_n} E_n \xra{\beta_n}
    G
  \end{equation}
  be a zig-zag of 2-morphisms in $\tildew{\ENH}_\R$ 
  where all morphisms $\alpha_i$ are objectwise homotopy
  equivalences.
  Then all $\delta(\alpha_i)$ are invertible in $\ENH_\R$ and
  \begin{equation}
    \delta(\beta_n) \delta(\alpha_n)\inv \cdots \delta(\beta_1)
    \delta(\alpha_1)\inv \colon \delta(F) \ra \delta(G)   
  \end{equation}
  defines a 2-morphism 
  $F \ra G$ in 
  $\ENH_\R$; this
  2-morphism is a 2-isomorphism if all $\beta_i$ are objectwise
  homotopy equivalences.
\end{remark}

\subsection{Some useful equalities}
\label{sec:some-usef-equal}

For later use we prove two additional results. 
We now work over the field $\kk$. In this subsection let
$\ulms{M}$ be either 
\begin{enumerate}
\item 
  \label{enum:M=CX-inj}
  $\ul\C(\mathcal{X})$ with $\C(\mathcal{X})$ carrying the
  $\II$-model structure, for
  a $\mathscr{U}$-small $\kk$-ringed site $(\mathcal{X},
  \mathcal{O})$, or 
\item
  \label{enum:M=ModC-inj-proj}
  $\ul\MMod(\ulms{C})$ with $\MMod(\ulms{C})$ carrying the
  $\II$-model structure, for a
  $\mathscr{U}$-small $\KK$-category $\ulms{C}$.
\end{enumerate}
We remind the reader that all objects of $\ulms{M}$ are
$\II$-cofibrant and hence the set of $\II$-fibrant objects in
$\ulms{M}$ 
coincides with the set of $\II$-bifibrant objects. Let $\ulms{M}_\bifib$
be the full subcategory of $\ulms{M}$ of $\II$-fibrant objects. 
So either $\ulms{M}_\bifib=\ul\II(\mathcal{X})$ or
$\ulms{M}_\bifib=\ul\IIMod(\ulms{C})$.

\fussnote{
  In den folgenden beiden Resultat k\"onnte vermutlich auch
  $\ulms{A}'$ und 
  $\ulms{B}'$ mitbetrachten, jedoch scheint mir dies nicht so
  wichtig. 
}

\begin{proposition}
  \label{p:compare-K-enri-I-fibr-repl-functors}
  Let $\ulms{M}$ be as above. 
  Let $\ulms{B}$ be a full additive
  $\KK$-subcategory of $\ulms{M}_\bifib$ such that
  $[\ulms{B}] \subset [\ulms{M}_\bifib]$ is a strictly full
  subcategory. 
  Let $\ulms{A}$ be any $\mathscr{V}$-small $\KK$-category
  and let
  \begin{equation}
    \label{eq:compare-iF-i'F}
    \delta \colon \DGCAT_\kk(\ulms{A},
    \ulms{B}) \ra \L_\ohe\DGCAT_\kk(\ulms{A},
    \ulms{B})
  \end{equation}
  be the additive $\kk$-localization
  of $\DGCAT_\kk(\ulms{A}, \ulms{B})$ with
  respect to the set of objectwise homotopy
  equivalences
  (see Corollary~\ref{c:R-localization-additive-small}).
  Let $(\ii, \iotaii)$ and $(\ii', \iotaii')$
  be $\KK$-enriched $\II$-fibrant resolution functors on
  $\ulms{M}$. 
  Let $F \colon \ulms{A} \ra
  \ulms{M}$ be a $\KK$-functor such that $\ii F \colon \ulms{A} \ra
  \ulms{M}_{\II\text{-}\fib}=\ulms{M}_\bifib$ lands in
  $\ulms{B}$. Then $\ii' F$ 
  and $\ii' \ii F$ and $\ii\ii' F$
  also land in $\ulms{B}$ and the zig-zags
  \begin{align}
    \label{eq:via-ii'-ii}
    \ii F 
    & \xra{\iotaii'\ii F} \ii'\ii F \xla{\ii'\iotaii F}
      \ii'F,
    \\
    \label{eq:via-ii-ii'}
    \ii F 
    & \xra{\ii\iotaii' F} \ii \ii' F \xla{\iotaii\ii' F}
      \ii' F
  \end{align}
  of objectwise homotopy equivalences in $\DGCAT_\kk(\ulms{A},
  \ulms{B})$ give rise to the same isomorphism
  \begin{equation}
    \phi_{\ii F \ra \ii'F}
    \colon \ii F \sira \ii'F
  \end{equation}
  in $\L_\ohe\DGCAT_\kk(\ulms{A}, \ulms{B})$ (cf.\
  Remark~\ref{rem:zig-zags-define-2-morph-in-ENH}).  

  If $(\ii'', \iotaii'')$ is another 
  $\KK$-enriched $\II$-fibrant resolution functor, we have
  \begin{align}
    \label{eq:phi-ii'-ii-assoc}
    \phi_{\ii' F \ra \ii''F}
    \circ
    \phi_{\ii F \ra \ii'F}
    & = \phi_{\ii F \ra \ii''F},\\
    \label{eq:phi-ii-ii-id}
    \phi_{\ii F \ra \ii F} 
    & = \id_{\ii F}
  \end{align}
  where the isomorphisms $\phi_{\ii' F \ra \ii''F}$,
  $\phi_{\ii F \ra \ii''F}$
  and $\phi_{\ii F \ra \ii F}$ are
  defined in the analog way.

  In particular, we obtain
  \begin{align}
    \label{eq:phi-ii-ii-id-corollar}
    \delta(\ii \iotaii F) 
    & = \delta(\iotaii \ii F),\\
    \label{eq:phi-ii-ii-id-corollar-2}
    \delta(\ii' \iotaii F)\inv \delta(\iota' \ii F)
    & = \delta(\iotaii \ii' F)\inv \delta(\ii \iotaii' F).
  \end{align}
\end{proposition}

\begin{proof}
  Let $A \in \ulms{A}$ and consider the commutative cube
  \begin{equation}
    \xymatrix{
      & {\ii''\ii' FA} 
      \ar[rr]^-{\ii''\ii'\iotaii_{FA}}
      &&
      {\ii''\ii' \ii FA}
      \\
      {\ii' FA} 
      \ar[rr]^(.7){\ii'\iotaii_{FA}}
      \ar[ur]^-{\iotaii''_{\ii' FA}}
      &&
      {\ii' \ii FA}
      \ar[ur]^-{\iotaii''_{\ii' \ii FA}}
      \\
      & 
      {\ii'' FA} 
      \ar[rr]^(.3){\ii''\iotaii_{FA}}
      \ar[uu]_(.3){\ii''\iotaii'_{FA}}
      &&
      {\ii''\ii FA}
      \ar[uu]_-{\ii''\iotaii'_{\ii FA}}
      \\
      {FA} 
      \ar[rr]_-{\iotaii_{FA}}
      \ar[uu]^-{\iotaii'_{FA}}
      \ar[ur]^-{\iotaii''_{FA}}
      &&
      {\ii FA.}
      \ar[uu]_(.7){\iotaii'_{\ii FA}}
      \ar[ur]_-{\iotaii''_{\ii FA}}
    }
  \end{equation}
  in $\ms{M}$. All morphisms are weak equivalences (by the
  2-out-of-3 property), and if we delete
  the object $FA$ and its outgoing arrows, all remaining
  morphisms are homotopy equivalences
  (cf.\
  Remark~\ref{rem:ii-maps-qisos-to-htpy-equis} below);  
  since $\ii F A$ is in
  $\ms{B}$, all the remaining objects are in $\ms{B}$ by our
  assumption that $[\ulms{B}] \subset
  [\ulms{M}_{\II\text{-}\fib}]$ is strictly full.

  We define $\phi_{\ii F \ra \ii' F}:= \delta(\ii'\iotaii F)\inv
  \delta(\iotaii' \ii F) \colon \ii F \sira \ii'
  F$ using the 
  zig-zag \eqref{eq:via-ii'-ii}, 
  and in a similar way, we define $\phi_{\ii' F \ra \ii''F} 
  := \delta(\ii''\iotaii' F)\inv
  \delta(\iotaii'' \ii' F)\colon \ii' F \sira \ii''F$
  and $\phi_{\ii F \ra \ii''F} := \delta(\ii''\iotaii F)\inv
  \delta(\iotaii'' \ii F)\colon \ii F \sira \ii''F$.
  Then \eqref{eq:phi-ii'-ii-assoc}
  follows immediately from the above cube (without using its
  corner $FA$).
 
  The equality \eqref{eq:phi-ii'-ii-assoc} for 
  $(\ii'', \iotaii'') = (\ii', \iotaii')= (\ii, \iotaii)$ says that
  $\phi_{\ii F \ra \ii F}$ is an idempotent isomorphism of $\ii
  F$. This implies 
  \eqref{eq:phi-ii-ii-id} and \eqref{eq:phi-ii-ii-id-corollar}.

  From \eqref{eq:phi-ii'-ii-assoc} for
  $(\ii'', \iotaii'') = (\ii, \iotaii)$
  and \eqref{eq:phi-ii-ii-id} we
  obtain
  $\phi_{\ii' F \ra \ii F} \circ \phi_{\ii F \ra \ii'F} = \phi_{\ii F
    \ra \ii F}=\id_{\ii F}$.
  Rewritten, this just means that
  \eqref{eq:phi-ii-ii-id-corollar-2} holds and shows that
  $\phi_{\ii F \ra \ii' F}$ can also be defined using the zig-zag
  \eqref{eq:via-ii-ii'}.
\end{proof}

The following result is similar and may be skipped on a first
reading. 

\begin{proposition}
  \label{p:compare-K-enri-E-fibr-repl-functors}
  Let $\ulms{M}$ be as above. 
  Let $\ulms{B}$ be a full additive
  $\KK$-subcategory of $\ulms{M}_\bifib$ such that
  $[\ulms{B}] \subset [\ulms{M}_\bifib]$ is a strictly full
  subcategory. 
  Let $\ulms{A}$ be any $\mathscr{V}$-small $\KK$-category
  and let
  \begin{equation}
    \label{eq:compare-FeG-Fe'G}
    \delta \colon \DGCAT_\kk(\ulms{A},
    \ulms{B}) \ra \L_\ohe\DGCAT_\kk(\ulms{A},
    \ulms{B})
  \end{equation}
  be the additive $\kk$-localization
  of $\DGCAT_\kk(\ulms{A}, \ulms{B})$ with
  respect to the set of objectwise homotopy
  equivalences
  (see Corollary~\ref{c:R-localization-additive-small}).
  Let $(\mathcal{Y}, \mathcal{O}_\mathcal{Y})$ be a
  $\mathscr{U}$-small $\kk$-ringed
  site, let
  $F \colon \ul\C(\mathcal{Y})_\hflat \ra
  \ulms{M}$ be a $\KK$-functor which maps
  quasi-isomorphisms to homotopy equivalences, and let 
  $G \colon \ulms{A} \ra \ul\C(\mathcal{Y})$ be a $\KK$-functor.
  Let $(\ee, \epsilonee)$ and  $(\ee', \epsilonee')$
  be $\KK$-enriched $\EE$-fibrant resolution functors on
  $\ul\C(\mathcal{Y})$.
  Assume that the composition
  \begin{equation}
    \label{eq:7}
    F \ee G \colon
    \ulms{A} \xra{G} \ul\C(\mathcal{Y}) \xra{\ee}
    \ul\EE(\mathcal{Y}) \subset \ul\C(\mathcal{Y})_\hflat \xra{F}
    \ulms{M} 
  \end{equation}
  lands in $\ulms{B}$.
  Then $F \ee' G$
  and $F \ee'\ee G$ and $F \ee\ee' G$
  also land in $\ulms{B}$ and the zig-zags
  \begin{align}
    \label{eq:via-ee'-ee}
    F \ee G 
    & \xla{F\epsilonee'\ee G} 
      F \ee'\ee G
      \xra{F\ee'\epsilonee G}
      F \ee' G
    \\
    \label{eq:via-ee-ee'}
    F \ee G
    & \xla{F \ee \epsilonee' G} 
      F \ee\ee' G
      \xra{F \epsilonee\ee' G}
      F \ee' G
  \end{align}
  of objectwise homotopy equivalences in $\DGCAT_\kk(\ulms{A},
  \ulms{B})$ give rise to the same isomorphism
  \begin{equation}
    \label{eq:psi-FeeG-to-Fee'G}
    \psi_{F\ee G \ra F\ee' G}
    \colon F\ee G \sira F\ee' G
  \end{equation}
  in $\L_\ohe\DGCAT_\kk(\ulms{A}, \ulms{B})$.
  If $(\ee'', \epsilonee'')$ is another 
  $\KK$-enriched $\EE$-fibrant resolution functor, we have
  \begin{align}
    \label{eq:psi-ee-ee'-assoc}
    \psi_{F\ee' G \ra F\ee'' G}
    \circ
    \psi_{F\ee G \ra F\ee' G}
    & = \psi_{F\ee G \ra F\ee'' G},\\
    \label{eq:psi-ee-ee'-id}
    \psi_{F\ee G \ra F \ee G}
    & = \id_{F\ee G}
  \end{align}
  where the isomorphisms
  $\psi_{F\ee' G \ra F\ee'' G}$, $\psi_{F\ee G \ra F\ee'' G}$ and 
  $\psi_{F\ee G \ra F\ee G}$ 
  are defined in the analog way.
  In particular, we obtain
  \begin{align}
    \label{eq:phi-ee-ee-id-corollar}
    \delta(F \epsilonee \ee G)
    & = \delta(F\ee \epsilonee G),\\
    \delta(F\ee'\epsilonee G) \delta(F\epsilonee'\ee G)\inv
    & = \delta(F\epsilonee\ee' G) \delta(F\ee\epsilonee' G)\inv.
  \end{align}
\end{proposition}

\fussnote{
  it seems that I could get rid of $\ul\C(\mathcal{Y})_\hflat$ by
  replacing it with
  $\ul\EE(\mathcal{Y})$. Aber so vielleicht besser, dann kann
  auch flache Modellstruktur nehmen f\"ur Garben auf
  $\kk$-geringten R\"aumen.
}

\begin{proof}
  Let $A \in \ulms{A}$, put $A':= G(A)$ and consider the
  commutative diagram
  \begin{equation}
    \xymatrix{
      & 
      {\ee''\ee'\ee A'} 
      \ar[rr]^-{\ee''\ee'\epsilonee_{A'}}
      \ar[dd]^(.7){\ee''\epsilonee'_{\ee A'}}
      \ar[dl]_-{\epsilonee''_{\ee'\ee A'}}
      &&
      {\ee''\ee'A'}
      \ar[dd]^-{\ee''\epsilonee'_{A'}}
      \ar[dl]_-{\epsilonee''_{\ee' A'}}
      \\
      {\ee'\ee A'} 
      \ar[rr]^(.7){\ee'\epsilonee_{A'}}
      \ar[dd]_-{\epsilonee'_{\ee A'}}
      &&
      {\ee'A'}
      \ar[dd]^(.3){\epsilonee'_{A'}}
      \\
      & {\ee''\ee A'} 
      \ar[rr]^(.3){\ee''\epsilonee_{A'}}
      \ar[dl]_-{\epsilonee''_{\ee A'}}
      &&
      {\ee''A'}
      \ar[dl]^-{\epsilonee''_{A'}}
      \\
      {\ee A'} 
      \ar[rr]_-{\epsilonee_{A'}}
      &&
      {A'}
    }
  \end{equation}
  in $\C(\mathcal{Y})$. All morphisms are quasi-isomorphisms, and
  all objects apart from $A'$ are h-flat. So if we remove the
  object $A'$ and all morphisms with target $A'$ from this
  diagram and apply $F$ we obtain a commutative diagram of
  homotopy equivalences in $\ms{M}$. Since 
  $[\ulms{B}]$ is a strictly full subcategory of
  $[\ulms{M}_{\II\text{-}\fib}]$ all objects of this new diagram lie
  in $\ulms{B}$, and we can define the isomorphism
  $\psi_{F\ee G \ra F\ee' G}:=\delta(F \ee'\epsilonee G)
  \delta(F\epsilonee' \ee G)\inv \colon F\ee G \ra
  F\ee'G$. 
  All claims of the proposition are now proved as the
  corresponding claims in
  Proposition~\ref{p:compare-K-enri-I-fibr-repl-functors}. 
\end{proof}

\subsection{The multicategory of formulas for the four operations}
\label{sec:categ-form-four}


Let $\R$ be a ring and $\RR=\C(\R)$.
Let $\fml'_\R$ be the free multicategory
whose objects are  
symbols $\ud{\mathcal{X}}$ and
$\ud{\mathcal{X}}^\opp$, for each
$\mathscr{U}$-small $\R$-ringed site $(\mathcal{X},
\mathcal{O}_\mathcal{X})$, and 
whose generating 
morphisms are given as follows (where all appearing $\R$-ringed
sites are assumed to be $\mathscr{U}$-small):
\begin{enumerate}
\item for each $\R$-ringed site $(\mathcal{X},
  \mathcal{O}_\mathcal{X})$ there are morphisms 
  $\ud{\otimes} \colon (\ud{\mathcal{X}},\ud{\mathcal{X}}) 
  \ra \ud{\mathcal{X}}$ and
  $\ud{\sheafHom} \colon (\ud{\mathcal{X}}^\opp,\ud{\mathcal{X}}) 
  \ra \ud{\mathcal{X}}$ and
  $\ud{\Hom} \colon (\ud{\mathcal{X}}^\opp,\ud{\mathcal{X}}) 
  \ra \ud{\pt}$ where $\ud{\pt}$ is the symbol for the
  $\R$-ringed site $(\pt, \R)$;
\item for each morphism 
  $\alpha \colon (\Sh(\mathcal{Y}), \mathcal{O}_\mathcal{Y}) \ra
  (\Sh(\mathcal{X}), \mathcal{O}_\mathcal{X})$ of $\R$-ringed
  topoi
  there are
  morphisms
  $\ud{\alpha}^* \colon \ud{\mathcal{X}} \ra \ud{\mathcal{Y}}$ and
  $\ud{\alpha}_* \colon \ud{\mathcal{Y}} \ra \ud{\mathcal{X}}$;
\item 
  for each $\R$-ringed site $(\mathcal{X},
  \mathcal{O}_\mathcal{X})$
  and each object $G \in \C(\mathcal{X})$ there is a morphism
  $\ud{G} \colon \emptyset \ra \ud{\mathcal{X}}$;
\item each of these morphisms has an ``opposite'' morphism:
  there are morphisms 
  $\ud{\otimes}^\opp \colon
  (\ud{\mathcal{X}}^\opp,\ud{\mathcal{X}}^\opp)  
  \ra \ud{\mathcal{X}}^\opp$,
  $\ud{\sheafHom}^\opp \colon
  (\ud{\mathcal{X}},\ud{\mathcal{X}}^\opp)  
  \ra \ud{\mathcal{X}}^\opp$,
  $\ud{\Hom}^\opp \colon
  (\ud{\mathcal{X}},\ud{\mathcal{X}}^\opp)  
  \ra \ud{\pt}^\opp$,
  $(\ud{\alpha}^*)^\opp \colon \ud{\mathcal{X}}^\opp \ra
  \ud{\mathcal{Y}}^\opp$, 
  $\ud{\alpha}_*^\opp \colon \ud{\mathcal{Y}}^\opp \ra
  \ud{\mathcal{X}}^\opp$, and $\ud{G}^\opp \colon \emptyset \ra
  \ud{\mathcal{X}}^\opp$.
\end{enumerate}
There is an obvious involution
\begin{equation}
  \label{eq:opp-fml'}
  (-)^\opp \colon \fml'_\R \ra \fml'_\R  
\end{equation}
swapping $\ud{\mathcal{X}}$ and
$\ud{\mathcal{X}}^\opp$ and exchanging each morphism with its
opposite. 

If $\sigma$ is the canonical morphism \eqref{eq:56}
we abbreviate $\ud\Gamma:=\ud\sigma_*$. 
The morphisms in $\fml'_\R$ may be viewed as planar rooted trees.
We prefer a formulaic notation, for
example $\ud\sheafHom(\ud\alpha^*(-),-)$ denotes the obvious
morphism $(\ud{\mathcal{X}}^\opp, \ud{\mathcal{Y}}) \ra
\ud{\mathcal{Y}}$. More precisely, this morphism would be
denoted $\ud\sheafHom((\ud\alpha^*)^\opp(-),-)$.

\begin{remark}
  \label{rem:enlarge-fml}
  In \ref{sec:categ-form-six} we will enlarge the multicategory
  $\fml'_\R$ by adding generating morphisms $\ud\alpha_!$,
  $\ud\alpha^!$, $\ud\alpha\inv$ for suitable morphisms $\alpha$
  of $\R$-ringed spaces, and call this enlarged multicategory
  $\fml_\R$. In \ref{sec:2-multicat-formulas} we even 
  add some 2-morphisms and define a $\kk$-linear 2-multicategory
  of formulas. 
\end{remark}

\begin{remark}
  \label{rem:name-fml}
  We call $\fml'_\R$ the multicategory of formulas.  
  A priori this seems to be a bad name: neither the objects nor
  the morphisms 
  of $\fml'_\R$ look like formulas. However, in the
  $\kk$-linear 2-multicategory of formulas defined later on
  there are 2-(iso)morphisms as
  $\ud\alpha_*\ud\sheafHom(\ud\alpha^*(-),-) \sira
  \ud\sheafHom(-, \ud\alpha_*(-))$ which really look like
  formulas and hence justify the name.
\end{remark}

\subsubsection{Interpretation of formulas in triangulated
  categories}
\label{sec:interpr-form-triang}

By definition of $\fml'_\R$, there is a unique interpretation
functor 
\begin{equation}
  \label{eq:interprete-fml'-tricat}
  \D \colon \fml'_\R \ra \trcat_\R  
\end{equation}
of multicategories given on objects by
$\ud{\mathcal{X}} \mapsto \D(\mathcal{X})$ and
$\ud{\mathcal{X}}^\opp \mapsto \D(\mathcal{X})^\opp$, and on
generating morphisms by 
$\ud\otimes \mapsto \otimes^\dL$, 
$\ud\sheafHom \mapsto \dR\sheafHom$,
$\ud\Hom \mapsto \dR\Hom$,
$\ud\alpha^* \mapsto \dL \alpha^*$,
$\ud\alpha_* \mapsto \dR \alpha_*$, and $\ud{G} \mapsto G$ (since
$G$ is also an object of $\D(\mathcal{X})$ we can view it as a
morphism $G \colon \emptyset \ra \D(\mathcal{X})$),
and by mapping
the opposites of these morphisms to the opposites of their
images. 
This is well-defined since $\D(\mathcal{X})$ is $\mathscr{V}$-small
(use for example Proposition~\ref{p:R-localization-additive-small}). 
The notation $\D$ for this functor is justified by
$\D(\ud{\mathcal{X}})=\D(\mathcal{X})$; note that
for example $\D(\ud\alpha_*)= \dR \alpha_*$. 

\fussnote{k\"onnte Funktor auch boldface $\mathbf{D}$ nennen.}

\begin{remark}
  \label{rem:functor-on-fml}
  In the following, when defining a functor from $\fml'_\R$ to
  another multicategory with an obvious involution $(-)^\opp$, we
  usually 
  only give the value of this functor on the objects 
  $\ud{\mathcal{X}}$ and the generating morphisms $\ud\otimes$,
  $\ud\sheafHom$, $\ud\Hom$, 
  $\ud\alpha^*$, 
  $\ud\alpha_*$, $\ud{G}$ and extend it tacitly so that it is
  compatible with the involutions. 
\end{remark}

\section{Four operations}
\label{sec:four-operations}

We lift the four functor formalism to the dg level, i.\,e.\ the
part of Grothendieck--Verdier--Spaltenstein's six functor
formalism concerning the four functors $\otimes^\dL$,
$\dR\sheafHom$, 
$\dL \alpha^*$, $\dR \alpha_*$.

We fix a field $\kk$ and write $\KK=\C(\kk)$.
We use the index ``$\hflat$'' (resp.\ ``$\hinj$'',
``$\whinj$'')
to indicate subcategories of h-flat (resp.\ h-injective, weakly
h-injective) objects. 

All ringed sites $(\mathcal{X}, \mathcal{O})$ in this section are
assumed to be $\mathscr{U}$-small. In particular, this applies to
the ringed sites giving rise to ringed topoi.

\subsection{Fixed data}
\label{sec:fixed-data-1}

 
For the rest of this article we fix for each $\kk$-ringed site
$(\mathcal{X}, \mathcal{O})$ a $\KK$-enriched $\EE$-cofibrant
resolution functor
$(\ee, \epsilonee)=(\ee_{(\mathcal{X}, \mathcal{O})},
\epsilonee_{(\mathcal{X}, \mathcal{O})})$
and a $\KK$-enriched $\II$-fibrant resolution functor
$(\ii, \iotaii)=(\ii_{(\mathcal{X}, \mathcal{O})},
\iotaii_{(\mathcal{X}, \mathcal{O})})$,
i.\,e.\ we have $\KK$-functors
\begin{equation}
  \label{eq:ee-and-ii}
  \ee \colon \ul\C(\mathcal{X}) \ra \ul\EE(\mathcal{X}) 
  \subset \ul\C(\mathcal{X})_\hflat 
  \quad \text{and} \quad
  \ii \colon \ul\C(\mathcal{X}) \ra \ul\II(\mathcal{X}) \subset
  \ul\C(\mathcal{X})_\hinj 
\end{equation}
and $\KK$-natural transformations $\ee \xra{\epsilonee} \id
\xra{\iotaii} \ii$ such that for each $M \in \ul\C(\mathcal{X})$
the morphism $\epsilonee_M \colon \ee(M) \ra M$ is a trivial
$\EE$-fibration and the morphism $\iotaii_M \colon M \ra \ii(M)$ is a
trivial $\II$-cofibration. 
This is
possible by
Theorem~\ref{t:enriched-functorial-fact-k-ringed-site};
the above inclusions were observed in 
Lemmas~\ref{l:E-cofibrant-(pullback)-h-flat} and
\ref{sec:inject-model-struct-1}.
For simplicity we assume that 
the $\KK$-enriched
resolution functors for $(\pt, \kk)$ are trivial, i.\,e.\ 
$(\ee_{(\pt, \kk)}, \epsilonee_{(\pt, \kk)})
=(\id,\id)=
(\ii_{(\pt, \kk)}, \iotaii_{(\pt, \kk)})$. This is allowed by 
Lemma~\ref{l:model-str-for-k}.



\begin{remark}
  \label{rem:ii-maps-qisos-to-htpy-equis}
  Here are some trivial facts that are frequently used
  in the following.
  Let $(\mathcal{X}, \mathcal{O})$ be a $\kk$-ringed site. Any
  quasi-isomorphism in $\C(\mathcal{X})$ between h-injective
  objects is invertible in the homotopy category
  $[\ul\C(\mathcal{X})]$.  In particular, the $\KK$-functor $\ii$
  maps quasi-isomorphisms to homotopy equivalences.  The
  $\KK$-functor $\ee$ obviously maps quasi-isomorphisms to
  quasi-isomorphisms.
\end{remark}

\begin{remark}
  \label{rem:choice-of-ii-and-ee}
  Subsequent constructions will depend on the fixed data.
  Nevertheless the
  constructions obtained from different choices are easy to
  compare. We will explain this in two examples, see
  Remarks~\ref{rem:compare-alpha_*-different-ii}
  and \ref{rem:compare-alpha^*-different-ee}.
\end{remark}

\subsection{Dg \texorpdfstring{$\kk$}{k}-enhancements considered}
\label{sec:enhanc-cons}


If $(\mathcal{X}, \mathcal{O})$ is a $\kk$-ringed site then
$\ul\II(\mathcal{X})$ is an object of $\ENH_\kk$. Recall
from \ref{sec:inject-enhanc}
that $\ul\II(\mathcal{X})$ together with the equivalence
\begin{equation}
  \label{eq:II-enhance-equiv}
  [\ul\II(\mathcal{X})] \sira \D(\mathcal{X}) 
\end{equation}
is a $\KK$-enhancement of $\D(\mathcal{X})$. Moreover, this
equivalence has  
\begin{equation}
  \label{eq:ol[ii]}
  \ol{[\ii]} \colon \D(\mathcal{X}) \sira [\ul\II(\mathcal{X})]
\end{equation}
as a quasi-inverse, by Lemma~\ref{l:ii-Drinfeld-quotient}.
If $q_\mathcal{X} \colon [\ul\C(\mathcal{X})] \ra
\D(\mathcal{X})$  
denotes the Verdier localization, then 
$\ol{[\ii]}$ is the unique triangulated $\kk$-functor such that 
$\ol{[\ii]} q_\mathcal{X}= [\ii]$.

\subsection{Lifts of derived functors}
\label{sec:lifts-deriv-funct}

The main players of this section are the following
$\KK$-functors.
If $\alpha \colon (\Sh(\mathcal{Y}), \mathcal{O}_\mathcal{Y}) \ra
(\Sh(\mathcal{X}), \mathcal{O}_\mathcal{X})$ is a morphism of $\kk$-ringed topoi,
define $\KK$-functors
\begin{align}
  \label{eq:25}
  \ul\alpha^* := \ii \alpha^*\ee 
  & \colon \ul\II(\mathcal{X})
    \xra{\ee} 
    \ul\EE(\mathcal{X}) 
    \xra{\alpha^*}
    \ul\C(\mathcal{Y})_\hflat
    \xra{\ii} 
    \ul\II(\mathcal{Y}), \\
  \label{eq:2555}
  \ul\alpha_* := \ii \alpha_* 
  & \colon \ul\II(\mathcal{Y})
    \xra{\alpha_*} 
    \ul\C(\mathcal{X})_\whinj \xra{\ii}
    \ul\II(\mathcal{X}).
\end{align}
Here $\alpha^*$ and $\alpha_*$ are the restrictions of the
obvious $\KK$-functors $\alpha^* \colon \ul\C(\mathcal{X}) \ra \ul\C(\mathcal{Y})$ 
and $\alpha_* \colon \ul\C(\mathcal{Y}) \ra \ul\C(\mathcal{X})$;
they land in the indicated
categories by 
Propositions~\ref{p:pullback-preserves-h-flat}
and \ref{p:spalt-5.15-sites}.
Note that $\ul\alpha^*$ and $\ul\alpha_*$ are 1-morphisms in the
$\kk$-linear
2-multicategories $\tildew{\ENH}_\kk$ and $\ENH_\kk$.

If $(\mathcal{X}, \mathcal{O}_\mathcal{X})$ is a $\kk$-ringed
site, define $\KK$-functors
\begin{align}
  \label{eq:33}
  \ul{\otimes}:= \ii(\ee(-) \otimes \ee(-)) 
  & \colon
  \ul\II(\mathcal{X}) \otimes \ul\II(\mathcal{X}) 
  \xra{\ee \otimes \ee}
  \ul\EE(\mathcal{X}) \otimes \ul\EE(\mathcal{X})
  \xra{\otimes}
  \ul\C(\mathcal{X})_\hflat
  \xra{\ii}
  \ul\II(\mathcal{X}),\\
  \ul\sheafHom := \ii \sheafHom(-,-) 
  & \colon 
  \ul\II(\mathcal{X})^\opp \otimes \ul\II(\mathcal{X}) 
  \xra{\sheafHom(-,-)}
  \ul\C(\mathcal{X})_\whinj 
  \xra{\ii} 
  \ul\II(\mathcal{X}).
\end{align}
Here we use the fact that the tensor product of two h-flat
objects is 
h-flat and Proposition~\ref{p:spalt-5.14-sites-sheafHom}.
Both $\KK$-functors $\ul\otimes$ and $\ul\sheafHom$ can and will
be viewed as 1-morphisms
$\ul\otimes \colon  (\ul\II(\mathcal{X}),\ul\II(\mathcal{X})) \ra
\ul\II(\mathcal{X})$ 
and $\ul\sheafHom \colon
(\ul\II(\mathcal{X})^\opp,\ul\II(\mathcal{X})) \ra 
\ul\II(\mathcal{X})$ 
in 
$\tildew{\ENH}_\kk$ and $\ENH_\kk$.

\begin{remark}
  Note that $\alpha^*$, $\alpha_*$,
  $\otimes$, $\sheafHom$ are
  $\KK$-functors. They are not the 
  underlying 
  functors of the $\KK$-functors $\ul\alpha^*$, $\ul\alpha_*$,
  $\ul\otimes$, $\ul\sheafHom$.
\end{remark}

Let $\sigma = \sigma_\mathcal{X}\colon (\Sh(\mathcal{X}),
\mathcal{O}_\mathcal{X}) \ra (\Sh(\pt), \kk)$ be the canonical 
morphism \eqref{eq:56} of $\kk$-ringed topoi, so
$\Gamma=\sigma_*$, by \ref{sec:final-topos}.
We define $\ul\Gamma:=\ul\sigma_*$.
Our assumption $\ii_{(\pt, \kk)}=\id$ implies that
$\ul\Gamma=\ul\sigma_*=\sigma_* = \Gamma \colon
\ul\II(\mathcal{X}) \ra \ul\KK=\ul\C(\pt,\kk)=\ul\II(\pt,\kk)$. 
Define the $\KK$-functor
\begin{equation}
  \label{eq:40}
  \ul\Hom=\ul\Hom_\mathcal{X}:= \Gamma
  \sheafHom \colon 
  \ul\II(\mathcal{X})^\opp \otimes 
  \ul\II(\mathcal{X}) \ra \ul\KK. 
\end{equation}
We write $\ul\End(I):= \ul\Hom(I,I)$.

\begin{remark}
  \label{rem:[ulC]-vs-[ulHom]}
  For $I$, $J \in \ul\II(\mathcal{X})$ we have
  $\ul\Hom(I,J)=\ul\C_\mathcal{X}(I,J)$.  The image of $\ul\Hom$
  under \eqref{eq:htpy-ENH-TRCAT} is the triangulated
  $\kk$-functor
  $[\ul\Hom] \colon [\ul\II(\mathcal{X})]^\opp \times
  [\ul\II(\mathcal{X})] \ra [\ul\KK]$;
  on objects, it coincides with \eqref{eq:40}, i.\,e.\
  $[\ul\Hom](I,J)=\ul\Hom(I,J)$ is a complex of $\kk$-vector
  spaces.  On the other hand, $[\ul\C_\mathcal{X}](I,J)$ is the
  $\kk$-vector space of closed degree zero morphisms $I \ra J$ up
  to homotopy, i.\,e.\
  $[\ul\C_\mathcal{X}](I,J) = H^0([\ul\Hom](I,J)) =
  H^0(\ul\Hom(I,J))$.
  To avoid possible conflicts of notation we will rarely use the
  symbol
  $\ul\C_\mathcal{X}$ for the $\KK$-functor
  $\ul\C(\mathcal{X})^\opp \otimes \ul\C(\mathcal{X}) \ra \KK$,
  $(E,F) \mapsto \ul\C_\mathcal{X}(E,F)$.
\end{remark}

All the 
derived
functors $\dL\alpha^*$, $\dR \alpha_*$, $\otimes^\dL$,
$\dR\sheafHom$, $\dR\Hom$ exist and can be computed in the
expected way, by 
Propositions~\ref{p:spalt-6.1+5-sites} and
\ref{p:spalt-6.7a-sites}.
It is intuitively clear that the $\KK$-functors $\ul\alpha^*$,
$\ul\alpha_*$, $\ul\otimes$, $\ul\sheafHom$, $\ul\Hom$ lift these
derived functors 
to the level of injective enhancements. 
Let us confirm this intuition.

\begin{proposition}
  \label{p:canonical-isotrafos}
  If $\alpha \colon (\Sh(\mathcal{Y}), \mathcal{O}_\mathcal{Y}) \ra
  (\Sh(\mathcal{X}), \mathcal{O}_\mathcal{X})$ is a morphism of 
  $\kk$-ringed topoi, there are canonical 2-isomorphisms
  \begin{align}
    \label{eq:omega-alpha^*}
    \omega_{\alpha^*} \colon 
    [\ul{\alpha}^*]\ol{[\ii]}
    & \sira \ol{[\ii]}\dL\alpha^*,\\
    \label{eq:omega-alpha_*}
    \omega_{\alpha_*} \colon [\ul{\alpha}_*]\ol{[\ii]}
    & \sira \ol{[\ii]}\dR\alpha_*
  \end{align}
  in $\TRCAT_\kk$ as illustrated in the diagrams
  \begin{equation}
    \label{eq:26}
    \xymatrix{
      {\D(\mathcal{X})} \ar[rr]^-{\dL \alpha^*}
      \ar[d]_-{\ol{[\ii]}}^-{\sim} 
      &&
      {\D(\mathcal{Y})}
      \ar[d]_-{\ol{[\ii]}}^-{\sim} 
      \\
      [\ul\II(\mathcal{X})] 
      \ar[rr]_-{[\ul\alpha^*]} 
      \ar@{=>}[rru]^-{\omega_{\alpha^*}}_-{\sim}
      &&
      {[\ul\II(\mathcal{Y})],} 
    }
    \quad \quad \quad
    \xymatrix{
      {\D(\mathcal{Y})} 
      \ar[rr]^-{\dR \alpha_*} 
      \ar[d]_-{\ol{[\ii]}}^-{\sim} 
      &&
      {\D(\mathcal{X})} 
      \ar[d]_-{\ol{[\ii]}}^-{\sim} 
      \\
      {[\ul\II(\mathcal{Y})]}
      \ar[rr]_-{[\ul\alpha_*]}
      \ar@{=>}[rru]^-{\omega_{\alpha_*}}_-{\sim}
      &&
      {[\ul\II(\mathcal{X})].}
    }
  \end{equation}
  In particular, for $\alpha=\sigma$ there is a canonical
  2-isomorphism 
  $\omega_\Gamma \colon [\ul\Gamma]\ol{[\ii]} \sira
  \dR \Gamma$.

  If $(\mathcal{X}, \mathcal{O})$ is a $\kk$-ringed site, there
  are canonical 
  2-isomorphisms
  $\omega_\otimes$, $\omega_\sheafHom$, and $\omega_\Hom$ in
  $\TRCAT_\kk$ as illustrated in the diagrams
  \begin{equation}
    \label{eq:34}
    \xymatrix{
      {\D(\mathcal{X}) \times \D(\mathcal{X})}
      \ar[rr]^-{\otimes^\dL}
      \ar[d]_-{\ol{[\ii]} \times \ol{[\ii]}}^-{\sim} 
      &
      &
      {\D(\mathcal{X})}
      \ar[d]_-{\ol{[\ii]}}^-{\sim} 
      \\
      {[\ul\II(\mathcal{X})] \times [\ul\II(\mathcal{X})]} 
      \ar[rr]_-{[\ul\otimes]}
      \ar@{=>}[rru]^-{\omega_{\otimes}}_-{\sim}
      &
      &
      {[\ul\II(\mathcal{X})],} 
    }
  \end{equation}
  \begin{equation}
    \label{eq:36}
    \xymatrix{
      {\D(\mathcal{X})^\opp \times \D(\mathcal{X})}
      \ar[rr]^-{\dR\sheafHom}
      \ar[d]_-{\ol{[\ii]} \times \ol{[\ii]}}^-{\sim} 
      &&
      {\D(\mathcal{X})}
      \ar[d]_-{\ol{[\ii]}}^-{\sim} 
      \\
      {[\ul\II(\mathcal{X})]^\opp \times [\ul\II(\mathcal{X})]} 
      \ar[rr]_-{[\ul\sheafHom]}
      \ar@{=>}[rru]^-{\omega_\sheafHom}_-{\sim}
      &
      &
      {[\ul\II(\mathcal{X})],} 
    }
  \end{equation}
  \begin{equation}
    \label{eq:omega-ulHom}
    \xymatrix{
      {\D(\mathcal{X})^\opp \times \D(\mathcal{X})}
      \ar[rr]^-{\dR\Hom}
      \ar[d]_-{\ol{[\ii]} \times \ol{[\ii]}}^-{\sim} 
      &&
      {\D(\pt)=[\KK]}
      \ar@<-3ex>[d]^-{\ol{[\ii]}=\id}
      \\
      {[\ul\II(\mathcal{X})]^\opp \times [\ul\II(\mathcal{X})]} 
      \ar[rr]_-{[\ul\Hom]}
      \ar@{=>}[rru]^-{\omega_\Hom}_-{\sim}
      &
      &
      {[\ul\II(\pt)]=[\KK]} 
    }
  \end{equation}
  (cf.\ \eqref{eq:[F]-def} for the definition of the lower
  horizontal arrows).
\end{proposition}

Let us caution the reader that
the vertical arrows
$\ol{[\ii]}$ in the above diagrams are equivalences and not
1-isomorphisms in general even though they are labeled $\sim$.

The canonical 2-isomorphisms $\omega_{\alpha_*}$ and
$\omega_{\alpha^*}$ can even be uniquely characterized, see
Propositions~\ref{p:ul-alpha_*-as-derived-functor} and
\ref{p:ul-alpha^*-as-derived-functor} below. There are similar
characterizations of $\omega_\otimes$, $\omega_\sheafHom$,
$\omega_\Hom$ 
left to the reader.

\fussnote{
  is the equivalence 
  \eqref{eq:diagram-for-bL-alpha^*}
  worth a separate statement?
}

\begin{proof}
  Construction of $\omega_{\alpha_*}$. 
  We can assume without
  loss of generality that $\dR \alpha_* =
  q_\mathcal{X}[\alpha_*]\ol{[\ii]}$. Then $\ol{[\ii]} \dR\alpha_*=
  \ol{[\ii]} q_\mathcal{X}[\alpha_*]\ol{[\ii]} =
  [\ii][\alpha_*]\ol{[\ii]}=[\ul\alpha_*]\ol{[\ii]}$ and we define
  $\omega_{\alpha_*}$ to 
  be the identity.

  Construction of $\omega_{\alpha^*}$. 
  Consider the obvious commutative diagram
  \begin{equation}
    \label{eq:diagram-for-bL-alpha^*}
    \xymatrix{
      {[\ul\C(\mathcal{X})]} 
      \ar[r]^-{[\ee]}
      \ar[d]^-{q_\mathcal{X} }
      &
      {[\ul\EE(\mathcal{X})]}
      \ar[r]^-{[\alpha^*]}
      \ar[d]
      & 
      {[\ul\C(\mathcal{Y}]}
      \ar[rd]^-{[\ii]}
      \ar[d]^-{q_\mathcal{Y}}
      \\
      {\D(\mathcal{X})}
      \ar[r]^-{\ol{[\ee]}}_-{\sim}
      &
      {[\ul\EE(\mathcal{X})]/[\ul\EE_\ac(\mathcal{X})]}
      \ar[r]^-{\ol{[\alpha^*]}}
      &
      {\D(\mathcal{Y})}
      \ar[r]^-{\ol{[\ii]}}_-{\sim}
      &
      {[\ul\II(\mathcal{Y})]}
    }
  \end{equation}
  whose lower horizontal arrows are the induced triangulated
  $\kk$-functors and $\ol{[\ee]}$ is an equivalence. Without loss
  of generality we can assume 
  that $\dL \alpha^*= \ol{[\alpha^*]}\;\ol{[\ee]}$. Then $\ol{[\ii]}
  \dL \alpha^* q_\mathcal{X}=[\ii] [\alpha^*] [\ee]$ by the above
  diagram. 
  For each object $E \in \C(\mathcal{X})$ the morphism
  \begin{equation}
    \label{eq:ii-alpha^*-ee-iotaii-E}
    \ii \alpha^* \ee E
    \xra{\ii \alpha^* \ee \iotaii_E} 
    \ii \alpha^* \ee \ii E  
    =
    \ul\alpha^* \ii E
  \end{equation}
  in $\II(\mathcal{Y})$ is a homotopy equivalence by
  Remark~\ref{rem:ii-maps-qisos-to-htpy-equis}
  and
  Proposition~\ref{p:pullback-preserves-h-flat}, i.\,e.\ an
  isomorphism in $[\ul\II(\mathcal{Y})]$.
  The family of these isomorphisms defines a natural
  isotransformation $\ol{[\ii]}\dL\alpha^* \sira
  [\ul\alpha^*]\ol{[\ii]}$ of triangulated $\kk$-functors
  $\D(\mathcal{X}) \ra [\ul\II(\mathcal{Y})]$. 
  We let 
  $\omega_{\alpha^*}$ be the inverse of this isotransformation.

  Construction of $\omega_\sheafHom$ and $\omega_\Hom$. We
  can assume that  
  \begin{equation}
    \dR\sheafHom(-,-)= q_\mathcal{X}[\sheafHom](-,\ol{[\ii]}
    (-)).
  \end{equation}
  For $E, F \in
  \C(\mathcal{X})$ the morphism
  \begin{equation}
    \ul\sheafHom(\ii E, \ii F)= \ii\sheafHom(\ii E, \ii F) 
    \xra{\ii\sheafHom(\iotaii_E, \id_{\ii F})} 
    \ii\sheafHom(E, \ii F) 
  \end{equation}
  defines an isomorphism in $[\ul\II(\mathcal{X})]$,
  by
  Proposition~\ref{p:spalt-5.20-sites-sheafHom}.\ref{enum:sheafHom-to-h-inj}. We
  define $\omega_{\sheafHom}$ to be the family of these
  isomorphisms. Similarly, 
  $\omega_\Hom$ is constructed using the homotopy equivalences
  $\ul\Hom(\ii E, \ii F)= \ul\C_{\mathcal{X}}(\ii E, \ii F) 
  \xra{\ul\C_{\mathcal{X}}(\iotaii_E, \id_{\ii F})} 
  \ul\C_{\mathcal{X}}(E, \ii F)$.

  Construction of $\omega_\otimes$. We can assume that $((-)
  \otimes^\dL (-))= (\ol{\ee}(-) \ol\otimes \ol{\ee}(-))$ where
  $\ol{\otimes} \colon 
  [\ul\EE(\mathcal{X})]/[\ul\EE_\ac(\mathcal{X})]
  \times [\ul\EE(\mathcal{X})]/[\ul\EE_\ac(\mathcal{X})] \ra
  \D(\mathcal{X})$ is the obvious functor, cf.\
  Proposition~\ref{p:spalt-5.7-sites}. 
  This proposition also shows that the morphism
  \begin{equation}
    \ii(\ee E \otimes \ee F)
    \xra{\ii(\ee\iotaii_E \otimes \ee\iotaii_F)}
    \ii(\ee\ii E \otimes \ee\ii F)
    = \ii E \;\ul\otimes\; \ii F
  \end{equation}
  defines an isomorphism in $[\ul\II(\mathcal{X})]$, for
  $E, F \in \C(\mathcal{X})$. Let $\omega_\otimes$ be given by
  the inverses of these isomorphisms.
\end{proof}


\begin{remark}
  \label{rem:pullback-flat}
  If
  $\alpha \colon (\Sh(\mathcal{Y}),\mathcal{O}_\mathcal{Y}) \ra
  (\Sh(\mathcal{X}), \mathcal{O}_\mathcal{X})$
  is a flat morphism of $\kk$-ringed topoi, then
  $\ul\alpha^* I =\ii \alpha^* \ee I \xra{\ii \alpha^*
    \epsilonee_I} \ii \alpha^* I$
  is a homotopy equivalence for all $I \in
  \ul\II(\mathcal{X})$.
  Therefore $\ul\alpha^* \ra \ii \alpha^*$ is an objectwise
  homotopy equivalence and defines a 2-isomorphism
  $\ul\alpha^*\sira\ii\alpha^*$ in $\ENH_\kk$.
\end{remark}

Let $(\mathcal{X}, \mathcal{O}_\mathcal{X})$ be a $\kk$-ringed 
site.
If $G \in \C(\mathcal{X})$ is an object we define
\begin{equation}
  \label{eq:define-ul-G}
  \ul{G}:= \ii G. 
\end{equation}
This is an object of $\II(\mathcal{X})$ and of
$[\ul\II(\mathcal{X})]$ and can therefore be viewed as a 1-morphism
in $\tildew{\ENH}_\kk$ and $\ENH_\kk$, by
Remark~\ref{rem:ENH-emptysource}.
Trivially, the diagram
\begin{equation}
  \label{eq:omega-G}
  \xymatrix{
    {\emptyset} 
    \ar[rr]^-{G}
    \ar@{=}[d]
    &&
    {\D(\mathcal{X})}
    \ar[d]_-{\ol{[\ii]}}^-{\sim} 
    \\
    {\emptyset}
    \ar[rr]_-{[\ul{G}]} 
    \ar@{=>}[rru]^-{\omega_{G}=\id}_-{\sim}
    &&
    {[\ul\II(\mathcal{X})]} 
  }
\end{equation}
in $\TRCAT_\kk$ commutes where $\omega_G := \id \colon
[\ul{G}] \xra{=} \ol{[\ii]}G$.

\begin{remark}
  \label{rem:ul-G-special-cases}
  For $G \in \KK=\C(\pt,\kk)$ we have $\ul{G}=G$ by assumption
  $\ii_{(\pt , \kk)}=\id$, in particular $\ul{\mathcal{O}}_\pt=\kk$.
  If $(\mathcal{X}, \mathcal{O}_\mathcal{X})$ is a $\kk$-ringed site and
  $\sigma = \sigma_\mathcal{X}\colon (\Sh(\mathcal{X}), 
  \mathcal{O}_\mathcal{X}) \ra (\Sh(\pt), \kk)$ is the associated 
  morphism of $\kk$-ringed topoi (see \eqref{eq:56}) then
  \begin{equation}
    \label{eq:112}
    \ul \sigma^* \ul{\mathcal{O}}_\pt=
    \ul \sigma^* \kk = 
    \ii \alpha^* \ee \kk=
    \ii \alpha^* \kk=
    \ii \mathcal{O}_\mathcal{X}= 
    \ul{\mathcal{O}}_\mathcal{X}=
    \ul{\mathcal{O}}
  \end{equation}
  using $\ee_{(\pt, \kk)}=\id$. This is a special case of 
  the morphism \eqref{eq:alpha^*-ul-obj}
  in Lemma~\ref{l:ul-G-and-functors} below.
\end{remark}

\subsubsection{Some uniqueness results}
\label{sec:some-uniq-results}

This subsubsection may be skipped on a first reading.
We provide the characterizations of $\omega_{\alpha_*}$ and
$\omega_{\alpha^*}$ mentioned above.
All derived functors are taken with respect to the obvious full
subcategories of acyclic objects.

\begin{proposition}
  \label{p:ul-alpha_*-as-derived-functor}
  Let $\alpha \colon (\Sh(\mathcal{Y}), \mathcal{O}_\mathcal{Y})
  \ra 
  (\Sh(\mathcal{X}), \mathcal{O}_\mathcal{X})$ be a morphism of
  $\kk$-ringed topoi.
  Consider 
  the $\KK$-natural transformation
  \begin{equation}
    \ii \alpha_* \iotaii \colon
    \ii \alpha_* \ra 
    \ii \alpha_* \ii = \ul\alpha_* \ii 
    \colon
    \ul\C(\mathcal{Y}) \ra \ul\II(\mathcal{X})
  \end{equation}
  and the induced transformation
  \begin{equation}
    [\ii \alpha_* \iotaii] 
    \colon
    [\ii \alpha_*] 
    \ra 
    [\ul\alpha_* \ii] = 
    [\ul\alpha_*] \ol{[\ii]} q_\mathcal{Y} 
    \colon 
    [\ul\C(\mathcal{Y})] \ra [\ul\II(\mathcal{X})]
  \end{equation} 
  of triangulated $\kk$-functors where we use $[\ii]=\ol{[\ii]}
  q_\mathcal{Y}$. 
  Then the pair
  $([\ul\alpha_*] \ol{[\ii]}, [\ii \alpha_* \iotaii])$
  is a right derived functor of
  $[\ii \alpha_*] \colon [\ul\C(\mathcal{Y})] \ra
  [\ul\II(\mathcal{X})]$.
  In particular, if 
  $(\dR \alpha_*, \rho)$
  is a right derived functor of
  $q_\mathcal{X}[\alpha_*] \colon
  [\ul\C(\mathcal{Y})] 
  \ra 
  \D(\mathcal{X})$
  where $\rho \colon q_\mathcal{X}
  [\alpha_*] \ra (\dR \alpha_*) q_\mathcal{Y}$
  then there is a unique isotransformation
  $\omega \colon 
  [\ul\alpha_*] \ol{[\ii]} \sira \ol{[\ii]} (\dR \alpha_*)$ 
  of triangulated $\kk$-functors such that
  the diagram
  \begin{equation}
    \xymatrix{
      & {[\ii \alpha_*]} \ar@{}[r]|-{=} 
      \ar[ld]_-{[\ii \alpha_* \iotaii]} 
      & {\ol{[\ii]}q_\mathcal{X}[\alpha_*]} 
      \ar[rd]^-{\ol{[\ii]}\rho}
      \\
      {[\ul\alpha_*] \ol{[\ii]} q_\mathcal{Y}} 
      \ar[rrr]^-{\omega q_\mathcal{Y}}_-{\sim} 
      &&& 
      {\ol{[\ii]} (\dR \alpha_*) q_\mathcal{Y}} 
    }
  \end{equation}
  commutes. This isotransformation $\omega$ coincides with the
  canonical isotransformation 
  \eqref{eq:omega-alpha_*}
  constructed in the proof of
  Proposition~\ref{p:canonical-isotrafos}. 
\end{proposition}

\begin{proof}
  The first claim is obvious. Since
  $(\ol{[\ii]}(\dR \alpha_*), \ol{[\ii]} \rho)$ is a right
  derived functor of
  $[\ii \alpha_*] \colon [\ul\C(\mathcal{Y})] \ra
  [\ul\II(\mathcal{X})]$
  the second claim follows from the universal property of a right
  derived functor.
  To see the last claim, note that
  $q_\mathcal{X}[\alpha_*] \ol{[\ii]}$
  together with 
  $q_\mathcal{X} [\alpha_*] [\iotaii] \colon q_\mathcal{X}
  [\alpha_*] \ra q_\mathcal{X}[\alpha_*] \ol{[\ii]}
  q_\mathcal{Y}= q_\mathcal{X}[\alpha_*] [\ii]$
  is a right derived functor of $q_\mathcal{X}[\alpha_*]$. In
  this case both isotransformations $\omega$ and
  $\omega_{\alpha_*}$ are obviously the identity.
\end{proof}


\begin{proposition}
  \label{p:ul-alpha^*-as-derived-functor}
  Let $\alpha \colon (\Sh(\mathcal{Y}), \mathcal{O}_\mathcal{Y}) \ra
  (\Sh(\mathcal{X}), \mathcal{O}_\mathcal{X})$ be a morphism of
  $\kk$-ringed topoi.
  Consider the 
  zig-zag of $\KK$-natural transformations
  \begin{equation}
     \ul\alpha^* \ii =
     \ii \alpha^* \ee \ii
     \xla{\ii \alpha^* \ee \iotaii}
     \ii \alpha^* \ee
     \xra{\ii \alpha^* \epsilonee}
     \ii \alpha^* 
     \colon
     \ul\C(\mathcal{X}) \ra \ul\II(\mathcal{Y}).
  \end{equation}
  Then the transformation on the left becomes invertible on
  homotopy categories, and we obtain the transformation
  \begin{equation}
    [\ii \alpha^* \epsilonee]
    \circ ([\ii \alpha^*\ee\iotaii])\inv
    \colon 
    [\ul\alpha^* \ii] =
     [\ul\alpha^*] \ol{[\ii]} q_\mathcal{X} 
     \ra
     [\ii \alpha^*]
     \colon
     [\ul\C(\mathcal{X})] \ra [\ul\II(\mathcal{Y})]
  \end{equation} 
  of triangulated $\kk$-functors.
  Then the pair
  $([\ul\alpha^*] \ol{[\ii]}, 
  [\ii \alpha^* \epsilonee]
  \circ ([\ii \alpha^*\ee\iotaii])\inv)$
  is a left derived functor of
  $[\ii \alpha^*] \colon [\ul\C(\mathcal{X})] \ra
  [\ul\II(\mathcal{Y})]$.
  In particular, if 
  $(\dL \alpha^*, \lambda)$
  is a left derived functor of
  $q_\mathcal{Y}[\alpha^*] \colon
  [\ul\C(\mathcal{X})] 
  \ra 
  \D(\mathcal{Y})$
  where $\lambda \colon
  (\dL \alpha^*) q_\mathcal{X}
  \ra q_\mathcal{Y}[\alpha^*]$
  then there is a unique isotransformation $\omega \colon
  [\ul\alpha^*] \ol{[\ii]} \sira \ol{[\ii]} (\dL \alpha^*)$ 
  of triangulated $\kk$-functors such that
  the diagram
  \begin{equation}
    \xymatrix{
      {[\ul\alpha^*] \ol{[\ii]} q_\mathcal{X}}
      \ar[rd]_-{[\ii \alpha^* \epsilonee]
        \circ ([\ii \alpha^*\ee\iotaii])\inv}
      \ar[rrr]^-{\omega q_\mathcal{X}}_-{\sim}
      &&&
      {\ol{[\ii]} (\dL \alpha^*) q_\mathcal{X}}
      \ar[ld]^-{\ol{[\ii]}\lambda}
      \\
      & {[\ii \alpha^*]} 
      \ar@{}[r]|-{=}  
      & {\ol{[\ii]}q_\mathcal{Y}[\alpha^*]} 
    }
  \end{equation}
  commutes.
  This isotransformation $\omega$ coincides with the
  canonical isotransformation 
  \eqref{eq:omega-alpha^*}
  constructed in the proof of
  Proposition~\ref{p:canonical-isotrafos}.
\end{proposition}

\begin{proof}
  We have seen 
  in 
  \eqref{eq:ii-alpha^*-ee-iotaii-E}
  in the proof of Proposition~\ref{p:canonical-isotrafos}
  that the morphism 
  $[\ii \alpha^*\ee\iotaii_E]$ in $[\ul\II(\mathcal{Y})]$ is
  invertible for any $E \in \C(\mathcal{X})$.
  With notation from the commutative
  diagram~\eqref{eq:diagram-for-bL-alpha^*}, 
  \begin{equation}
    \label{eq:obvious-L-alpha^*}
    (\ol{[\alpha^*]} \; \ol{[\ee]},\;
    \ol{[\alpha^*]} \; \ol{[\ee]}q_\mathcal{X} =
    q_\mathcal{Y} [\alpha^* \ee] 
    \xra{q_\mathcal{Y} [\alpha^* \epsilonee]} 
    q_\mathcal{Y} [\alpha^*])
  \end{equation}
  is a left derived functor of $q_\mathcal{Y} [\alpha^*]$, and,
  by postcomposing with $\ol{[\ii]}$,
  \begin{equation}
    (\ol{[\ii]} \; \ol{[\alpha^*]} \; \ol{[\ee]}, \;
    \ol{[\ii]} \; \ol{[\alpha^*]} \; \ol{[\ee]}q_\mathcal{X} =
    [\ii \alpha^* \ee] 
    \xra{[\ii \alpha^* \epsilonee]} 
    [\ii \alpha^*])
  \end{equation}
  is a left derived functor of $[\ii \alpha^*]$.  This implies
  that the pair in the proposition is a left derived functor of
  $[\ii \alpha^*]$. The same statement is true for
  $(\ol{[\ii]}(\dL \alpha^*), \ol{[\ii]} \lambda)$ and implies 
  the second claim. The last claim is now immediate from the
  construction 
  of $\omega_{\alpha^*}$.
\end{proof}

\begin{remark}
  \label{rem:omega-alpha^*-wlog}
  If we take \eqref{eq:obvious-L-alpha^*}
  as the left derived functor of $q_\mathcal{Y} [\alpha^*]$ the
  isotransformation $\omega_{\alpha^*}$ is given by
  \begin{equation}
    \label{eq:omega-alpha^*-wlog}
    \omega_{\alpha^*} 
    = \ol{[\ii]} \; \ol{[\alpha^*]} \; \ol{[\ee]}
    \; \ol{[\iotaii]}\inv
    \colon
    [\ul\alpha^*] \ol{[\ii]}
    = 
    \ol{[\ii]} \; \ol{[\alpha^*]} \; \ol{[\ee]} q_\mathcal{X}
    \ol{[\ii]}  
    \sira \ol{[\ii]} \; \ol{[\alpha^*]}
    \; 
    \ol{[\ee]}  
  \end{equation}
  where $\ol{[\iotaii]} \colon \id \ra q_\mathcal{X} \ol{[\ii]}$ is
  the obvious isotransformation induced by $\iotaii$.
  This is of course the same isotransformation we constructed in
  the proof 
  of Proposition~\ref{p:canonical-isotrafos}.
\end{remark}

\fussnote{
  I guess we could also assume that 
  \begin{equation}
    (q_\mathcal{Y}[\alpha^*] \; [\ee] \; \ol{[\ii]},\;
    ...)
  \end{equation}
  is a left derived functor of $q_\mathcal{Y} [\alpha^*]$.
  Then $\omega_{\alpha^*}$ is identity.
}

\subsubsection{Dependence on choices}
\label{sec:dependence-choices}

This subsubsection may be skipped on a first reading.

\begin{remark}
  \label{rem:compare-alpha_*-different-ii}
  We explain Remark~\ref{rem:choice-of-ii-and-ee}
  for $\ul\alpha_*$.
  Let 
  $(\ii', \iotaii')=(\ii'_{(\mathcal{X}, \mathcal{O})},
  \iotaii'_{(\mathcal{X}, \mathcal{O})})$ be another choice
  of a
  $\KK$-enriched $\II$-fibrant resolution functor.
  Then
  Proposition~\ref{p:compare-K-enri-I-fibr-repl-functors}
  provides a 2-isomorphism
  $\ul\alpha_* = \ii
  \alpha_* \sira \ii' 
  \alpha_*$
  between 1-morphisms $\ul\II(\mathcal{Y}) \ra \ul\II(\mathcal{X})$
  in $\ENH_\kk$, and the 2-isomorphisms obtained in this way are
  compatible in the expected manner, cf.\
  \eqref{eq:phi-ii'-ii-assoc},
  \eqref{eq:phi-ii-ii-id}.
\end{remark}

\begin{remark}
  \label{rem:compare-alpha^*-different-ee}
  We explain Remark~\ref{rem:choice-of-ii-and-ee} for
  $\ul\alpha^*$. Let
  $(\ee', \epsilonee')=(\ee'_{(\mathcal{X}, \mathcal{O})},
  \epsilonee'_{(\mathcal{X}, \mathcal{O})})$
  be another $\KK$-enriched $\EE$-cofibrant resolution functor.
  Then Proposition~\ref{p:compare-K-enri-E-fibr-repl-functors}
  provides a 2-isomorphism 
  $\psi \colon$
  $\ul\alpha^*=\ii\alpha^*\ee 
  \sira 
  \ii\alpha^*\ee'$ between 1-morphisms $\ul\II(\mathcal{X}) \ra
  \ul\II(\mathcal{Y})$ in $\ENH_\kk$, and the 2-isomorphisms obtained
  in this way are compatible as expected, cf.\
  \eqref{eq:psi-ee-ee'-assoc}, \eqref{eq:psi-ee-ee'-id}. 
  Similarly, using
  Proposition~\ref{p:compare-K-enri-I-fibr-repl-functors}, 
  one may also include different choices of 
  $\KK$-enriched $\II$-fibrant resolution functors
  into this discussion.
\end{remark}

\subsection{Interpretation of formulas in enhancements}
\label{sec:interpr-form-enhanc}

Recall that $\enh_\kk$ denotes the underlying multicategory of the
$\kk$-linear 2-multicategory
$\ENH_\kk$. 
Recall the category $\fml'_\kk$ of formulas from
\ref{sec:categ-form-four} and the interpretation functor $\D
\colon \fml'_\kk \ra \trcat_\kk$.  
Similarly, there is a unique interpretation functor
\begin{equation}
  \label{eq:II-interpret}
  \ul\II \colon
  \fml'_\kk \ra \enh_\kk
\end{equation}
of multicategories given on objects by 
$\ud{\mathcal{X}} \mapsto \ul\II(\mathcal{X})$ 
and on 
generating morphisms by $\ud\otimes \mapsto \ul\otimes$,
$\ud\sheafHom \mapsto \ul\sheafHom$,
$\ud\alpha^* \mapsto \ul\alpha^*$,
$\ud\alpha_* \mapsto \ul\alpha_*$,
$\ud\Hom \mapsto \ul\Hom$, $\ud{G} \mapsto \ul{G}$
(cf.\ Remark~\ref{rem:functor-on-fml}).
The notation $\ul\II$ for this
functor is justified by
$\ul\II(\ud{\mathcal{X}})=\ul\II(\mathcal{X})$; note that for
example $\ul\II(\ud\alpha_*)= \ul\alpha_*$. Generalizing this we
abbreviate $\ul{t}:=\ul\II(\ud{t})$ if $\ud{t}$ is any morphism in
$\fml'_\kk$.

\subsubsection{Relation to interpretation of formulas in triangulated
  categories}
\label{sec:relat-interpr-triang}

Let $[\ul\II] \colon \fml'_\kk \ra \trcat_\kk$ denote the
composition of 
\eqref{eq:II-interpret} with 
\eqref{eq:[]-enh}, mapping $\ud{\mathcal{X}}$ to
$[\ul\II(\mathcal{X})]$ and, for example, $\ud\alpha^*$ to
$[\ul\alpha^*]$. Then there is a unique 
pseudo-natural 
transformation (see \cite[7.5]{borceux-cat},
generalized to 2-multicategories; here we view the functors $\D$ and
$[\ul\II]$ of multicategories in the trivial way as functors of
2-multicategories)
\begin{equation}
  \label{eq:pseudo-nat-trafo-D-to-[I]}
  (\ol{[\ii]}, \omega) \colon \D \ra [\ul\II] 
\end{equation}
that maps an object $\mathcal{X}$ (resp.\ $\mathcal{X}^\opp$)
to the 1-morphism \eqref{eq:ol[ii]} (resp.\ its opposite)
(which is an equivalence of triangulated $\kk$-categories) 
and that maps the generating morphisms $\ud\alpha^*$,
$\ud\alpha_*$, $\ud\otimes$, $\ud\sheafHom$, $\ud\Hom$, $\ud{G}$
to the 
2-isomorphisms 
in the diagrams \eqref{eq:26}, \eqref{eq:34}, \eqref{eq:36},
\eqref{eq:omega-ulHom},
\eqref{eq:omega-G},
respectively, and accordingly for the opposite generating
morphisms.
If $\ud t$ is an arbitrary morphism in $\fml'_\kk$ 
the corresponding 2-isomorphism is denoted
\begin{equation}
  \label{eq:omega-ud-t}
  \omega_{\ud t} \colon [\ul{t}] \ol{[\ii]} \sira \ol{[\ii]}
  \D(\ud{t}).
\end{equation}
For example, 
$\omega_{\ud{\alpha}^*}=\omega_{\alpha^*}$ and
\begin{equation}
  \omega_{\ud\sheafHom(\ud\alpha^*(-),-)}
  \colon 
  \ul\sheafHom(\ul\alpha^* \ol{[\ii]} (-), \ol{[\ii]}(-))
  \sira
  \ol{[\ii]} \dR\sheafHom(\dL\alpha^*(-),-).  
\end{equation}

\fussnote{
  Second condition in  
  in \cite[Def.~7.5.2]{borceux-cat}
  i.\,e.\ 
  Diagram 7.15,
  says how to extend
  (compatibly) from
  generating morphisms to trees. Vertical arrows in this diagram
  are equalities in our case.
}

\begin{example}
  \label{exam:omega-alpha-alpha}
  The 2-isomorphism $\omega_{\ud\alpha_*\ud\alpha^*}$ is
  obtained by juxtaposing the two diagrams in 
  \eqref{eq:26}, i.\,e.\ it is 
  the
  composition
  \begin{equation}
    \label{eq:omega-alpha_*-alpha^*}
    \omega_{\ud\alpha_*\ud\alpha^*} \colon
    [\ul\alpha_*][\ul\alpha^*]\ol{[\ii]}
    \xsira{[\ul\alpha_*]\omega_{\alpha^*}}
    [\ul\alpha_*]\ol{[\ii]}\dL\alpha^*
    \xsira{\omega_{\alpha_*}\dL\alpha^*}
    \ol{[\ii]}\dR\alpha_*\dL\alpha^*.
  \end{equation}
  Similarly, the 2-isomorphism $\omega_{\ud\alpha^*\ud\alpha_*}$ is
  the
  composition
  \begin{equation}
    \label{eq:omega-alpha^*-alpha_*}
    \omega_{\ud\alpha^*\ud\alpha_*} \colon
    [\ul\alpha^*][\ul\alpha_*]\ol{[\ii]}
    \xsira{[\ul\alpha^*]\omega_{\alpha_*}}
    [\ul\alpha^*]\ol{[\ii]}\dR\alpha_*
    \xsira{\omega_{\alpha^*}\dR\alpha_*}
    \ol{[\ii]}\dL\alpha^*\dR\alpha_*.
  \end{equation}
  If we choose $\dR\alpha_*$ as in the proof of 
  Proposition~\ref{p:canonical-isotrafos},
  then $\omega_{\alpha_*}$ is the identity
  and we have
  $\omega_{\ud\alpha_*\ud\alpha^*} =
  [\ul\alpha_*]\omega_{\alpha^*}$
  and
  $\omega_{\ud\alpha^*\ud\alpha_*} =
  \omega_{\alpha^*}\dR\alpha_*$.
\end{example}

\begin{remark}
  \label{rem:outlook:t:compare-interpretations-FML}
  Later on, in \ref{sec:2-multicat-formulas},
  we will extend $\fml'_\kk$ to a $\kk$-linear 2-multicategory
  $\FML_\kk$, the 
  functors $\D$ and $\ul\II$ (and $[\ul\II]$) to
  functors of $\kk$-linear 2-multicategories, and $(\ol{[\ii]},
  \omega)$ to a 
  pseudo-natural transformation $\D \ra [\ul\II]$ between these
  extensions,
  see Theorem~\ref{t:compare-interpretations-FML}.
  This theorem summarizes many results of this article.
\end{remark}

\subsection{Lifts of relations}
\label{sec:lifts-relations-2}

\subsubsection{Lifts of 2-(iso)morphisms}
\label{sec:lifts-2-morphisms}

\begin{definition}
  \label{d:enhance-2-morphism}
  Let $\ud{v}$, $\ud{w} \colon
  ((\ud{\mathcal{X}}_1)^{\epsilon_1},
  \dots,
  (\ud{\mathcal{X}}_n)^{\epsilon_n})  
  \ra \ud{\mathcal{Y}}^{\delta}$ be morphisms in $\fml'_\kk$. We say
  that a 2-morphism $\tau \colon \ul{v}=\ul\II(\ud{v}) \ra \ul{w}
  = \ul\II(\ud{w})$ in $\ENH_\kk$
  \define{$(\ud{v},\ud{w})$-enhances} a 2-morphism
  $\sigma \colon \D(\ud{v}) \ra \D(\ud{w})$ in $\TRCAT_\kk$ if
  the equality of 2-morphisms
  \begin{equation}
    \label{eq:sigma-omega=omega-tau}
    \omega_{\ud{w}} 
    ([\tau]\ol{[\ii]})
    =
    (\ol{[\ii]}\sigma) 
    \omega_{\ud{v}}
  \end{equation}
  in $\TRCAT_\kk$ holds, i.\,e.\ if the diagram
  \begin{equation}
    \label{eq:sigma-omega=omega-tau-diagram}
    \xymatrix{
      {[\ul{v}]\ol{[\ii]}} 
      \ar[d]_-{\omega_{\ud{v}}}^-{\sim}
      \ar[r]^-{[\tau]\ol{[\ii]}} &
      {[\ul{w}]\ol{[\ii]}} 
      \ar[d]_-{\omega_{\ud{w}}}^-{\sim}\\
      {\ol{[\ii]}\D(\ud{v})} \ar[r]^-{\ol{[\ii]}\sigma} &
      {\ol{[\ii]}\D(\ud{w})}
    }
  \end{equation}
  in $\TRCAT(\D(\mathcal{X}_1)^{\epsilon_1},
  \dots,
  \D(\mathcal{X}_n)^{\epsilon_n};  
  [\ul\II(\mathcal{Y})]^{\delta})$ is commutative.
  The following diagram
  illustrates the situation.
  \begin{equation}
    \label{eq:2-cat-diagram-enhances}
    \begin{tikzpicture}[baseline=(current  bounding  box.center)]
      \matrix (m) [matrix of math nodes, row sep=6em, column
      sep=10em, minimum width=2em, text height=1.5ex, text
      depth=0.25ex] 
      { 
        \D(\mathcal{X}_1)^{\epsilon_1} \times \dots \times
        \D(\mathcal{X})^{\epsilon_n}    
        & \D(\mathcal{Y})^\delta \\
        {[\ul\II(\mathcal{X})]^{\epsilon_1} \times \dots \times
          [\ul\II(\mathcal{X})]^{\epsilon_n}}
        & {[\ul\II(\mathcal{Y})]^\delta} \\
      };
      \path[->, font=\small]
      (m-1-1) edge node [left] {$\ol{[\ii]} \times \dots \times \ol{[\ii]}$} (m-2-1)
      (m-1-1.north east) edge [bend left=20] node [above] {$\D(\ud{v})$} 
                                             node (Dv) {} 
                                             node [pos=0.2] (Dvleft) {} (m-1-2)
      (m-1-1.south east) edge [bend right=20] node [below] {$\D(\ud{w})$}
                                              node (Dw) {} 
                                              node [pos=0.8]
                                              (Dwright) {} (m-1-2)
      (m-2-1.south east) edge [bend right=20] node [below] {$[\ul{w}]=[\ul\II(\ud{w})]$}
                                              node (Iw) {}
                                              node [pos=0.8] (Iwright) {} (m-2-2)
      (m-1-2) edge node [right] {$\ol{[\ii]}$} (m-2-2);
  
      \draw[double, double equal sign distance, -implies] (Iwright) --
      node [pos=0.55, right]{\scriptsize{$\omega_\ud{w}$}} (Dwright);    
      \path[->, font=\small]
      (m-2-1.north east) edge [bend left=20, -, line width=6pt, draw=white] (m-2-2)
                         edge [bend left=20] node [above] {[$\ul{v}]=[\ul\II(\ud{v})]$}
                                             node (Iv) {} 
                                             node [pos=0.2] (Ivleft) {} (m-2-2);

      \draw[double, double equal sign distance, -implies] (Dv) --
      node [right] {\scriptsize{$\sigma$}} (Dw);    
      \draw[-, line width=6pt, draw=white] (Dvleft) -- (Ivleft);    
      \draw[double, double equal sign distance, -implies] (Ivleft) --
      node [pos=0.45, left]{\scriptsize{$\omega_\ud{v}$}} (Dvleft);    
      \draw[double, double equal sign distance, -implies] (Iv) --
      node [left] {\scriptsize{$[\tau]$}} (Iw);    
    \end{tikzpicture}
  \end{equation}
  Usually, $\ud{v}$ and $\ud{w}$ are clear from the context and
  we just say that $\tau$ \define{enhances} $\sigma$.
\end{definition}

\begin{remark}
  \label{rem:enhance-2-morphism-compose-and-invert}
  Keep the setting of Definition~\ref{d:enhance-2-morphism}. Let
  $\ud u$ be another morphism in $\fml'_\kk$, and assume that
  $\tau' \colon \ul{w}=\ul\II(\ud{w}) \ra \ul{u} =\ul\II(\ud{u})$
  $(\ud{w},\ud{u})$-enhances
  $\sigma' \colon \D(\ud{w}) \ra \D(\ud{u})$. Then
  $\tau' \tau$ 
  $(\ud{v},\ud{u})$-enhances $\sigma' \sigma$.
  If $\tau$ 
  $(\ud{v},\ud{w})$-enhances $\sigma$ and both $\tau$ and
  $\sigma$ are
  2-isomorphisms,
  then 
  $\tau\inv$ 
  $(\ud{w},\ud{v})$-enhances $\sigma\inv$.
\end{remark}

\begin{proposition}
  \label{p:adjunction-morphisms-push-pull}
  Let
  $\alpha \colon (\Sh(\mathcal{Y}), \mathcal{O}_\mathcal{Y}) \ra
  (\Sh(\mathcal{X}), \mathcal{O}_\mathcal{X})$
  be a morphism of $\kk$-ringed topoi. 
  Then the zig-zag
  \begin{equation}
    \label{eq:zigzag-aa-id}
    \ul\alpha^* \ul\alpha_* =\ii \alpha^* \ee \ii\alpha_* 
    \xla{[\sim]}
    \ii \alpha^* \ee \alpha_*
    \ra
    \ii \alpha^* \alpha_*
    \ra
    \ii 
    \xla{[\sim]}
    \id
  \end{equation}
  of 2-morphisms in $\tildew{\ENH}_\kk$ defines a 2-morphism
  \begin{equation}
    \label{eq:aa-id}
    \ul\alpha^* \ul\alpha_* \ra \id
  \end{equation}
  in $\ENH_\kk$ that enhances the counit $\dL\alpha^* \dR\alpha_*
  \ra \id$ (see Remarks~\ref{rem:fixed-data-alpha*-adjunction}
  and
  \ref{rem:aa-id-precisely} for more details),
  and the zig-zag
  \begin{equation}
    \label{eq:zigzag-id-aa}
    \id \xra{[\sim]}
    \ii 
    \xla{[\sim]}
    \ii \ee 
    \ra
    \ii \alpha_* \alpha^*\ee
    \ra
    \ii \alpha_* \ii \alpha^*\ee
    = \ul\alpha_* \ul\alpha^*
  \end{equation}
  of 2-morphisms in $\tildew{\ENH}_\kk$ defines a 2-morphism
  \begin{equation}
    \label{eq:id-aa}
    \id 
    \ra \ul\alpha_* \ul\alpha^* 
  \end{equation}
  in $\ENH_\kk$ 
  that enhances the unit morphism $\id \ra \dR\alpha_*\dL\alpha^*$.
\end{proposition}

\begin{remark}
  \label{rem:fixed-data-alpha*-adjunction}
  In the formulation of
  Proposition~\ref{p:adjunction-morphisms-push-pull} we have
  implicitly fixed an adjunction 
  $(\alpha^*, \alpha_*, \id \xra{\eta} \alpha_* \alpha^*,
  \alpha^* \alpha_* \xra{\theta} \id)$ giving
  rise to an adjunction 
  $(\dL\alpha^*, \dR\alpha_*, \eta^\D, \theta^\D)$. We will
  tacitly fix and use similar data in the following.
\end{remark}

\begin{remark}
  \label{rem:aa-id-precisely}
  The precise meaning of the first claim in
  Proposition~\ref{p:adjunction-morphisms-push-pull} is:
  in the zig-zag
  \begin{equation}
    \ul\alpha^* \ul\alpha_* =\ii \alpha^* \ee \ii\alpha_* 
    \xla{\ii\alpha^*\ee\iotaii\alpha_*}
    \ii \alpha^* \ee \alpha_*
    \xra{\ii \alpha^* \epsilonee \alpha_*}
    \ii \alpha^* \alpha_*
    \xra{\ii \theta}
    \ii 
    \xla{\iotaii}
    \id
  \end{equation}
  of 2-morphisms (between 1-morphisms $\ul\II(\mathcal{Y}) \ra
  \ul\II(\mathcal{Y})$)  
  in $\tildew{\ENH}_\kk$
  all left-pointing arrows are objectwise homotopy equivalences
  (here $\theta$ is as in
  Remark~\ref{rem:fixed-data-alpha*-adjunction}); 
  therefore, 
  as explained in Remark~\ref{rem:zig-zags-define-2-morph-in-ENH}, 
  \begin{equation}
    \label{eq:aa-id-in-detail}
    \delta(\iotaii)\inv
    \delta(\ii \theta)
    \delta(\ii \alpha^* \epsilonee \alpha_*)
    \delta(\ii\alpha^*\ee\iotaii\alpha_*)\inv
    \colon
    \ul\alpha^* \ul\alpha_* \ra \id
  \end{equation}
  defines a 2-morphism in $\ENH_\kk$ 
  (where $\delta \colon \tildew{\ENH}_\kk \ra \ENH_\kk$ is the
  functor
  \eqref{eq:delta-tilde-ENH-to-ENH})
  and this 2-morphism
  $(\ud\alpha^* \ud\alpha_*, \id)$-enhances the counit 2-morphism
  $\theta^\D \colon \dL\alpha^* \dR\alpha_* \ra \id$ in $\TRCAT_\kk$.
  We leave it to the reader to interpret similar claims
  accordingly. 
\end{remark}

\fussnote{
  could prove similar adjunction morphisms for $(\otimes,
  \sheafHom)$.
}

\fussnote{
  Are these adjunction morphisms compatible with 
  say 
  Lemma~\ref{l:enhanced-pull-push-ulHom}.
}

\begin{proof}
  Let $\eta$ and $\theta$ be as in
  Remark~\ref{rem:fixed-data-alpha*-adjunction}. The evaluation
  of \eqref{eq:zigzag-aa-id} at $J \in \ul\II(\mathcal{Y})$ is
  the zig-zag
  \begin{equation}
    \label{eq:aa-id-construct}
    \ii \alpha^* \ee \ii\alpha_* J
    \xla{\ii \alpha^* \ee \iotaii_{\alpha_* J}}
    \ii \alpha^* \ee \alpha_* J
    \xra{\ii \alpha^* \epsilonee_{\alpha_*J}}
    \ii \alpha^* \alpha_* J
    \xra{\ii \theta_J}
    \ii J
    \xla{\iotaii_J}
    J
  \end{equation}
  in $\II(\mathcal{Y})$.
  The two arrows pointing to the left are homotopy equivalences,
  by Proposition~\ref{p:pullback-preserves-h-flat} (and
  Remark~\ref{rem:ii-maps-qisos-to-htpy-equis}).
  This shows that the left-pointing arrows in
  \eqref{eq:zigzag-aa-id} are objectwise homotopy equivalences;
  hence this zig-zag defines the 2-morphism \eqref{eq:aa-id}.
  Let us call it $\tau$.

  By choosing $\dR \alpha_*$ and $\dL\alpha^*$ as in
  the proof of Proposition~\ref{p:canonical-isotrafos}
  (more precisely, as in the proof of
  Proposition~\ref{p:ul-alpha_*-as-derived-functor} and
  as in Remark~\ref{rem:omega-alpha^*-wlog})
  we can assume that
  $\omega_{\alpha_*}$ is the identity and that
  $\omega_{\alpha^*}$ is \eqref{eq:omega-alpha^*-wlog}. Then
  the counit $\theta^\D \colon \dL \alpha^* \dR \alpha_* \ra \id$
  is the composition
  \begin{multline*}
    \theta^\D \colon \dL \alpha^* \dR \alpha_* 
    =
    \ol{[\alpha^*]} \; \ol{[\ee]}
    q_\mathcal{X}[\alpha_*] \ol{[\ii]}
    =
    q_\mathcal{Y}[\alpha^*] [\ee]
    [\alpha_*] \ol{[\ii]}\\
    \xra{q_\mathcal{Y}[\alpha^*][\epsilonee][\alpha_*]
      \ol{[\ii]}}
    q_\mathcal{Y}[\alpha^*\alpha_*] \ol{[\ii]}
    \xra{q_\mathcal{Y}[\theta] \ol{[\ii]}}
    q_\mathcal{Y} \ol{[\ii]}
    \xra{\ol{[\iotaii]}\inv}
    \id,
  \end{multline*}
  cf.\ \eqref{eq:diagram-for-bL-alpha^*}.
  We need to show (cf.\ \eqref{eq:sigma-omega=omega-tau} and
  Example~\ref{exam:omega-alpha-alpha})
  that
  \begin{equation}
    [\tau] \ol{[\ii]} = (\ol{[\ii]} \theta^\D) (\omega_{\alpha^*}
    \dR \alpha_*)
  \end{equation}
  as natural transformations
  $[\ul\alpha^*][\ul\alpha_*] \ol{[\ii]} = 
  [\ul\alpha^*] \ol{[\ii]} \dR\alpha_* 
  \ra \ol{[\ii]}$ between
  triangulated $\kk$-functors
  $\D(\mathcal{Y}) \ra [\ul\II(\mathcal{Y})]$.  It is enough to
  show that the evaluations of both transformations at an
  arbitrary object 
  $F \in \D(\mathcal{Y})$ coincide.
  Consider the composition
  \begin{equation}
    \ii \alpha^* \ee \ii \alpha_* \ii F 
    \xra{[\ii\alpha^*\ee\iotaii_{\alpha_*\ii F}]\inv}
    \ii \alpha^* \ee \alpha_* \ii F
    \xra{[\ii \alpha^* \epsilonee_{\alpha_* \ii F}]}
    \ii \alpha^* \alpha_* \ii F
    \xra{[\ii \theta_{\ii F}]}
    \ii \ii F.
  \end{equation}
  If we compose it with
  $[\iotaii_{\ii F}]\inv \colon\ii \ii F \ra \ii F$ we obtain the
  evaluation of $[\tau] \ol{[\ii]}$ at $F$; if we compose it with
  $[\ii \iotaii_F]\inv \colon \ii \ii F \ra \ii F$
  we obtain the evaluation of $(\ol{[\ii]} \theta^\D)
  (\omega_{\alpha^*}\dR \alpha_*)$ at $F$ 
  (cf.\ \eqref{eq:ii-alpha^*-ee-iotaii-E}). 
  These two evaluations coincide because 
  $[\iotaii_{\ii F}] = [\ii
  \iota_F]$, by
  Lemma~\ref{l:l-iotaii-ii-vs-ii-iotaii} below. This shows that
  $\ul\alpha^* \ul\alpha_* \ra \id$
  enhances $\dL \alpha^* \dR \alpha_* \ra \id$. 
  
  The evaluation of the zig-zag \eqref{eq:zigzag-id-aa}
  at $I \in \ul\II(\mathcal{X})$ 
  is the zig-zag 
  \begin{equation}
    \label{eq:id-aa-construct}
    I 
    \xra{\iotaii_I}
    \ii I 
    \xla{\ii \epsilonee_I}
    \ii \ee I
    \xra{\ii \eta_{\ee I}} 
    \ii \alpha_* \alpha^*\ee I
    \xra{\ii\alpha_* \iotaii_{\alpha^*\ee I}}
    \ii \alpha_* \ii \alpha^*\ee I
    = \ul\alpha_* \ul\alpha^* I
  \end{equation}
  in $\II(\mathcal{X})$.
  Its first two arrows are obviously homotopy equivalences.
  This shows that the indicated two arrows in
  \eqref{eq:zigzag-id-aa} are objectwise homotopy equivalences;
  hence this zig-zag defines the 2-morphism \eqref{eq:id-aa}.
  Let us call it $\rho$.

  \fussnote{
    Ob der Term
    $\ii \alpha_* \alpha^*\ee I$
    in \eqref{eq:id-aa-construct}
    f\"ur $I \in \ul\QQcoh(X)$ (und $\alpha$ konzentriert) in
    $\ul\QQcoh(X)$ liegt, ist unklar?
  }

  We choose $\dR \alpha^*$ and $\dL \alpha^*$ as before.
  Consider the composition 
  \begin{equation}
    \id \xla{[\epsilonee]}
    [\ee] \xra{[\eta][\ee]}
    [\alpha_* \alpha^* \ee]
    \xra{[\alpha_*][\iotaii][\alpha^*][\ee]} 
    [\alpha_* \ii \alpha^* \ee]
  \end{equation}
  of transformations of
  endofunctors of
  $[\ul\C(\mathcal{X})]$.
  If we apply $q_\mathcal{X}$, the left-pointing arrow becomes
  invertible and we obtain a transformation
  \begin{equation}
    \label{eq:theta-tilde}
    \tilde{\eta} \colon 
    q_\mathcal{X} \ra
    q_X[\alpha_*] [\ii] [\alpha^*] [\ee]
    =
    q_X[\alpha_*] \ol{[\ii]} q_\mathcal{Y} [\alpha^*] [\ee]
    = \dR \alpha_* \ol{[\alpha^*]} \; \ol{[\ee]} q_\mathcal{X}
    = \dR \alpha_* \dL \alpha^* q_\mathcal{X} 
  \end{equation}
  using \eqref{eq:diagram-for-bL-alpha^*}.
  It is easy to see (using that $[\alpha_* \ii \alpha^* \ee]$
  preserves quasi-isomorphisms) that there is a unique
  transformation
  $\eta^\D \colon \id \ra \dR \alpha_* \dL \alpha^*$
  such that $\eta^\D q_\mathcal{X}=     \tilde{\eta}$;
  moreover, this transformation is the
  unit of 
  the adjunction $(\dL \alpha^*, \dR \alpha_*)$.

  We need to show (cf.\ \eqref{eq:sigma-omega=omega-tau} and
  Example~\ref{exam:omega-alpha-alpha})
  that 
  $([\ul\alpha_*] \omega_{\alpha^*})([\rho] \ol{[\ii]}) =
  \ol{[\ii]} \eta^\D$. By precomposing with $q_\mathcal{X}$
  it is enough to show that
  \begin{equation}
    ([\ul\alpha_*] \omega_{\alpha^*} q_\mathcal{X})([\rho] [\ii]) =
    \ol{[\ii]} \tilde{\eta}
  \end{equation}
  as natural transformations $[\ii] \ra
  \ol{[\ii]}\dR \alpha_* \dL \alpha^* q_\mathcal{X} =
  [\ii][\alpha_*][\ii][\alpha^*] [\ee]$ between triangulated
  functors 
  $[\ul\C(\mathcal{X})] \ra [\ul\II(\mathcal{X})]$. 
  The evaluation of 
  $\ol{[\ii]} \tilde{\eta}$
  at $E \in [\ul\C(\mathcal{X})]$ is
  \begin{equation}
    \label{eq:eval-ii-eta}
    \ii E
    \xra{[\ii\epsilonee_E]\inv}
    \ii\ee E 
    \xra{[\ii\eta_{\ee E}]}
    \ii\alpha_* \alpha^* \ee E
    \xra{[\ii\alpha_*\iotaii_{\alpha^*\ee E}]} 
    \ii\alpha_* \ii \alpha^* \ee E.
  \end{equation}
  On the other hand, the evaluation of
  $([\ul\alpha_*] \omega_{\alpha^*} q_\mathcal{X})([\rho] [\ii])$ at
  $E$ is (the first four arrows are 
  obtained from 
  \eqref{eq:id-aa-construct} and compose to $([\rho] [\ii])_E$,
  the last arrow is
  $([\ul\alpha_*] \omega_{\alpha^*})_E$, cf.\ 
  \eqref{eq:ii-alpha^*-ee-iotaii-E}
  or
  \eqref{eq:omega-alpha^*-wlog})
  \begin{equation}
    \label{eq:eval-omeage-rho-ii}
    \ii E 
    \xra{[\iotaii_{\ii E}]}
    \ii \ii E 
    \xra{[\ii \epsilonee_{\ii E}]\inv}
    \ii \ee \ii E
    \xra{[\ii \eta_{\ee \ii E}]} 
    \ii \alpha_* \alpha^*\ee \ii E
    \xra{[\ii\alpha_* \iotaii_{\alpha^*\ee \ii E}]}
    \ii \alpha_* \ii \alpha^*\ee \ii E
    \xra{[\ii \alpha_*\ii \alpha^* \ee \iotaii_E]\inv}
    \ii \alpha_* \ii \alpha^* \ee E.  
  \end{equation}
  To see that the compositions in \eqref{eq:eval-ii-eta}
  and \eqref{eq:eval-omeage-rho-ii} coincide consider the
  commutative diagram
  \begin{equation}
    \xymatrix{
      {\ii E} 
      \ar[d]_-{\ii \iotaii_E}^-{[\sim]}
      & 
      {\ii \ee E} 
      \ar[d]_-{\ii \ee \iotaii_E}^-{[\sim]} 
      \ar[l]_-{\ii \epsilonee_E}^-{[\sim]}
      \ar[r]^-{\ii \eta_{\ee E}}
      & 
      {\ii \alpha_*\alpha^*\ee E} 
      \ar[d]_-{\ii \alpha_*\alpha^* \ee \iotaii_E} 
      \ar[rr]^-{\ii \alpha_*\iotaii_{\alpha^* \ee E}} 
      & &
      {\ii \alpha_* \ii \alpha^*\ee E} 
      \ar[d]_-{\ii \alpha_*\ii\alpha^* \ee \iotaii_E}^-{[\sim]}
      \\
      {\ii\ii E}
      &
      {\ii \ee \ii E}
      \ar[l]_-{\ii \epsilonee_{\ii E}}^-{[\sim]}
      \ar[r]^-{\ii \eta_{\ee \ii E}}
      &
      {\ii \alpha_*\alpha^*\ee \ii E}
      \ar[rr]^-{\ii \alpha_*\iotaii_{\alpha^* \ee \ii E}} 
      & &
      {\ii \alpha_* \ii \alpha^*\ee \ii E}
    }
  \end{equation}
  in $\II(\mathcal{X})$, note that the arrows labeled $[\sim]$
  become 
  invertible in $[\ul\II(\mathcal{X})]$, and that moreover 
  $[\ii\iotaii_E]=[\iotaii_{\ii E}]$ 
  by Lemma~\ref{l:l-iotaii-ii-vs-ii-iotaii} below.
  This shows that $\id \ra \ul\alpha_* \ul\alpha^*$ enhances $\id
  \ra \dR \alpha_* \dL \alpha^*$.
\end{proof}

\fussnote{
  folgendes Lemma:
  koennte und sollte? vorziehen. koennte in Annahme
  schreiben:  
  let 
  $(\ii, \iotaii)=(\ii_{(\mathcal{X}, \mathcal{O})},
  \iotaii_{(\mathcal{X}, \mathcal{O})})$
  be a $\KK$-enriched $\II$-fibrant resolution functor 
  
  Funktor muss nicht enriched sein.
}

\fussnote{``same'' result true for 
  $\II$-fibrant and $\PP$-cofibrant resolution functor
  on $\MMod(\ulms{C})$.
}

\fussnote{
  Folgt wohl auch aus
  Proposition~\ref{p:compare-K-enri-I-fibr-repl-functors}. 
  sieht aber nach overkill aus. Andererseits ist unser Beweis
  hier so kurz und zeigt im wesentlichen auch einen Teil der
  dortigen 
  Proposition. 
}

\begin{lemma}
  \label{l:l-iotaii-ii-vs-ii-iotaii}
  Let $\mathcal{X}$ be a $\kk$-ringed site and $F \in
  \ul\C(\mathcal{X})$. Then the two morphisms
  $[\iotaii_{\ii F}]$, $[\ii \iota_F] \colon \ii F \ra \ii\ii F$
  in $[\ul\C(\mathcal{X})]$ coincide, 
  i.\,e.\ $[\iotaii_{\ii F}]=[\ii \iota_F]$.
\end{lemma}

\begin{proof}
  Consider the commutative square
  \begin{equation}
    \label{eq:ii-iota-vs-iota-ii}
    \xymatrix{
      {\ii F} \ar[r]^-{\ii \iotaii_F} & {\ii\ii F}\\
      {F} \ar[u]^-{\iotaii_F} \ar[r]^-{\iotaii_F}
      & {\ii F} \ar[u]^-{\iotaii_{\ii F}}
    }
  \end{equation}
  in $\C(\mathcal{X})$.
  Since $\iotaii_F$
  is a quasi-isomorphism and $\ii\ii F$ is h-injective, the map
  \begin{equation}
    [\iotaii_F]^* \colon [\ul\C_\mathcal{Y}](\ii F, \ii\ii F) \ra
    [\ul\C_\mathcal{Y}](F, \ii\ii F)  
  \end{equation}
  is an isomorphism; it maps 
  $[\iotaii_{\ii F}]$ and $[\ii \iota_F]$ to the same element.
\end{proof}

\begin{lemma}
  \label{l:composition-inverse-direct-image}
  Let $(\Sh(\mathcal{Z}), \mathcal{O}_\mathcal{Z}) \xra{\beta}
  (\Sh(\mathcal{Y}), \mathcal{O}_\mathcal{Y}) \xra{\alpha}
  (\Sh(\mathcal{X}), \mathcal{O}_\mathcal{X})$ be morphisms of
  $\kk$-ringed 
  topoi.
  Then the zig-zags
  \begin{align}
    \label{eq:id_*}
    \ul\id_* 
    & = \ii \xla{[\sim]} \id,\\
    \label{eq:alphabeta_*}
    \ul{(\alpha\beta)}_*
    & =
      \ii (\alpha\beta)_*
      \sira \ii \alpha_* \beta_* 
      \xra{[\sim]}
      \ii\alpha_*\ii\beta_*
      =      
      \ul\alpha_*\ul\beta_*,\\
    \ul\id^* 
    & = \ii\ee \xra{[\sim]} \ii
      \xla{[\sim]}
      \id,\\
    \label{eq:alphabeta^*-tilde-ENH}
    \ul{(\alpha\beta)}^* 
    & =  
      \ii(\alpha\beta)^* \ee 
      \sira 
      \ii\beta^* \alpha^*\ee 
      \xla{[\sim]}
      \ii\beta^*\ee \alpha^*\ee 
      \xra{[\sim]}
      \ii\beta^*\ee\ii \alpha^*\ee = 
      \ul\beta^*\ul\alpha^*
  \end{align}
  of objectwise homotopy 
  equivalences in $\tildew{\ENH}_\kk$ define 2-isomorphisms
  \begin{align}
    \label{eq:id_*-ENH}
    \ul\id_* 
    & \sira \id,
    \\
    \label{eq:alphabeta_*-ENH}
    \ul{(\alpha \beta)}_* 
    & \sira \ul\alpha_* \ul\beta_*,
    \\
    \label{eq:id^*-ENH}
    \ul\id^*
    & \sila \id,
    \\ 
    \label{eq:alphabeta^*-ENH}
    \ul{(\alpha\beta)}^*
    & \sila \ul\beta^*\ul\alpha^*
  \end{align}
  in $\ENH_\kk$
  that enhance the isomorphisms
  $\dR \id_* \cong \id$, 
  $\dR(\alpha \beta)_* \cong \dR \alpha_* \dR \beta_*$, 
  $\dL \id^* \cong \id$,
  $\dL(\alpha \beta)^* \cong \dL \beta^* \dL \alpha^*$.
\end{lemma}

In the formulation of the lemma, we use the usual isomorphisms
$(\alpha\beta)_* \sira \alpha_* \beta_*$ and
$(\alpha\beta)^* \sira \beta^*\alpha^*$, cf.
Remark~\ref{rem:fixed-data-alpha*-adjunction}.

\begin{proof}
  Let $J \in \ul\II(\mathcal{Z})$. Obviously, 
  $J \xra{\iotaii_J} \ii J = \ul\id_* J$ is a homotopy equivalence.
  Proposition~\ref{p:spalt-5.15-sites}
  shows that $\beta_*
  J \xra{\iotaii_{\beta_* J}} \ii\beta_* J$ is a
  quasi-isomorphism between 
  weakly h-injective complexes and that
  \begin{equation}
    \label{eq:32}
    \ul{(\alpha\beta)}_*J
    =
    \ii (\alpha\beta)_*J
    \sira
    \ii \alpha_* \beta_* J
    \xra{\ii\alpha_*\iotaii_{\beta_*J}}
    \ii\alpha_*\ii\beta_* J
    =      
    \ul\alpha_*\ul\beta_*J    
  \end{equation}
  is a homotopy equivalence. 

  Let $I \in \ul\II(\mathcal{X})$.
  Obviously, 
  $I \xra{\iotaii_I} \ii I
  \xla{\ii \epsilonee_I} \ii\ee I = \ul\id^* I$ consists of
  homotopy equivalences.
  Proposition~\ref{p:pullback-preserves-h-flat}
  shows that $F:=\alpha^*\ee I$ is h-flat, so both morphisms
  \begin{equation}
    \label{eq:29}
    F \xla{\epsilonee_F} \ee F \xra{\ee \iotaii_F} \ee\ii F
  \end{equation}
  are quasi-isomorphisms between h-flat complexes, and that 
  $\ii \beta^*$ maps them to homotopy equivalences.
  We obtain homotopy equivalences
  \begin{equation}
    \ul{(\alpha\beta)}^* I 
    =
    \ii(\alpha\beta)^* \ee I 
    \sira \ii\beta^* \alpha^* \ee I
    =
    \ii\beta^* F 
    \xla{\ii\beta^*\epsilonee_F}
    \ii\beta^*\ee F 
    \xra{\ii\beta^*\ee \iotaii_F}
    \ii\beta^*\ee\ii F = 
    \ul\beta^*\ul\alpha^*I. 
  \end{equation}
  We leave the rest of the proof to the reader.
\end{proof}

The following result uses fixed data turning
$(\C(\mathcal{X}), \otimes,
\mathcal{O}_\mathcal{X})$ into a symmetric monoidal category.

\begin{lemma}
  \label{l:ulO-otimes-neutral}
  Let $(\mathcal{X},\mathcal{O})$ be a $\kk$-ringed site.
  Then the zig-zags
  \begin{align*}
    (\ul{\mathcal{O}} 
    &
      \; \ul\otimes -) 
      = \ii(\ee \ii \mathcal{O} \otimes \ee(-))
      \xla{[\sim]} \ii(\ee \mathcal{O} \otimes \ee(-))
      \xra{[\sim]} \ii(\mathcal{O} \otimes \ee (-))
      \sira \ii \ee
      \xra{[\sim]}
      \ii 
      \xla{[\sim]}
      \id,\\
    (- 
    &
      \ul\otimes \; \ul{\mathcal{O}})
      = \ii(\ee(-) \otimes \ee \ii \mathcal{O})
      \xla{[\sim]} \ii(\ee(-) \otimes \ee \mathcal{O})
      \xra{[\sim]} \ii(\ee(-) \otimes \mathcal{O})
      \sira \ii \ee
      \xra{[\sim]}
      \ii 
      \xla{[\sim]}
      \id,\\
    ((- 
    & \ul\otimes -) 
      \ul\otimes -)
      = \ii(\ee\ii(\ee(-) \otimes \ee(-)) \otimes \ee(-))
      \xla{[\sim]} \ii(\ee(\ee(-) \otimes \ee(-)) \otimes \ee(-))\\
    \notag
    & \xra{[\sim]} \ii((\ee(-) \otimes
      \ee(-)) \otimes \ee(-))
      \sira \ii(\ee(-) \otimes (\ee(-) \otimes \ee(-)))
      \xla{[\sim]} \dots
      \xra{[\sim]}  
      (- \ul\otimes (- \ul\otimes -)),\\
    (- 
    & \ul\otimes ?) 
      = \ii(\ee(-) \otimes \ee(?))
      \sira  \ii(\ee(?) \otimes \ee(-)) = (? \ul\otimes -)
  \end{align*}
  of objectwise homotopy equivalences in $\tildew{\ENH}_\kk$ define
  2-isomorphisms
  \begin{align}
    \label{eq:left}
    (\ul{\mathcal{O}} \;\ul\otimes -)
    & \sira \id,
    \\
    \label{eq:right}
    (- \ul\otimes \; \ul{\mathcal{O}})
    & \sira \id,\\
    \label{eq:ass-and-symm}
    ((- \ul\otimes -) \ul\otimes -)
    & \sira
    (- \ul\otimes (- \ul\otimes -)),
    \\
    \label{eq:ulotimes-swap}
    (- \ul\otimes ?) 
    & \sira
    (? \ul\otimes -),
  \end{align}
  that enhance the corresponding isomorphisms
  $(\mathcal{O} \otimes^\dL -) \sira \id$, 
  $(- \otimes^\dL \mathcal{O}) \sira \id$, 
  $((- \otimes^\dL -) \otimes^\dL -) \sira 
  (- \otimes^\dL (- \otimes^\dL -))$,
  $(- \otimes^\dL ?) \sira (? \otimes^\dL -)$.
\end{lemma}

\begin{proof}
  Obvious and left to the reader.
\end{proof}

\begin{lemma}
  \label{l:ulGamma-ulsheafHom}
  Let $(\mathcal{X},\mathcal{O})$ be a $\kk$-ringed site.
  Then the 2-morphism
  \begin{equation}
    \ul\Hom(-,-) = \Gamma\sheafHom \ra
    \Gamma \ii \sheafHom(-,-) 
    = \ul\Gamma \; \ul\sheafHom(-,-)  
  \end{equation}
  in $\tildew{\ENH}_\kk$ is an
  objectwise homotopy equivalence and the induced
  isomorphism
  \begin{equation}
    \label{eq:46}
    \ul\Hom(-,-) \sira
    \ul\Gamma \; \ul\sheafHom(-,-) 
  \end{equation}
  in $\ENH_\kk$ enhances 
  $\dR \Hom(-,-) \sira \dR \Gamma \dR\sheafHom(-,-)$. 
\end{lemma}

\begin{proof}
  For $I, J \in \II(\mathcal{X})$ the morphism
  $\sheafHom(I,J) \ra \ii\sheafHom(I,J)$ is a quasi-isomorphism
  between weakly h-injective complexes, by
  Proposition~\ref{p:spalt-5.14-sites-sheafHom}, so that we can
  apply Proposition~\ref{p:spalt-5.15-sites}.
  The rest of the proof is left to the
  reader.
\end{proof}

\begin{lemma}
  \label{l:pull-tensor}
  Let
  $\alpha \colon (\Sh(\mathcal{Y}), \mathcal{O}_\mathcal{Y}) \ra
  (\Sh(\mathcal{X}), \mathcal{O}_\mathcal{X})$
  be a morphism of $\kk$-ringed topoi. 
  Then the zig-zag
  \begin{multline*}
    \ul\alpha^* (- \ul\otimes -)
    =
    \ii\alpha^*\ee\ii(\ee(-) \otimes \ee(-))
    \xla{[\sim]}
    \ii\alpha^*\ee(\ee(-) \otimes \ee(-))
    \xra{[\sim]}
    \ii\alpha^*(\ee(-) \otimes \ee(-))\\
    \xra{\sim}
    \ii(\alpha^*\ee(-) \otimes \alpha^*\ee(-))
    \xla{[\sim]}
    \ii(\ee \alpha^*\ee(-) \otimes \ee \alpha^*\ee(-))
    \\
    \xra{[\sim]}
    \ii(\ee\ii\alpha^*\ee(-) \otimes \ee\ii\alpha^*\ee(-))
    = 
    (\ul\alpha^* -) \ul\otimes 
    (\ul\alpha^* -))
  \end{multline*}
  of objectwise homotopy equivalences in $\tildew{\ENH}_\kk$ defines
  a 2-isomorphism
  \begin{equation}
    \label{eq:ul-alpha^*-otimes}
    \ul\alpha^* (- \ul\otimes -)
    \sira
    (\ul\alpha^* -) \ul\otimes 
    (\ul\alpha^* -) 
  \end{equation}
  in $\ENH_\kk$ that enhances
  $\dL \alpha^*(- \otimes^\dL -) \sira \dL \alpha^*(-)
  \otimes^\dL \dL \alpha^*(-)$ (cf.\
  \cite[Prop.~3.2.4.(i)]{lipman-notes-on-derived-functors-and-GD}
  for 
  ringed spaces).
\end{lemma}

\begin{proof}
  The proof uses Proposition~\ref{p:pullback-preserves-h-flat}
  and is left to the reader.
\end{proof}

\fussnote{
  can maybe also generalize notion of pseudofunctor of
  monoidal cats (oder wie das bei Lipman heisst) zu meiner
  $\kk$-linearen 2-Multikategorie. Er verlangt aber glaube ich
  strikte Identit\"at.

  Kompatibiltiaet mit assoziator, recht- und links einheiten,
  symmetrie-swap. 
}

\begin{proposition}
  \label{p:enhanced-pull-push-sheafHom}
  Let
  $\alpha \colon (\Sh(\mathcal{Y}), \mathcal{O}_\mathcal{Y}) \ra
  (\Sh(\mathcal{X}), \mathcal{O}_\mathcal{X})$
  be a morphism of $\kk$-ringed topoi. Then the zig-zag
  \begin{multline}
    \label{eq:enhanced-pull-push-sheafHom-zig-zag}
    \ul\alpha_* \ul\sheafHom(\ul\alpha^* -, -)
    = 
    \ii \alpha_* \ii\sheafHom(\ii \alpha^*\ee(-), -)
    \xla{[\sim]}
    \ii \alpha_* \sheafHom(\ii \alpha^*\ee(-), -)
    \\
    \xra{[\sim]}
    \ii \alpha_* \sheafHom(\alpha^*\ee(-), -)
    \xra{\sim}
    \ii \sheafHom(\ee(-), \alpha_*(-))
    \xra{[\sim]}
    \ii \sheafHom(\ee(-), \ii\alpha_*(-))
    \\
    \xla{[\sim]}
    \ii \sheafHom(-, \ii\alpha_*(-))
    =
    \ul\sheafHom(-, \ul\alpha_*-)
  \end{multline}
  of objectwise homotopy equivalences
  in $\tildew{\ENH}_\kk$ defines a 2-isomorphism
  \begin{equation}
    \label{eq:pull-push-sheafHom-ENH}
    \ul\alpha_* \ul\sheafHom(\ul\alpha^* -, -)
    \sira
    \ul\sheafHom(-, \ul\alpha_*-)
  \end{equation}
  in $\ENH_\kk$ that enhances 
  $\dR \alpha_* \dR \sheafHom(\dL
  \alpha^* -,-) \sira \dR \sheafHom(-, \dR \alpha_* -)$.
\end{proposition}

\fussnote{
  relation to 
  Proposition~\eqref {p:adjunction-*-pull-push-in-ENH}?
  same with next lemma, Hom statt sheafHom.
}

\begin{proof}
  Let $I \in \ul\II(\mathcal{X})$ and $J \in \ul\II(\mathcal{Y})$
  and evaluate the zig-zag
  \eqref{eq:enhanced-pull-push-sheafHom-zig-zag} at $(I,J)$. 
  Then the first arrow is a homotopy equivalence because
  $L:=\sheafHom(\ii\alpha^*\ee I, J)$ is weakly h-injective by
  Proposition~\ref{p:spalt-5.14-sites-sheafHom}, so that we can
  apply
  Proposition~\ref{p:spalt-5.15-sites}
  to $L \xra{\iotaii} \ii L$. The second arrow is a
  homotopy equivalence because
  $\sheafHom(\ii\alpha^*\ee I, J) \ra \sheafHom(\alpha^*\ee I,
  J)$
  is a quasi-isomorphism between weakly h-injective complexes, by
  Propositions~\ref{p:spalt-5.20-sites-sheafHom}.\ref{enum:sheafHom-to-h-inj} 
  and
  Proposition~\ref{p:spalt-5.14-sites-sheafHom},
  so that we can
  apply
  Proposition~\ref{p:spalt-5.15-sites}.
  The third arrow 
  comes from the usual
  isomorphism.
  The fourth arrow
  is a homotopy equivalence by
  Proposition~\ref{p:spalt-5.15-sites}
  and
  Corollary~\ref{c:spalt-5.20-sites-sheafHom}.
  Proposition~\ref{p:spalt-5.20-sites-sheafHom}.\ref{enum:sheafHom-to-h-inj} 
  shows that the last arrow is a homotopy equivalence.
  The rest of the proof is left to the reader.
\end{proof}

\fussnote{
  Fussnote bei ``The third arrow 
  comes from the usual''
  war:
  
  reference? Yoneda und vorher bewiesene
  pull-push Adjunktion
  \citestacks{03D7},
  tensor-sheafHom Adjunktion \citestacks{03EO}
  und \citestacks{03EL}. 
  
  existiert das als individueller Eintrag im Stack Project?
}

\begin{proposition}
  \label{p:enhanced-tensor-sheafhom}
  Let $(\mathcal{X}, \mathcal{O}_\mathcal{X})$ be a $\kk$-ringed
  site. Then the zig-zag
  \begin{multline}
    \label{eq:enhanced-tensor-sheafhom-zigzag}
    \ul\sheafHom(- \ul\otimes -, -)
    =
    \ii\sheafHom(\ii(\ee(-) \otimes \ee(-)), -)
    \\
    \xra{[\sim]}
    \ii\sheafHom(\ee(-) \otimes \ee(-), -)
    \xla{[\sim]}
    \ii\sheafHom(\ee(-) \otimes -, -)
    \\
    \sira
    \ii\sheafHom(\ee(-), \sheafHom(-, -))
    \xra{[\sim]}
    \ii\sheafHom(\ee(-), \ii\sheafHom(-, -))
    \\
    \xla{[\sim]}
    \ii\sheafHom(-, \ii\sheafHom(-,-))
    =
    \ul\sheafHom(-, \ul\sheafHom(-,-))
  \end{multline}
  of objectwise homotopy equivalences in $\tildew{\ENH}_\kk$ defines a
  2-isomorphism
  \begin{equation}
    \label{eq:48}
    \ul\sheafHom(- \ul\otimes -, -)
    \sira
    \ul\sheafHom(-, \ul\sheafHom(-,-))
  \end{equation}
  in $\ENH_\kk$ that enhances
  $\dR\sheafHom(- \otimes^\dL
  -,-) 
  \sira \dR\sheafHom(-, \dR\sheafHom(-,-))$.
\end{proposition}

\begin{proof}
  Let $I,J,L \in \ul\II(\mathcal{X})$ and evaluate the zig-zag
  \eqref{eq:enhanced-tensor-sheafhom-zigzag} at $(I,J,L)$.  The
  first two arrows are homotopy equivalences by
  Proposition~\ref{p:spalt-5.20-sites-sheafHom}.\ref{enum:sheafHom-to-h-inj}
  and h-flatness of $\ee I$, the third arrow is the usual
  isomorphism, the fourth arrow is a homotopy equivalence by
  Proposition~\ref{p:spalt-5.14-sites-sheafHom} and
  Corollary~\ref{c:spalt-5.20-sites-sheafHom}, and the fifth
  arrow is a homotopy equivalence by
  Proposition~\ref{p:spalt-5.20-sites-sheafHom}.\ref{enum:sheafHom-to-h-inj}.
  We leave the rest of the proof to the reader.
\end{proof}

\fussnote{
  could provide morphisms $\ul\sheafHom(E, -) \ul\otimes E \ra
  \id$ and $\id \ra \ul\sheafHom(E, (-) \ul\otimes E)$ and show
  that they form an adjunction. Also could clarify relation to
  above 2-isomorphisms
  \eqref{eq:28}
  or \eqref{eq:28}.
}

\subsubsection{Lifts of commutative diagrams}
\label{sec:lifts-comm-diagr}

We indicate some relations between the 2-morphisms we have
constructed so far. 

\begin{proposition}
  \label{p:adjunction-*-pull-push-in-ENH}
  Let
  $\alpha \colon (\Sh(\mathcal{Y}), \mathcal{O}_\mathcal{Y}) \ra
  (\Sh(\mathcal{X}), \mathcal{O}_\mathcal{X})$
  be a morphism of $\kk$-ringed topoi.
  Then the 
  two 1-morphisms $\ul\alpha^*$ and $\ul\alpha_*$ and the
  two 2-morphisms 
  \eqref{eq:id-aa}
  and 
  \eqref{eq:aa-id}
  form an adjunction in $\ENH_\kk$,
  i.\,e.\ the two diagrams
  in $\ENH_\kk$
  in \eqref{eq:intro:alpha*-triangles}
  commute.
\end{proposition}

\fussnote{
  Sicherlich kann ich analoge Resultate f\"ur andere Adjunktionen
  zeigen...!? 
}

\fussnote{
  Hat diese Adjunktion andere wichtige Konsequenzen? Lohnt es,
  andere Aussagen dieser Art zu beweisen?
}

\begin{proof}
  Let 
  $\ul\eta \colon \id \ra \ul\alpha_* \ul\alpha^*$
  and
  $\ul\theta \colon \ul\alpha^* \ul\alpha_* \ra \id$ 
  denote 
  \eqref{eq:id-aa}
  and
  \eqref{eq:aa-id}, respectively.
  We claim that the composition
  \begin{equation}
    \label{eq:eta-alpha_*-alpha_*-theta}
    \ul\alpha_*
    \xra{\ul\eta\,\ul\alpha_*} 
    \ul\alpha_* \ul\alpha^* 
    \ul\alpha_* 
    \xra{\ul\alpha_* \ul\theta}
    \ul\alpha_*
  \end{equation}
  of 2-morphisms in $\ENH_\kk$ is the identity 2-morphism.

  Let $J \in \ul\II(\mathcal{Y})$ and consider the commutative
  diagram 
  \begin{equation}
    \label{eq:eta-alpha_*-alpha_*-theta-big-diagram}
    \xymatrix@C+3em@R+1ex{
      {\ii \ee \ii\alpha_* J} 
      \ar@[red][dd]|(.3){\text{red}}^-{\ii \epsilonee_{\ii \alpha_* J}}_-{[\sim]}
      \ar@[red][r]|(.3){\text{red}}^-{\ii \eta_{\ee\ii\alpha_* J}}
      &
      {\ii \alpha_* \alpha^*\ee \ii\alpha_* J} 
      \ar@[red][r]|(.35){\text{red}}^-{\ii \alpha_* \iotaii_{\alpha^*\ee\ii\alpha_* J}}
      &
      {\ii \alpha_* \ii \alpha^*\ee \ii\alpha_* J} 
      \\
      &
      {\ii \ee \alpha_* J} 
      \ar[ul]_-{\ii \ee \iotaii_{\alpha_* J}}^-{[\sim]}
      \ar[r]^-{\ii \eta_{\ee\alpha_* J}}
      \ar[dd]^-{\ii \epsilonee_{\alpha_* J}}_-{[\sim]}
      &
      {\ii \alpha_* \alpha^* \ee \alpha_*J} 
      \ar[ul]_-{\ii \alpha_* \alpha^* \ee \iotaii_{\alpha_* J}}
      \ar[r]^-{\ii \alpha_* \iotaii_{\alpha^*\ee\alpha_* J}}
      \ar[dd]^-{\ii \alpha_* \alpha^* \epsilonee_{\alpha_* J}}
      &
      {\ii \alpha_* \ii \alpha^* \ee \alpha_*J} 
      \ar@[red][ul]|(.4){\text{red}}_-{\ii \alpha_* \ii \alpha^* \ee
        \iotaii_{\alpha_* J}}^(0.7){[\sim]}
      \ar@[red][dd]|(.3){\text{red}}^-{\ii \alpha_* \ii \alpha^*
        \epsilonee_{\alpha_* J}}
      \\
      {\ii\ii \alpha_* J}
      \\
      &
      {\ii \alpha_* J} 
      \ar@[green][lu]|(.3){\text{green}}_-{\ii \iotaii_{\alpha_* J}}^-{[\sim]}
      \ar[r]^-{\ii \eta_{\alpha_* J}}
      \ar[rdd]^-{\id}
      &
      {\ii \alpha_* \alpha^* \alpha_*J} 
      \ar[r]^-{\ii \alpha_* \iotaii_{\alpha^*\alpha_* J}}
      \ar[dd]^-{\ii \alpha_* \theta_J}
      &
      {\ii \alpha_* \ii \alpha^* \alpha_*J} 
      \ar@[red][dd]|(.3){\text{red}}^-{\ii \alpha_* \ii \theta_J}
      \\
      {\ii \alpha_* J}
      \ar@[blue][uu]|(.3){\text{blue}}_-{\iotaii_{\ii \alpha_* J}}^-{[\sim]}
      \\
      &
      &
      {\ii \alpha_* J} 
      \ar@[red][r]|(.3){\text{red}}^-{\ii \alpha_* \iotaii_J}_-{[\sim]}
      &
      {\ii \alpha_* \ii J} 
    }
  \end{equation}
  in $\II(\mathcal{X})$; the arrows labeled $[\sim]$ become
  invertible in $[\ul\II(\mathcal{X})]$. We view $J$ as a
  variable in the rest of this proof. Then our diagram is a
  commutative diagram in
  $\tildew{\ENH}_\kk(\ul\II(\mathcal{Y}); \ul\II(\mathcal{X}))$
  and all arrows labeled $[\sim]$ are objectwise homotopy
  equivalences.  If we pass to $\ENH_\kk$, these labeled arrows
  become invertible, and we define $\beta$ (resp.\ $\gamma$) to
  be the composition of the path starting with the blue (resp.\
  green) arrow and then following the red arrows and the inverses
  of the red arrows. Then $\beta$ is the 2-morphism
  \eqref{eq:eta-alpha_*-alpha_*-theta} in $\ENH_\kk$, cf.\
  \eqref{eq:aa-id-construct}, \eqref{eq:id-aa-construct}.
  Commutativity of the diagram obviously shows that $\gamma$ is
  the identity 2-morphism in $\ENH_\kk$.
  Proposition~\ref{p:compare-K-enri-I-fibr-repl-functors}, more
  precisely \eqref{eq:phi-ii-ii-id-corollar}, shows that
  $\delta(\ii\iotaii\alpha_*)=\delta(\iotaii \ii \alpha_*)$ in
  $\ENH_\kk$, i.\,e.\ blue and green arrow are equal in
  $\ENH_\kk$, and our claim follows.
  
  Similarly, one shows that the composition $\ul\alpha^*
  \xra{\ul\alpha^* \ul\eta} \ul\alpha^* \ul\alpha_* \ul\alpha^*
  \xra{\ul\theta\, \ul\alpha^*} \ul\alpha^*$ is the identity.
  This shows that the two diagrams
  in \eqref{eq:intro:alpha*-triangles}
  commute.
\end{proof}

In the following result we use results and terminology from
\cite[IV.7]{maclane-working-mathematician},
\cite[3.3]{lipman-notes-on-derived-functors-and-GD}, generalized
to 2-categories in the obvious way (see also
Remark~\ref{rem:morphism-between-left-adjoints}). 

\begin{proposition}
  \label{p:pullback-comp-adjoint-pushforward-comp}
  Let $(\Sh(\mathcal{Z}), \mathcal{O}_\mathcal{Z}) \xra{\beta}
  (\Sh(\mathcal{Y}), \mathcal{O}_\mathcal{Y}) \xra{\alpha}
  (\Sh(\mathcal{X}), \mathcal{O}_\mathcal{X})$ be morphisms of
  $\kk$-ringed 
  topoi.
  Then the 2-isomorphisms
  $\ul\id_* \xsira{\eqref{eq:id_*-ENH}} \id$ and $\id
  \xsira{\eqref{eq:id^*-ENH}} \ul\id^*$ 
  are conjugate
  and so are the 
  2-isomorphisms
  $\ul{(\alpha \beta)}_* 
  \xsira{\eqref{eq:alphabeta_*-ENH}} \ul\alpha_* \ul\beta_*$
  and
  $\ul\beta^*\ul\alpha^*
  \xsira{\eqref{eq:alphabeta^*-ENH}}
  \ul{(\alpha\beta)}^*$, i.\,e.\ the diagrams
  \eqref{eq:intro:pullback-id-adjoint-pushforward-id} and
  \eqref{eq:intro:pullback-comp-adjoint-pushforward-comp} commute.
  Here we use the two adjunctions
  $(\ul\id^*,\ul\id_*)$ and $(\id,\id)$, and the two adjunctions
  $(\ul{(\alpha\beta)}^*, \ul{(\alpha\beta)}_*)$ and
  $(\ul\beta^*\ul\alpha^*, \ul\alpha_*\ul\beta_*)$
  obtained from
  Proposition~\ref{p:adjunction-*-pull-push-in-ENH}.
\end{proposition}

\begin{proof}
  We only show the second claim and leave the first to the reader.
  We need to show that the diagram 
  \begin{equation}
    \label{eq:pullback-comp-adjoint-pushforward-comp}
    \xymatrix{
      {\id} 
      \ar@[green][r]|(.3){\text{green}}^-{\eqref{eq:id-aa}}
      \ar@[blue][d]|(.3){\text{blue}}_-{\eqref{eq:id-aa}}
      &
      {\ul\alpha_*\ul\alpha^*}
      \ar@[green][rr]|(.3){\text{green}}^-{\ul\alpha_*\eqref{eq:id-aa}\ul\alpha^*}
      &&
      {\ul\alpha_*\ul\beta_*\ul\beta^*\ul\alpha^*}
      \ar@[green][d]|(.3){\text{green}}^-{\ul\alpha_*\ul\beta_*\eqref{eq:alphabeta^*-ENH}}_-{\sim}  
      \\
      {\ul{(\alpha\beta)}_*\ul{(\alpha\beta)}^*}
      \ar@[blue][rrr]|(.3){\text{blue}}^-{\eqref{eq:alphabeta_*-ENH}\ul{(\alpha\beta)}^*}_-{\sim}  
      &&&
      {\ul\alpha_*\ul\beta_*\ul{(\alpha\beta)}^*}
    }
  \end{equation}
  of 2-morphisms in $\ENH_\kk$ commutes, cf.\
  \cite[3.6.2)]{lipman-notes-on-derived-functors-and-GD}. 
  Replacing the arrows in this diagram by their defining zig-zags
  of 2-morphisms in $\tildew{\ENH}_\kk$ yields the outer arrows in
  diagram 
  \eqref{eq:pullback-comp-adjoint-pushforward-comp-big}
  on page 
  \pageref{eq:pullback-comp-adjoint-pushforward-comp-big} where
  $I \in \ul\II(\mathcal{X})$; this diagram lives in 
  $\II(\mathcal{X})$ and is commutative if we remove the bent green
  arrow; here $\eta$, $\eta'$ and
  $\eta''$ denote units of the tacitly fixed adjunctions
  $(\alpha^*, \alpha_*)$, $(\beta^*, \beta_*)$ and
  $((\alpha\beta)^*, (\alpha\beta)_*)$, 
  and of course we have tacitly assumed that the
  isomorphisms  
  $\sigma \colon (\alpha\beta)_* \sira \alpha_* \beta_*$ and
  $\tau \colon \beta^*\alpha^* \ra (\alpha\beta)^*$ 
  used to construct
  \eqref{eq:alphabeta_*-ENH}
  and \eqref{eq:alphabeta^*-ENH}
  are
  conjugate 
  (this implies that the ``pentagon'' in diagram
  \eqref{eq:pullback-comp-adjoint-pushforward-comp-big}
  commutes); 
  the arrows labeled $[\sim]$ become invertible in
  $[\ul\II(\mathcal{X})]$.  In the rest of this proof we view $I$
  as a variable. Then diagram
  \eqref{eq:pullback-comp-adjoint-pushforward-comp-big} is a
  commutative diagram in
  $\tildew{\ENH}_\kk(\ul\II(\mathcal{X}), \ul\II(\mathcal{X}))$
  if we forget the bent green arrow, and
  the arrows labeled $[\sim]$ are objectwise homotopy
  equivalences. Its red and blue (resp.\ red and green) arrows
  give rise to the blue (resp.\ green)
  arrows in 
  \eqref{eq:pullback-comp-adjoint-pushforward-comp}.
  Proposition~\ref{p:compare-K-enri-I-fibr-repl-functors},
  more precisely
  \eqref{eq:phi-ii-ii-id-corollar}, shows that
  $\delta(\ii \iotaii\alpha^*\ee)=\delta(\iotaii\ii\alpha^*\ee)$
  in 
  $\ENH_\kk$.
  Hence 
  $\delta(\ii\alpha_* \ii
  \iotaii\alpha^*\ee)=\delta(\ii\alpha_*\iotaii\ii\alpha^*\ee)$,
  i.\,e.\ 
  the bent green arrow and the cyan arrow become equal in
  $\ENH_\kk$. 
  This implies our claim.
\end{proof}

\begin{landscape}
  \begin{equation}
    \label{eq:pullback-comp-adjoint-pushforward-comp-big}
    \hspace{-3.1cm}
    \xymatrix@R=3pc{
      & {I} 
      \ar@[red][d]|(.3){\text{red}}_-{\iotaii_I}^-{[\sim]}
      \\
      &
      {\ii I}
      &&&
      {\ii\alpha_*\ii \ii\alpha^*\ee I}
      \\
      &
      {\ii\ee I} 
      \ar@[red][u]|(.3){\text{red}}^-{\ii \epsilonee_I}_-{[\sim]}
      \ar@[green][r]|(.3){\text{green}}^-{\ii\eta_{\ee I}}
      \ar@[blue][d]|(.3){\text{blue}}_-{\ii\eta''_{\ee I}}
      &
      {\ii\alpha_*\alpha^*\ee I} 
      \ar@[green][r]|(.3){\text{green}}^-{\ii \alpha_* \iotaii_{\alpha^*\ee I}}
      \ar[d]_-{\ii \alpha_* \eta'_{\alpha^* \ee I}}
      &
      {\ii\alpha_*\ii\alpha^*\ee I}
      \ar@[green]@/^4pc/[ru]|(.15){\text{green}}^-{\ii\alpha_*\iotaii_{\ii\alpha^*\ee
          I}}_-{[\sim]} 
      \ar@[cyan][ru]|(.3){\text{cyan}}^-{\ii\alpha_*\ii\iotaii_{\alpha^*\ee
          I}}_-{[\sim]} 
      \ar[d]_-{\ii\alpha_*\ii \eta'_{\alpha^*\ee I}}
      &
      &
      {\ii\alpha_*\ii\ee \ii\alpha^*\ee I} 
      \ar@[green][lu]|(.3){\text{green}}_-{\ii\alpha_*\ii \epsilonee_{\ii\alpha^*\ee I}}
      ^-{[\sim]}
      \ar@[green][d]|(.3){\text{green}}|-{\ii\alpha_*\ii \eta'_{\ee\ii\alpha^*\ee I}}
      \\
      &
      {\ii(\alpha\beta)_*(\alpha\beta)^*\ee I} 
      \ar@[blue][dl]|(.7){\text{blue}}_-{\ii(\alpha\beta)_*\iotaii_{(\alpha\beta)^*\ee I}}
      \ar[d]_-{\ii \sigma_{(\alpha\beta)^*\ee I}}^-{\sim}
      &
      {\ii\alpha_*\beta_*\beta^*\alpha^*\ee I} 
      \ar[r]^-{\ii \alpha_* \iotaii_{\beta_*\beta^*\alpha^*\ee
          I}}
      \ar[dl]_-{\ii\alpha_*\beta_*\tau_{\ee I}}^-{\sim}
      &
      {\ii\alpha_*\ii\beta_*\beta^*\alpha^*\ee I}
      \ar[dl]_-{\ii\alpha_*\ii\beta_*\tau_{\ee I}}^-{\sim}
      \ar[d]|-{\ii\alpha_*\ii \beta_* \iotaii_{\beta^*\alpha^*\ee I}}
      &
      {\ii\alpha_*\ii\ee\alpha^*\ee I}
      \ar[lu]_-{\ii\alpha_*\ii \epsilonee_{\alpha^*\ee
          I}}^-{[\sim]}
      \ar[ru]^-{\ii\alpha_*\ii\ee\iotaii_{\alpha^*\ee I}}_-{[\sim]}
      \ar[d]|(.6){\ii\alpha_*\ii \eta'_{\ee\alpha^*\ee I}}
      &
      {\ii\alpha_*\ii\beta_*\beta^*\ee \ii\alpha^*\ee I} 
      \ar@[green][d]|(.3){\text{green}}|-{\ii\alpha_*\ii \beta_* \iotaii_{\beta^*\ee
          \ii\alpha^*\ee I}} 
      \\
      {\ii(\alpha\beta)_*\ii(\alpha\beta)^*\ee I} 
      \ar@[blue][d]|(.3){\text{blue}}_-{\ii \sigma_{\ii(\alpha\beta)^*\ee I}}^-{\sim}
      &
      {\ii\alpha_*\beta_*(\alpha\beta)^*\ee I} 
      \ar[dl]^-{\ii\alpha_*\beta_*\iotaii_{(\alpha\beta)^*\ee I}}
      \ar[r]^-{\ii \alpha_* \iotaii_{\beta_*(\alpha\beta)^*\ee I}}
      &
      {\ii\alpha_*\ii\beta_*(\alpha\beta)^*\ee I} 
      \ar[d]_-{\ii\alpha_*\ii \beta_* \iotaii_{(\alpha\beta)^*\ee I}}
      &
      {\ii\alpha_*\ii\beta_*\ii\beta^*\alpha^*\ee I}
      \ar@[green][dl]|(.3){\text{green}}_-{\ii\alpha_*\ii\beta_*\ii\tau_{\ee I}}^-{\sim}
      &
      {\ii\alpha_*\ii\beta_*\beta^*\ee\alpha^*\ee I}
      \ar[ru]^-{\ii\alpha_*\ii \beta_*\beta^*\ee\iotaii_{\alpha^*\ee
          I}}
      \ar[lu]_-{\ii\alpha_*\ii \beta_*\beta^*
        \epsilonee_{\alpha^*\ee I}} 
      \ar[d]|-{\ii\alpha_*\ii \beta_* \iotaii_{\beta^*\ee\alpha^*\ee I}}
      &
      {\ii\alpha_*\ii\beta_*\ii\beta^*\ee \ii\alpha^*\ee I} 
      \\
      {\ii\alpha_*\beta_*\ii(\alpha\beta)^*\ee I} 
      \ar@[blue][rr]|(.3){\text{blue}}^-{\ii \alpha_*
        \iotaii_{\beta_*\ii(\alpha\beta)^*\ee I}}_-{[\sim]}
      &&
      {\ii\alpha_*\ii\beta_*\ii(\alpha\beta)^*\ee I} 
      &&
      {\ii\alpha_*\ii\beta_*\ii\beta^*\ee\alpha^*\ee I}
      \ar@[green][ru]|(.7){\text{green}}_-{\ii\alpha_*\ii
        \beta_*\ii\beta^*\ee\iotaii_{\alpha^*\ee I}}^-{[\sim]}
      \ar@[green][lu]|(.7){\text{green}}^-{\ii\alpha_*\ii \beta_*\ii\beta^*
        \epsilonee_{\alpha^*\ee I}}_-{[\sim]} 
    }
  \end{equation}
\end{landscape}

Recall that a (symmetric) monoidal category is a
(symmetric) monoidal object in the 2-category of categories.
We generalize this definitions in the obvious way to
$\kk$-linear 2-multicategories.

\begin{lemma}
  \label{l:monoidal-enhancement}
  Let $(\mathcal{X}, \mathcal{O})$ be a $\kk$-ringed site,
  define 
  $l:=\eqref{eq:left}$,
  $r:=\eqref{eq:right}$,
  $a:=\eqref{eq:ass-and-symm}$,
  $s:=\eqref{eq:ulotimes-swap}$,
  and 
  consider $\ul{\mathcal{O}}$ as a  
  1-morphism
  $\ul{\mathcal{O}} \colon \emptyset  \ra \ul\II(\mathcal{X})$,
  cf.\ Remark~\ref{rem:ENH-emptysource}.
  Then 
  $(\ul\II(\mathcal{X}), \ul\otimes, \ul{\mathcal{O}}, a, l, r,
  s)$
  is a symmetric monoidal object in 
  $\ENH_\kk$,
  i.\,e.\ the diagrams
  \eqref{eq:intro:ulotimes-ass},
  \eqref{eq:intro:ulotimes-unital},
  \eqref{eq:intro:ulotimes-symm}
  in $\ENH_\kk$ commute.
\end{lemma}

\begin{proof}
  Obvious and left to the reader.
\end{proof}

\begin{remark}
  \label{rem:monoidal-enhancement}
  One may call
  $(\ul\II(\mathcal{X}), \ul\otimes, \ul{\mathcal{O}}, a, l, r,
  s)$
  a ``symmetric monoidal enhancement''; similar terminology in
  the context of derivators appears in
  \cite[2.1]{groth-monoidal-derivators}). If we apply the
  functor \eqref{eq:htpy-ENH-TRCAT}, we
  obtain the symmetric monoidal triangulated $\kk$-category
  $[\ul\II(\mathcal{X})]$.  Lemma~\ref{l:ulO-otimes-neutral}
  shows that \eqref{eq:ol[ii]} is an equivalence
  of
  symmetric monoidal triangulated $\kk$-categories.
\end{remark}

\fussnote{
  Fussnote bei ''shows that \eqref{eq:ol[ii]} is an equivalence
  of
  symmetric monoidal triangulated $\kk$-categories.'':
  
  nicht
  nachgeschaut, was terminologie hier... (co)lax ...  
}

\fussnote{
  Maybe could define/say that 
  $(\ul\II(\mathcal{X}), \ul\otimes, \ul{\mathcal{O}}, 
  \eqref{eq:left},
  \eqref{eq:right},
  \eqref{eq:ass-and-symm})$ 
  together with 
  \eqref{eq:ol[ii]}
  (Hier will richtige Richtung, cf.\ fuassnoate (if still
  existent) in
  Remark~\ref{rem:ENH-objects-are-enhancements})
  is a monoidal $\KK$-enhancement of 
  the 
  monoidal triangulated $\kk$-category
  $(\D(\mathcal{X}), \otimes^\dL, \mathcal{O})$.

  Definition: Let ... be a monoidal triangulated
  $\kk$-category. A  monoidal enhancement...
}

\fussnote{
  what about closed? Gibt das nun Zugriff auf extranatural
  trafos?

  Schaue \cite{groth-monoidal-derivators} 1.3 an? adjunction of
  two variables. Vielleicht
  sollte wieder Begriff in 2-multicategory definieren?

  (Diese kommen uebrigens implizit in \cite{KS} vor, siehe
  Gleichung (2.6.10).)
  Hilft es, dass Modellstruktur auf $\KK$-Funktoren (wie ich
  sie zur Definition von $\ENH_\kk$ verwende) ebenfalls
  ``$\KK$-model structure'' ist?
}

\subsubsection{Some other lifts}
\label{sec:some-other-lifts}

\fussnote{
  allgemeine Bemerkung: strong Yoneda in enriched category
  setting in particular applies to 2-categories! Ja, aber wohl
  nicht so zu $\ENH_\kk$ wie ich das gern haette. oder?
}

\fussnote{
  Beachte, dass im folgenden Lemma 
  $\oldul{\ul\C(E,F)}$ schreiben muss. Denn zur Zeit ist
  $\ul\Hom$ nur 
  definiert, falls beide Argumente in $\ul\II(\mathcal{X})$. 

  Using Remark~\ref{rem:[ulC]-vs-[ulHom]}
}

Recall that each object $G \in \C(\mathcal{X})$ gives rise to an
object $\ul{G}:=\ii G$ of $\II(\mathcal{X})$ and to a
1-morphism $\ul{G}$ in $\ENH_\kk$, cf.\ \eqref{eq:define-ul-G}.
Similarly, any morphism $g \colon G \ra G'$ in $\C(\mathcal{X})$
gives rise to a morphism
\begin{equation}
  \label{eq:define-ul-g}
  \ul{g}:=\ii g \colon \ul{G} \ra \ul{G}' 
\end{equation}
in $\II(\mathcal{X})$ 
and to a 2-morphism $\ul{g} \colon \ul{G} \ra \ul{G}'$ in
$\ENH_\kk$.  

\begin{lemma}
  \label{l:ul-G-and-functors}
  Let
  $\alpha \colon (\Sh(\mathcal{Y}), \mathcal{O}_\mathcal{Y}) \ra
  (\Sh(\mathcal{X}), \mathcal{O}_\mathcal{X})$
  be a morphism of $\kk$-ringed topoi, let 
  $E$, $F \in \C(\mathcal{X})$, and let $g \colon G \ra G'$ be a
  morphism in $\C(\mathcal{Y})$.
  Then the zig-zags
  \begin{align}
    \ul\alpha^* \ul E 
    & = \ii \alpha^* \ee \ii E \xla{[\sim]} \ii
      \alpha^* \ee E \ra \ii \alpha^* E = \ul{\alpha^* E},\\
    \ul E \;\ul\otimes\; \ul F 
    & = \ii(\ee \ii E \otimes \ee \ii F)
      \xla{[\sim]} \ii(\ee E \otimes \ee F)
      \ra \ii (E \otimes F) = \ul{E \otimes F},\\
    \ul{\alpha_* G} 
    & = \ii \alpha_* G \ra \ii\alpha_* \ii G = \ul\alpha_* \ul G,\\
    \ul{\sheafHom(E,F)} 
    & = \ii\sheafHom(E,F) 
      \ra \ii\sheafHom(E,\ii F) 
      \xla{[\sim]} \ii\sheafHom(\ii E,\ii F)
      = \ul\sheafHom(\ul E, \ul F),\\
    \oldul{\ul\C(E,F)} 
    & 
      = \ul\C(E,F) 
      \ra \ul\C(E,\ii F) 
      \xla{[\sim]} \ul\C(\ii E,\ii F)
      = \ul\Hom(\ii E,\ii F)
      = \ul\Hom(\ul E, \ul F),\\
    \ul G 
    &
      \xra{\ul{g}} \ul G'
  \end{align}
  of obvious 2-morphisms in
  $\tildew{\ENH}_\kk$
  define 2-morphisms
  \begin{align}
    \label{eq:alpha^*-ul-obj}
    \ul\alpha^* \ul E & \ra \ul{\alpha^* E} 
    & \text{(2-isomorphism if $E$ is h-flat)},\\
    \label{eq:otimes-ul-obj}
    \ul E \;\ul\otimes\; \ul F 
    & \ra \ul{E \otimes F}
    & \text{(2-isomorphism if $E$ or $F$ is h-flat)},\\
    \label{eq:alpha_*-ul-obj}
    \ul{\alpha_* G} 
    & \ra \ul\alpha_* \ul G
    & \text{(2-isomorphism if $G$ is weakly h-injective)},\\
    \label{eq:sheafHom-ul-obj}
    \ul{\sheafHom(E,F)} 
    & \ra \ul\sheafHom(\ul E, \ul F)
    & \text{(2-isom.\ if $F$ h-inj., or $E$ h-flat, $F$
      weakly h-inj.)},\\ 
    \label{eq:Hom-ul-obj}
    \oldul{\ul\C(E,F)} 
    & \ra \ul\Hom(\ul E, \ul F)
    & \text{(2-isom.\ if $F$ h-inj., or $E$ h-flat, $F$
      weakly h-inj.)},\\
    \label{eq:gul-obj-D}
    \ul G 
    & \xra{\ul g} \ul G'
    & \text{(2-isom.\ if $g$ is a quasi-isomorphism)}
  \end{align}
  in $\ENH_\kk$ that are 2-isomorphisms in the indicated cases
  and that enhance the 2-morphisms
  $\dL\alpha^*(E) \ra \alpha^* E$,
  $E \otimes^\dL F \ra E \otimes F$,
  $\alpha_* G \ra \dR\alpha_*(G)$,
  $\sheafHom(E,F) \ra \dR\sheafHom(E, F)$,
  $\ul\C(E,F) \ra \dR\Hom(E, F)$,
  $G \xra{g} G'$.
  The first five of these 2-morphisms come from the $\kk$-natural
  transformation which is part of the datum of the corresponding
  derived functor. 
  In particular, \eqref{eq:alpha^*-ul-obj}
  provides a 2-isomorphism
  $\ul\alpha^* \ul{\mathcal{O}}_\mathcal{X}
  \sira
  \oldul{\alpha^* \mathcal{O}_\mathcal{X}}
  =\ul{\mathcal{O}}_\mathcal{Y}$
  in $\ENH_\kk$ that enhances $\dL \alpha^*
  \mathcal{O}_\mathcal{X} \sira \mathcal{O}_\mathcal{Y}$.
\end{lemma}

\begin{proof}
  Left to the reader. 
\end{proof}

\begin{lemma}
  \label{l:ul-object-4-functors}
  Let
  $\alpha \colon (\Sh(\mathcal{Y}), \mathcal{O}_\mathcal{Y}) \ra
  (\Sh(\mathcal{X}), \mathcal{O}_\mathcal{X})$
  be a morphism of $\kk$-ringed topoi, let 
  $e \colon E \ra E'$, $f \colon F \ra F'$ be morphisms in
  $\C(\mathcal{X})$, let $g, g' \colon G \ra G'$ and $h \colon G'
  \ra G''$ be 
  morphisms in $\C(\mathcal{Y})$, and let $r, r' \in \kk$.
  Then the following diagrams in $\ENH_\kk$ are commutative.
  \begin{equation}
    \label{eq:ul-alpha-comm}
    \xymatrix{
      {\ul\alpha^* \ul{E}} 
      \ar[r]^-{\eqref{eq:alpha^*-ul-obj}}
      \ar[d]_-{\ul\alpha^*\ul{e}}
      & 
      {\ul{\alpha^*E}} 
      \ar[d]_-{\ul{\alpha^*e}}
      \\
      {\ul\alpha^*\ul{E}'} 
      \ar[r]^-{\eqref{eq:alpha^*-ul-obj}}
      & 
      {\ul{\alpha^*E'}} 
    }
    \quad\quad\quad
    \xymatrix{
      {\ul{\alpha_*G}} 
      \ar[r]^-{\eqref{eq:alpha_*-ul-obj}}
      \ar[d]_-{\ul{\alpha_*g}}
      & 
      {\ul\alpha_*\ul{G}}
      \ar[d]_-{\ul\alpha_*\ul{g}}
      \\
      {\ul{\alpha_*G'}} 
      \ar[r]^-{\eqref{eq:alpha_*-ul-obj}}
      & 
      {\ul\alpha_*\ul{G}}
    }
  \end{equation}
  \begin{equation}
    \label{eq:otimes-sheafHom-ul-obj-comm}
    \xymatrix{
      {\ul E \;\ul\otimes\; \ul F}
      \ar[r]^-{\eqref{eq:otimes-ul-obj}}
      \ar[d]_-{\ul{e}\;\ul\otimes\;\ul{f}}
      & 
      {\ul{E \otimes F}}
      \ar[d]_-{\ul{e \otimes f}}
      \\
      {\ul E' \;\ul\otimes\; \ul F'}
      \ar[r]^-{\eqref{eq:otimes-ul-obj}}
      & 
      {\ul{E' \otimes F'}}
    }
    \quad\quad\quad
    \xymatrix{
      {\ul{\sheafHom(E',F)}}
      \ar[r]^-{\eqref{eq:sheafHom-ul-obj}}
      \ar[d]_-{\ul{\sheafHom(e,f)}}
      & 
      {\ul\sheafHom(\ul E', \ul F)}
      \ar[d]_-{\ul{\sheafHom}(\ul{e},\ul{f})}
      \\
      {\ul{\sheafHom(E,F')}}
      \ar[r]^-{\eqref{eq:sheafHom-ul-obj}}
      & 
      {\ul\sheafHom(\ul E, \ul F')}
    }
  \end{equation}
  \begin{equation}
    \label{eq:Hom-ul-obj-comm}
    \xymatrix{
      {\oldul{\ul\C(E',F)}}
      \ar[r]^-{\eqref{eq:Hom-ul-obj}}
      \ar[d]_-{\oldul{\ul\C(e,f)}}
      & 
      {\ul\Hom(\ul E', \ul F)}
      \ar[d]_-{\ul{\Hom}(\ul{e},\ul{f})}
      \\
      {\oldul{\ul\C(E,F')}}
      \ar[r]^-{\eqref{eq:Hom-ul-obj}}
      & 
      {\ul\Hom(\ul E, \ul F')}
    }
  \end{equation}
  \begin{equation}
    \label{eq:ulg-ulg'-and-rg+r'g'}
    \xymatrix{
      {\ul{G}} 
      \ar[r]^-{\ul{g}}
      \ar[rd]_-{\ul{hg}}
      &
      {\ul{G}'}
      \ar[d]^-{\ul{h}}
      \\
      &
      {\ul{G}''}
    }
    \quad\quad\quad
    \xymatrix{
      {\ul{G}} 
      \ar@/^/[d]^-{\ul{rg+r'g'}}
      \ar@/_/[d]_-{r\ul{g}+r'\ul{g}'}
      \\
      {\ul{G}'}
    }
  \end{equation}
\end{lemma}

\begin{proof}
  Left to the reader. The last two diagrams are commutative
  because the $\KK$-functor $\ii \colon \ul\C(\mathcal{Y}) \ra
  \ul\II(\mathcal{Y})$ 
  gives rise to a
  $\kk$-functor $\C(\mathcal{Y}) \ra \II(\mathcal{Y}) \ra
  [\ul\II(\mathcal{Y})]$.  
\end{proof}

\subsubsection{Subsequently constructed lifts of
  2-(iso)morphisms}
\label{sec:subs-constr-lifts}

We give some examples as to how to produce more lifts from the
lifts we have constructed so far.

\begin{lemma}
  \label{l:enhanced-pull-push-ulHom}
  Let
  $\alpha \colon (\Sh(\mathcal{Y}), \mathcal{O}_\mathcal{Y}) \ra
  (\Sh(\mathcal{X}), \mathcal{O}_\mathcal{X})$
  be a morphism of $\kk$-ringed topoi.
  Then the composition
  \begin{multline}
    \label{eq:enhanced-pull-push-ulHom}
    \ul\Hom(\ul\alpha^* -, -)
    \xsira{\eqref{eq:46}(\ul\alpha^*,\id)}
    \ul\Gamma\;\ul\sheafHom(\ul\alpha^* -, -)
    \xsira{\eqref{eq:alphabeta_*-ENH}\ul\sheafHom(\ul\alpha^* -, -)}
    \ul\Gamma\;\ul\alpha_*\ul\sheafHom(\ul\alpha^* -, -)\\
    \xsira{\ul\Gamma\eqref{eq:pull-push-sheafHom-ENH}}
    \ul\Gamma\;\ul\sheafHom(-, \ul\alpha_*-)
    \xsira{\eqref{eq:46}\inv(\id, \ul\alpha_*)}
    \ul\Hom(-, \ul\alpha_*-)
  \end{multline}
  of 2-isomorphisms
  in $\ENH_\kk$ enhances 
  $\dR \Hom(\dL
  \alpha^* -,-) \sira \dR \Hom(-, \dR \alpha_* -)$.
\end{lemma}

\begin{proof}
  Use 
  Remark~\ref{rem:enhance-2-morphism-compose-and-invert},
  Lemmas~\ref{l:composition-inverse-direct-image}, 
  \ref{l:ulGamma-ulsheafHom}, and
  Proposition~\ref{p:enhanced-pull-push-sheafHom}.
\end{proof}

\begin{remark}
  \label{rem:enhanced-pull-push-Hom}
  Let us explain how the adjunction isomorphisms 
  $\D_\mathcal{Y}(\dL\alpha^* E, F) \sira
  \D_\mathcal{X}(E, \dR\alpha_* F)$ are encoded in this formalism.  
  Lemma~\ref{l:enhanced-pull-push-ulHom}
  implies that the diagram
  \begin{equation}
    \xymatrix{
      {[\ul\Hom]([\ul\alpha^*]\ol{[\ii]}(-), \ol{[\ii]}(-))}
      \ar[rr]_-{\sim}^-{[\eqref{eq:enhanced-pull-push-ulHom}]\ol{[\ii]}}
      \ar[d]_{\omega_{\ud{\Hom}(\ud{\alpha}^*(-),-)}}^-{\sim}
      &&
      {[\ul\Hom](\ol{[\ii]}(-), [\ul\alpha_*]\ol{[\ii]}(-))}
      \ar[d]_{\omega_{\ud{\Hom}(-,\ud{\alpha}_*(-))}}^-{\sim}
      \\
      {\dR\Hom(\dL\alpha^*(-), -)}
      \ar[rr]^-{\sim}
      &&
      {\dR\Hom(-, \dR\alpha_*(-))}
    }
  \end{equation}
  of 1- and 2-morphisms in $\TRCAT_\kk$ 
  commutes, cf.\ \eqref{eq:sigma-omega=omega-tau-diagram} (we use
  that $\ol{[\ii_{(\pt,\kk)}]}=\id$).
  If we evaluate this diagram at $E \in \D(\mathcal{X})$
  and 
  $F \in \D(\mathcal{Y})$, take $H^0 \colon [\KK] \ra
  \Mod(\kk)$ and
  use Remark~\ref{rem:[ulC]-vs-[ulHom]} we obtain the upper
  square in the
  commutative diagram 
  \begin{equation}
    \xymatrix{
      {[\ul\C_\mathcal{Y}](\ul\alpha^*\ii E, \ii F)}
      \ar[r]^-{\sim}
      \ar[d]^-{\sim}
      &
      {[\ul\C_\mathcal{X}](\ii E, \ul\alpha_* \ii F)}
      \ar[d]^-{\sim}
      \\
      {H^0\dR\Hom(\dL\alpha^*E, F)}
      \ar[r]^-{\sim}
      \ar[d]^-{\sim}
      &
      {H^0\dR\Hom(E, \dR\alpha_*F)}
      \ar[d]^-{\sim}
      \\
      {\D_\mathcal{Y}(\dL\alpha^* E, F)}
      \ar[r]^-\sim
      &
      {\D_\mathcal{X}(E, \dR\alpha_* F)}  
    }
  \end{equation}
  in $\Mod(\kk)$. The vertical arrows in the lower square are the
  usual isomorphisms and the lower horizontal isomorphism is the
  adjunction isomorphism.
\end{remark}

\begin{lemma}
  \label{l:enhanced-tensor-sheafhom-Hom}
  Let $(\mathcal{X}, \mathcal{O}_\mathcal{X})$ be a $\kk$-ringed
  site.
  Then the composition
  \begin{multline}
    \label{eq:28}
    \ul\Hom(- \ul\otimes -, -)
    \xsira{\eqref{eq:46}(- \ul\otimes -, -)}
    \ul\Gamma\;\ul\sheafHom(- \ul\otimes -, -)\\
    \xsira{\ul\Gamma\eqref{eq:48}}
    \ul\Gamma\;\ul\sheafHom(-, \ul\sheafHom(-,-))
    \xsira{\eqref{eq:46}\inv(-, \sheafHom(-, -))}
    \ul\Hom(-, \ul\sheafHom(-,-)) 
  \end{multline}
  of 2-isomorphisms in $\ENH_\kk$ enhances
  $\dR\!\Hom(- \otimes^\dL -,-)
  \sira \dR\!\Hom(-, \dR\sheafHom(-,-))$.
\end{lemma}

The adjunction
isomorphisms $\D_\mathcal{X}(E \otimes^\dL F, G) \cong \D_\mathcal{Y}(E,
\dR\sheafHom(F,G))$
can be obtained as in Remark~\ref{rem:enhanced-pull-push-Hom}.

\begin{proof}
  Use 
  Remark~\ref{rem:enhance-2-morphism-compose-and-invert},
  Lemma~\ref{l:ulGamma-ulsheafHom} and
  Proposition~\ref{p:enhanced-tensor-sheafhom}.
\end{proof}

\begin{lemma}
  \label{l:push-sheafHom}
  Let
  $\alpha \colon (\Sh(\mathcal{Y}), \mathcal{O}_\mathcal{Y}) \ra
  (\Sh(\mathcal{X}), \mathcal{O}_\mathcal{X})$
  be a morphism of $\kk$-ringed topoi. Then the composition
  \begin{equation}
    \label{eq:push-sheafHom-ENH}
    \ul\alpha_* \ul\sheafHom(-, -)
    \xra{\ul\alpha_* \ul\sheafHom(\eqref{eq:aa-id}, -)}
    \ul\alpha_*\ul\sheafHom(\ul\alpha^*\ul\alpha_*(-), -)
    \xsira{\eqref{eq:pull-push-sheafHom-ENH}(\ul\alpha_*(-),-)}
    \ul\sheafHom(\ul\alpha_*(-), \ul\alpha_*(-))
  \end{equation}
  in $\ENH_\kk$ enhances
  $\dR \alpha_* \dR \sheafHom(-,-) \ra \dR \sheafHom(\dR
  \alpha_*(-), \dR \alpha_* (-))$, cf.\ \cite[(2.6.24)]{KS}. 
\end{lemma}

\begin{proof}
  Obvious and left to the reader.
\end{proof}

\begin{remark}
  \label{rem:alternative-constructions}
  Some of the 2-(iso)morphisms in $\ENH_\kk$ admit
  alternative constructions. 
  For example, 
  the compositions 
  in \eqref{eq:enhanced-pull-push-ulHom}, \eqref{eq:28},
  \eqref{eq:push-sheafHom-ENH}
  can be constructed
  directly from zig-zags, 
  cf.\
  Propositions~\ref{p:enhanced-pull-push-sheafHom},
  \ref{p:enhanced-tensor-sheafhom}.  
\end{remark}

\section{Two operations}
\label{sec:two-operations}

We lift the remaining part of
Grothendieck--Verdier--Spaltenstein's six functor 
formalism involving the two functors $\dR \alpha_!$ and $\alpha^!$
to the dg level.

All ringed spaces and algebras in this section are assumed to be
$\mathscr{U}$-small; all algebras are assumed to be commutative. 

\subsection{Proper direct image}
\label{sec:proper-direct-image}
 
We use results on the proper direct image functor and
on separated and locally proper morphisms of topological spaces from
\cite{wolfgang-olaf-locallyproper};
all results there extend in a
straightforward manner 
from sheaves of abelian groups to sheaves of modules over some
fixed ground ring (for the modification of
\cite[Thm.~6.3]{wolfgang-olaf-locallyproper} cf.\ 
\cite[Prop.~2.6.6]{KS}), and we use these extended results
tacitly; instead of
the notation $\alpha_{(!)}$, $\alpha_!$, etc.\ there we use the
traditional notation 
$\alpha_!$, $\dR \alpha_!$ etc.\ here.

Let $\R$ be a ring and $A$ a (commutative) $\R$-algebra.
If $X$ is a topological space,
the constant sheaf $A_X$ of $\R$-algebras 
with stalk $A$ turns $X$ into an $\R$-ringed space $X_A :=(X, A_X)$.
Any morphism $\alpha \colon Y \ra X$ of topological spaces
gives rise to a morphism $\alpha \colon Y_A \ra X_A$ of
$\R$-ringed spaces whose comorphism is adjoint to the identity
morphism $\alpha\inv A_X \xra{\id} A_Y$. 
Then $\alpha\inv=\alpha^* \colon \Mod(X_A) \ra \Mod(Y_A)$
is an exact functor and 
$\alpha\inv=\alpha^* \colon \ul\C(X_A) \ra \ul\C(Y_A)$ and 
$\dL\alpha\inv=\dL\alpha^* \colon \D(X_A) \ra \D(Y_A)$.
To emphasize that we work with sheaves of $A$-modules we
usually prefer the notation $\alpha\inv$ and $\dL\alpha\inv$. 

Let $\alpha_! \colon \Mod(Y_A) \ra \Mod(X_A)$
denote the proper direct image functor. It induces the $\RR$-functor
$\alpha_! \colon \ul\C(Y_A) \ra \ul\C(X_A)$
and the derived functor $\dR\alpha_! \colon \D(Y_A) \ra \D(X_A)$. 

If $\alpha$ is a separated and locally proper morphism of
topological spaces such that 
$\alpha_! \colon \Mod(Y_A) \ra \Mod(X_A)$ has finite cohomological
dimension, then $\dR\alpha_!$ admits a right adjoint functor
$\alpha^! \colon \D(X_A) \ra \D(Y_A)$, by
\cite[Thm.~8.3.(a)]{wolfgang-olaf-locallyproper}. We fix an
adjunction $(\dR\alpha_!, \alpha^!)$.

\subsection{The multicategory of formulas for the six operations}
\label{sec:categ-form-six}

Recall the category $\fml'_\R$ of formulas from
\ref{sec:categ-form-four}.  For each $\R$-algebra $A$ and each
topological space $X$ we have the object $\ud{X}_A:=\ud{X_A}$ of
this category.

Let $\fml_\R$ be the free multicategory 
having the same objects as
$\fml'_\R$ and whose generating morphisms consist of the
generating morphisms of $\fml'_\R$ and the following morphisms:
\begin{enumerate}
\item for each $\R$-algebra $A$ and each morphism
  $\alpha \colon Y \ra X$ of topological spaces 
  there are morphisms
  $\ud\alpha\inv \colon \ud{X}_A \ra \ud{Y}_A$ and 
  $\ud\alpha_! \colon \ud{Y}_A \ra \ud{X}_A$ 
  and their opposites
  $(\ud\alpha\inv)^\opp \colon \ud{X}_A^\opp \ra \ud{Y}_A^\opp$ and
  $\ud\alpha_!^\opp \colon \ud{Y}_A^\opp \ra \ud{X}_A^\opp$; 
\item for each $\R$-algebra $A$ and each separated, locally
  proper morphism $\alpha \colon Y \ra X$ with
  $\alpha_! \colon \Mod(Y_A) \ra \Mod(X_A)$ of finite
  cohomological dimension there are a morphism
  $\ud\alpha^! \colon \ud{X}_A \ra \ud{Y}_A$ and its opposite
  $(\ud\alpha^!)^\opp \colon \ud{X}_A^\opp \ra \ud{Y}_A^\opp$.
\end{enumerate}
Obviously, $\fml_\R$ contains
$\fml'_\R$ as a full subcategory, and the involution $(-)^\opp$
on $\fml'_\R$ extends to $\fml_\R$.
%
%
The interpretation functor \eqref{eq:interprete-fml'-tricat}
of formulas in triangulated categories extends uniquely to a
functor 
\begin{equation}
  \label{eq:interprete-fml-tricat}
  \D \colon \fml_\R \ra \trcat_\R  
\end{equation}
of multicategories by $\ud\alpha\inv \mapsto \dL\alpha\inv$,
$\ud\alpha_! \mapsto \dR \alpha_!$, $\ud\alpha^! \mapsto
\alpha^!$.


\subsection{More fixed data}
\label{sec:more-fixed-data}

We keep the conventions from \ref{sec:fixed-data-1}.
Additionally, for each $\kk$-algebra $A$ and each separated,
locally proper morphism $\alpha \colon Y \ra X$ of topological
spaces with $\alpha_! \colon \Mod(Y_A) \ra \Mod(X_A)$ of finite
cohomological dimension, we fix a bounded complex
$\mathcal{L}=\mathcal{L}_\alpha \in \C(Y_A)$ with flat and
$\alpha$-c-soft components and a quasi-isomorphism
\begin{equation}
  \label{eq:flat-alpha-c-soft-reso}
  A_Y \ra \mathcal{L}=\mathcal{L}_\alpha  
\end{equation}
in $\C(Y_A)$; this is possible by
the proof of 
\cite[Thm.~7.7]{wolfgang-olaf-locallyproper}.

\begin{proposition}
  \label{p:alpha_!^L-adjoints}
  Let $A$ be a
  $\kk$-algebra and
  $\alpha \colon Y \ra X$ a separated, locally proper
  morphism of 
  topological spaces with
  $\alpha_! \colon 
  \Mod(Y_A) \ra \Mod(X_A)$ 
  of finite cohomological dimension.
  Let $\mathcal{M} \in \C(Y_A)$ be a bounded complex
  of flat $A_Y$-modules
  such that one of the following two conditions is satisfied.
  \begin{enumerate}
  \item 
    \label{enum:components-alpha-c-soft}
    All components of $\mathcal{M}$ are 
    $\alpha$-c-soft.
  \item 
     \label{enum:otimes-components-alpha_!-acyclic}
    For all components $\mathcal{M}^p$ of $\mathcal{M}$ the
    functor 
    $(- \otimes \mathcal{M}^p) \colon \Mod(Y_A) \ra \Mod(Y_A)$
    lands in the subcategory of $\alpha_!$-acyclic objects.
  \end{enumerate}
  Then the 
  $\KK$-functor 
  \begin{equation}
    \label{eq:alpha_!^L}
    \alpha_!^{\mathcal{M}} := \alpha_!(\mathcal{M} \otimes -)
    \colon
    \ul\C(Y_A) \ra \ul\C(X_A)
  \end{equation}
  admits a right adjoint $\KK$-functor
  \begin{equation}
    \label{eq:alpha^!_L}
    \alpha^!_{\mathcal{M}} 
    \colon
    \ul\C(X_A) \ra \ul\C(Y_A).
  \end{equation}
  Moreover, $\alpha_!^{\mathcal{M}}$ preserves acyclic objects and 
  $\alpha^!_{\mathcal{M}}$ 
  preserves h-injective objects, componentwise injective
  objects and, in particular, 
  $\II$-fibrant objects.
\end{proposition}

\begin{proof}
  If condition~\ref{enum:components-alpha-c-soft}
  holds, this follows from
  \cite[Prop.~7.5 and proofs of Thms~7.7 and
  8.3.(a)]{wolfgang-olaf-locallyproper}. 
  By inspection of the proof of
  \cite[Prop.~7.5]{wolfgang-olaf-locallyproper}
  we see that 
  condition~\ref{enum:otimes-components-alpha_!-acyclic}
  is also sufficient for our conclusion.
\end{proof}

Given any datum as in Proposition~\ref{p:alpha_!^L-adjoints} 
we fix 
an adjunction $(\alpha_!^{\mathcal{M}},
\alpha^!_{\mathcal{M}})$ in $\DGCAT_\kk$.

\begin{remark}
  \label{rem:choice-of-L}
  Subsequent constructions will depend on the fixed 
  componentwise flat and $\alpha$-c-soft resolution
  \eqref{eq:flat-alpha-c-soft-reso} of $A_Y$. Nevertheless the
  constructions obtained from different choices are easy to
  compare, starting from the following observation: if $A_Y \ra
  \mathcal{L}'$ is another componentwise 
  flat and $\alpha$-c-soft resolution, its tensor product with
  $A_Y \ra 
  \mathcal{L}=\mathcal{L}_\alpha$ provides a resolution
  $A_Y \ra \mathcal{L} \otimes \mathcal{L}'$ of the same type
  which allows the comparison of $\alpha_!^\mathcal{L}$ and
  $\alpha_!^{\mathcal{L}'}$ via
  $\alpha_!^{\mathcal{L} \otimes \mathcal{L}'}$, and then also of
  $\alpha^!_\mathcal{L}$ and $\alpha^!_{\mathcal{L}'}$, by
  Remark~\ref{rem:morphism-between-left-adjoints} below. 
\end{remark}

\subsection{Lifts of derived functors}
\label{sec:lifts-deriv-funct-2}

Let $A$ be a $\kk$-algebra and $\alpha \colon Y \ra X$ a morphism
of topological spaces.
Define $\KK$-functors
\begin{align}
  \ul\alpha_! := \ii \alpha_!
  & \colon \ul\II(Y_A)
    \xra{\alpha_!} \ul\C(X_A) \xra{\ii}
    \ul\II(X_A),\\
  \ul\alpha\inv := \ii \alpha\inv 
  & \colon \ul\II(X_A)
    \xra{\alpha\inv} \ul\C(Y_A) \xra{\ii}
    \ul\II(Y_A).
\end{align}
Similarly to \eqref{eq:omega-alpha_*}, \eqref{eq:omega-alpha^*},
there are canonical 
2-isomorphisms 
\begin{align}
  \label{eq:omega-alpha_!}
  \omega_{\alpha_!} 
  & \colon
    [\ul{\alpha}_!]\ol{[\ii]}
    \sira \ol{[\ii]}\dR\alpha_!,\\
  \label{eq:omega-alpha-inv}
  \omega_{\alpha\inv} 
  & \colon
    [\ul{\alpha}\inv]\ol{[\ii]}
    \sira \ol{[\ii]}\dL\alpha\inv.
\end{align}

If $\alpha$ is separated and locally proper with $\alpha_! \colon
\Mod(Y_A) \ra \Mod(X_A)$ of finite cohomological dimension, 
Proposition~\ref{p:alpha_!^L-adjoints}
allows to 
define the $\KK$-functor  
\begin{equation}
  \ul\alpha^! := \alpha^!_\mathcal{L} 
  = \alpha^!_{\mathcal{L}_\alpha} 
  \colon \ul\II(X_A) \ra
  \ul\II(Y_A).   
\end{equation}
It follows from the proof of
\cite[Thm.~8.3.(a)]{wolfgang-olaf-locallyproper} that 
the composition 
\begin{equation}
  \label{eq:2}
  \D(X_A) 
  \xra{\ol{[\ii]}} 
  [\ul\II(X_A)]
  \xra{[\ul\alpha^!]}
  [\ul\II(Y_A)]
  \ra
  [\ul\C(Y_A)]
  \xra{q_{Y_A}}
  \D(Y_A)
\end{equation}
is right adjoint to $\dR \alpha_!$ with explicitly given unit and
counit morphisms obtained from $(\alpha_!^\mathcal{L},
\alpha^!_\mathcal{L})$, and therefore canonically 
isomorphic to $\alpha^!$. We obtain a canonical 2-isomorphism
\begin{equation}
  \label{eq:omega-alpha^!}
  \omega_{\alpha^!} \colon [\ul\alpha^!]\ol{[\ii]} \sira
  \ol{[\ii]} \alpha^!.
\end{equation}

The results of \ref{sec:interpr-form-enhanc}
generalize in the obvious way. The interpretation functor
\eqref{eq:II-interpret} of formulas in enhancements extends
uniquely to a functor
\begin{equation}
  \label{eq:II-interpret-fml}
  \ul\II \colon
  \fml_\kk \ra \enh_\kk
\end{equation}
of multicategories by $\ud\alpha\inv \mapsto \ul\alpha\inv$,
$\ud\alpha_! \mapsto \ul\alpha_!$,
$\ud\alpha^! \mapsto \ul\alpha^!$. The pseudo-natural
transformation \eqref{eq:pseudo-nat-trafo-D-to-[I]} accordingly
extends by mapping $\ud\alpha_!$, $\ud\alpha\inv$, $\ud\alpha^!$
to
$\omega_{\alpha_!}$, $\omega_{\alpha\inv}$, $\omega_{\alpha^!}$, 
respectively.
Definition~\ref{d:enhance-2-morphism} and all results using this
definition extend to and will be used
for arbitrary
morphisms $\ud{v}$ and $\ud{w}$ in $\fml_\kk$.


\subsection{Lifts of relations}
\label{sec:lifts-relations-1}

\subsubsection{Lifts of 2-(iso)morphisms}
\label{sec:lifts-2-isomorphisms}

\begin{lemma}
  \label{l:ulalpha*-inv}
  Let $A$ be a $\kk$-algebra and $\alpha \colon Y \ra X$ a morphism
  of topological spaces.
  The obvious objectwise homotopy equivalence
  $\ul\alpha^* =\ii\alpha^*\ee=\ii \alpha\inv \ee \xra{[\sim]} 
  \ii\alpha\inv = \ul\alpha\inv$
  in $\tildew{\ENH}_\kk$ defines an isomorphism 
  \begin{equation}
    \label{eq:ulalpha*-inv}
    \ul\alpha^* \sira \ul\alpha\inv
  \end{equation}
  in $\ENH_\kk$ which enhances the identity $\dL \alpha^* = \dL
  \alpha\inv$. 
\end{lemma}

\begin{proof}
  Obvious.
\end{proof}

We deduce the following consequence.

\begin{lemma}
  \label{l:adjunction-morphisms-push-pullinv}
  Let $A$ be a $\kk$-algebra and $\alpha \colon Y \ra X$ a morphism
  of topological spaces.
  The compositions
  \begin{align}
    \label{eq:ainva-id}
    \ul\alpha\inv \ul\alpha_* 
    & \xsira{\eqref{eq:ulalpha*-inv}\inv \ul\alpha_*}
      \ul\alpha^*\ul\alpha_* 
      \xra{\eqref{eq:aa-id}}
      \id,\\
    \label{eq:id-aainv}
    \id 
    & \xra{\eqref{eq:id-aa}}
      \ul\alpha_* \ul\alpha^*
      \xsira{\ul\alpha_*\eqref{eq:ulalpha*-inv}}
      \ul\alpha_* \ul\alpha\inv
  \end{align}
  of 2-morphisms 
  in $\ENH_\kk$ enhance counit $\dL\alpha\inv \dR \alpha_* \ra \id$
  and unit $\id \ra \dR\alpha_* \dL\alpha\inv$ of the adjunction
  $(\dL\alpha\inv, \dR\alpha_*)$.
  Moreover, the 
  two 1-morphisms $\ul\alpha\inv$ and $\ul\alpha_*$ and the
  two 2-morphisms 
  \eqref{eq:id-aainv}
  and
  \eqref{eq:ainva-id}
  form an adjunction in $\ENH_\kk$.
\end{lemma}

\begin{proof}
  The first statement follows from
  Proposition~\ref{p:adjunction-morphisms-push-pull}
  and Lemma~\ref{l:ulalpha*-inv}.
  The second statement follows from
  Proposition~\ref{p:adjunction-*-pull-push-in-ENH}
\end{proof}

\begin{lemma}
  Let $A$ be a
  $\kk$-algebra and
  $\alpha \colon Y \ra X$ a 
  morphism of 
  topological spaces.
  Then the obvious 2-morphism $\alpha_! \ra \alpha_*$ between
  1-morphisms $\ul\C(Y_A) \ra \ul\C(X_A)$ in $\DGCAT_\kk$ induces
  a 2-morphism $\ul\alpha_!=\ii\alpha_! \ra \ii \alpha_*
  =\ul\alpha_*$ in $\tildew{\ENH}_\kk$ and then a 2-morphism
  \begin{equation}
    \label{eq:alpha_!-to-alpha_*}
    \ul\alpha_! \ra \ul\alpha_*
  \end{equation}
  in $\ENH_\kk$ that enhances $\dR\alpha_! \ra \dR \alpha_*$.
  If $\alpha$ is proper, then 
  \eqref{eq:alpha_!-to-alpha_*} is a 2-isomorphism.
\end{lemma}

\begin{proof}
  Clearly, $\ul\alpha_! \ra \ul\alpha_*$ enhances $\dR \alpha_!
  \ra \dR\alpha_*$. If
  $\alpha$ is proper then 
  $\alpha_!=\alpha_*$ and hence 
  $\ul\alpha_! \ra \ul\alpha_*$ is a 2-isomorphism.
\end{proof}

We need some preparations for Proposition~\ref{p:compos-alpha-!}.

\begin{remark}
  \label{rem:L-id}
  For 
  $\alpha=\id_X$ the morphism
  $\id \sira A_X \otimes (-)
  \xra{\eqref{eq:flat-alpha-c-soft-reso} \otimes(-)}
  \mathcal{L} \otimes (-) =
  \id_!^\mathcal{L}$ 
  between left adjoints induces, by 
  Remark~\ref{rem:morphism-between-left-adjoints} below, a
  morphism 
  \begin{equation}
    \label{eq:id^!_L-to-id}
    \id^!_\mathcal{L} \ra \sheafHom(A_X,-) \sira \id
  \end{equation}
  between the corresponding right adjoints. 
  We have
  $\id_\mathcal{L}^! \cong \sheafHom(\mathcal{L},-)$. In
  particular,
  the 
  evaluation of
  \eqref{eq:id^!_L-to-id}
  at an h-injective object is
  a quasi-isomorphism, and even a homotopy equivalence because
  $\id_\mathcal{L}^!$ preserves h-injectives. 
  Of course, we could also assume that 
  $\mathcal{L}_{\id_X}=A_X$.
\end{remark}

\begin{lemma}
  \label{l:pull-tensor-is-alphabeta-acyclic}
  Let $A$ be a
  $\kk$-algebra and
  let $Z \xra{\beta} Y \xra{\alpha} X$ be separated, locally
  proper 
  morphisms of 
  topological spaces with
  $\alpha_! \colon 
  \Mod(Y_A) \ra \Mod(X_A)$ 
  and $\beta_! \colon 
  \Mod(Z_A) \ra \Mod(Y_A)$ 
  of finite cohomological dimension.
  Assume that $\mathcal{A} \in \Mod(Y_A)$ 
  is flat and $\alpha$-c-soft 
  and that $\mathcal{B} \in \Mod(Z_A)$
  is flat and $\beta$-c-soft. 
  Then 
  $\beta\inv(\mathcal{A}) \otimes \mathcal{B} \otimes
  \mathcal{T}$ is $(\alpha\beta)_!$-acyclic, for any $\mathcal{T}
  \in \Mod(Z_A)$. 
\end{lemma}

\begin{proof}
  Let 
  \begin{equation}
    \label{eq:betainvA}
    0 \ra
    \beta\inv(\mathcal{A}) \otimes \mathcal{B} \otimes
    \mathcal{T} \ra I^0 \ra I^1 \ra \dots
  \end{equation}
  be an injective resolution.
  By
  \cite[Lemmas~5.5, 7.4]{wolfgang-olaf-locallyproper},
  $\beta\inv(\mathcal{A}) \otimes \mathcal{B} \otimes
  \mathcal{T}$
  is $\beta$-c-soft and hence $\beta_!$-acyclic.
  Hence 
  \begin{equation}
    \label{eq:beta_!betainvA}
    0 \ra \beta_!(\beta\inv(\mathcal{A}) \otimes \mathcal{B} \otimes
    \mathcal{T}) \ra \beta_!(I^0) \ra \beta_!(I^1) \ra
  \end{equation}
  is an exact sequence. All objects $\beta_!(I^n)$ are
  $\alpha$-c-soft, by
  \cite[5.3, Lemma~5.7]{wolfgang-olaf-locallyproper}, and hence
  $\alpha_!$-acyclic. We also have an isomorphism
  \begin{equation}
    \label{eq:5}
    \beta_!(\beta\inv(\mathcal{A}) \otimes \mathcal{B} \otimes
    \mathcal{T}) \sila
    \mathcal{A} \otimes \beta_!(\mathcal{B} \otimes
    \mathcal{T})
  \end{equation}
  by \cite[Thm.~6.2]{wolfgang-olaf-locallyproper}, and the object
  on the right is $\alpha$-c-soft and hence $\alpha_!$-acyclic by 
  \cite[Lemmas~5.5, 7.4]{wolfgang-olaf-locallyproper}. 
  Hence \eqref{eq:beta_!betainvA} stays exact when we apply
  $\alpha_!$. Hence, using $\alpha_! \beta_!=(\alpha\beta)_!$,
  see \cite[Thm.~3.6]{wolfgang-olaf-locallyproper}, 
  \eqref{eq:betainvA} stays exact when applying
  $(\alpha\beta)_!$.
  This implies the lemma.
\end{proof}

\begin{remark}
  \label{rem:L-alpha-beta}
  Let $A$ be a
  $\kk$-algebra and let
  $Z \xra{\beta} Y \xra{\alpha} X$ be separated, locally proper
  morphisms of 
  topological spaces with
  $\alpha_! \colon 
  \Mod(Y_A) \ra \Mod(X_A)$ 
  and $\beta_! \colon 
  \Mod(Z_A) \ra \Mod(Y_A)$ 
  of finite cohomological dimension.
  Then
  $\alpha\beta$ is also separated and locally proper with
  $(\alpha\beta)_! \colon \Mod(Z_A) \ra \Mod(X_A)$ of finite
  cohomological dimension, by
  \cite[Thm.~8.3.(b)]{wolfgang-olaf-locallyproper}.

  Consider the quasi-isomorphisms
  $A_Z=\beta\inv A_Y \ra
  \beta\inv(\mathcal{L}_\alpha)$, $A_Z \ra \mathcal{L}_\beta$ and
  $A_Z \ra \mathcal{L}_{\alpha\beta}$ between bounded complexes
  with flat 
  components.
  Suitable tensor products of these quasi-isomorphisms yield
  quasi-isomorphisms 
  \begin{equation}
    \label{eq:quisos-flat-alphabeta-c-soft}
    \mathcal{L}_{\alpha\beta}
    \ra
    \mathcal{M}:=\mathcal{L}_{\alpha\beta} \otimes
    \beta\inv(\mathcal{L}_\alpha) 
    \otimes \mathcal{L}_\beta 
    \la
    \mathcal{N}:=
    \beta\inv(\mathcal{L}_\alpha)
    \otimes \mathcal{L}_\beta 
  \end{equation}
  between bounded complexes with flat 
  components.
  We can apply 
  Proposition~\ref{p:alpha_!^L-adjoints}
  to the map
  $\alpha \beta$
  and the three complexes
  $\mathcal{L}_{\alpha\beta}$,
  $\mathcal{M}$ and $\mathcal{N}$:
  this is obvious for
  $\mathcal{L}_{\alpha\beta}$
  and follows from 
  Lemma~\ref{l:pull-tensor-is-alphabeta-acyclic}
  for 
  $\mathcal{M}$ and $\mathcal{N}$.
  Hence the adjunctions
  $((\alpha\beta)_!^{\mathcal{L}_{\alpha\beta}},
  (\alpha\beta)^!_{\mathcal{L}_{\alpha\beta}})$, 
  $((\alpha\beta)_!^\mathcal{M},
  (\alpha\beta)^!_\mathcal{M})$ 
  and
  $((\alpha\beta)_!^\mathcal{N}, (\alpha\beta)^!_\mathcal{N})$
  are available. 

  If $T \in \C(Z_A)$ is arbitrary, applying $(-) \otimes T$
  to 
  \eqref{eq:quisos-flat-alphabeta-c-soft}
  yields quasi-isomorphisms between componentwise
  $(\alpha\beta)_!$-acyclic complexes, by 
  \cite[Lemmas~5.5, 7.4]{wolfgang-olaf-locallyproper}
  and Lemma~\ref{l:pull-tensor-is-alphabeta-acyclic}. 
  Hence
  $(\alpha\beta)_!$ maps these quasi-isomorphisms to
  quasi-isomorphisms, by 
  \cite[Lemma~12.4.(b)]{wolfgang-olaf-locallyproper}
  using the fact that 
  $(\alpha\beta)_! \colon \Mod(Z_A) \ra \Mod(X_A)$ has finite
  cohomological dimension. This shows that the
  evaluation of the zig-zag 
  \begin{equation}
    (\alpha\beta)_!^{\mathcal{L}_{\alpha\beta}} 
    \ra
    (\alpha\beta)_!^{\mathcal{M}}
    \la
    (\alpha\beta)_!^{\mathcal{N}}
  \end{equation}
  of 2-morphisms 
  in $\DGCAT_\kk$ at any $T \in \ul\C(Z_A)$ is a zig-zag of 
  quasi-isomorphisms.
  Remark~\ref{rem:morphism-between-left-adjoints} below therefore
  yields a zig-zag 
  \begin{equation}
    \label{eq:alphabeta^!-LMN}
    (\alpha\beta)^!_{\mathcal{L}_{\alpha\beta}} 
    \la
    (\alpha\beta)^!_{\mathcal{M}}
    \ra
    (\alpha\beta)^!_{\mathcal{N}}
  \end{equation}
  of 2-morphisms in $\DGCAT_\kk$. We claim that the evaluation of
  this zig-zag at 
  an h-injective object $I \in \ul\C(X_A)$ consists of 
  homotopy equivalences (between h-injectives). 
  
  Let $T \in \ul\C(Z_A)$ and consider the commutative diagram 
  \begin{equation}
    \xymatrix{
      {[\ul\C_{X_A}]((\alpha\beta)_!^{\mathcal{M}}T, I)}
      \ar[r]^-{\sim}
      \ar[d]_-{\sim}
      & 
      {[\ul\C_{X_A}]((\alpha\beta)_!^{\mathcal{N}}T, I),}
      \ar[d]_-{\sim}
      \\
      {[\ul\C_{Z_A}](T, (\alpha\beta)^!_{\mathcal{M}}I)}
      \ar[r]
      &
      {[\ul\C_{Z_A}](T, (\alpha\beta)^!_{\mathcal{N}}I)}
    }
  \end{equation}
  cf.\ \eqref{eq:conjugate-maps}.
  Its upper horizontal arrow is an isomorphism because $I$ is
  h-injective and $(\alpha\beta)_!^\mathcal{N} T \ra
  (\alpha\beta)_!^\mathcal{M} T$ is a quasi-isomorphism as
  observed above. The vertical isomorphisms 
  come
  from 
  the adjunctions.
  Since $T$ was arbitrary, the Yoneda lemma shows that 
  $(\alpha\beta)^!_{\mathcal{M}}I \ra
  (\alpha\beta)^!_{\mathcal{N}}I$ is an isomorphism in
  $[\ul\C(Z_A)]$ and hence a homotopy equivalence in $\C(Z_A)$.
  The same argument with $\mathcal{N}$
  replaced by $\mathcal{L}_{\alpha\beta}$ shows the second half
  of the claim. 

  Since $\mathcal{L}_\alpha$ has flat components,
  \cite[Thms.~3.6, 6.2]{wolfgang-olaf-locallyproper} 
  provide a 2-isomorphism
  \begin{equation}
    \alpha_!^{\mathcal{L}_\alpha}\beta_!^{\mathcal{L}_\beta}
    =
    \alpha_!(\mathcal{L}_\alpha \otimes
    \beta_!(\mathcal{L}_\beta \otimes (-)))
    \sira
    \alpha_!\beta_!(\beta\inv(\mathcal{L}_\alpha) \otimes
    \mathcal{L}_\beta \otimes (-))
    =   
    (\alpha\beta)_!^\mathcal{N} 
  \end{equation}
  in $\DGCAT_\kk$ which corresponds by  
  Remark~\ref{rem:morphism-between-left-adjoints} below to a 
  2-isomorphism
  \begin{equation}
    \label{eq:alphabeta^!-N-alpha_!-beta_!}
    (\alpha\beta)^!_{\beta\inv(\mathcal{L}_\alpha) \otimes
    \mathcal{L}_\beta}
    \sira
    \beta^!_{\mathcal{L}_\beta}
    \alpha^!_{\mathcal{L}_\alpha}.
  \end{equation}
\end{remark}

\begin{proposition}
  \label{p:compos-alpha-!}
  Let $A$ be a $\kk$-algebra and let
  $Z \xra{\beta} Y \xra{\alpha} X$ be separated, locally proper
  morphisms of topological spaces with
  $\alpha_! \colon \Mod(Y_A) \ra \Mod(X_A)$ and
  $\beta_! \colon \Mod(Z_A) \ra \Mod(Y_A)$ of finite
  cohomological dimension.  Then the zig-zags
  \begin{align}
    \label{eq:id_!}
    \ul\id_! 
    & = \ii \xla{[\sim]} \id,\\
    \label{eq:alphabeta_!}
    \ul{(\alpha\beta)}_!
    & =
      \ii (\alpha\beta)_!
      =\ii \alpha_! \beta_! 
      \xra{[\sim]}
      \ii\alpha_!\ii\beta_!
      =      
      \ul\alpha_!\ul\beta_!,
    \\
    \ul\id^! 
    & = \id^!_{\mathcal{L}_{\id}}
      \xra[{\eqref{eq:id^!_L-to-id}}]{[\sim]}
      \id,\\
    \label{eq:alphabeta^!-tilde-ENH}
    \ul{(\alpha\beta)}^! 
    & = (\alpha\beta)^!_{\mathcal{L}_{\alpha\beta}} 
      \xla[{\eqref{eq:alphabeta^!-LMN}_{\text{left}}}]{[\sim]}
      (\alpha\beta)^!_{\mathcal{L}_{\alpha\beta} \otimes
      \beta\inv(\mathcal{L}_\alpha) 
      \otimes \mathcal{L}_\beta}
      \xra[{\eqref{eq:alphabeta^!-LMN}_{\text{right}}}]{[\sim]}
      (\alpha\beta)^!_{\beta\inv(\mathcal{L}_\alpha)
      \otimes \mathcal{L}_\beta}
    \\
    \notag
    & \qquad \qquad \qquad \qquad \qquad \qquad \qquad \qquad
      \qquad \qquad
      \xra[{\eqref{eq:alphabeta^!-N-alpha_!-beta_!}}]{\sim}
      \beta^!_{\mathcal{L}_\beta}
      \alpha^!_{\mathcal{L}_\alpha} 
      =
      \ul\beta^!\ul\alpha^!
  \end{align}
  of objectwise homotopy 
  equivalences in $\tildew{\ENH}_\kk$ define 2-isomorphisms
  \begin{align}
    \label{eq:id_!-ENH}
    \ul\id_! 
    & \sila \id,
    \\
    \label{eq:alphabeta_!-ENH}
    \ul{(\alpha \beta)}_! 
    & \sila \ul\alpha_! \ul\beta_!,
    \\
    \label{eq:id^!-ENH}
    \ul\id^!
    & \sira \id,
    \\ 
    \label{eq:alphabeta^!-ENH}
    \ul{(\alpha\beta)}^!
    & \sira \ul\beta^!\ul\alpha^!
  \end{align}
  in $\ENH_\kk$
  that enhance the isomorphisms
  $\dR \id_! \cong \id$, 
  $\dR(\alpha \beta)_! \cong \dR \alpha_! \dR \beta_!$,
  $\id^! \cong \id$,
  $(\alpha \beta)^! \cong \beta^! \alpha^!$ from
  \cite[Thm.~8.3.(b)]{wolfgang-olaf-locallyproper}
\end{proposition}

\fussnote{
  Aus der Konstruktion sollte recht schnell folgen, dass die
  beiden Morphismen 
  \eqref{eq:alphabeta^!-ENH}
  und 
  \eqref{eq:alphabeta_!-ENH}
  konjugiert zueinander sind (und ebenso f\"ur die beiden
  anderen).
}

\begin{proof}
  The claim for $\ul\id_!$ is obvious, and that for 
  $\ul{(\alpha \beta)}_!$ follows from the proof of
  \cite[Thm.~8.3.(b)]{wolfgang-olaf-locallyproper}.  
  The claims for $\ul\id^!$ and $\ul{(\alpha\beta)}^!$ follow
  from Remarks~\ref{rem:L-id} and \ref{rem:L-alpha-beta}. 
\end{proof}

We need some preparations for 
Proposition~\ref{p:adjunction-morphisms-proper-push-pull}. 

\begin{remark}
  \label{rem:alpha_!-vs-alpha_!^L}
  Let $A$ be a
  $\kk$-algebra and
  $\alpha \colon Y \ra X$ a separated, locally proper
  morphism of 
  topological spaces with
  $\alpha_! \colon 
  \Mod(Y_A) \ra \Mod(X_A)$ 
  of finite cohomological dimension.
  We claim that the evaluation of the 2-morphism
  \begin{equation}
    \label{eq:alpha_!-vs-alpha_!^L}
    \alpha_! \sira \alpha_!^{A_Y}= \alpha_!(A_Y \otimes -) 
    \xra{\alpha_!(\eqref{eq:flat-alpha-c-soft-reso} \otimes(-))}
    \alpha_!(\mathcal{L} \otimes -) 
    = \alpha_!^\mathcal{L}.
  \end{equation}
  in $\DGCAT_\kk$
  at any componentwise injective object $J \in
  \ul\C(Y_A)$, and in particular at 
  any object of $\ul\II(Y_A)$, is a
  quasi-isomorphism. 

  Let $J \in \ul\C(Y_A)$ be componentwise injective. Then
  the quasi-isomorphism  
  \eqref{eq:flat-alpha-c-soft-reso} between h-flat objects gives
  rise to the quasi-isomorphism
  \begin{equation}
    J \sira A_Y \otimes J \xra{\eqref{eq:flat-alpha-c-soft-reso}
      \otimes \id} \mathcal{L} \otimes J
  \end{equation}
  between complexes of
  $\alpha_!$-acyclic objects, 
  by 
  Proposition~\ref{p:spalt-5.7-sites}
  and \cite[5.3, Lemmas 5.5,
  7.4]{wolfgang-olaf-locallyproper}.
  Therefore applying $\alpha_!$ preserves this quasi-isomorphism,
  by \cite[Lemma~12.4.(b)]{wolfgang-olaf-locallyproper}. This
  shows the claim.
\end{remark}

\begin{proposition}
  \label{p:adjunction-morphisms-proper-push-pull}
  Let $A$ be a
  $\kk$-algebra and
  $\alpha \colon Y \ra X$ a separated, locally proper
  morphism of 
  topological spaces with
  $\alpha_! \colon 
  \Mod(Y_A) \ra \Mod(X_A)$ 
  of finite cohomological dimension.
  Then the zig-zags
  \begin{align}
    \label{eq:zigzag-a_!a^!-id}
    \ul\alpha_! \ul\alpha^! 
    & = \ii\alpha_! \alpha^!_\mathcal{L}
      \xra[{\ii\eqref{eq:alpha_!-vs-alpha_!^L}
      \alpha^!_\mathcal{L}}]{[\sim]}  
      \ii\alpha_!^\mathcal{L} \alpha^!_\mathcal{L}
      \ra
      \ii 
      \xla{[\sim]}
      \id,
    \\
    \label{eq:zigzag-id-a^!a_!}
    \id 
    & \ra
      \alpha^!_\mathcal{L} \ii \alpha_!^\mathcal{L}
      \xla[{\alpha^!_\mathcal{L}\ii\eqref{eq:alpha_!-vs-alpha_!^L}}]
      {[\sim]} 
      \alpha^!_\mathcal{L} \ii \alpha_!
      = \ul\alpha^! \ul\alpha_!
  \end{align}
  of 2-morphisms in $\tildew{\ENH}_\kk$ 
  where the first arrow in \eqref{eq:zigzag-id-a^!a_!}
  is the composition
  $\id \ra
  \alpha^!_\mathcal{L} \alpha_!^\mathcal{L}
  \ra
  \alpha^!_\mathcal{L} \ii \alpha_!^\mathcal{L}$
  define 2-morphisms
  \begin{align}
    \label{eq:a_!a^!-id}
    \ul\alpha_! \ul\alpha^! & \ra \id,\\
    \label{eq:id-a^!a_!}
    \id & \ra \ul\alpha^! \ul\alpha_! 
  \end{align}
  in $\ENH_\kk$ that enhance the counit $(\dR\alpha_!) \alpha^! \ra
  \id$ and the unit $\id \ra \alpha^!(\dR\alpha_!)$
  of the adjunction $(\dR \alpha_!, \alpha^!)$.
\end{proposition}

The first arrow in \eqref{eq:zigzag-id-a^!a_!}
is not expanded as the given composition since
the $\KK$-functor $\alpha^!_\mathcal{L} \alpha_!^\mathcal{L}$ 
does not land in $\ul\II(Y_A)$ in general.

\begin{proof}
  Note that 
  $\ii\eqref{eq:alpha_!-vs-alpha_!^L} \colon \ii \alpha_! \ra \ii
  \alpha_!^\mathcal{L}$ considered as a 2-morphism between
  1-morphisms $\ul\II(Y_A) \ra \ul\II(X_A)$ in $\tildew{\ENH}_\kk$
  is an objectwise homotopy equivalence, by
  Remark~\ref{rem:alpha_!-vs-alpha_!^L}.
  We leave the rest of the proof to the reader.
\end{proof}

\begin{proposition}
  \label{p:proper-base-change}
  Let  $A$ be a
  $\kk$-algebra and
  \begin{equation}
    \xymatrix{
      {Y'} \ar[r]^-{\beta'} \ar[d]^-{\alpha'} &
      {Y} \ar[d]^-{\alpha} \\
      {X'} \ar[r]^-{\beta} &
      {X}
    }
  \end{equation}
  a cartesian diagram of topological spaces with $\alpha$
  separated, locally proper, and with $\alpha_! \colon
  \Mod(Y_A) \ra \Mod(X_A)$ of finite cohomological dimension.
  Let $t \colon
  \beta\inv \alpha_! \sira \alpha'_! \beta'^{-1}$ be the 
  isomorphism obtained from
  \cite[Thm.~4.4]{wolfgang-olaf-locallyproper}. Then the zig-zag 
  \begin{equation}
    \label{eq:proper-base-change-zigzag}
    \ul{\beta}\inv \ul\alpha_! 
    = 
    \ii \beta\inv \ii \alpha_! 
    \xlongla{[\sim]} 
    \ii \beta\inv \alpha_!
    \xlongsira{\ii t} 
    \ii \alpha'_! \beta'^{-1}
    \xlongra{[\sim]} 
    \ii \alpha'_! \ii \beta'^{-1}
    =
    \ul\alpha'_! \ul\beta'^{-1}
  \end{equation}
  of objectwise homotopy equivalences
  in $\tildew{\ENH}_\kk$ defines a 2-isomorphism
  \begin{equation}
    \label{eq:proper-base-change-ENH}
    \ul{\beta}\inv \ul\alpha_! 
    \sira 
    \ul\alpha'_! \ul\beta'^{-1}
  \end{equation}
  in $\ENH_\kk$ that enhances the derived proper
  base change isomorphism
  \begin{equation}
    \label{eq:proper-base-change-D}
    \dL\beta\inv \dR\alpha_! 
    \sira 
    \dR\alpha'_! \dL\beta'^{-1}  
  \end{equation}
  from
  \cite[Thm.~8.3.(c)]{wolfgang-olaf-locallyproper}. 
\end{proposition}

\begin{proof}
  We can assume that $\dL \beta\inv E= \beta\inv E$ and $\dR
  \alpha_! F= \alpha_! \ii F$, for $E \in \D(X_A)$ and $F \in
  \D(Y_A)$, and similarly for $\dL \beta'^{-1}$ and $\dR
  \alpha'_!$. Then 
  $\dL\beta\inv \dR\alpha_! F
  \sira 
  \dR\alpha'_! \dL\beta'^{-1} F$ is obtained from the zig-zag
  \begin{equation}
    \beta\inv \alpha_! \ii F 
    \xsira{t\ii}
    \alpha'_! \beta'^{-1} \ii F
    \xra{\alpha'_! \iotaii_{\beta'^{-1} \ii F}}
    \alpha'_! \ii \beta'^{-1} \ii F
    \xla{\alpha'_! \ii \beta'^{-1}\iotaii_F}
    \alpha'_! \ii \beta'^{-1} F
  \end{equation}
  of quasi-isomorphisms in $\C(X'_A)$; this follows from 
  \cite[8.3]{wolfgang-olaf-locallyproper} which also shows that
  the  
  third arrow in \eqref{eq:proper-base-change-zigzag}
  is an objectwise homotopy equivalence.
  The rest of the proof
  is obvious and left to the reader.
\end{proof}

\begin{proposition}
  \label{p:projection-fml}
  Let $A$ be a
  $\kk$-algebra and
  $\alpha \colon Y \ra X$ a separated, locally proper morphism of
  topological spaces with
  $\alpha_! \colon 
  \Mod(Y_A) \ra \Mod(X_A)$ of finite cohomological dimension.
  Consider the composition
  \begin{equation}
    \label{eq:compos-for-proj-fml}
    \alpha_!(-) \otimes (-)
    \ra 
    \alpha_!((-) \otimes \alpha\inv(-))
    \xlongra{\alpha_! \iotaii ((-) \otimes \alpha\inv(-))}
    \alpha_!\ii((-) \otimes \alpha\inv(-))
  \end{equation}
  where the first arrow is obtained from
  \cite[6.1]{wolfgang-olaf-locallyproper}.
  This composition gives rise to the third arrow in the zig-zag 
  \begin{multline}
    \label{eq:projection-fml-zigzag}
    \ul{\alpha}_!(-) \ul\otimes (-)
    = 
    \ii(\ee\ii\alpha_!(-) \otimes \ee(-))
    \xlongra{[\sim]} 
    \ii(\ii\alpha_!(-) \otimes \ee(-))
    \xlongla{[\sim]} 
    \ii(\alpha_!(-) \otimes \ee(-))
    \\
    \xlongra{[\sim]}
    \ii\alpha_!\ii((-) \otimes \alpha\inv\ee(-))
    \xlongla{[\sim]}
    \ii\alpha_!\ii(\ee(-) \otimes \alpha\inv\ee(-))
    \xlongra{[\sim]}
    \ii\alpha_!\ii(\ee(-) \otimes \alpha\inv(-))
    \\
    \xlongra{[\sim]}
    \ii\alpha_!\ii(\ee(-) \otimes \ii\alpha\inv(-))
    \xlongla{[\sim]}
    \ii\alpha_!\ii(\ee(-) \otimes \ee\ii\alpha\inv(-))
    =
    \ul\alpha_!((-) \ul\otimes \;\ul\alpha\inv(-))
  \end{multline}
  of objectwise homotopy equivalences in $\tildew{\ENH}_\kk$ and defines
  a 2-isomorphism
  \begin{equation}
    \label{eq:projection-fml-ENH}
    \ul{\alpha}_!(-) \ul\otimes (-)
    \sira
    \ul\alpha_!((-) \ul\otimes \;\ul\alpha\inv(-))
  \end{equation}
  in $\ENH_\kk$ that enhances the isomorphism (the derived projection
  formula) 
  \begin{equation}
    \label{eq:projection-fml-D}
    \dR\alpha_!(-) \otimes^\dL (-)
    \sira
    \dR\alpha_!((-) \otimes^\dL \;\dL\alpha\inv(-))
  \end{equation}
  from
  \cite[Thm.~8.3.(d)]{wolfgang-olaf-locallyproper}. 
\end{proposition}

\begin{proof}
  Let $S \in \C(Y_A)$ have $\alpha$-c-soft components and let $F
  \in \C(X_A)$ be h-flat.
  We claim that the evaluation of the composition of
  \eqref{eq:compos-for-proj-fml} 
  at $(S,F)$ is a quasi-isomorphism.
  
  By \cite[Prop.~5.6]{spaltenstein} and
  with the notation from there, there is a
  quasi-isomorphism $P \ra F$ with $P \in
  \underrightarrow{\mathfrak{P}}(X_A)$. 
  Consider the commutative
  diagram
  \begin{equation}
    \xymatrix{
      {\alpha_! S \otimes P} \ar[r] \ar[d]
      &
      {\alpha_!(S \otimes \alpha\inv P)} \ar[r] \ar[d]
      &
      {\alpha_!\ii(S \otimes \alpha\inv P)} \ar[d]
      \\
      {\alpha_! S \otimes F} \ar[r]
      &
      {\alpha_!(S \otimes \alpha\inv F)} \ar[r]
      &
      {\alpha_!\ii(S \otimes \alpha\inv F)}
    }
  \end{equation}
  whose rows come from \eqref{eq:compos-for-proj-fml}.
  The proof of \cite[Thm.~8.3.(d)]{wolfgang-olaf-locallyproper}
  shows that the upper left horizontal arrow is an isomorphism
  and that the 
  upper right horizontal arrow is a quasi-isomorphism. The left
  and the right vertical arrows are quasi-isomorphisms because
  $F$, $P$, $\alpha\inv F$, $\alpha\inv P$ are h-flat. Hence the
  composition of the two lower horizontal arrows is a
  quasi-isomorphism. This proves our claim. 

  This claim
  and the fact that  
  $\II$-fibrant objects have injective and hence
  $\alpha$-c-soft components
  imply that 
  the third arrow in \eqref{eq:projection-fml-zigzag}
  is an objectwise homotopy equivalence. 

  We can assume that $\dL\alpha\inv E= \alpha\inv E$ and $\dR
  \alpha_! G= \alpha_! \ii G$, for $E \in \D(X_A)$ and $G \in
  \D(Y_A)$, and that $(-\otimes^\dL -) = (-\otimes \ee(-))$.
  Then \eqref{eq:projection-fml-D} evaluated at $(G,E)$
  is obtained from the zig-zag
  \begin{multline}
    \alpha_!\ii G \otimes \ee E 
    \xra{\eqref{eq:compos-for-proj-fml}_{(\ii G,\ee E)}}
    \alpha_! \ii(\ii G \otimes \alpha\inv \ee E)
    \la 
    \alpha_! \ii(\ii G \otimes \ee\alpha\inv \ee E)
    \\
    \la 
    \alpha_! \ii(G \otimes \ee\alpha\inv \ee E)
    \ra 
    \alpha_! \ii(G \otimes \ee\alpha\inv E)
  \end{multline}
  of quasi-isomorphisms in $\C(X_A)$ (the first arrow is a
  quasi-isomorphism by the above claim, and the second one
  because $\alpha\inv\ee E$ is h-flat), as follows from 
  \cite[8.3]{wolfgang-olaf-locallyproper}.
  The rest of the proof
  is obvious and left to the reader.
\end{proof}

\fussnote{
  K\"unneth formel in diesem Setting? e.\,g.\ \cite[p.~135]{KS}.
  Ich glaube aber, es ist \"aquivalent zum Proper base change...
  vermutlich habe ich das privat oder per
  \cite{iversen-coho-sheaves} gelernt.
}

Our next aim is Proposition~\ref{p:!-adjunction-sheafHom-ENH}.

Let $A$ be a $\kk$-algebra and $\alpha \colon Y \ra X$ a
separated, locally proper morphism of topological spaces.
The map constructed in \cite[6.1]{wolfgang-olaf-locallyproper}
gives rise to a 2-morphism
\begin{equation}
  \label{eq:proj-fml-cxs}
  \alpha_!(-) \otimes (-) 
  \ra \alpha_!((-) \otimes \alpha\inv(-))
\end{equation}
between 1-morphisms $(\ul\C(Y_A), \ul\C(X_A)) \ra \ul\C(X_A)$ in
$\DGCAT_\kk$. 

\begin{lemma}
  \label{l:proj-fml-cxs-compwise-flat}
  Let $A$ be a $\kk$-algebra and $\alpha \colon Y \ra X$ a
  separated, locally proper morphism of topological spaces.
  Let $F \in \C(X_A)$ 
  be componentwise
  flat and 
  $T \in \C(Y_A)$.
  Then 
  the evaluation 
  of \eqref{eq:proj-fml-cxs} at $(T,F)$ 
  is an isomorphism
 \begin{equation}
   \label{eq:proj-fml-cxs-compwise-flat}
   (\alpha_! T) \otimes F \sira \alpha_!(T \otimes \alpha\inv F)
 \end{equation}
\end{lemma}

\begin{proof}
  From 
  \cite[Thm.~6.2, Lemma 4.6]{wolfgang-olaf-locallyproper}
  we obtain the following two statements. 
  For $p$, $q \in \bZ$ the morphism
  $(\alpha_! T^p) \otimes F^q \sira \alpha_!(T^p \otimes
  \alpha\inv F^q)$ is an isomorphism.
  For
  fixed $n \in 
  \bN$, the morphism 
  $\bigoplus_{p+q =n}\alpha_!(T^p \otimes 
  \alpha\inv F^q) 
  \sira 
  \alpha_!(\bigoplus_{p+q =n}T^p \otimes 
  \alpha\inv F^q)$ is an isomorphism. The lemma follows.
\end{proof}

Assume in addition that   
$\alpha_! \colon \Mod(Y_A) \ra \Mod(X_A)$ 
has finite cohomological dimension.
From \eqref{eq:proj-fml-cxs} and symmetry we obtain 2-morphisms
\begin{align}
  \label{eq:proj-fml-cxs-with-L}
  \alpha^\mathcal{L}_!(-) \otimes (-) 
  & \ra
  \alpha^\mathcal{L}_!((-) \otimes \alpha\inv (-)),\\
  \label{eq:proj-fml-cxs-right-with-L}
  (-) \otimes \alpha^\mathcal{L}_!(-) 
  & \ra
  \alpha^\mathcal{L}_!(\alpha\inv(-) \otimes (-))
\end{align}
between 1-morphisms $(\ul\C(Y_A), \ul\C(X_A)) \ra \ul\C(X_A)$
and
$(\ul\C(X_A), \ul\C(Y_A)) \ra \ul\C(X_A)$ in
$\DGCAT_\kk$, respectively. 

Let $F$, $I \in \C(X_A)$ and $J \in \C(Y_A)$.  
The evaluation of the 2-morphism 
\eqref{eq:proj-fml-cxs-right-with-L}
at $(F, J)$, the 
adjunctions $(\alpha\inv, \alpha_*)$ and   
$(\alpha_!^\mathcal{L}, \alpha^!_\mathcal{L})$ and the
$\otimes$-$\sheafHom$-adjunction yield morphisms 
\begin{multline}
  \label{eq:derivation-!-adjunction-sheafHom}
  \ul\C_{X_A}(F, \alpha_*\sheafHom(J, \alpha^!_\mathcal{L} I))
  \sila
  \ul\C_{Y_A}(\alpha\inv F, \sheafHom(J, \alpha^!_\mathcal{L} I))
  \\
  \sila
  \ul\C_{Y_A}(\alpha\inv(F) \otimes J, \alpha^!_\mathcal{L} I)
  \sila
  \ul\C_{X_A}(\alpha_!^\mathcal{L}(\alpha\inv(F) \otimes J), I)
  \\
  \xra{\eqref{eq:proj-fml-cxs-right-with-L}_{(F,J)}^*}
  \ul\C_{X_A}(F \otimes \alpha_!^\mathcal{L}(J), I)
  \sira
  \ul\C_{X_A}(F, \sheafHom(\alpha_!^\mathcal{L}(J), I))
\end{multline}
in $\KK$. Hence the Yoneda lemma (or, more
concretely, evaluating at $F=v_!(A_V)$, for $v
\colon V \ra Y$ 
the embedding of an open subset $V$ of $Y$, for varying $V$)
and naturality in $J$ and $I$ yield a 
2-morphism 
\begin{equation}
  \label{eq:!-adjunction-sheafHom}
  \alpha_*\sheafHom(-, \alpha^!_\mathcal{L}(-))
  \ra 
  \sheafHom(\alpha_!^\mathcal{L}(-), -)
\end{equation}
between 1-morphisms $(\ul\C(Y_A)^\opp, \ul\C(X_A)) \ra \ul\C(X_A)$ in
$\DGCAT_\kk$. 

\begin{lemma}
  \label{l:!-adjunction-sheafHom-compwise-flat}
  Let $A$ be a $\kk$-algebra and $\alpha \colon Y \ra X$ a
  separated, locally proper morphism of topological spaces
  with $\alpha_! \colon 
  \Mod(Y_A) \ra \Mod(X_A)$ 
  of finite cohomological dimension.
  Then the 2-morphisms 
  \eqref{eq:proj-fml-cxs-with-L},
  \eqref{eq:proj-fml-cxs-right-with-L}
  and \eqref{eq:!-adjunction-sheafHom}
  are 2-isomorphisms.
\end{lemma}

\begin{proof}
  In order to show that \eqref{eq:!-adjunction-sheafHom} is a
  2-isomorphism we need to show that the evaluation 
  \begin{equation}
    \label{eq:!-adjunction-sheafHom-at-IJ}
    \alpha_*\sheafHom(J, \alpha^!_\mathcal{L} I)
    \ra 
    \sheafHom(\alpha_!^\mathcal{L} J, I).
  \end{equation}
  of \eqref{eq:!-adjunction-sheafHom} at 
  $J \in \ul\C(Y_A)$
  and $I \in \ul\C(X_A)$ is an isomorphism. 
  Let $F \in \ul\C(X_A)$ be componentwise flat.
  Then $\ul\C(F,-)$ applied to 
  \eqref{eq:!-adjunction-sheafHom-at-IJ} is an
  isomorphism
  because all arrows in
  \eqref{eq:derivation-!-adjunction-sheafHom}
  are isomorphisms, by
  Lemma~\ref{l:proj-fml-cxs-compwise-flat}.
  Let $v \colon V \subset Y$ be the embedding of an arbitrary
  open subset $V$ 
  of $Y$ and take $F=v_!(A_V)$. 
  This shows that  
  \eqref{eq:!-adjunction-sheafHom-at-IJ}
  evaluated at $V$ is an isomorphism.
  Hence \eqref{eq:!-adjunction-sheafHom} is a
  2-isomorphism.

  If $F \in \ul\C(X_A)$ is now an arbitrary object, we deduce that 
  all morphisms in
  \eqref{eq:derivation-!-adjunction-sheafHom}
  are in fact isomorphisms.
  Hence the Yoneda lemma implies that 
  \eqref{eq:proj-fml-cxs-right-with-L}
  is a 2-isomorphism.
  By symmetry, \eqref{eq:proj-fml-cxs-with-L} is also a
  2-isomorphism.  
  (The statement that 
  \eqref{eq:proj-fml-cxs-right-with-L} is a 2-isomorphism can
  also be proved directly: its evaluation 
  at a componentwise flat
  object $F$ and an arbitrary object $J$ is an
  isomorphism, and an arbitrary object $G \in \C(X_A)$ admits a
  short exact sequence $F' \ra F \ra G \ra 0$ in $\C(X_A)$ where
  $F$ and $F'$ have flat components.)
\end{proof}
  
\begin{proposition}
  \label{p:!-adjunction-sheafHom-ENH}
  Let $A$ be a
  $\kk$-algebra and
  $\alpha \colon Y \ra X$ a separated, locally proper
  morphism of 
  topological spaces with
  $\alpha_! \colon 
  \Mod(Y_A) \ra \Mod(X_A)$ 
  of finite cohomological dimension.
  Then the zig-zag
  \begin{multline}
    \label{eq:!-adjunction-sheafHom-tilde-ENH}
    \ul\alpha_*\ul\sheafHom(-, \ul\alpha^!(-)) 
    =
    \ii\alpha_*\ii\sheafHom(-, \alpha^!_\mathcal{L}(-))
    \xla{[\sim]}
    \ii\alpha_*\sheafHom(-, \alpha^!_\mathcal{L}(-))
    \\
    \xsira{\ii\eqref{eq:!-adjunction-sheafHom}}
    \ii\sheafHom(\alpha_!^\mathcal{L}(-), -)
    \xla{[\sim]}
    \ii\sheafHom(\ii\alpha_!^\mathcal{L}(-), -)
    =
    \ii\sheafHom(\ii\alpha_!(\mathcal{L} \otimes -), -)
    \\
    \xra{[\sim]}
    \ii\sheafHom(\ii\alpha_!(A_Y \otimes -), -)
    \xla{\sim}
    \ii\sheafHom(\ii\alpha_!(-), -)
    =
    \ul\sheafHom(\ul\alpha_!(-),-) 
  \end{multline}
  of objectwise homotopy equivalences in $\tildew{\ENH}_\kk$ defines
  a 2-isomorphism
  \begin{equation}
    \label{eq:!-adjunction-sheafHom-ENH}
    \ul\alpha_*\ul\sheafHom(-, \ul\alpha^!(-)) 
    \sira
    \ul\sheafHom(\ul\alpha_!(-),-) 
  \end{equation}
  in $\ENH_\kk$ that enhances the 2-isomorphism
  \begin{equation}
    \label{eq:!-adjunction-sheafHom-D}
    \dR\alpha_*\dR\sheafHom(-, \alpha^!(-)) 
    \sira
    \dR\sheafHom(\dR\alpha_!(-),-) 
  \end{equation}
  from \cite[Thm.~8.3.(d)]{wolfgang-olaf-locallyproper}.
\end{proposition}

\begin{proof}
  If we evaluate 
  \eqref{eq:!-adjunction-sheafHom-tilde-ENH} 
  at $J \in \ul\II(Y_A)$ and $I \in \ul\II(X_A)$   
  the first arrow is a
  homotopy equivalence by
  Propositions~\ref{p:spalt-5.14-sites-sheafHom},
  \ref{p:spalt-5.15-sites}
  and h-injectivity of 
  $\alpha^!_\mathcal{L} I$.
  The second arrow is 
  an isomorphism by 
  Lemma~\ref{l:!-adjunction-sheafHom-compwise-flat},
  the third arrow is obviously a homotopy equivalence. 
  The fourth arrow is a homotopy equivalence because $J
  \cong A_Y
  \otimes J \ra \mathcal{L} \otimes J$ is a quasi-isomorphism
  between componentwise $\alpha$-c-soft complexes, by
  Proposition~\ref{p:spalt-5.7-sites} and
  \cite[5.3, Lemma~7.4]{wolfgang-olaf-locallyproper}, and hence
  $\alpha_!(A_Y \otimes J) \ra \alpha_!(\mathcal{L} \otimes J)$
  is a quasi-isomorphism by
  \cite[Lemma~12.4.(b)]{wolfgang-olaf-locallyproper}. 
  The fifth arrow clearly is an isomorphism.

  For $F$, $I \in \C(X_A)$ and $J \in \C(Y_A)$, 
  all arrows in   
  \eqref{eq:derivation-!-adjunction-sheafHom} are isomorphisms, by
  Lemma~\ref{l:!-adjunction-sheafHom-compwise-flat}.
  Taking the 0-th cohomology of this sequence of isomorphisms
  gives isomorphisms 
  \begin{multline}
    \label{eq:derivation-!-adjunction-sheafHom-[]}
    [\ul\C_{X_A}](F, \alpha_*\sheafHom(J, \alpha^!_\mathcal{L} I))
    \sila
    [\ul\C_{Y_A}](\alpha\inv F, \sheafHom(J, \alpha^!_\mathcal{L} I))
    \\
    \sila
    [\ul\C_{Y_A}](\alpha\inv(F) \otimes J, \alpha^!_\mathcal{L} I)
    \sila
    [\ul\C_{X_A}](\alpha_!^\mathcal{L}(\alpha\inv(F) \otimes J), I)
    \\
    \xsira{\eqref{eq:proj-fml-cxs-right-with-L}_{(F,J)}^*}
    [\ul\C_{X_A}](F \otimes \alpha_!^\mathcal{L}(J), I)
    \sira
    [\ul\C_{X_A}](F, \sheafHom(\alpha_!^\mathcal{L}(J), I)).
  \end{multline}
  Now assume that $F$ is h-flat and that $I$ is h-injective.
  Then all objects in this sequence are spaces of morphisms in
  either 
  $[\ul\C(Y_A)]$ or
  $[\ul\C(X_A)]$ which have either h-injective target or
  h-flat source and weakly h-injective target:
  this follows from
  Propositions~\ref{p:pullback-preserves-h-flat},
  \ref{p:spalt-5.14-sites-sheafHom},
  \ref{p:spalt-5.15-sites}
  and the fact that 
  $\alpha^!_\mathcal{L}$ preserves h-injectives.
  Therefore, by
  Corollary~\ref{c:spalt-5.20-sites-new},
  all these morphism spaces map isomorphically to the corresponding
  morphism spaces in the derived categories
  $\D(Y_A)$ and $\D(X_A)$, respectively.
  We can moreover assume that $\dR\sheafHom$ (resp.\ $\otimes^\dL$)
  is computed naively 
  if its second argument is h-injective (resp.\ if one of its
  arguments is h-flat), that $\dL \alpha\inv$ is computed
  naively, that $\dR \alpha_*$ is computed naively if its
  argument is weakly h-injective, that $\alpha^!=\dR
  \alpha^!_\mathcal{L}$ is computed 
  naively if its argument is h-injective, and that
  $\dR
  \alpha_!=\alpha_!^\mathcal{L}$ 
  (by the proof of
  \cite[Thm.~8.3.(a)]{wolfgang-olaf-locallyproper}). 
  Hence we get a sequence
  of isomorphisms
  \begin{multline}
    \label{eq:derivation-!-adjunction-sheafHom-D}
    \D_{X_A}(F, \dR\alpha_*\dR\sheafHom(J, \alpha^! I))
    \sila
    \D_{Y_A}(\dL\alpha\inv F, \dR\sheafHom(J, \alpha^! I))
    \\
    \sila
    \D_{Y_A}(\dL \alpha\inv(F) \otimes^\dL J, \alpha^! I)
    \sila
    \D_{X_A}(\dR\alpha_!(\dL\alpha\inv(F) \otimes^\dL J), I)
    \\
    \sira
    \D_{X_A}(F \otimes^\dL \dR \alpha_!(J), I)
    \sira
    \D_{X_A}(F, \dR\sheafHom(\dR\alpha_!(J), I))
  \end{multline}
  which combines with
  \eqref{eq:derivation-!-adjunction-sheafHom-[]} to a commutative
  diagram in the obvious way.  Moreover, the isomorphisms in
  \eqref{eq:derivation-!-adjunction-sheafHom-D} are the
  isomorphisms obtained from the adjunctions
  $(\dL\alpha\inv, \dR\alpha_*)$ and $(\dR\alpha_!, \alpha^!)$,
  the $\otimes$-$\sheafHom$-adjunction, and the derived
  projection formula 
  \eqref{eq:projection-fml-D}.  Now recall that
  \eqref{eq:!-adjunction-sheafHom-D} is by definition the
  isomorphism giving rise to the composition in
  \eqref{eq:derivation-!-adjunction-sheafHom-D}.  
  The rest of the
  proof is left to the reader.
\end{proof}

Let $A$ be a $\kk$-algebra and $\alpha \colon Y \ra X$ a
separated, locally proper morphism of topological spaces
with $\alpha_! \colon \Mod(Y_A) \ra \Mod(X_A)$ 
of finite cohomological dimension.
Let $T \in \C(Y_A)$ and $F$, $J \in \C(X_A)$.  
The evaluation of the 2-isomorphism
\eqref{eq:proj-fml-cxs-with-L} (see Lemma~\ref{l:!-adjunction-sheafHom-compwise-flat}) at $(T,F)$,
the adjunctions  
$(\alpha_!^\mathcal{L}, \alpha^!_\mathcal{L})$ and the
$\otimes$-$\sheafHom$-adjunction yield isomorphisms 
\begin{multline}
  \label{eq:derivation-upper-!-sheafHom}
  \ul\C_{Y_A}(T, \sheafHom(\alpha\inv F, \alpha^!_\mathcal{L}J))
  \sila
  \ul\C_{Y_A}(T \otimes \alpha\inv F, \alpha^!_\mathcal{L}J)
  \\
  \sila
  \ul\C_{X_A}(\alpha_!^\mathcal{L} (T \otimes \alpha\inv F), J)
  \xsira{\eqref{eq:proj-fml-cxs-with-L}_{(T,F)}^*}
  \ul\C_{X_A}(\alpha_!^\mathcal{L} (T) \otimes F, J)
  \\
  \sira
  \ul\C_{X_A}(\alpha_!^\mathcal{L} (T), \sheafHom(F, J))
  \sira
  \ul\C_{Y_A}(T, \alpha^!_\mathcal{L}\sheafHom(F, J))
\end{multline}
in $\KK$. Hence the Yoneda lemma (or, more
concretely, evaluating at $T=v_!(A_V)$ for $v \colon V \ra Y$
the embedding of an open subset $V$ of $Y$, for varying $V$)
and naturality in $F$ and $J$ yield a 
2-isomorphism 
\begin{equation}
  \label{eq:upper-!-sheafHom}
  \sheafHom(\alpha\inv(-), \alpha^!_\mathcal{L}(-))
  \sira \alpha^!_\mathcal{L}\sheafHom(-, -)
\end{equation}
between 1-morphisms $(\ul\C(X_A),\ul\C(X_A)) \ra \ul\C(Y_A)$ in
$\DGCAT_\kk$. 

\begin{proposition}
  \label{p:upper-!-sheafHom-ENH}
  Let $A$ be a
  $\kk$-algebra and
  $\alpha \colon Y \ra X$ a separated, locally proper
  morphism of 
  topological spaces with
  $\alpha_! \colon 
  \Mod(Y_A) \ra \Mod(X_A)$ 
  of finite cohomological dimension.
  Then the zig-zag
  \begin{multline}
    \label{eq:upper-!-sheafHom-tilde-ENH}
    \ul\sheafHom(\ul\alpha\inv(-), \ul\alpha^!(-)) 
    =
    \ii \sheafHom(\ii\alpha\inv(-), \alpha^!_\mathcal{L}(-)) 
    \xra{[\sim]}
    \ii \sheafHom(\alpha\inv(-), \alpha^!_\mathcal{L}(-)) 
    \\
    \xra{[\sim]}
    \ii \sheafHom(\alpha\inv\ee(-), \alpha^!_\mathcal{L}(-)) 
    \xsira{\ii\eqref{eq:upper-!-sheafHom}(\ee(-), -)}
    \ii \alpha^!_\mathcal{L} \sheafHom(\ee(-),-) 
    \\
    \xra{[\sim]}
    \ii \alpha^!_\mathcal{L}\ii\sheafHom(\ee(-),-) 
    \xla{[\sim]}
    \alpha^!_\mathcal{L}\ii\sheafHom(\ee(-),-) 
    \\
    \xla{[\sim]}
    \alpha^!_\mathcal{L}\ii\sheafHom(-,-) 
    =
    \ul\alpha^!\ul\sheafHom(-,-) 
  \end{multline}
  of objectwise homotopy equivalences in $\tildew{\ENH}_\kk$ defines
  a 2-isomorphism
  \begin{equation}
    \label{eq:upper-!-sheafHom-ENH}
    \ul\sheafHom(\ul\alpha\inv(-), \ul\alpha^!(-)) 
    \sira
    \ul\alpha^!\ul\sheafHom(-,-) 
  \end{equation}
  in $\ENH_\kk$ that enhances the 2-isomorphism
  \begin{equation}
    \label{eq:upper-!-sheafHom-D}
     \dR\sheafHom(\dL\alpha\inv(-), \alpha^!(-))
    \sira
    \alpha^!\dR\sheafHom(-,-) 
  \end{equation}
  from \cite[Thm.~8.3.(d)]{wolfgang-olaf-locallyproper}.
\end{proposition}

\begin{proof}
  If we evaluate \eqref{eq:upper-!-sheafHom-tilde-ENH} at $I$, $J
  \in \ul\II(X_A)$  
  the first two arrows are
  homotopy equivalences because $\alpha^!_\mathcal{L}$ preserves
  h-injectives (and $\II$-fibrant objects), the third arrow is an
  isomorphism because we have seen above that 
  \eqref{eq:upper-!-sheafHom} is an isomorphism,
  the fourth arrow
  is a
  homotopy 
  equivalence 
  because $\sheafHom(\ee I, J)$ is h-injective by
  Proposition~\ref{p:spalt-5.14-sites-sheafHom},
  the fifth arrow is a homotopy equivalence because
  $\alpha^!_\mathcal{L}$ preserves h-injectives
  and the sixth arrow is obviously a homotopy
  equivalence. Hence we get the 2-isomorphism
  \eqref{eq:upper-!-sheafHom-ENH}.

  For $T \in \C(Y_A)$ and $F$, $J \in \C(X_A)$, 
  consider the sequence 
  \eqref{eq:derivation-upper-!-sheafHom} of isomorphisms.
  Taking the 0-th cohomology of this sequence gives isomorphisms
  \begin{multline}
    \label{eq:derivation-upper-!-sheafHom-[]}
    [\ul\C_{Y_A}](T, \sheafHom(\alpha\inv F, \alpha^!_\mathcal{L}J))
    \sila
    [\ul\C_{Y_A}](T \otimes \alpha\inv F, \alpha^!_\mathcal{L}J)
    \\
    \sila
    [\ul\C_{X_A}](\alpha_!^\mathcal{L} (T \otimes \alpha\inv F), J)
    \xsira{\eqref{eq:proj-fml-cxs-with-L}_{(T,F)}^*}
    [\ul\C_{X_A}](\alpha_!^\mathcal{L} (T) \otimes F, J)
    \\
    \sira
    [\ul\C_{X_A}](\alpha_!^\mathcal{L} (T), \sheafHom(F, J))
    \sira
    [\ul\C_{Y_A}](T, \alpha^!_\mathcal{L}\sheafHom(F, J)).
  \end{multline}
  Now assume that
  $F$ is h-flat and that $J$ is 
  h-injective. 
  Then all objects in this sequence are spaces of morphisms in
  either 
  $[\ul\C(Y_A)]$ or
  $[\ul\C(X_A)]$ whose target objects are h-injective:
  this follows from
  Propositions~\ref{p:pullback-preserves-h-flat},
  \ref{p:spalt-5.14-sites-sheafHom} and the fact that
  $\alpha^!_\mathcal{L}$ preserves h-injectives.  Therefore, all
  these morphism spaces map isomorphically to the corresponding
  morphism spaces 
  in the derived categories $\D(Y_A)$ and $\D(X_A)$,
  respectively. We can moreover assume that $\dR\sheafHom$
  (resp.\ $\otimes^\dL$) is computed naively if its second
  argument is h-injective (resp.\ if one of its arguments is
  h-flat), that $\dL \alpha\inv$ is computed naively, that
  $\alpha^!=\dR \alpha^!_\mathcal{L}$ is computed naively if its
  argument is h-injective, and that
  $\dR \alpha_!=\alpha_!^\mathcal{L}$.  Hence we get a sequence
  of isomorphisms
  \begin{multline}
    \label{eq:derivation-upper-!-sheafHom-D}
    \D_{Y_A}(T, \dR\sheafHom(\dL\alpha\inv F, \alpha^! J))
    \sila
    \D_{Y_A}(T \otimes^\dL \dL\alpha\inv F, \alpha^! J)
    \\
    \sila
    \D_{X_A}(\dR\alpha_! (T \otimes^\dL \dL\alpha\inv F), J)
    \sira
    \D_{X_A}((\dR \alpha_! T) \otimes^\dL F, J)
    \\
    \sira
    \D_{X_A}(\dR \alpha_! T, \dR\sheafHom(F, J))
    \sira
    \D_{Y_A}(T, \alpha^!\dR\sheafHom(F, J))
  \end{multline}
  which combines with 
  \eqref{eq:derivation-upper-!-sheafHom-[]} 
  to a commutative diagram in the obvious way.
  Moreover, the isomorphisms in
  \eqref{eq:derivation-upper-!-sheafHom-D} are the
  isomorphisms obtained from
  the $\otimes$-$\sheafHom$-adjunction,
  the adjunction $(\dR\alpha_!,
  \alpha^!)$, 
  and the 
  derived projection formula 
  \eqref{eq:projection-fml-D}.
  Now recall that
  \eqref{eq:upper-!-sheafHom-D} is by definition the isomorphism
  giving rise to the composition in 
  \eqref{eq:derivation-upper-!-sheafHom-D}.
  The rest of the proof is left to the reader.
\end{proof}

\subsubsection{Lifts of commutative diagrams}
\label{sec:lifts-comm-diagr-1}

\begin{proposition}
  \label{p:adjunction-!-push-pull-in-ENH}
  Let $A$ be a
  $\kk$-algebra and
  $\alpha \colon Y \ra X$ a separated, locally proper
  morphism of 
  topological spaces with
  $\alpha_! \colon 
  \Mod(Y_A) \ra \Mod(X_A)$ 
  of finite cohomological dimension.
  Then the 
  two 1-morphisms $\ul\alpha_!$ and $\ul\alpha^!$ and the
  two 2-morphisms 
  \eqref{eq:id-a^!a_!}
  and
  \eqref{eq:a_!a^!-id}
  form an adjunction in $\ENH_\kk$, i.\,e.\ the two diagrams in
  $\ENH_\kk$ in
  \eqref{eq:intro:alpha!-triangles}
  commute.
\end{proposition}

\begin{proof}
  The proof is similar to that of 
  Proposition~\ref{p:adjunction-*-pull-push-in-ENH} and left to
  the reader. 
\end{proof}

\fussnote{
  relation to 
  Proposition~\eqref{p:!-adjunction-sheafHom-ENH}?
}

\subsubsection{Some other lifts}
\label{sec:some-other-lifts-1}

\begin{lemma}
  \label{l:ul-G-and-functors-!}
  Let $A$ be a
  $\kk$-algebra and
  $\alpha \colon Y \ra X$ a 
  morphism of 
  topological spaces.
  Let $E \in \C(X_A)$ and $F \in \C(Y_A)$.
  The 2-morphisms
  \begin{align}
    \ul\alpha\inv \ul E
    & =
      \ii \alpha\inv \ii E 
      \xla{[\sim]}
      \ii \alpha\inv E
      =
      \ul{\alpha\inv E},\\
    \ul{\alpha_! F}
    & = \ii \alpha_! F
      \ra \ii\alpha_! \ii F = \ul\alpha_! \ul F
  \end{align}
  in $\tildew{\ENH}_\kk$
  define 2-morphisms
  \begin{align}
    \label{eq:alpha-inv-ul-obj}
    \ul\alpha\inv \ul E
    & \sira
      \ul{\alpha\inv E} 
    & \text{(always a 2-isomorphism)},\\
    \label{eq:alpha_!-ul-obj}
    \ul{\alpha_! F}
    & \ra \ul\alpha_! \ul F
    & \text{(2-isomorphism if $F$ is h-inj.)}
  \end{align}
  in $\ENH_\kk$ that enhance the obvious 2-morphisms
  $\dL\alpha\inv(E) \sira \alpha\inv E$ and
  $\alpha_! F \ra \dR\alpha_!(F)$.
\end{lemma}

Note that $\alpha^!$ does not appear here because it is usually
not viewed as a derived functor.

\begin{proof}
  Left to the reader. 
\end{proof}

\begin{lemma}
  \label{l:ul-object-2-functors}
  Let $A$ be a
  $\kk$-algebra and
  $\alpha \colon Y \ra X$ a 
  morphism of 
  topological spaces.
  Let $e \colon E \ra E'$ in $\C(X_A)$ and $f \colon F \ra F'$ in
  $\C(Y_A)$ be morphisms.
  Then the following diagrams in $\ENH_\kk$ are commutative.
  \begin{equation}
    \label{eq:ul-alpha-2-comm}
    \xymatrix{
      {\ul{\alpha}\inv\ul{E}} 
      \ar[d]_-{\ul{\alpha}\inv\ul{e}}
      \ar[r]^-{\eqref{eq:alpha-inv-ul-obj}}_-{\sim}
      & 
      {\ul{\alpha\inv E}} 
      \ar[d]_-{\ul{\alpha\inv e}}
      \\
      {\ul{\alpha}\inv\ul{E'}} 
      \ar[r]^-{\eqref{eq:alpha-inv-ul-obj}}_-{\sim}
      & 
      {\ul{\alpha\inv E'}} 
    }
    \quad\quad\quad
    \xymatrix{
      {\ul{\alpha_!G}} 
      \ar[r]^-{\eqref{eq:alpha_!-ul-obj}}
      \ar[d]_-{\ul{\alpha_!g}}
      & 
      {\ul{\alpha}_!\ul{G}}
      \ar[d]_-{\ul\alpha_!\ul{g}}
      \\
      {\ul{\alpha_!G'}} 
      \ar[r]^-{\eqref{eq:alpha_!-ul-obj}}
      & 
      {\ul{\alpha}_!\ul{G}'}
    }
  \end{equation}
\end{lemma}

\begin{proof}
  Left to the reader. 
\end{proof}

\subsubsection{Subsequently constructed lifts of
  2-(iso)morphisms}
\label{sec:subs-constr-lifts-1}

As in \ref{sec:subs-constr-lifts} we deduce some consequences.

\begin{lemma}
  \label{l:c:proper-base-change}
  Under the assumptions of Proposition~\ref{p:proper-base-change}
  the composition
  \begin{equation}
    \label{eq:alpha_!beta'_*-beta_*alpha'_!-ENH}
    \ul\alpha_!\ul\beta'_* 
    \xlongra{\eqref{eq:id-aainv}\ul\alpha_!\ul\beta'_*}
    \ul\beta_*\ul\beta\inv\ul\alpha_!\ul\beta'_* 
    \xlongsira{\ul\beta_*\eqref{eq:proper-base-change-ENH}\ul\beta'_*}
    \ul\beta_*\ul\alpha'_!\ul{\beta}'^{-1}\ul\beta'_* 
    \xlongra{\ul\beta_*\ul\alpha'_!\eqref{eq:ainva-id}}
    \ul\beta_*\ul\alpha'_!
  \end{equation}
  in $\ENH_\kk$ enhances the 2-morphism
  \begin{equation}
    \label{eq:alpha_!beta'_*-beta_*alpha'_!-D}
    \dR\alpha_!\dR\beta'_* \ra \dR\beta_*\dR\alpha'_!  
  \end{equation}
  which is adjoint to the composition
  $\dL\beta\inv \dR\alpha_! \dR \beta'_* 
  \xsira{\eqref{eq:proper-base-change-D}\dR \beta'_*}
  \dR\alpha'_! \dL\beta'^{-1} \dR \beta'_* \ra \dR\alpha'_!$. 
\end{lemma}

\begin{proof}
  The composition
  $\ul\beta\inv\ul\alpha_!\ul\beta'_* 
  \xlongra{\eqref{eq:proper-base-change-ENH}\ul\beta'_*}
  \ul\alpha'_!\ul\beta'^{-1}\ul\beta'_* 
  \xlongra{\ul\alpha'_!\eqref{eq:ainva-id}}
  \ul\alpha'_!$ enhances the last composition in the formulation
  of the lemma, by 
  Lemma~\ref{l:adjunction-morphisms-push-pullinv},
  Proposition~\ref{p:proper-base-change} and
  Remark~\ref{rem:enhance-2-morphism-compose-and-invert}.
  This implies the lemma.
\end{proof}

\begin{proposition}
  \label{p:proper-base-change-upper-shriek}
  Under the assumptions of Proposition~\ref{p:proper-base-change}
  the composition
  \begin{multline}
    \label{eq:proper-base-change-upper-shriek-ENH-def}
    \ul{\beta}'_*\ul\alpha'^! 
    \xra{\eqref{eq:id-a^!a_!} \ul{\beta}'_*\ul\alpha'^!} 
    \ul\alpha^!\ul\alpha_! \ul{\beta}'_*\ul\alpha'^! 
    \xra{\ul\alpha^! \eqref{eq:id-aainv} \ul\alpha_!
      \ul{\beta}'_*\ul\alpha'^!} 
    \ul\alpha^!\ul\beta_*\ul\beta\inv
    \ul\alpha_!\ul{\beta}'_*\ul\alpha'^!  
    \\
    \xra[\sim]{\ul\alpha^!\ul\beta_*
      \eqref{eq:proper-base-change-ENH} 
      \ul{\beta}'_*\ul\alpha'^!}  
    \ul\alpha^!\ul\beta_*\ul\alpha'_!\ul\beta'^{-1}
    \ul{\beta}'_*\ul\alpha'^! 
    \xra{\ul\alpha^!\ul\beta_* \ul\alpha'_! \eqref{eq:ainva-id}
      \ul\alpha'^!}  
    \ul\alpha^!\ul\beta_* \ul\alpha'_! \ul\alpha'^!
    \xra{\ul\alpha^!\ul\beta_* \eqref{eq:a_!a^!-id}} 
    \ul\alpha^!\ul\beta_*
  \end{multline}
  in $\ENH_\kk$ is a 2-isomorphism 
  \begin{equation}
    \label{eq:proper-base-change-upper-shriek-ENH}
    \ul{\beta}'_*\ul\alpha'^! 
    \sira
    \ul\alpha^!\ul\beta_* 
  \end{equation}
  that enhances the 2-isomorphism 
  \begin{equation}
    \label{eq:proper-base-change-upper-shriek-D}
    (\dR{\beta}'_*) \alpha'^! 
    \sira
    \alpha^! \dR\beta_*
  \end{equation}
  from
  \cite[Thm.~8.3.(c)]{wolfgang-olaf-locallyproper}. 
\end{proposition}

\begin{remark}
  Note that
  \eqref{eq:proper-base-change-upper-shriek-ENH}
  is a 2-isomorphism even though, in general, not all 2-morphisms
  in \eqref{eq:proper-base-change-upper-shriek-ENH-def}
  are objectwise homotopy equivalences.
\end{remark}

\begin{proof}
  We claim that the 2-isomorphism $\ul{\beta}\inv \ul\alpha_! 
  \sira
  \ul\alpha'_! \ul\beta'^{-1}$
  in \eqref{eq:proper-base-change-ENH} has 
  \eqref{eq:proper-base-change-upper-shriek-ENH}
  as its conjugate 2-morphism, and hence
  \eqref{eq:proper-base-change-upper-shriek-ENH}
  is a 2-isomorphism.
  To see the claim, specialize
  Remarks~\ref{rem:morphism-between-left-adjoints-2-categorical}  
  and \ref{rem:compose-adjunctions} below
  to the case $L= L_1L_2=\ul\beta\inv
  \ul\alpha_!$ and
  $L'= L'_1L'_2=\ul\alpha'_! \ul\beta'^{-1}$ and use 
  the adjunctions
  $(\ul\alpha_!, \ul\alpha^!)$,
  $(\ul\alpha'_!, \ul\alpha'^!)$,
  $(\ul{\beta}\inv, \ul\beta_*)$ and
  $(\ul{\beta}'^{-1}, \ul\beta'_*)$ in $\ENH_\kk$
  (see Proposition~\ref{p:adjunction-!-push-pull-in-ENH}
  and Lemma~\ref{l:adjunction-morphisms-push-pullinv}).

  The 2-isomorphism \eqref{eq:proper-base-change-upper-shriek-D}
  is by definition the 2-morphism that is conjugate to the
  2-isomorphism
  $\dL\beta\inv \dR\alpha_! 
  \sira 
  \dR\alpha'_! \dL\beta'^{-1}$
  from
  \eqref{eq:proper-base-change-D}. By the above argument, now
  using $L= L_1L_2=\dL\beta\inv
  \dR\alpha_!$ and
  $L'= L'_1L'_2=\dR\alpha'_! \dL\beta'^{-1}$, 
  it is equal to the composition
  \begin{multline}
    (\dR\beta'_*)\alpha'^! 
    \xra{\eta_2 (\dR\beta'_*)\alpha'^!} 
    \alpha^!\dR \alpha_! (\dR\beta'_*)\alpha'^! 
    \xra{\alpha^! \eta_1 \dR \alpha_! (\dR\beta'_*)\alpha'^!}
    \alpha^!\dR \beta_*\dL \beta\inv \dR
    \alpha_!(\dR\beta'_*)\alpha'^! 
    \\
    \xra{\alpha^!\dR \beta_* \eqref{eq:proper-base-change-D}
      (\dR\beta'_*)\alpha'^!}  
    \alpha^!\dR
    \beta_*\dR\alpha'_!\dL\beta'^{-1}(\dR\beta'_*)\alpha'^! 
    \\
    \xra{\alpha^!\dR \beta_* \dR\alpha'_!\theta'_2 \alpha'^!} 
    \alpha^!\dR \beta_* (\dR\alpha'_!) \alpha'^!
    \xra{\alpha^!\dR \beta_*\theta'_1} 
    \alpha^!\dR \beta_* 
  \end{multline}
  where $\eta_1$, $\eta_2$ and
  $\theta'_1$, $\theta'_2$ denote
  units and counits of the obvious adjunctions.

  Therefore, Lemma~\ref{l:adjunction-morphisms-push-pullinv},
  Propositions~\ref{p:adjunction-morphisms-proper-push-pull},
  \ref{p:proper-base-change},
  and
  Remark~\ref{rem:enhance-2-morphism-compose-and-invert}
  imply that the composition in
  \eqref{eq:proper-base-change-upper-shriek-ENH-def}
  enhances the 2-isomorphism
  \eqref{eq:proper-base-change-upper-shriek-D}. 
\end{proof}

\fussnote{
  Let us combine these two observations. Assume that
  $(L, R, \eta, \theta)$ is
  as in 
  observation 2 and that $(L'=L'_1L'_2, R'=R'_2 R'_1, \eta',
  \theta')$ is given similarly. Then 
  \eqref{eq:rho-adjunction} is given by 
  the composition
  \begin{multline}
    R'_2R'_1 
    \xra{\eta_2 R'_2R'_1} 
    R_2L_2 R'_2R'_1 
    \xra{R_2 \eta_1 L_2 R'_2R'_1}
    R_2R_1L_1L_2R'_2R'_1
    \\
    \xra{R_2R_1\lambda R'_2R'_1} 
    R_2R_1L'_1L'_2R'_2R'_1
    \xra{R_2R_1 L'_1\theta'_2 R'_1} 
    R_2R_1 L'_1 R'_1
    \xra{R_2R_1\theta'_1} 
    R_2R_1. 
  \end{multline}
  
  Now specialize to the case $L= L_1L_2=\dL\beta\inv
  \dR\alpha_!$ and 
  $L'= \dR\alpha'_! \dL\beta'^{-1}$ and let $\lambda \colon L \ra
  L'$ be given by 
  \eqref{eq:proper-base-change-D}.
}

\begin{remark}
  \label{rem:morphism-between-left-adjoints}
  Assume that $(L, R, \eta, \theta)$ and $(L', R', \eta', \theta')$
  are 
  adjunctions between functors $L,L' \colon \mathcal{A} \ra
  \mathcal{C}$ and $R, R' \colon \mathcal{C} \ra \mathcal{A}$,
  respectively. Then any transformation $\lambda \colon L \ra L'$
  gives rise to a unique transformation $\rho \colon R' \ra R$
  such that 
  \begin{equation}
    \label{eq:conjugate-maps}
    \xymatrix{
      {\mathcal{C}(L'A, C)} 
      \ar[r]^-{(\lambda_A)^*} 
      \ar[d]^-{\sim}
      &
      {\mathcal{C}(LA, C)} 
      \ar[d]^-{\sim}
      \\
      {\mathcal{A}(A, R'C)} 
      \ar[r]^-{(\rho_C)_*} 
      &
      {\mathcal{A}(A, RC)} 
    }
  \end{equation}
  commutes for all $A \in \mathcal{A}$ and $C \in \mathcal{C}$,
  and conversely.
  Then $\rho$ and $\lambda$ are conjugate in the terminology 
  of 
  \cite[IV.7]{maclane-working-mathematician},
  \cite[3.3]{lipman-notes-on-derived-functors-and-GD}.
  Moreover, $\lambda$ is an isotransformation if and only if $\rho$
  is an isotransformation
  (cf.\
  Remark~\ref{rem:morphism-between-left-adjoints-2-categorical}
  below). 
  Similar statements hold for adjunctions between
  $\KK$-functors. In fact, there is an abstract 2-categorical
  statement explained in the following
  Remark~\ref{rem:morphism-between-left-adjoints-2-categorical}. 
\end{remark}

\begin{remark}
  \label{rem:morphism-between-left-adjoints-2-categorical} 
  This remark is a variant of 
  Remark~\ref{rem:morphism-between-left-adjoints}. 
  Assume that $(L, R, \eta,
  \theta)$ and $(L', R', \eta', \theta')$ 
  are 
  adjunctions in a 2-category where $L$ and $L'$ are 1-morphisms
  $\mathcal{A} \ra \mathcal{C}$. 
  Then any 2-morphism $\lambda \colon L \ra L'$ gives rise to a
  conjugate 2-morphism $\rho \colon R' \ra R$ defined as the
  composition
  \begin{equation}
    \label{eq:rho-adjunction}
    R' \xra{\eta R'} RLR' \xra{R\lambda R'} RL'R' \xra{R\theta'} R
  \end{equation}
  which is illustrated by the following diagram.
  \begin{equation}
    \label{eq:rho-adjunction-diagram}
    \begin{tikzpicture}[baseline=(current bounding box.center),
      description/.style={fill=white,inner sep=2pt},
      natural/.style={double,double equal sign distance,-implies} ]
      
      \matrix (m) [matrix of math nodes, row sep=3em,
      column sep=2em, text height=1.5ex, text depth=0.25ex]
      {          & \mathcal{C} && \mathcal{C} && \mathcal{C}\\
        \mathcal{A} && \mathcal{A} && \mathcal{A}\\};
      \path[->,font=\scriptsize]
      (m-1-4) edge node[above]{$\id$} node[pos=0.18, below](end){} (m-1-2)
      (m-2-5) edge node[below]{$\id$} node[pos=0.82, above](start){} (m-2-3)
      (m-2-3) edge node[right] {$L$} (m-1-2)
      (m-2-5) edge node[left] {$L'$} (m-1-4)
      (m-1-2) edge node[left] {$R$} (m-2-1)
      (m-1-6) edge node[right] {$R'$} (m-2-5)

      (m-2-3) edge node[below]{$\id$} node[above](gleichD'){} (m-2-1)
      (m-1-6) edge node[above]{$\id$} node[below](gleichC){} (m-1-4);
            
      \path[->,font=\scriptsize]
      (start) edge[natural]
      node[right] {$\lambda$} (end)
      (gleichD') edge[natural,shorten >=0.2em]
      node[pos=0.4, right] {$\eta$} (m-1-2) 
      (m-2-5) edge[natural,shorten <=0.2em]
      node[pos=0.6, right] {$\theta'$} (gleichC); 
    \end{tikzpicture}
  \end{equation}
  This construction defines a bijection between the set of
  2-morphisms $L \ra L'$ and the set of 2-morphisms $R' \ra R$.
  
  Let $(L'', R'', \eta'', \theta'')$ 
  be another 
  adjunction where $L'' \colon \mathcal{A} \ra \mathcal{C}$ is a
  1-morphism. Then any 2-morphism $\lambda' \colon L' \ra L''$
  similarly has a conjugate 2-morphism $\rho' \colon R''
  \ra R'$. 
  Pasting together the diagram~\eqref{eq:rho-adjunction-diagram}
  and the corresponding
  diagram defining $\rho'$, the triangle identities
  easily imply that the
  2-morphism conjugate to $\lambda' \circ \lambda \colon L \ra
  L''$ is $\rho \circ \rho'$.
  In particular, $\lambda$ is a 2-isomorphism if and only if
  $\rho$ is a 2-isomorphism.
\end{remark}

\begin{remark}
  \label{rem:compose-adjunctions}
  Assume that $(L_1, R_1, \eta_1, \theta_1)$ and $(L_2, R_2,
  \eta_2, \theta_2)$ 
  are  
  adjunctions between 1-morphisms
  $\mathcal{A} \xrightleftarrows[\tiny{$L_2$}][\tiny{$R_2$}]
  \mathcal{B} 
  \xrightleftarrows[\tiny{$L_1$}][\tiny{$R_1$}] \mathcal{C}$ in a
  2-category.
  Then $(L=L_1L_2, R=R_2R_1)$ together with 
  \begin{align}
    \eta \colon 
    & \id 
    \xra{\eta_2} R_2L_2 \xra{R_2 \eta_1 L_2} R_2 R_1 L_1L_2,\\
    \theta \colon
    & L_1L_2R_2R_1 
    \xra{L_1\theta_2 R_1} L_1 R_1 \xra{\theta_1} \id        
  \end{align}
  as unit and counit form an adjunction.
\end{remark}

\begin{lemma}
  \label{l:!-adjunction-Hom-ENH}
  Let $A$ be a
  $\kk$-algebra and
  $\alpha \colon Y \ra X$ a separated, locally proper
  morphism of 
  topological spaces with
  $\alpha_! \colon 
  \Mod(Y_A) \ra \Mod(X_A)$ 
  of finite cohomological dimension.
  Then the composition
  \begin{multline}
    \label{eq:!-adjunction-Hom-ENH}
    \ul\Hom(-, \ul\alpha^!(-))
    \xsira{\eqref{eq:46}(\id, \ul\alpha^!)}
    \ul\Gamma\;\ul\sheafHom(-, \ul\alpha^!(-))
    \\
    \xsira{\eqref{eq:alphabeta_*-ENH}\ul\sheafHom(-, \ul\alpha^!(-))}
    \ul\Gamma\;\ul\alpha_*\ul\sheafHom(-, \ul\alpha^!(-))
    \xsira{\ul\Gamma \eqref{eq:!-adjunction-sheafHom-ENH}}
    \ul\Gamma\;\ul\sheafHom(\ul\alpha_!(-), -)
    \\
    \xsira{\eqref{eq:46}\inv(\ul\alpha_!(-), \id)}
    \ul\Hom(\ul\alpha_!-,-) 
  \end{multline}
  of 2-isomorphisms in $\ENH_\kk$ enhances
  $\dR\!\Hom(-, \alpha^!(-))
  \sira \dR\!\Hom(\dR\alpha_!(-),-)$.
\end{lemma}

The adjunction
isomorphisms $\D_{Y_A}(G,\alpha^!E) \cong \D_{X_A}(\dR
\alpha_!(G), E)$ 
can be obtained as in Remark~\ref{rem:enhanced-pull-push-Hom}.

\begin{proof}
  Use 
  Remark~\ref{rem:enhance-2-morphism-compose-and-invert},
  Lemmas~\ref{l:composition-inverse-direct-image}, 
  \ref{l:ulGamma-ulsheafHom}, and
  Proposition~\ref{p:!-adjunction-sheafHom-ENH}.
\end{proof}

\section{The 2-multicategory of formulas}
\label{sec:2-multicat-formulas}

We provide a language which allows a compact formulation of
almost all main results of 
this article,
see Theorems~\ref{t:interprete-FML-ENH} and
\ref{t:compare-interpretations-FML}
and Remark~\ref{rem:comm-diagrams}.
All ringed sites, ringed spaces and algebras in this section are
assumed to be $\mathscr{U}$-small; all algebras are commutative. 

Let $\kk$ be a field.
Recall the category $\fml_\kk$ from \ref{sec:categ-form-six}.
Let $\tildew{\FML}_\kk$ be the free $\kk$-linear 2-multicategory
which has the same objects as $\fml_\kk$, whose generating
1-morphisms are the generating morphisms of 
$\fml_\kk$, and whose generating
2-morphisms are 
the 27 morphisms in the left column of 
table~\ref{tab:main} on page \pageref{tab:main}
and the following generating 2-morphisms, and their opposites
(cf.\ Remark~\ref{rem:involution-FML} below).
\begin{enumerate}
\item 
  for each 
  $\kk$-ringed site $(\mathcal{X},
  \mathcal{O})$ and each morphism $g \colon G \ra G'$ in
  $\C(\mathcal{X})$ there is a 2-morphism
  $\ud{g} \colon \ud{G} \ra \ud{G'}$;
\item 
  for each $\kk$-ringed site $(\mathcal{X}, \mathcal{O})$ and
  each pair $E, F \in \C(\mathcal{X})$ of objects there are
  2-morphisms $\ud E \;\ud\otimes\; \ud F \ra \ud{E \otimes F}$,
  $\ud{\sheafHom(E,F)} \ra \ud\sheafHom(\ud E, \ud F)$ and
  $\oldud{\ul\C(E,F)} \ra \ud\Hom(\ud E, \ud F)$;
\item 
  for each 
  morphism
  $\alpha \colon (\Sh(\mathcal{Y}), \mathcal{O}_\mathcal{Y}) \ra
  (\Sh(\mathcal{X}), \mathcal{O}_\mathcal{X})$
  of $\kk$-ringed topoi and each object
  $E \in \C(\mathcal{X})$ (resp.\ $F \in \C(\mathcal{Y})$) there
  is a 2-morphism
  $\ud\alpha^* \ud E  \ra \ud{(\alpha^* E)}$ 
  (resp.\ $\ud{(\alpha_* F)} \ra \ud\alpha_* \ud F$);
\item 
  for each 
  $\kk$-algebra $A$ and each 
  morphism 
  $\alpha \colon Y \ra X$ of 
  topological spaces and each object
  $E \in \C(X_A)$ (resp.\ $F \in \C(Y_A)$) there is a
  2-morphism
  $\ud\alpha\inv \ud E \ra \ud{(\alpha\inv E)}$
  (resp.\ $\ud{(\alpha_! F)} \ra \ud\alpha_! \ud F$).
\end{enumerate}

\fussnote{
 (cf.\ \cite[Sect.~1]{mimram-rewriting}).
}

The $\kk$-linear 2-multicategory $\FML_\kk$ is obtained from
$\tildew{\FML}_\kk$ as follows. We require the following 
generating 2-morphisms and their opposites to be invertible:
all 2-morphisms labeled $\sim$ in the left column of
table~\ref{tab:main};
$\ud{g} \colon \ud G \ra \ud G'$ if $g$ is a quasi-isomorphism;
$\ud E \;\ud\otimes\; \ud F \ra \ud{E \otimes F}$ if $E$ or $F$
is h-flat;
$\ud{\sheafHom(E,F)} \ra \ud\sheafHom(\ud E, \ud F)$ and
$\oldud{\ul\C(E,F)} \ra \ud\Hom(\ud E, \ud F)$ if $F$ is
h-injective or if $E$ is h-flat and $F$ is weakly h-injective;
$\ud\alpha^* \ud E \ra \ud{(\alpha^* E)}$ if $E$ is h-flat;
$\ud{(\alpha_* F)} \ra \ud\alpha_* \ud F$ if $F$ is weakly
h-injective; 
$\ud\alpha\inv \ud E \ra \ud{(\alpha\inv E)}$;
$\ud{(\alpha_! F)} \ra \ud\alpha_! \ud F$ if $F$ is h-injective.
Moreover, we impose the relations $\ud{f} \;\ud{g}=\ud{(fg)}$ for
all composable morphisms $f$ and $g$ in $\C(\mathcal{X})$, and
$\ud{rg+r'g'}=r\ud{g}+r'\ud{g'}$ for all morphisms
$g, g' \colon G \ra G'$ in $\C(\mathcal{X})$ and all scalars
$r, r' \in \kk$. 

These conditions imply that
$\ud{\id}_G=\id_{\ud{G}}$ in $\FML_\kk$, for $G \in
\C(\mathcal{X})$, and that 
two morphisms $g$ and $g'$ in $\C(\mathcal{X})$ which
become equal in $\D(\mathcal{X})$ give rise to the same
2-morphism $\ud{g}=\ud{g}'$ in $\FML_\kk$.

\begin{remark}
  \label{rem:involution-FML}
  The involution $(-)^\opp$ on $\fml_\kk$ from
  \ref{sec:categ-form-six} 
  extends to
  $\tildew{\FML}_\kk$ and 
  $\FML_\kk$ in the obvious way by changing the direction of the
  2-morphisms. For example,
  the 2-morphism $\tau \colon \ud{\alpha}^*\ud{\alpha}_* \ra \id
  \colon 
  \ud{\mathcal{Y}} \ra \ud{\mathcal{Y}}$ is mapped to the
  2-morphism
  $\tau^\opp \colon \id \ra
  (\ud{\alpha}^*)^\opp(\ud{\alpha}_*)^\opp \colon
  \ud{\mathcal{Y}}^\opp \ra \ud{\mathcal{Y}}^\opp$.
  More formally, the involution is a functor $(-)^\opp \colon
  \FML_\kk 
  \ra \FML_\kk^\co$ where $\FML_\kk^\co$ is obtained from
  $\FML_\kk$ by reversing the directions of the 2-morphisms.
\end{remark}

There is a canonical functor $\tildew{\FML}_\kk \ra \FML_\kk$
between $\kk$-linear 2-multicategories. The underlying
multicategories of 
$\tildew{\FML}_\kk$ and 
$\FML_\kk$ are both $\fml_\kk$.
We extend our convention from Remark~\ref{rem:functor-on-fml}
to functors with source $\tildew{\FML}_\kk$ or $\FML_\kk$ in the
obvious way. 

\begin{lemma}
  \label{l:interprete-FML-TRCAT}
  There is a unique functor
  \begin{equation}
    \label{eq:interprete-FML-TRICAT}
    \D \colon \FML_\kk \ra \TRCAT_\kk
  \end{equation}
  of $\kk$-linear 2-multicategories whose composition with
  $\tildew{\FML}_\kk \ra \FML_\kk$ extends the interpretation
  functor $\D \colon \fml_\kk \ra \trcat_\kk$ of multicategories
  from \eqref{eq:interprete-fml-tricat}, sends the 27 generating
  2-morphisms in the left column of table~\ref{tab:main} to the
  corresponding entries in the right column, and sends
  $\ud{g} \colon \ud{G} \ra \ud{G}'$ 
  to $g \colon G \ra G'$,
  $\ud E \;\ud\otimes\; \ud F \ra \ud{E \otimes F}$ 
  to 
  $E \otimes^\dL F \ra E \otimes F$,
  $\ud{\sheafHom(E,F)} \ra \ud\sheafHom(\ud E, \ud F)$ 
  to $\sheafHom(E,F) \ra \dR\sheafHom(E, F)$,
  $\oldud{\ul\C(E,F)} \ra \ud\Hom(\ud E, \ud F)$ to
  $\ul\C(E,F) \ra \dR\Hom(E, F)$,
  $\ud\alpha^* \ud E \ra \ud{(\alpha^* E)}$ to
  $\dL\alpha^*(E) \ra \alpha^* E$,
  $\ud{(\alpha_* F)} \ra \ud\alpha_* \ud F$ to
  $\alpha_* F \ra \dR\alpha_* F$,
  $\ud\alpha\inv \ud E \ra \ud{(\alpha\inv E)}$ to
  $\dL\alpha\inv(E) \ra \alpha\inv E$, and
  $\ud{(\alpha_! F)} \ra \ud\alpha_! \ud F$ to
  $\alpha_! F \ra \dR\alpha_! F$.
\end{lemma}

\begin{proof}
  This is obvious.
\end{proof}

The main results of sections~\ref{sec:four-operations} and
\ref{sec:two-operations}
can now be summarized in the following two theorems.

\begin{theorem}
  \label{t:interprete-FML-ENH}
  Let $\kk$ be a field.
  There is a unique functor
  \begin{equation}
    \label{eq:interprete-FML-ENH}
    \ul\II \colon \FML_\kk \ra \ENH_\kk
  \end{equation}
  of $\kk$-linear 2-multicategories whose composition with
  $\tildew{\FML}_\kk \ra \FML_\kk$ extends the interpretation
  functor $\ul\II \colon \fml_\kk \ra \enh_\kk$ of
  multicategories from \eqref{eq:II-interpret-fml}, sends the 27
  generating 2-morphisms in the left column of
  table~\ref{tab:main} to the corresponding entries in the middle
  column, and 
  sends 
  $\ud{g} \colon \ud{G} \ra \ud{G}'$ 
  to $\ul{g} \colon \ul{G} \ra \ul{G}'$, 
  $\ud E \;\ud\otimes\; \ud F \ra \ud{E \otimes F}$ to 
  $\ul E \;\ul\otimes\; \ul F \ra \ul{E \otimes F}$,
  $\ud{\sheafHom(E,F)} \ra \ud\sheafHom(\ud E, \ud F)$ to
  $\ul{\sheafHom(E,F)} \ra \ul\sheafHom(\ul E, \ul F)$, 
  $\oldud{\ul\C(E,F)} \ra \ud\Hom(\ud E, \ud F)$ to
  $\oldul{\ul\C(E,F)} \ra \ul\Hom(\ul E, \ul F)$,
  $\ud\alpha^* \ud E \ra \ud{(\alpha^* E)}$ to
  $\ul\alpha^* \ul E \ra \ul{(\alpha^* E)}$,
  $\ud{(\alpha_* F)} \ra \ud\alpha_* \ud F$ to
  $\ul{(\alpha_* F)} \ra \ul\alpha_* \ul F$,
  $\ud\alpha\inv \ud E \ra \ud{(\alpha\inv E)}$ to
  $\ul\alpha\inv \ul E \ra \ul{(\alpha\inv E)}$,
  and
  $\ud{(\alpha_! F)} \ra \ud\alpha_! \ud F$ to
  $\ul{(\alpha_! F)} \ra \ul\alpha_! \ul F$.
\end{theorem}

\begin{proof}
  The results in sections~\ref{sec:four-operations} and
  \ref{sec:two-operations}
  provide the relevant images in $\ENH_\kk$ for the generating
  2-morphisms of $\tildew{\FML}_\kk$ and show that they satisfy the
  defining relations of $\FML_\kk$.
\end{proof}

\begin{remark}
  \label{rem:comm-diagrams}
  If we view all
  diagrams in \ref{sec:some-comm-diagr} 
  (except diagram \eqref{eq:intro:alpha*-triangles-derived})
  as diagrams in $\FML_\kk$
  by replacing underlines by underdots, the interpretation
  functor
  \eqref{eq:interprete-FML-ENH}
  maps these diagrams to commutative diagrams (namely to the
  commutative diagrams  
  in \ref{sec:some-comm-diagr}), and similarly for the
  commutative diagrams in 
  Lemmas~\ref{l:ul-object-4-functors},
  \ref{l:ul-object-2-functors}.
  This follows from the results in 
  \ref{sec:lifts-comm-diagr},
  \ref{sec:lifts-comm-diagr-1}
  and Lemmas~\ref{l:ul-object-4-functors},
  \ref{l:ul-object-2-functors}.
  Hence, when defining $\FML_\kk$,
  we could additionally impose the
  condition that all these 
  diagrams 
  are commutative.
  If doing so,
  we would in particular have adjunctions 
  $(\ud\alpha^*, \ud\alpha_*)$
  and $(\ud\alpha_!, \ud\alpha^!)$ in $\FML_\kk$.
  This would imply that 
  the 2-morphism
  $\ud{\beta}'_*\ud\alpha'^! \ra \ud\alpha^!\ud\beta_*$
  in the left column of row
  \eqref{eq:tab:proper-base-change-upper-shriek-ENH} 
  of table~\ref{tab:subsequent}
  would be invertible (as the other 2-morphisms in this row): 
  in analogy to
  \eqref{eq:proper-base-change-upper-shriek-ENH-def}, 
  it is defined as the composition
  \begin{multline}
    \label{eq:proper-base-change-upper-shriek-FML-def}
    \ud{\beta}'_*\ud\alpha'^! 
    \ra
    \ud\alpha^!\ud\alpha_! \ud{\beta}'_*\ud\alpha'^! 
    \ra
    \ud\alpha^!\ud\beta_*\ud\beta\inv
    \ud\alpha_!\ud{\beta}'_*\ud\alpha'^!  
    \\
    \xsira{\ud\alpha^!\ud\beta_*
      \eqref{eq:tab:proper-base-change-ENH}
      \ud{\beta}'_*\ud\alpha'^!}  
    \ud\alpha^!\ud\beta_*\ud\alpha'_!\ud\beta'^{-1}
    \ud{\beta}'_*\ud\alpha'^! 
    \ra
    \ud\alpha^!\ud\beta_* \ud\alpha'_! \ud\alpha'^!
    \ra
    \ud\alpha^!\ud\beta_*,
  \end{multline}
  hence it is conjugate to
  the 2-isomorphism 
  $\ud\beta\inv\ud\alpha_!
  \xsira{\eqref{eq:tab:proper-base-change-ENH}}
  \ud\alpha'_!\ud\beta'^{-1}$ 
  in $\FML_\kk$ and invertible, by 
  Remarks~\ref{rem:morphism-between-left-adjoints-2-categorical}  
  and \ref{rem:compose-adjunctions}.
\end{remark}

\begin{theorem}
  \label{t:compare-interpretations-FML}
  Let $\kk$ be a field. The pseudo-natural transformation
  \eqref{eq:pseudo-nat-trafo-D-to-[I]} defines in fact a
  pseudo-natural transformation 
  \begin{equation}
    \label{eq:pseudo-nat-trafo-D-to-[I]-2-multicat}
    (\ol{[\ii]}, \omega) \colon \D \ra [\ul\II] = [-] \circ \ul\II
  \end{equation}
  between functors $\FML_\kk \ra \TRCAT_\kk$ of $\kk$-linear
  2-multicategories where $\D$ is
  \eqref{eq:interprete-FML-TRICAT} and $[\ul\II]$ is the 
  composition
  $\FML_\kk \xra[\II]{\eqref{eq:interprete-FML-ENH}} \ENH_\kk
  \xra[{[-]}]{\eqref{eq:htpy-ENH-TRCAT}} \TRCAT_\kk$
  (see \eqref{eq:intro:pseudonatural-trafo-omega} for a diagram).
\end{theorem}

\begin{proof}
  The data required to define a pseudo-natural transformation
  $\D \ra [\ul\II]$ where $\D$ and $[\ul\II]$ are
  considered as functors $\fml_\kk \ra \trcat_\kk$ of
  multicategories or as functors
  $\FML_\kk \ra \TRCAT_\kk$ of
  $\kk$-linear 2-multicategories coincide, by
  \cite[Def.~7.5.2]{borceux-cat}: the required natural
  transformation consists of 
  a collection of
  morphisms 
  indexed by the 
  set $\fml_\kk(S_1, \dots, S_n;T)=\Obj(\FML_\kk(S_1, \dots,
  S_n;T))$, for each given collection $S_1, \dots, S_n,
  T \in 
  \Obj(\fml_\kk)=\Obj(\FML_\kk)$ of objects.
  However, 
  in the case of functors of $\kk$-linear 2-multicategories more
  conditions 
  are imposed: the collection of morphisms is required to be
  compatible with morphisms in $\FML_\kk(S_1, \dots, S_n;T)$.
  For the generating 2-morphisms of $\tildew{\FML}_\kk$ this
  compatibility boils down to our main results of 
  sections~\ref{sec:four-operations} and
  \ref{sec:two-operations},
  using Definition~\ref{d:enhance-2-morphism}, for example to the
  fact that our 2-morphism $\ul\alpha^* \ul\alpha_* \ra \id$ 
  $(\ud\alpha^*\ud\alpha_*, \id)$-enhances the 2-morphism
  $\dL\alpha^*\dR\alpha_* \ra \id$.
  The theorem follows.
\end{proof}

\section{Some remarks on 2-morphisms in \texorpdfstring{$\ENH_\kk$}{ENH}}
\label{sec:some-remarks-2-morphisms-ENH}

Recall the functor 
$\delta \colon \tildew{\ENH}_\kk \ra \ENH_\kk$ from 
\eqref{eq:delta-tilde-ENH-to-ENH}.
There are three statements about 2-morphisms in $\ENH_\kk$ we would
like to be true 
but do not know how to prove in general:
\begin{enumerate}[label=(\Roman*)]
\item 
  \label{enum:useful-ohe-iff-delta-2-isom-wish}
  A 2-morphism $\tau$ in 
  $\tildew{\ENH}_\kk$ 
  is an objectwise homotopy equivalence if and
  only if $\delta(\tau)$ is a 2-isomorphism in $\ENH_\kk$.
\item   
  \label{enum:useful-represent-2-morphisms-wish}
  Any 2-morphism $\tau'$ in 
  $\ENH_\kk$ can be represented by a roof, i.\,e.\ it 
  has the form
  $\tau' = \delta(v) \circ \delta(u)\inv$ where $u$ and $v$ are
  2-morphisms in $\tildew{\ENH}_\kk$ and $u$ is an
  objectwise homotopy equivalence; 
  moreover, for any such representation, $\tau'$ is a
  2-isomorphism if and only if $v$ is an objectwise homotopy
  equivalence.
\item 
  \label{enum:useful-reflects-2-isos-wish}
  The functor
  $[-] \colon \ENH_\kk \ra \TRCAT_\kk$
  reflects 2-isomorphisms (i.\,e.\ it is conservative).
\end{enumerate}
The aim of this section is to prove these
statements under suitable assumptions on the sources and target
of the involved
1-morphisms. The target must be a suitable category of sheaves or
modules that we can control with
model categorical techniques. Then either the sources must be 
sufficiently small (see
Proposition~\ref{p:sources-small})  
or the target must be made sufficiently large (see
Proposition~\ref{p:target-large}). 
We expect these results to be useful in applications.
Our approach is based on the existence of certain model
structures on functor categories which only seem to be available
if the source category is small compared to the target category.

\subsection{Localizations of functor categories}
\label{sec:localizations-of-funct-categ}

\begin{proposition}
  \label{p:localization-wrt-obj-htpy-equ}
  Let $\ulms{M}$ be a combinatorial $\KK$-enriched model category
  (with respect to the universe $\mathscr{U}$).
  Assume that the 
  localization functor $\lambda \colon \ms{M} \ra
  \Ho(\ms{M})$ 
  factors through the canonical functor
  $\ms{M} \ra [\ulms{M}]$
  as
  \begin{equation}
    \ms{M} \ra [\ulms{M}] \xra{\ol{\lambda}} \Ho(\ms{M})
  \end{equation}
  and that the restriction
  \begin{equation}
    \label{eq:[M-bifib]-HoM}
    [\ulms{M}_\bifib] \ra \Ho(\ms{M})
  \end{equation}
  of $\ol{\lambda}$ is an equivalence, where $\ulms{M}_\bifib$ is
  the full $\KK$-category of $\ulms{M}$ of bifibrant (= fibrant
  and cofibrant) objects. Let $\ulms{B}$ be a full
  $\KK$-subcategory of $\ulms{M}_\bifib$ such that
  $[\ulms{B}] \subset [\ulms{M}_\bifib]$ is a strictly full
  subcategory. 
  Let $\ulms{A}$ be any $\mathscr{U}$-small
  $\KK$-category.
  Let 
  \begin{equation}
    \label{eq:loc-wrt-obj-htpy-equiv}
    \delta \colon \DGCAT_\kk(\ulms{A},
    \ulms{B}) \ra \L_\ohe\DGCAT_\kk(\ulms{A},
    \ulms{B})    
  \end{equation}
  be the localization
  of $\DGCAT_\kk(\ulms{A}, \ulms{B})$ with
  respect to the set
  of objectwise homotopy
  equivalences. Then
  \begin{enumerate}
  \item 
    \label{enum:K:localization-of-U-exists}
    we may and will assume that
    the category $\L_\ohe\DGCAT_\kk(\ulms{A},\ulms{B})$ has
    $\mathscr{U}$-small Hom-sets, that it has
    the same set of 
    objects as $\DGCAT_\kk(\ulms{A},\ulms{B})$ and that $\delta$
    is the identity on objects;
  \item 
    \label{enum:K:weak-equiv-vs-iso}
    a morphism $\tau$ in 
    $\DGCAT_\kk(\ulms{A},\ulms{B})$
    is an objectwise homotopy equivalence if and
    only if $\delta(\tau)$ is an isomorphism;
  \item   
    \label{enum:K:simplify-rep-if-U-fibrant}
    any morphism $\tau'$ in 
    $\L_\ohe\DGCAT_\kk(\ulms{A},\ulms{B})$    
    has the form
    $\tau' = \delta(v) \circ \delta(u)\inv$ where $u$ is an
    objectwise homotopy equivalence in 
    $\DGCAT_\kk(\ulms{A},\ulms{B})$ and $v$ is a morphism
    in 
    $\DGCAT_\kk(\ulms{A},\ulms{B})$; moreover, in any such
    representation, $\tau'$ is an
    isomorphism if and only if $v$ is an objectwise homotopy
    equivalence;
  \item 
    \label{enum:K:olhU-reflects-isos}
    the unique functor
    $\ol{[-]} \colon
    \L_\ohe\DGCAT_\kk(\ulms{A}, \ulms{B}) \ra
    \CAT_\kk([\ulms{A}], [\ulms{B}])$ such that $\ol{[-]} \circ
    \delta = [-]$ 
    reflects isomorphisms.
  \end{enumerate}
\end{proposition}

\fussnote{
  zu loeschen:
  Is the following true (and is it then useful)?
  Given $\tau, \tau' \colon F \ra G$ with
  $\delta(\tau)=\delta(\tau')$ implies $\tau=\tau'$? (Also
  $\delta$ faithful.)
  
  NEIN (vermutlich), denn wenn $F$ und $G$ bifasernd, dann ist
  $\delta(\tau)=\delta(\tau')$ genau dann, wenn $\tau$ und
  $\tau'$ (links/rechts-)homotop sind. Kann also vermutlich
  explizit gegenbeispiel geben, etwa $\ulms{A}=\ul\kk$, so dass
  die 
  Lokalisierung
  $\ms{B} \ra [\ulms{B}]$ ist (wie ich glaube ich oben
  berechne). 
}

\fussnote{
  $\ulms{A}$ should be small
}

\begin{proof}
  Recall that 
  $\KK$ is an excellent model category, by
  Proposition~\ref{p:KK-excellent-and-MM-KK-model-cat}.\ref{enum:KK-excellent}. 
  Since $\ulms{A}$ is $\mathscr{U}$-small, we may consider
  $\DGCAT_\kk(\ulms{A}, \ulms{M})$ as a model category 
  equipped with the projective model
  structure given by
  \cite[Prop.~A.3.3.2]{lurie-higher-topos}. 
  By definition, a morphism 
  $\tau \colon F \ra G$ 
  in 
  $\DGCAT_\kk(\ulms{A}, \ulms{M})$
  is a weak equivalence (resp.\ a fibration) if and only if
  $\tau_A \colon F(A) \ra G(A)$ is a weak equivalence (resp.\ a
  fibration) for all
  $A \in \ulms{A}$.

  We view 
  $\DGCAT_\kk(\ulms{A}, \ulms{B})$ as the full
  subcategory of 
  $\DGCAT_\kk(\ulms{A}, \ulms{M})$ consisting of $\KK$-functors
  $\ulms{A} \ra \ulms{M}$ that factor as 
  $\ulms{A} \ra \ulms{B} \ra \ulms{M}$.
  We want to apply Theorem~\ref{t:localization-exists}
  to 
  \begin{equation}
    \mathcal{S}:=
    \DGCAT_\kk(\ulms{A}, \ulms{B}) 
    \subset  
    \mathcal{N}:=\DGCAT_\kk(\ulms{A}, \ulms{M})  
  \end{equation}
  and 
  $h$ 
  the composition 
  \begin{equation}
    h \colon
    \mathcal{N}=\DGCAT_\kk(\ulms{A}, \ulms{M}) 
    \xra{[-]} 
    \CAT_\kk([\ulms{A}], [\ulms{M}])
    \ra
    \CAT([\ulms{A}], [\ulms{M}])
    \xra{\ol{\lambda}_*} \CAT([\ulms{A}], \Ho(\ms{M}))
  \end{equation}
  where $\CAT$ is the 2-category of categories.
  Note that $\tau$ as above is a weak equivalence if and only if
  $h(\tau)$ 
  is an isomorphism, by \cite[Thm.~8.3.10]{hirschhorn-model}.

  We need to show that all objects of $\mathcal{S}$ admit fibrant
  cofibrant 
  approximations and cofibrant fibrant approximations in
  $\mathcal{S}$. Note that the initial (resp.\ terminal) object of
  $\mathcal{N}$ is the $\mathcal{K}$-functor $\ulms{A} \ra
  \ulms{M}$ mapping each object of $\ulms{A}$ to the initial (resp.\
  terminal) object of $\ms{M}$. Therefore, all objects of
  $\mathcal{S}$ are fibrant by the definition of the projective
  model structure and our assumption $\ulms{B} \subset
  \ulms{M}_\bifib \subset \ulms{M}_\fib$; hence cofibrant fibrant
  approximations in $\mathcal{S}$ exist trivially for all objects
  of $\mathcal{S}$. 

  Let $E$ be an object of $\mathcal{S}$
  and $\rho \colon G \ra E$ a fibrant cofibrant approximation
  in $\mathcal{N}$. 
  For $A \in \ulms{A}$ the morphism $\rho_A
  \colon G(A) \ra E(A)$ is a trivial fibration 
  by the definition of the projective model structure.
  Since $E(A)$ is fibrant so is $G(A)$.
  By
  \cite[Prop.~3.3]{toen-homotopy-of-dg-cats-morita}, 
  $G(A)$ is cofibrant because $G$ is cofibrant.
  The object $E(A)$ is bifibrant by the assumption $\ulms{B}
  \subset \ulms{M}_\bifib$. 
  So $\rho_A$ is a weak equivalence between
  bifibrant objects, and 
  the equivalence \eqref{eq:[M-bifib]-HoM}
  shows that $\rho_A$ is an isomorphism in
  $[\ulms{M}_\bifib]$. 
  Since $E(A) \in \ulms{B}$ and $[\ulms{B}] \subset
  [\ulms{M}_\bifib]$ is strictly full we 
  obtain $G(A) \in \ulms{B}$. 
  Hence 
  $G \in \mathcal{S}$.
  This shows that fibrant cofibrant
  approximations in $\mathcal{S}$ exist for all objects
  of $\mathcal{S}$. 

  Now
  Theorem~\ref{t:localization-exists}.\ref{enum:localization-of-S-exists} 
  applies to the localization of 
  $\mathcal{S}=\DGCAT_\kk(\ulms{A}, \ulms{B})$ 
  with respect to the set of weak equivalences.

  We claim that a morphism in $\mathcal{S}$ is a weak equivalence
  if and only if it is an objectwise homotopy equivalence.
  The restriction of $h$ to
  $\mathcal{S}$ is equal to the 
  composition 
  \begin{equation}
    \mathcal{S}=\DGCAT_\kk(\ulms{A}, \ulms{B})
    \xra{[-]} 
    \CAT_\kk([\ulms{A}], [\ulms{B}])
    \ra
    \CAT([\ulms{A}], [\ulms{B}])
    \ra \CAT([\ulms{A}], \Ho(\ms{M}))
  \end{equation}
  whose third arrow is induced by
  the
  full and faithful composition
  $[\ulms{B}] \subset [\ulms{M}_\bifib] \sira \Ho(\ms{M})$.
  Therefore this third arrow is full and faithful 
  and in particular reflects isomorphisms.
  The second arrow obviously reflects isomorphisms.
  These statements imply the claim, proving
  \ref{enum:K:localization-of-U-exists}.

  Moreover, 
  \ref{enum:K:weak-equiv-vs-iso},
  \ref{enum:K:simplify-rep-if-U-fibrant}, and
  \ref{enum:K:olhU-reflects-isos} 
  follow from the above facts and
  parts 
  \ref{enum:weak-equiv-vs-iso},
  \ref{enum:f'-we-iff-f-iso},
  \ref{enum:simplify-rep-if-S-fibrant}, and
  \ref{enum:olhS-reflects-isos}
  of Theorem~\ref{t:localization-exists}.
\end{proof}

\begin{proposition}
  \label{p:examples-localization-wrt-obj-htpy-equ}
  Let $\ulms{M}$ be either
  \begin{enumerate}
  \item 
    \label{enum:CX-inj}
    $\ul\C(\mathcal{X})$ with $\C(\mathcal{X})$ carrying the
    $\II$-model structure, for
    a $\mathscr{U}$-small $\kk$-ringed site $(\mathcal{X},
    \mathcal{O})$, or 
  \item
    \label{enum:ModC-inj-proj}
    $\ul\MMod(\ulms{C})$ with $\MMod(\ulms{C})$ carrying the
    $\II$- or $\PP$-model structure, for a $\mathscr{U}$-small
    $\KK$-category $\ulms{C}$.
  \end{enumerate}
  Then $\ulms{M}$ 
  satisfies the assumptions imposed on 
  $\ulms{M}$ in
  Proposition~\ref{p:localization-wrt-obj-htpy-equ}.
\end{proposition}

\begin{proof}
  By Theorem~\ref{t:examples-enriched-model-cats}, $\ulms{M}$ is
  a combinatorial $\KK$-enriched model category.  Two morphisms
  in $\ms{M}$ that are homotopic in the sense that they are equal
  in $[\ulms{M}]$ are left (resp.\ right) homotopic in the model
  categorical sense (use the obvious cylinder (resp.\ path)
  object if $\ms{M}$ carries the $\II$- (resp.\ $\PP$-)model
  structure).  Therefore, using
  \cite[Lemma~8.3.4]{hirschhorn-model}, the localization functor
  $\lambda \colon \ms{M} \ra \Ho(\ms{M})$ factors through the
  canonical functor $\ms{M} \ra [\ulms{M}]$ as
  $\ms{M} \ra [\ulms{M}] \xra{\ol{\lambda}} \Ho(\ms{M})$.
  Moreover, the induced functor
  $[\ulms{M}_\bifib] \ra \Ho(\ms{M})$ is an equivalence, by
  \cite[Thm.~8.3.6]{hirschhorn-model}.
  Here we use the fact that maps in $\ms{M}$ between bifibrant
  objects are 
  homotopic in the model categorical sense if and only if they
  are homotopic in the sense that they are equal in $[\ulms{M}]$:
  use 
  \cite[Cor.~1.2.6]{hovey-model-categories} and
  cylinder (resp.\ path) objects as above.
\end{proof}

\begin{remark}
  \label{rem:reminder-fibrant-cofibrant}
  We remind the reader that all objects of $\ul\C(\mathcal{X})$
  are $\II$-cofibrant and that
  $\ul\II(\mathcal{X})=\ul\C(\mathcal{X})_{\II\text{-}\fib}=\ul\C(\mathcal{X})_{\II\text{-}\bifib}$,
  see 
  \ref{sec:inject-model-struct-1} and
  \ref{sec:inject-model-struct}.
  Similarly,
  $\ul\MMod(\ulms{C})=\ul\MMod(\ulms{C})_{\II\text{-}\cof}$ and
  $\ul\IIMod(\ulms{C})=\ul\MMod(\ulms{C})_{\II\text{-}\fib}=\ul\MMod(\ulms{C})_{\II\text{-}\bifib}$,
  and 
  $\ul\MMod(\ulms{C})=\ul\MMod(\ulms{C})_{\PP\text{-}\fib}$ and
  $\ul\PPMod(\ulms{C})=\ul\MMod(\ulms{C})_{\PP\text{-}\cof}=\ul\MMod(\ulms{C})_{\PP\text{-}\bifib}$,
  see \ref{sec:model-structures}.
\end{remark}

\begin{remark}
  \label{rem:examples-localization-wrt-obj-htpy-equ}
  Consider $\ul\C(\mathcal{X})$ with $\C(\mathcal{X})$ carrying
  the $\II$-model structure, for a $\kk$-ringed site
  $(\mathcal{X}, \mathcal{O})$.
  The composition $\C(\mathcal{X}) \ra [\ul\C(\mathcal{X})] \ra
  \D(\mathcal{X})$ maps weak equivalences to isomorphisms and
  hence factors through the homotopy category
  $\C(\mathcal{X}) \ra \Ho(\C(\mathcal{X}))$ to a functor
  $\Ho(\C(\mathcal{X})) \ra \D(\mathcal{X})$. This functor is an
  isomorphism of categories.
  The subcategory of bifibrant objects of $\C(\mathcal{X})$ 
  is $\II(\mathcal{X})$, and it is well-known that
  $[\ul\II(\mathcal{X})] \ra \D(\mathcal{X})$ is an equivalence,
  cf.\ \ref{sec:inject-enhanc}.
  This gives another point of view on the proof of part
  \ref{enum:CX-inj} of   
  Proposition~\ref{p:examples-localization-wrt-obj-htpy-equ}, and
  similarly for part \ref{enum:ModC-inj-proj} where
  $\Ho(\MMod(\ulms{C})) \sira \D(\ulms{C})$ is an isomorphism.
\end{remark}

\subsection{Small sources}
\label{sec:small-sources}

Let $\mathscr{U} \in \mathscr{V}$ be universes as in
\ref{sec:some-2-mult}.

Let $\ulms{A}_1, \dots, \ulms{A}_n$, $\ulms{B}$ be objects of
$\ENH_\kk$.
Recall that 
\begin{equation}
  \label{eq:delta-ENH-A-B}
  \delta \colon \tildew{\ENH}_\kk(\ulms{A}_1, \dots,
  \ulms{A}_n; \ulms{B}) \ra \ENH_\kk(\ulms{A}_1, \dots,
  \ulms{A}_n; \ulms{B}) 
\end{equation}
denotes the additive $\kk$-localization with respect to the set of
objectwise homotopy equivalences.
We can now prove the statements
\ref{enum:useful-ohe-iff-delta-2-isom-wish},
\ref{enum:useful-represent-2-morphisms-wish},
\ref{enum:useful-reflects-2-isos-wish}
if the sources $\ulms{A}_1, \dots, \ulms{A}_n$ are
$\mathscr{U}$-small and the target 
$\ulms{B}$ is a suitable category of sheaves or modules.

\begin{proposition}
  \label{p:sources-small}
  Let $\ulms{A}_1, \dots, \ulms{A}_n$ be 
  additive pretriangulated 
  $\mathscr{U}$-small $\KK$-categories.
  Let $\ulms{N}$ be either
  \begin{itemize}
  \item 
    $\ul\II(\mathcal{X})$, for a
    $\mathscr{U}$-small 
    $\kk$-ringed site 
    $(\mathcal{X}, \mathcal{O})$, or
  \item 
    $\ul\IIMod(\ulms{C})$   
    or $\ul\PPMod(\ulms{C})$, for a $\mathscr{U}$-small
    $\KK$-category $\ulms{C}$.
  \end{itemize}
  Assume that $\ulms{B}$ is an additive pretriangulated full
  $\KK$-subcategory of $\ulms{N}$ such that
  $[\ulms{B}]$ is a strictly full subcategory of
  $[\ulms{N}]$. 
  Then the following statements hold true.
  \begin{enumerate}
  \item 
    \label{enum:useful-ohe-iff-delta-2-isom}
    A morphism $\tau$ in
    $\tildew{\ENH}_\kk(\ulms{A}_1, \dots, \ulms{A}_n; \ulms{B})$
    is an objectwise homotopy equivalence if and
    only if $\delta(\tau)$ 
    is an isomorphism.
  \item   
    \label{enum:useful-represent-2-morphisms}
    Any morphism $\tau'$ in 
    $\ENH_\kk(\ulms{A}_1, \dots,
    \ulms{A}_n; \ulms{B})$
    has the form
    $\tau' = \delta(v) \circ \delta(u)\inv$ where $u$ and $v$ are
    morphisms in 
    $\tildew{\ENH}_\kk(\ulms{A}_1, \dots, \ulms{A}_n; \ulms{B})$
    and $u$ is an
    objectwise homotopy equivalence; 
    moreover, for any such representation, $\tau'$ is a
    2-isomorphism if and only if $v$ is an objectwise homotopy
    equivalence.
  \item 
    \label{enum:useful-reflects-2-isos}
    The functor
    $[-] \colon 
    \ENH_\kk(\ulms{A}_1, \dots,
    \ulms{A}_n; \ulms{B}) \ra
    \TRCAT_\kk([\ulms{A}_1], \dots,
    [\ulms{A}_n]; [\ulms{B}])$
    reflects isomorphisms.
  \end{enumerate}
\end{proposition}

\begin{remark}
  \label{rem:U-small-objects-ENH}
  This remark provides examples of possible source categories
  $\ulms{A}_1, \dots, \ulms{A}_n$ in
  Proposition~\ref{p:sources-small}. 
  Let $(\mathcal{X}, \mathcal{O})$ be a $\mathscr{U}$-small
  $\kk$-ringed site. 
  Then $\ul\II(\mathcal{X})$ is an object of $\ENH_\kk$ which 
  has $\mathscr{U}$-small $\Hom$-sets and
  objects in $\mathscr{U}$, as explained in 
  Remark~\ref{rem:ENH-objects-are-enhancements}.
  In general, however, $\ul\II(\mathcal{X})$ is not
  $\mathscr{U}$-small. 
  But there are many additive pretriangulated
  $\mathscr{U}$-small full $\KK$-subcategories $\ulms{A}$ of 
  $\ul\II(\mathcal{X})$ (that may enhance some interesting
  subcategories of $\D(\mathcal{X})$ defined by finiteness
  conditions). 
  For example, take any strongly pretriangulated full 
  $\KK$-subcategory of $\ul\II(\mathcal{X})$ whose set of objects
  is $\mathscr{U}$-small. 
  One may produce such a subcategory iteratively by starting with
  any $\mathscr{U}$-small set of objects of 
  $\ul\II(\mathcal{X})$.
\end{remark}

\begin{proof}
  If $\ulms{N}$ is $\ul\II(\mathcal{X})$, let
  $\ulms{M}=\ul\C(\mathcal{X})$ with $\ms{M}$ carrying the
  $\II$-model structure. 
  If $\ulms{N}$ is $\ul\IIMod(\ulms{C})$ (resp.\
  $\ul\PPMod(\ulms{C})$), let $\ulms{M}=\ul\MMod(\ulms{C})$ with
  $\ms{M}$ carrying the $\II$-model structure (resp.\ the
  $\PP$-model structure).
  By Proposition~\ref{p:examples-localization-wrt-obj-htpy-equ}
  and the fact that $\ulms{B} \subset \ulms{N}=\ulms{M}_\bifib$ (see 
  Remark~\ref{rem:reminder-fibrant-cofibrant})
  we can apply 
  Proposition~\ref{p:localization-wrt-obj-htpy-equ}
  (using the universe $\mathscr{U}$ and the $\mathscr{U}$-small
  $\KK$-category 
  $\ulms{A}_1 \otimes \dots \otimes \ulms{A}_n$). 
  All claims follow if we remember that the underlying functor
  of the additive $\kk$-localization is the ordinary
  localization, see 
  Proposition~\ref{p:R-localization-additive-small}.
\end{proof}

\begin{remark}
  \label{rem:sources-small}
  The assumption in Proposition~\ref{p:sources-small}
  that the $\KK$-categories $\ulms{A}_1, \dots, \ulms{A}_n$ are
  $\mathscr{U}$-small is important for the proof. The assumption
  that they are 
  additive pretriangulated was just imposed to ensure that they
  are objects of $\ENH_\kk$. If we omit it, 
  Proposition~\ref{p:sources-small} remains true if
  we replace
  $\tildew{\ENH}_\kk(\ulms{A}_1, \dots, \ulms{A}_n; \ulms{B})
  \ra \ENH_\kk(\ulms{A}_1, \dots, \ulms{A}_n; \ulms{B})$
  by 
  $\DGCAT_\kk(\ulms{A}_1, \dots, \ulms{A}_n; \ulms{B}) \ra
  \L_\ohe\DGCAT_\kk(\ulms{A}_1, \dots, \ulms{A}_n; \ulms{B})$ and
  $\TRCAT_\kk$ by $\CAT_\kk$.
\end{remark}

\begin{example}
  \label{exam:base-change}
  Let
  \begin{equation}
    \label{eq:cartesian-schemes}
    \xymatrix{
      {Y'} \ar[r]^-{\beta'} \ar[d]^-{\alpha'} &
      {Y} \ar[d]^-{\alpha} \\
      {X'} \ar[r]^-{\beta} &
      {X}
    }
  \end{equation}
  be a cartesian diagram of $\mathscr{U}$-small
  $\kk$-schemes. Then there is a base 
  change 2-morphism
  \begin{equation}
    \label{eq:BC-schemes-TRCAT}
    \dL\beta^* \dR\alpha_* \ra \dR\alpha'_*\dL\beta'^* \colon
    \D(Y) \ra \D(X')
  \end{equation}
  in $\TRCAT_\kk$ constructed in the usual way using
  adjunction units and counits, see e.\,g.\
  \cite[Prop.~3.7.2]{lipman-notes-on-derived-functors-and-GD}. 
  This construction has an obvious analog in $\ENH_\kk$
  providing a 2-morphism
  \begin{equation}
    \label{eq:BC-schemes-ENH}
    \ul\beta^* \ul\alpha_* \ra \ul\alpha'_*\ul\beta'^* \colon
    \ul\II(Y) \ra \ul\II(X')
  \end{equation}
  in $\ENH_\kk$ that enhances \eqref{eq:BC-schemes-TRCAT}.
  Let $\D_\qc(Y)$ (resp.\ $\ul\II_\qc(Y)$) be the subcategory of
  $\D(Y)$ (resp. $\ul\II(Y)$) of objects with 
  quasi-coherent cohomology. 
  We 
  assume that the
  2-morphism \eqref{eq:BC-schemes-TRCAT} restricts to a 
  2-isomorphism 
  \begin{equation}
    \label{eq:BC-schemes-TRCAT-Dqc}
    \dL\beta^* \dR\alpha_* \sira \dR\alpha'_*\dL\beta'^* \colon
    \D_\qc(Y) \ra \D(X')
  \end{equation}
  (see \cite[Prop.~3.9.5 or
  Thm.~3.10.3]{lipman-notes-on-derived-functors-and-GD}
  for sufficient conditions).
  Let 
  \begin{equation}
    \label{eq:BC-schemes-ENH-IIqc}
    \tau' \colon \ul\beta^* \ul\alpha_* \ra
    \ul\alpha'_*\ul\beta'^* \colon 
    \ul\II_\qc(Y) \ra \ul\II(X')
  \end{equation}
  be the corresponding restriction of \eqref{eq:BC-schemes-ENH}.
  It is a 2-morphism in $\ENH_\kk$ but we do not know whether
  it is invertible.
  The fact that \eqref{eq:BC-schemes-TRCAT-Dqc} is a 2-isomorphism 
  implies (use
  Definition~\eqref{d:enhance-2-morphism}) that  
  \begin{equation}
    \label{eq:BC-schemes-ENH-Iqc}
    [\tau'] \colon [\ul\beta^*][\ul\alpha_*] \ra
    [\ul\alpha'_*][\ul\beta'^*] \colon 
    [\ul\II_\qc(Y)] \ra [\ul\II(X')]
  \end{equation}
  is a 2-isomorphism.
  Let $\ulms{A}$ be a strongly pretriangulated
  $\mathscr{U}$-small full $\KK$-subcategory of 
  $\ul\II_\qc(Y)$, cf.\ Remark~\ref{rem:U-small-objects-ENH}.
  Then the restriction 
  \begin{equation}
    \label{eq:BC-schemes-ENH-IIqc-restricted}
    \tau'|_\ulms{A} \colon \ul\beta^* \ul\alpha_* \ra
    \ul\alpha'_*\ul\beta'^* \colon 
    \ulms{A} \ra \ul\II(X')
  \end{equation}
  of \eqref{eq:BC-schemes-ENH-IIqc}
  to $\ulms{A}$ is an isomorphism
  in $\ENH_\kk(\ulms{A}, \ul\II(X'))$,
  by
  Proposition~\ref{p:sources-small}.\ref{enum:useful-reflects-2-isos}.
  Moreover, by part~\ref{enum:useful-represent-2-morphisms}, this
  restriction is represented by a roof
  \begin{equation}
    \ul\beta^*\ul\alpha_* \xla{u} F \xra{v}
    \ul\alpha'_*\ul\beta'^*
  \end{equation}
  of objectwise homotopy equivalences in
  $\DGCAT_\kk(\ulms{A}, \ul\II(X'))$
  where
  $F \colon \ulms{A} \ra \ul\II(X')$
  is some $\KK$-functor.
\end{example}

\subsection{Large targets}
\label{sec:large-targets}

Let $\mathscr{U} \in \mathscr{V}$ be universes as in
\ref{sec:some-2-mult}. In this subsection we need two
additional universes $\mathscr{U}' \in \mathscr{V}'$ with
$\mathscr{V} \subset \mathscr{U}'$.


Assume that $\ulms{B}$ is an additive pretriangulated
full $\KK$-subcategory of 
$\ul\II(\mathcal{X})=\ul\II(\mathcal{X}; \mathscr{U})$ 
where $(\mathcal{X}, \mathcal{O})$
is a $\mathscr{U}$-small $\kk$-ringed site;
we have included $\mathscr{U}$ in the notation to emphasize that
we consider sheaves of $\mathscr{U}$-small modules.

By Proposition~\ref{p:h-inj-change-universe-sheaves}
we know that we can view 
$\ul\II(\mathcal{X}; \mathscr{U})$ as a full
$\KK_{\mathscr{U}'}$-subcategory of $\ul\II(\mathcal{X};
\mathscr{U}')$.
Let $\ulms{B}_{\mathscr{U}'}$ be 
any pretriangulated full $\KK_{\mathscr{U}'}$-subcategory of 
$\ul\II(\mathcal{X}; \mathscr{U}')$
containing $\ulms{B}$ whose homotopy category
is closed under isomorphisms in 
$[\ul\II(\mathcal{X}; \mathscr{U}')]$ (possible choices are 
$\ul\II(\mathcal{X}; \mathscr{U}')$ itself, or
the
smallest such subcategory: its set of objects is the
set of objects of the essential image of
$[\ulms{B}]$ in $[\ul\II(\mathcal{X}; \mathscr{U}')]$). 
It is automatically
strongly pretriangulated and in particular additive.

Remember that all objects of $\tildew{\ENH}_\kk$ and   
$\ENH_\kk$
are 
$\mathscr{V}$-small $\KK$-categories. 
Define $\tildew{\ENH}_\kk^{\mathscr{V}'}$ and
$\ENH_\kk^{\mathscr{V}'}$ in the obvious way by taking as objects all
additive pretriangulated $\mathscr{V}'$-small $\KK$-categories.
There are obvious change of universe functors $\tildew{\ENH}_\kk \ra
\tildew{\ENH}_\kk^{\mathscr{V}'}$ and
$\ENH_\kk \ra
\ENH_\kk^{\mathscr{V}'}$. Define $\DGCAT_\kk^{\mathscr{V}'}$ and
$\TRCAT_\kk^{\mathscr{V}'}$ similarly by allowing
$\mathscr{V}'$-small categories as objects.

Let $\ulms{A}_1, \dots, \ulms{A}_n$ be objects of $\ENH_\kk$.
Composition with $\ulms{B} \ra \ulms{B}_{\mathscr{U}'}$
yields the horizontal functors in the following
commutative diagram of $\kk$-linear categories.
\begin{equation}
  \xymatrix{
    {\tildew{\ENH}_\kk(\ulms{A}_1, \dots, \ulms{A}_n; \ulms{B})}
    \ar[d]^-{\delta}
    \ar[r] &
    {\tildew{\ENH}_\kk^{\mathscr{V}'}(\ulms{A}_1, \dots, \ulms{A}_n;
      \ulms{B}_{\mathscr{U}'})} 
    \ar[d]^-{\delta}
    \\
    {\ENH_\kk(\ulms{A}_1, \dots, \ulms{A}_n; \ulms{B})}
    \ar[r] &
    {\ENH_\kk^{\mathscr{V}'}(\ulms{A}_1, \dots,\ulms{A}_n;
      \ulms{B}_{\mathscr{U}'})} 
  }
\end{equation}
The advantage of the right column in this diagram is that the
$\mathscr{V}$-small categories $\ulms{A}_i$ are in particular
$\mathscr{U}'$-small and that $\ms{B}_{\mathscr{U}'}$ lives in
the model category $\C(\mathcal{X}; \mathscr{U}')$ which is a
model category with respect to the universe $\mathscr{U}'$.
In this situation we can prove the following version of statements
\ref{enum:useful-ohe-iff-delta-2-isom-wish},
\ref{enum:useful-represent-2-morphisms-wish},
\ref{enum:useful-reflects-2-isos-wish}.

\begin{proposition}
  \label{p:target-large}
  Under the above assumptions, the following statements hold true.
  \begin{enumerate}
  \item 
    \label{enum:useful-ohe-iff-delta-2-isom-large}
    A morphism $\tau$ in
    $\tildew{\ENH}_\kk^{\mathscr{V}'}(\ulms{A}_1, \dots,
    \ulms{A}_n; \ulms{B}_{\mathscr{U}'})$ 
    is an objectwise homotopy equivalence if and
    only if $\delta(\tau)$ 
    is an isomorphism.
  \item   
    \label{enum:useful-represent-2-morphisms-large}
    Any morphism $\tau'$ in 
    $\ENH_\kk^{\mathscr{V}'}(\ulms{A}_1, \dots,
    \ulms{A}_n; \ulms{B}_{\mathscr{U}'})$
    has the form
    $\tau' = \delta(v) \circ \delta(u)\inv$ where $u$ and $v$ are
    morphisms in 
    $\tildew{\ENH}_\kk^{\mathscr{V}'}(\ulms{A}_1, \dots,
    \ulms{A}_n; \ulms{B}_{\mathscr{U}'})$ 
    and $u$ is an
    objectwise homotopy equivalence; 
    moreover, for any such representation, $\tau'$ is a
    2-isomorphism if and only if $v$ is an objectwise homotopy
    equivalence.
  \item 
    \label{enum:useful-reflects-2-isos-large}
    The functor
    $[-] \colon 
    \ENH_\kk^{\mathscr{V}'}(\ulms{A}_1, \dots,
    \ulms{A}_n; \ulms{B}_{\mathscr{U}'}) \ra
    \TRCAT_\kk^{\mathscr{V}'}([\ulms{A}_1], \dots,
    [\ulms{A}_n]; [\ulms{B}_{\mathscr{U}'}])$
    reflects isomorphisms.
  \end{enumerate}
\end{proposition}

\begin{proof}
  This is a direct application of
  Proposition~\ref{p:sources-small}
  if we work with the universes $\mathscr{U}' \in \mathscr{V}'$
  instead of the universes $\mathscr{U} \in \mathscr{V}$.
\end{proof}

\begin{remark} 
  If we assume that $\ulms{B}$ is an additive pretriangulated
  full $\KK$-subcategory of 
  $\ul\IIMod(\ulms{C})$
  or $\ul\PPMod(\ulms{C})$, for a
  $\mathscr{U}$-small 
  $\KK$-category $\ulms{C}$,
  and define $\ulms{B}_{\mathscr{U}'}$ accordingly using 
  Proposition~\ref{p:h-inj-change-universe-dg-modules},
  Proposition~\ref{p:target-large} literally remains true.
\end{remark}

\begin{example}
  \label{exam:base-change-large}
  Consider the situation of Example~\ref{exam:base-change}
  and assume in particular that \eqref{eq:BC-schemes-TRCAT-Dqc}
  is a 2-isomorphism.
  We do not know whether the 2-morphism 
  \eqref{eq:BC-schemes-ENH-IIqc}
  in $\ENH_\kk$ is invertible.
  But if we postcompose it with
  $\ul\II(X')=\ul\II(X'; \mathscr{U}) 
  \subset \ul\II(X'; \mathscr{U}')$
  (which is well-defined by
  Proposition~\ref{p:h-inj-change-universe-sheaves}), 
  the 
  2-isomorphism
  \eqref{eq:BC-schemes-ENH-Iqc}
  and
  Proposition~\ref{p:target-large}.\ref{enum:useful-reflects-2-isos-large}
  show that this postcomposition
  is a 2-isomorphism
  \begin{equation}
    \label{eq:BC-schemes-ENH-IIqc-large}
    \ul\beta^* \ul\alpha_* \sira
    \ul\alpha'_*\ul\beta'^* \colon 
    \ul\II_\qc(Y)=\ul\II_\qc(Y;\mathscr{U}) \ra \ul\II(X';
    \mathscr{U}') 
  \end{equation}
  in $\ENH_\kk^{\mathscr{V}'}$.
  It is clear that this 2-isomorphism
  lifts the 2-isomorphism
  \begin{equation}
    \label{eq:BC-schemes-TRCAT-Dqc-large}
    \dL\beta^* \dR\alpha_* \sira \dR\alpha'_*\dL\beta'^* \colon
    \D_\qc(Y)=\D_\qc(Y,\mathscr{U}) \ra \D(X';\mathscr{U}')
  \end{equation}
  in $\TRCAT_\kk^{\mathscr{V}'}$.
  Moreover, by
  part~\ref{enum:useful-represent-2-morphisms-large}
  of Proposition~\ref{p:target-large}, the 
  2-isomorphism \eqref{eq:BC-schemes-ENH-IIqc-large}
  is represented by a roof
  \begin{equation}
    \ul\beta^* \ul\alpha_* 
    \xla{u} F \xra{v}
    \ul\alpha'_*\ul\beta'^*
  \end{equation}
  of objectwise homotopy equivalences in
  $\DGCAT_\kk^{\mathscr{V}'}(\ul\II_\qc(Y;\mathscr{U}), \ul\II(X';
  \mathscr{U}'))$ 
  where
  $F$ is some $\KK_{\mathscr{U}'}$-functor 
  $\ul\II_\qc(Y;\mathscr{U}) \ra \ul\II(X'; \mathscr{U}')$.
  The induced functor 
  $[F] \colon [\ul\II_\qc(Y;\mathscr{U})] \ra
  [\ul\II(X'; \mathscr{U}')]$
  lands in the essential image of 
  $[\ul\II(X';\mathscr{U})] \subset [\ul\II(X'; \mathscr{U}')]$.
\end{example} 

\appendix

\section{Spaltenstein's results generalized to ringed topoi}
\label{sec:spalt-results-ring}

\fussnote{
  Hashimoto in Lipman-Hashimoto hat Kapitel: Derived categories
  and derived functors of sheaves on ringed sites. But I guess I
  don't like it...
}

We generalize some results from \cite{spaltenstein}
to ringed sites and ringed topoi.
We use the prefix ``h'' (for ``homotopically'')
instead of Spaltenstein's prefix ``K'', for 
example we say h-limp instead of K-limp.
We use the notation and a few
results 
from \ref{sec:ringed-sites}. 

Let $(\mathcal{X}, \mathcal{O}=\mathcal{O}_\mathcal{X})$ be a
ringed site. The definition of an h-injective (resp.\ h-flat)
complex of $\mathcal{O}$-modules is well-known.

\begin{proposition}
  [{cf.\ \cite[Prop.~5.7]{spaltenstein}}]
  \label{p:spalt-5.7-sites}
  If $F \in \C(\mathcal{X})$ is h-flat acyclic, then $F
  \otimes A$ is acyclic for 
  every $A \in \C(\mathcal{X})$.
\end{proposition}

\begin{proof}
  Let $G \ra A$ be an h-flat resolution \citestacks{06YS}.
  Then $F \otimes G \ra F \otimes A$ is a quasi-isomorphism, and
  $F \otimes G$ is acyclic.
\end{proof}

\begin{proposition}
  [{cf.\ \cite[Prop.~5.4.(b)]{spaltenstein}}]
  \label{p:pullback-preserves-h-flat}
  Let $\alpha \colon (\Sh(\mathcal{Y}), \mathcal{O}_\mathcal{Y})
  \ra (\Sh(\mathcal{X}), \mathcal{O}_\mathcal{X})$ be a morphism
  of ringed topoi. 
  Let $F \in \C(\mathcal{X})$ be h-flat. Then
  $\alpha^*F \in \C(\mathcal{Y})$ is 
  h-flat. It is acyclic if $F$ is in addition acyclic.
\end{proposition}

\begin{proof}
  See \citestacks{06YX}
  for the first claim
  (the assumption there that $\mathcal{Y}$ has enough
  points is not necessary by the 
  comments at the beginning of
  \citestacks{06YV}). 
  Let $F$ be h-flat acyclic, then $\alpha\inv F \in
  \C(\alpha\inv\mathcal{O}_\mathcal{X})$ is acyclic and h-flat,
  by the first claim applied to the morphism $(\Sh(\mathcal{Y}),
  \alpha\inv 
  \mathcal{O}_\mathcal{X}) \ra (\Sh(\mathcal{X}),
  \mathcal{O}_\mathcal{X})$ of ringed topoi. Then $\alpha^*F =
  \mathcal{O}_\mathcal{Y} 
  \otimes_{\alpha\inv \mathcal{O}_\mathcal{X}} \alpha\inv F$ is
  acyclic by Proposition~\ref{p:spalt-5.7-sites}
  (and h-flat). 
\end{proof}

Write
$\mathscr{P}(\mathcal{O}) = \mathscr{P}(\mathcal{O}_\mathcal{X})$
for the set of all bounded above 
complexes $E \in \C(\mathcal{X})$ all of whose components $E^n$
are direct sums of $\mathcal{O}$-modules of the form
$j_{U!}\mathcal{O}_U$, for $U \in \mathcal{X}$. Here we use
the notation from \ref{sec:extension-zero-model}. All
objects of $\mathscr{P}(\mathcal{O})$ are h-flat
\citestacks{06YQ}.
 
\begin{definition}
  [{cf.\ \cite[5.11]{spaltenstein}}]
  A complex $W$ of $\mathcal{O}$-modules is 
  \define{weakly h-injective} if
  $\ul\C_{\mathcal{X}}(F, W)$ is acyclic for all h-flat acyclic 
  objects $F$.
  A complex $L$ of $\mathcal{O}$-modules is 
  \define{h-limp} if $\ul\C_{\mathcal{X}}(P, L)$ is
  acyclic for all acyclic $P \in \mathscr{P}(\mathcal{O})$.
\end{definition}

We have implications h-injective $\Rightarrow$ weakly h-injective
$\Rightarrow$ h-limp.

\begin{lemma}
  \label{l:restriction-to-U-preserves-?-limp}
  Let $U \in \mathcal{X}$. If $A \in \C(\mathcal{X})$
  is h-limp (resp.\ weakly h-injective, h-injective),
  so is $j_U^*A$.
\end{lemma}

\begin{proof}
  We use results from \citestacks{03DH}. 
  Recall that 
  there is an adjunction
  $(j_{U!}, j_U^*)$ of functors between $\Mod(\mathcal{X})$ and
  $\Mod(\mathcal{X}/U)$, and $j_{U!}$ is exact. 
  Therefore we have
  \begin{equation}
    \label{eq:988}
    \ul\C_{\mathcal{X}/U}(T, j_U^*A) \cong
    \ul\C_{\mathcal{X}}(j_{U!}T, A) 
  \end{equation}
  for $T \in \C(\mathcal{X}/U)$, and
  $j_{U!}T$ is acyclic if $T$ is acyclic.
  Hence $j_U^*$ preserves h-injectivity.

  If $(V/U)$ is an object of $\mathcal{X}/U$, then $j_{U!} (j_{V/U
    !}\mathcal{O}_{V/U}) \cong j_{V!}\mathcal{O}_V$. Therefore
  $j_{U!}(\mathscr{P}(\mathcal{O}_U)) \subset
  \mathscr{P}(\mathcal{O})$.
  Hence $j_U^*$ preserves h-limpness.

  In order to see that $j_U^*$ preserves weak h-injectivity we
  need to show that $j_{U!}$ preserves h-flatness.
  Recall the projection formula 
  \begin{equation}
    \label{eq:911}
    j_{U!}(F \otimes j_U^*G) \cong
    (j_{U!}F) \otimes G
  \end{equation}
  for all $F \in \C(\mathcal{X}/U)$ and $G \in \C(\mathcal{X})$,
  see 
  \cite[Exp.~IV, 
  Prop.~12.11.(b)]{SGA4-1}. Since $j_U^*$ and $j_{U!}$ are exact
  this proves that $j_{U!}$ preserves h-flatness.
\end{proof}

\begin{proposition}
  \label{p:spalt-5.20-sites-sheafHom}
  \begin{enumerate}
  \item
    \label{enum:sheafHom-to-h-inj}
    If 
    $T \in
    \C(\mathcal{X})$  is
    acyclic 
    and
    $I \in \C(\mathcal{X})$ 
    is h-injective, then 
    $\sheafHom(T, I)$ is acyclic.
  \item 
    \label{enum:sheafHom-h-flat-weakly-h-inj}
    If 
    $F \in \C(\mathcal{X})$ is h-flat acyclic
    and
    $W \in \C(\mathcal{X})$ is weakly h-injective,
    then
    $\sheafHom(F, W)$ is acyclic.
  \end{enumerate}
\end{proposition}

\begin{proof}
  Let $U \in \mathcal{X}$.

  \ref{enum:sheafHom-to-h-inj}
  Let $I$ be h-injective and $T$ acyclic. Then 
  $j_U^*I$ is h-injective by
  Lemma~\ref{l:restriction-to-U-preserves-?-limp},
  and $j_U^*T$ is acyclic. Therefore
  $\sheafHom(T,I)(U) = \ul\C_{\mathcal{X}/U}(j_U^*T, j_U^*I)$ is
  acyclic. Since $U$ is arbitrary, $\sheafHom(T,I)$ is acyclic as
  a complex of presheaves, and a fortiori as a complex of
  sheaves \citestacks{03EI}.

  \ref{enum:sheafHom-h-flat-weakly-h-inj}
  Let $W$ be weakly h-injective and $F$ h-flat acyclic.
  Then 
  $j_U^*W$ is weakly h-injective by
  Lemma~\ref{l:restriction-to-U-preserves-?-limp},
  and $j_U^*F$ is acyclic, and h-flat, by
  Proposition~\ref{p:pullback-preserves-h-flat}.
  This implies that 
  $\sheafHom(F,W)(U) = \ul\C_{\mathcal{X}/U}(j_U^*F, j_U^*W)$ is
  acyclic, and hence $\sheafHom(F,W)$ is acyclic. 
\end{proof}

\begin{proposition}
  [{cf.\ \cite[Prop.~5.14]{spaltenstein}}]
  \label{p:spalt-5.14-sites-sheafHom}
  If $T \in \C(\mathcal{X})$ is arbitrary (resp.\ h-flat)
  and $I \in \C(\mathcal{X})$ is h-injective, 
  then 
  $\sheafHom(T,I)$ is 
  weakly h-injective (resp.\ is h-injective). 
\end{proposition}

\begin{proof}
  Let $F \in \C(\mathcal{X})$ be h-flat acyclic (resp.\ be acyclic).
  Then $F \otimes T$ is acyclic by 
  Proposition~\ref{p:spalt-5.7-sites} (resp.\ because $T$ is
  h-flat), 
  and hence
  $\ul\C_{\mathcal{X}}(F, \sheafHom(T,I)) \cong
  \ul\C_{\mathcal{X}}(F \otimes T, I)$
  is acyclic because $I$ is h-injective. 
\end{proof}

\begin{proposition}
  [{cf.\ \cite[Prop.~5.16]{spaltenstein}}]
  \label{p:spalt-5.16-sites}
  Let $A \in \C(\mathcal{X})$ be an h-limp acyclic complex and
  $U \in \mathcal{X}$. Then $\Gamma(U, A)=\Gamma(U, j_U^*A)$ is
  acyclic. 
\end{proposition}

\begin{proof}
  The proof of \cite[Prop.~5.16]{spaltenstein} generalizes.
\end{proof}

\begin{proposition}
  [{cf.\ \cite[Prop.~5.15]{spaltenstein}}]
  \label{p:spalt-5.15-sites-h-limp}
  Let $\alpha \colon (\mathcal{Y}, \mathcal{O}_\mathcal{Y}) \ra
  (\mathcal{X}, \mathcal{O}_\mathcal{X})$ be a morphism of
  ringed 
  sites. 
  Let $W \in \C(\mathcal{Y})$ be h-limp.
  Then $\alpha_*W$ is 
  h-limp. It is acyclic if $W$ is in addition acyclic.
\end{proposition}

\begin{proof}
  We claim that
  $\alpha^*(\mathscr{P}(\mathcal{O}_\mathcal{X}) \subset
  \mathscr{P}(\mathcal{O}_\mathcal{Y})$.
  Let $\alpha$ be given by the continuous functor
  $u \colon \mathcal{X} \ra \mathcal{Y}$.  Let
  $U \in \mathcal{X}$ and denote the induced morphism
  $(\mathcal{Y}/u(U), \mathcal{O}_{u(U)}) \ra (\mathcal{X}/U,
  \mathcal{O}_U)$
  of ringed sites also by $\alpha$.  Then
  $j_U^*\alpha_* = \alpha_*j_{u(U)}^*$ (see \citestacks{04IZ})
  implies $\alpha^*j_{U!} \cong j_{u(U)!}\alpha^*$ using the
  Yoneda lemma and the obvious adjunctions. Therefore the claim
  follows from $\alpha^* \mathcal{O}_U=\mathcal{O}_{u(U)}$.

  Let $W \in \C(\mathcal{Y})$ be h-limp.
  For $F \in \mathscr{P}(\mathcal{O}_\mathcal{X})$
  we have just seen that
  $\alpha^*F \in \mathscr{P}(\mathcal{O}_\mathcal{Y})$.
  If $F$ is in addition acyclic, so is $\alpha^*F$ by 
  Proposition~\ref{p:pullback-preserves-h-flat}.
  Then
  $\ul\C_{\mathcal{X}}(F, \alpha_*W) \cong
  \ul\C_{\mathcal{Y}}(\alpha^*F, W)$
  is acyclic, so $\alpha_*W$ is
  h-limp.

  For the second claim note that
  $\alpha_* \colon \C(\mathcal{Y}) \ra
  \C(\mathcal{X})$ is given by 
  $(\alpha_* A) (U) = A(u(U))= \Gamma(u(U), A)$ for $A \in
  \C(\mathcal{Y})$ 
  and $U \in \mathcal{X}$. 
  Therefore Proposition~\ref{p:spalt-5.16-sites} shows that
  $\alpha_*$ maps 
  h-limp acyclic complexes to acyclic 
  complexes.
\end{proof}

\begin{proposition}
  [{cf.\ \cite[Prop.~5.15]{spaltenstein}}]
  \label{p:spalt-5.15-sites}
  Let $\alpha \colon (\Sh(\mathcal{Y}), \mathcal{O}_\mathcal{Y}) \ra
  (\Sh(\mathcal{X}), \mathcal{O}_\mathcal{X})$ be a morphism of
  ringed 
  topoi. 
  Let $W \in \C(\mathcal{Y})$ be weakly h-injective.
  Then $\alpha_*W$ is 
  weakly h-injective. It is acyclic if $W$ is in addition acyclic.
\end{proposition}

\begin{proof}
  Let $W \in \C(\mathcal{Y})$ be weakly h-injective.
  Let $F \in \C(\mathcal{X})$ be h-flat acyclic. Then
  $\alpha^*F$ is 
  h-flat acyclic by
  Proposition~\ref{p:pullback-preserves-h-flat}. Therefore
  $\ul\C_{\mathcal{X}}(F, \alpha_*W) \cong
  \ul\C_{\mathcal{Y}}(\alpha^*F, W)$
  is acyclic, so $\alpha_*W$ is
  weakly h-injective. 


  We can factor $\alpha$ as a composition
  \begin{equation}
    \label{eq:50}
    (\Sh(\mathcal{Y}), \mathcal{O}_\mathcal{Y}) \xra{f}
    (\Sh(\mathcal{Y}'), \mathcal{O}_{\mathcal{Y}'}) \xra{g}
    (\Sh(\mathcal{X}), \mathcal{O}_\mathcal{X})
  \end{equation}
  of an equivalence $f$ of ringed topoi followed by a
  morphism of ringed topoi $g$ induced by a morphism of ringed
  sites, 
  cf.\ \cite[\sptag{03A2}, \sptag{03CR}]{stacks-project}.  
  If $A \in \C(\mathcal{Y})$ is weakly h-injective acyclic,
  then $f_*A$ is weakly h-injective acyclic, in particular h-limp
  acyclic, and hence
  $g_*f_*A=\alpha_*A$ is acyclic by 
  Proposition~\ref{p:spalt-5.15-sites-h-limp}.
\end{proof}


\begin{proposition}
  [{cf.\ \cite[Prop.~5.20.(a)]{spaltenstein}}]
  \label{p:spalt-5.20-sites}
  If $F \in \C(\mathcal{X})$ is h-flat
  and
  $W \in \C(\mathcal{X})$ is weakly h-injective acyclic, then
  $\ul\C_{\mathcal{X}}(F, W)$ is acyclic.
\end{proposition}

\begin{proof}
  Let $L \in \C(\mathcal{X})$ be h-limp acyclic and $U \in
  \mathcal{X}$.  
  Proposition~\ref{p:spalt-5.16-sites} shows that 
  $\ul\C_{\mathcal{X}}(j_{U!}\mathcal{O}_U, L) =
  \ul\C_{\mathcal{X}/U}(\mathcal{O}_U, 
  j_U^* L)=\Gamma(U,L)$ is acyclic.
  The proof of 
  \cite[Prop.~5.20.(c)]{spaltenstein}
  then shows that 
  $\ul\C_{\mathcal{X}}(T,L)$ is acyclic for all objects $T \in
  \underrightarrow{\mathscr{P}}(\mathcal{O})$ (see
  \cite[2.9]{spaltenstein} for the definition of
  $\underrightarrow{\mathscr{P}}(\mathcal{O})$). 
  
  Let $p \colon P \ra F$ be a quasi-isomorphism with  
  $P \in \underrightarrow{\mathscr{P}}(\mathcal{O})$
  (which exists by the proof of \cite[Prop.~5.6]{spaltenstein},
  or by \citestacks{077J}). Then $P$ and $\Cone(p)$ are
  h-flat, and $\Cone(p)$ is of course acyclic. Since $W$ is 
  weakly h-injective,  
  $\ul\C_{\mathcal{X}}(\Cone(p),W)$ is acyclic by definition,
  so 
  $\ul\C_{\mathcal{X}}(F,W) \ra \ul\C_{\mathcal{X}}(P,W)$
  is a quasi-isomorphism whose target is acyclic by the above
  applied to $L=W$.
\end{proof}

\begin{corollary}
  \label{c:spalt-5.20-sites-new}
  If 
  $F \in \C(\mathcal{X})$ is h-flat
  and
  $W \in \C(\mathcal{X})$ is weakly h-injective, then
  $[\ul\C_{\mathcal{X}}](F, W) \ra \D_\mathcal{X}(F, W)$ is an
  isomorphism.
\end{corollary}

\begin{proof}
  Let $W \ra I$ be a quasi-isomorphism with $I$ h-injective.
  Consider the commutative diagram
  \begin{equation}
    \xymatrix{
      {[\ul\C_\mathcal{X}](F, W)}
      \ar[r]^-{\sim}
      \ar[d]
      &
      {[\ul\C_\mathcal{X}](F, I)}
      \ar[d]^-{\sim}
      \\
      {\D_\mathcal{X}(F, W)}
      \ar[r]^-{\sim}
      &
      {\D_\mathcal{X}(F, I).}
    }
  \end{equation}
  Its upper horizontal arrow is an isomorphism by 
  Proposition~\ref{p:spalt-5.20-sites}. The lower
  horizontal and the right vertical arrow are certainly
  isomorphisms. The claim follows.
\end{proof}

\begin{corollary}
  \label{c:spalt-5.20-sites-sheafHom}
  If 
  $F \in \C(\mathcal{X})$ is h-flat and
  $W \in \C(\mathcal{X})$ is weakly h-injective acyclic,
  then
  $\sheafHom(F, W)$ is acyclic.
\end{corollary}

\begin{proof}
  Let $U \in \mathcal{X}$.
  Then 
  $j_U^*W$ is weakly h-injective acyclic by
  Lemma~\ref{l:restriction-to-U-preserves-?-limp},
  and $j_U^*F$ is h-flat, by
  Proposition~\ref{p:pullback-preserves-h-flat}.
  Therefore
  $\sheafHom(F,W)(U) = \ul\C_{\mathcal{X}/U}(j_U^*F, j_U^*W)$ is
  acyclic
  by Proposition~\ref{p:spalt-5.20-sites}. Since $U \in
  \mathcal{X}$ was arbitrary, $\sheafHom(F,W)$
  is acyclic.  
\end{proof}

\begin{proposition}
  [{cf.\ \cite[Prop.~6.1, 6.5]{spaltenstein}}]
  \label{p:spalt-6.1+5-sites}
  Let $(\mathcal{X}, \mathcal{O})$ be a ringed site. Let $A, B
  \in \D(\mathcal{X})$. 
  \begin{enumerate}
  \item 
    \label{enum:RHom}
    Then $\dR\Hom(A,B)$ and
    $\dR\sheafHom(A,B)$ are defined and can be computed by any of
    the following methods:
    \begin{enumerate}
    \item 
      \label{enum:h-injective-derived-Homs}
      using an h-injective resolution of $B$;
    \item 
      \label{enum:h-flat-to-acyclic-w-hinj-derived-Homs}
      using an h-flat resolution of $A$ and a
      weakly h-injective resolution of $B$.
    \end{enumerate}
    \item 
      \label{enum:Lotimes}
      Then $A \otimes^\dL B$ is defined and can be computed
      using 
      an h-flat resolution of $A$ or $B$. 
  \end{enumerate}
\end{proposition}

\begin{proof}
  \ref{enum:RHom}.\ref{enum:h-injective-derived-Homs} Use the
  definition for 
  $\dR\Hom$, and 
  Proposition~\ref{p:spalt-5.20-sites-sheafHom}.\ref{enum:sheafHom-to-h-inj}
  for $\dR\sheafHom$.

  \ref{enum:RHom}.\ref{enum:h-flat-to-acyclic-w-hinj-derived-Homs}
  Use the definition and
  Proposition~\ref{p:spalt-5.20-sites} 
  for $\dR\Hom$. Use
  Proposition~\ref{p:spalt-5.20-sites-sheafHom}.\ref{enum:sheafHom-h-flat-weakly-h-inj} 
  and
  Corollary~\ref{c:spalt-5.20-sites-sheafHom}
  for $\dR\sheafHom$.

  \ref{enum:Lotimes}
  Clear from Proposition~\ref{p:spalt-5.7-sites}.
\end{proof}

\begin{proposition}
  \label{p:spalt-6.7a-sites}
  If
  $\alpha \colon (\Sh(\mathcal{Y}), \mathcal{O}_\mathcal{Y}) \ra
  (\Sh(\mathcal{X}), \mathcal{O}_\mathcal{X})$
  is a morphism of ringed topoi, $\dR \alpha_*$ is defined and
  may be 
  computed
  using weakly h-injective resolutions, and
  $\dL \alpha^*$ exists and may be computed using
  h-flat resolutions.

  If $\alpha$ comes from a morphism $(\mathcal{Y},
  \mathcal{O}_\mathcal{Y}) \ra (\mathcal{X},
  \mathcal{O}_\mathcal{X})$ of ringed sites, $\dR \alpha_*$ may
  even be computed using h-limp resolutions.
\end{proposition}

\begin{proof}
  Certainly, $\dR \alpha_*$ exists and can be computed
  using 
  h-injective resolutions. 
  Proposition~\ref{p:spalt-5.15-sites} shows that it can even be
  computed using weakly h-injective resolutions.
  
  Certainly $\dL \alpha^*$ exists \citestacks{06YY}, and
  Proposition~\ref{p:pullback-preserves-h-flat} shows that it can
  be computed using h-flat resolutions.

  The last claim follows from
  Proposition~\ref{p:spalt-5.15-sites-h-limp}. 
\end{proof}

The usual formulas
$\dR(\alpha \beta)_* \cong \dR \alpha_*
\dR \beta_*$, 
$\dL(\alpha \beta)^* \cong \dL \beta^* \dL
\alpha^*$,
\begin{align}
  \dR \alpha_* \dR \sheafHom(\dL \alpha^* -,-) 
  & \cong \dR \sheafHom(-, \dR \alpha_* -),\\
  \dR\sheafHom(- \otimes^\dL -,-) 
  & \cong \dR\sheafHom(-, \dR\sheafHom(-,-)),
\end{align}
etc.\ 
follow readily.

\section{Change of universe}
\label{sec:change-universe}

Let $\mathscr{U} \subset \mathscr{V}$ be universes.
If $(\mathcal{X}, \mathcal{O})$ is a $\mathscr{U}$-small
$\R$-ringed site, 
it is certainly expected that the derived category
$\D(\mathcal{X}; \mathscr{U})$ of sheaves of
$\mathscr{U}$-small modules embeds fully faithfully into the
corresponding derived category $\D(\mathcal{X}; \mathscr{V})$ of
$\mathscr{V}$-small modules. We could not find this statement in
the literature (in the setting of unbounded derived categories).
In the first part of this appendix we prove this and similar
statements 
concerning h-injective and $\II$-fibrant complexes.
In the second part we provide analog statements for
dg modules.

\subsection{Sheaves of modules on ringed sites}
\label{sec:sheaves-modules}

Let $(\mathcal{X}, \mathcal{O}_\mathcal{X})$ be a
$\mathscr{U}$-small $\R$-ringed site.
Let $\Mod(\mathcal{X}; \mathscr{U})$
be the $\R$-category of $\mathscr{U}$-small
$\mathcal{O}_\mathcal{X}$-modules and
$\C(\mathcal{X}; \mathscr{U})$ the $\R$-category 
of complexes of such sheaves. Previously these categories were
denoted $\Mod(\mathcal{X})$ and $\C(\mathcal{X})$.

Now, using the universe $\mathscr{V}$, there is an obvious 
``change of universe'' $\R$-functor
\begin{equation}
  \label{eq:change-universe}
  \Mod(\mathcal{X}; \mathscr{U})
  \ra
  \Mod(\mathcal{X}; \mathscr{V}).
\end{equation}

\begin{remark}
  \label{rem:change-universe}
  The 
  change of universe functor \eqref{eq:change-universe}
  is fully faithful,
  reflects isomorphisms, preserves all $\mathscr{U}$-small limits
  and colimits (hence is exact), preserves injectives,
  and is compatible with tensor products and sheaf homomorphisms
  (see \cite[Prop.~V.1.9]{SGA4-2}).
  More trivially, the essentially $\mathscr{U}$-small set
  (i.\,e.\ it is isomorphic to a $\mathscr{U}$-small set) of
  subobjects (resp.\ quotients) of an 
  object of $\Mod(\mathcal{X}; \mathscr{U})$
  is in the obvious way identified with the 
  set of subobjects (resp.\ quotients) of the corresponding object of
  $\Mod(\mathcal{X}; \mathscr{V})$.
\end{remark}


\begin{proposition}
  \label{p:h-inj-change-universe-sheaves}
  Let $\mathscr{U} \subset \mathscr{V}$ be
  universes.
  Let $(\mathscr{X}, \mathcal{O}_\mathcal{X})$ be a
  $\mathscr{U}$-small $\R$-ringed site.
  Then the change of universe functor 
  $\C(\mathcal{X}; \mathscr{U})
  \ra
  \C(\mathcal{X}; \mathscr{V})$ preserves h-injective
  complexes and 
  $\II$-fibrant objects.
  In particular, the induced change of universe functor
  \begin{equation}
    \D(\mathcal{X}; \mathscr{U}) \ra \D(\mathcal{X}; \mathscr{V})
  \end{equation}
  on the level of derived categories is full and faithful.
\end{proposition}

\begin{proof}
  We first prove the statement concerning $\II$-fibrant objects.
  Our proof is heavily based on results in \cite[9,
  14.1]{KS-cat-sh}.
  Note that an object of $\C(\mathcal{X}; \mathscr{U})$ is
  $\II$-fibrant if 
  and only if it is QM-injective in the terminology of
  \cite[Def.~9.5.1, (14.1.2)]{KS-cat-sh} where QM is the
  set of morphisms that are both quasi-isomorphisms and
  monomorphisms. 

  Let us look at the proof of
  \cite[Thm.~14.1.7]{KS-cat-sh}
  in the case of the Grothendieck abelian category
  $\Mod(\mathcal{X}; \mathscr{U})^\bZ$ of graded
  $\mathscr{U}$-small $\mathcal{O}_\mathcal{X}$-modules with
  obvious translation. Kashiwara and Schapira show that there is
  an essentially $\mathscr{U}$-small full subcategory
  $\mathcal{S}$ of $\C(\mathcal{X}; \mathscr{U})$ having the
  properties (i)-(iv) in \cite[(14.1.4)]{KS-cat-sh} (and which
  can and should also be assumed to be closed under translation;
  a category is essentially $\mathscr{U}$-small if it is
  equivalent to a $\mathscr{U}$-small category).
  Then they show that there is a $\mathscr{U}$-small set
  $\mathcal{F}$ of monomorphic quasi-isomorphisms $X \ra
  Y$ between objects $X, Y \in \mathcal{S}$ such that an object
  of $\C(\mathcal{X}; \mathscr{U})$ is QM-injective (i.\,e.\
  $\II$-fibrant) if and only if it is $\mathcal{F}$-injective
  (this statement relies on \cite[Thm.~9.5.5]{KS-cat-sh}).
  
  We need to look more closely at the construction of
  $\mathcal{S}$ using \cite[9.3, in particular
  Corollary~9.3.8]{KS-cat-sh}. Consider the generator 
  \begin{equation}
    G:= \bigoplus_{n \in \bZ} \bigoplus_{U \in \mathcal{X}}
    [n]\iCone(j_{U!}\mathcal{O}_U)  
  \end{equation}
  of $\C(\mathcal{X}; \mathscr{U})$. Let $\pi_0$ be an infinite
  regular cardinal (in the universe $\mathscr{U}$) such that 
  \begin{equation}
    \label{eq:card-G}
    \card(G(G)) < \pi_0 
    \quad \text{and} \quad
    \card(G^{\oplus G(G)}(G)) < \pi_0
  \end{equation}
  where we abbreviate $A(B):=\C_{\mathcal{X}}(A,B)$; choose
  another infinite regular cardinal $\pi \in \mathscr{U}$ such
  that the 
  conditions in \cite[(9.3.4)]{KS-cat-sh} are satisfied for the
  generator $G$ and the category $\C(\mathcal{X}; \mathscr{U})$.
  Then, by the proof of \cite[Cor.\ 9.3.8]{KS-cat-sh} we can
  and will assume that $\mathcal{S}$ is the full subcategory
  $\C(\mathcal{X}; \mathscr{U})_\pi$ of $\pi$-accessible objects
  in $\C(\mathcal{X}; \mathscr{U})$.

  Now consider the change of universe functor 
  \begin{equation}
    \epsilon \colon \C(\mathcal{X}; \mathscr{U}) 
    \ra \C(\mathcal{X}; \mathscr{V}).
  \end{equation}
  Let $\mathcal{S}':= \epsilon(\mathcal{S})$ be the essential
  image of $\mathcal{S}$ under this functor.
  We claim that every object of $\mathcal{S}'$ is
  $\pi$-accessible, i.\,e.\ $\mathcal{S}' \subset \C(\mathcal{X};
  \mathscr{V})_\pi$. 

  To see this let $S \in \mathcal{S}=\C(\mathcal{X};
  \mathscr{U})_\pi$. By \cite[Thm.~9.3.4]{KS-cat-sh} this means
  that $\card(S(G))< \pi$. Note that $\epsilon(G)$, $\pi$ and
  $\pi_0$ satisfy the conditions in 
  \cite[(9.3.4)]{KS-cat-sh} for the category $\C(\mathcal{X};
  \mathscr{V})$: (a) and (c) are obvious, (b) is satisfied by
  \cite[Prop.~9.3.2]{KS-cat-sh} and the fact that
  \eqref{eq:card-G}
  also holds for $G$ replaced by $\epsilon(G)$ (by
  Remark~\ref{rem:change-universe}),  
  and (d) is satisfied since any
  $\mathscr{V}$-small set $A$ with $\card(A)< \pi_0$ is
  isomorphic to a $\mathscr{U}$-small set $B$ since $\pi_0 \in
  \mathscr{U}$. 
  Since $\card(\epsilon(S)(\epsilon(G)))=\card(S(G))< \pi$ 
  we obtain from \cite[Thm.~9.3.4]{KS-cat-sh} that $\epsilon(S)$
  is $\pi$-accessible. This proves our claim that any object of
  $\mathcal{S}'$ is $\pi$-accessible.

  We now claim that $\mathcal{S}'$ (which is obviously closed
  under translation) also has the 
  properties (i)-(iv) in \cite[(14.1.4)]{KS-cat-sh}, now with
  respect to $\C(\mathcal{X}; \mathscr{V})$: property (i) is
  clear since $\epsilon(G)$ is a generator, and properties (ii)
  and (iv) are obvious (cf.\ Remark~\ref{rem:change-universe}).
  To establish (iii) it is sufficient to check the following
  condition: Given any epimorphism $f \colon Z \sra Y$ with $Y
  \in \mathcal{S}'$, there is an object $S \in \mathcal{S}'$ and
  a morphism $g \colon S \ra Z$ such that the composition $f
  \circ g$ is an epimorphism. 
  
  To see this, we copy the argument from the proof of
  \cite[Cor.~9.3.8]{KS-cat-sh}. Consider the epimorphisms
  \begin{equation}
    \epsilon(G)^{\oplus Z(\epsilon(G))} \sra Z \sra Y
  \end{equation}
  and the $\mathscr{V}$-small $\pi$-filtered ordered set
  \begin{equation}
    I:= \{A \subset Z(\epsilon(G)) \mid \card(A) < \pi\}.
  \end{equation}
  Note that $\colim_{A \in I} \epsilon(G)^{\oplus A} \sira
  \epsilon(G)^{\oplus Z(\epsilon(G))}$. Since $Y \in \mathcal{S}'
  \subset \C(\mathcal{X}; \mathscr{V})_\pi$ there is, by
  \cite[Lemma~9.3.1]{KS-cat-sh}, an element $A \in I$ such that
  the composition $\epsilon(G)^{\oplus A} \ra Z \sra Y$ is
  already an epimorphism. It remains to show that
  $\epsilon(G)^{\oplus A} \in \mathcal{S}'$.
  Since $\card(A)< \pi \in \mathscr{U}$ there is a
  $\mathscr{U}$-small set $B$ such that $B \cong A$, and it is
  enough to show that $G^{\oplus B} \in \mathcal{S} =
  \C(\mathcal{X}; \mathscr{U})_\pi$.
  By \cite[Thm.~9.3.4]{KS-cat-sh} it is enough to show that 
  $\card(G^{\oplus B}(G))< \pi$; but this is true
  by 
  \cite[Lemma~9.3.3]{KS-cat-sh}.
  Now we have proven the claim that $\mathcal{S}'$ has the properties
  (i)-(iv) in \cite[(14.1.4)]{KS-cat-sh}.

  Looking at the proof of \cite[Thm.~14.1.7]{KS-cat-sh} again, we
  see that an object of $\C(\mathcal{X}; \mathscr{V})$ is
  QM-injective (i.\,e.\ $\II$-fibrant) if and only if it is
  $\epsilon(\mathcal{F})$-injective. 
  Since $\epsilon$ obviously sends $\mathcal{F}$-injective
  objects to $\epsilon(\mathcal{F})$-injective objects, we deduce
  that $\epsilon$ preserves $\II$-fibrant objects.

  Now it is easy to see all other claims.
  Let $E \in \C(\mathcal{X}; \mathscr{U})$ be h-injective.
  Let $r \colon E \ra I_E$ be a quasi-isomorphism with
  $\II$-fibrant $I_E$. 
  Recall that the $\II$-fibrant objects are precisely the
  h-injective objects with injective components.
  Since $r$ is a quasi-isomorphism between h-injective objects, 
  it is invertible in the homotopy category
  $[\C(\mathcal{X}; \mathscr{U})]$, and 
  $\epsilon(r)$ is invertible in 
  $[\C(\mathcal{X}; \mathscr{V})]$.
  By the above we know that $\epsilon(I_E)$ is $\II$-fibrant and hence
  h-injective. Since h-injectivity is invariant under
  isomorphisms in the homotopy category, $\epsilon(E)$ is
  h-injective as well.

  The claim concerning derived categories follows since 
  \eqref{eq:change-universe} is exact and fully faithful and the
  derived categories are equivalent to the corresponding homotopy
  categories of h-injectives.
\end{proof}

\subsection{Dg Modules over dg categories}
\label{sec:dg-modules-over}

Recall that $\RR=\C(\Mod(\R))$. More precisely, we have
$\RR=\RR_\mathscr{U}=\C(\Mod(\R;\mathscr{U}))$ where
$\Mod(\R;\mathscr{U})$ denotes the category of
$\mathscr{U}$-small $\R$-modules.
Similarly, $\RR_\mathscr{V}$ is defined using the universe
$\mathscr{V}$. 

Let $\ulms{C}$ be a $\mathscr{U}$-small $\RR$-category. Let
$\MMod(\ulms{C};\mathscr{U})$ be the $\R$-category of
$\ulms{C}$-modules (with values in $\ul\RR$), i.\,e.\ of
$\RR$-functors $\ulms{C}^\opp \ra \ul\RR$. 
Similarly, let 
$\MMod(\ulms{C}; \mathscr{V})$ be the $\R$-category of
$\ulms{C}$-modules (with values in $\ul\RR_\mathscr{V}$), i.\,e.\ of
$\RR_\mathscr{V}$-functors $\ulms{C}^\opp \ra \ul\RR_\mathscr{V}$. 

There is an obvious change of universe $\R$ functor $\MMod(\ulms{C};
\mathscr{U}) \ra \MMod(\ulms{C}; \mathscr{V})$ which is fully
faithful, reflects isomorphisms and preserves $\mathscr{U}$-small
limits and colimits (hence is exact). 

\begin{proposition}
  \label{p:h-inj-change-universe-dg-modules}
  Let $\mathscr{U} \subset \mathscr{V}$ be
  universes.
  Let $\ulms{C}$ be a
  $\mathscr{U}$-small $\RR$-category.
  Then the change of universe functor 
  $\MMod(\ulms{C}; \mathscr{U})
  \ra
  \MMod(\ulms{C}; \mathscr{V})$ preserves h-injective,
  $\II$-fibrant,
  h-projective, and $\PP$-cofibrant
  objects.
  In particular, the induced change of universe functor
  \begin{equation}
    \D(\ulms{C}; \mathscr{U}) \ra \D(\ulms{C}; \mathscr{V})
  \end{equation}
  on the level of derived categories is full and faithful.
\end{proposition}

\begin{proof}
  The proof that h-injective and $\II$-fibrant objects are
  preserved follows the proof of
  Proposition~\ref{p:h-inj-change-universe-sheaves}, taking the
  generator 
  $G:= \bigoplus_{n \in \bZ} \bigoplus_{C \in \ulms{C}}
  [n]\iCone(\Yo(C))$.

  The proof that $\PP$-cofibrant objects are preserved uses the
  fact that every $\PP$-cofibrant object is a retract of a
  semi-free $\ulms{C}$-module, see e.\,g.\
  \cite[Lemma~2.7]{valery-olaf-smoothness-equivariant}, and
  semi-free $\ulms{C}$-modules are clearly preserved.
  Since all $\PP$-cofibrant objects are h-projective, 
  see e.\,g.\
  \cite[Lemma~2.6]{valery-olaf-smoothness-equivariant},
  the claim concerning h-projective objects follows as the
  corresponding claim for h-injectives in the proof of 
  Proposition~\ref{p:h-inj-change-universe-sheaves}.
  
  The statement about derived categories is then obvious.
\end{proof}

\section{Some results on model categories}
\label{sec:some-results-model}

Our aim is to prove Theorem~\ref{t:localization-exists}.
This result should be well-known to the experts and is not
difficult to prove but we could not
find a reference in the literature.

As a preparation we need
Proposition~\ref{p:morphisms-in-htpy-cat}.
Its proof provides some details to the
last claim in 
the proof of
\cite[Thm.~14.4.6]{may-ponto-more-concise-alg-topo}, and also
to the last sentence in the proof of
\cite[Thm.~8.3.5]{hirschhorn-model}.
We use the terminology of fibrant cofibrant and cofibrant
fibrant approximations from
\cite[Def.~8.1.2]{hirschhorn-model} and in particular do not
assume that these approximations are functorial.

\begin{proposition}
  [{cf.\ \cite[Thm.~14.4.6]{may-ponto-more-concise-alg-topo}}]
  \label{p:morphisms-in-htpy-cat}
  Let $\mathcal{M}$ be a model category (with respect to some
  universe). 
  For each object $M \in \mathcal{M}$ choose a fibrant cofibrant
  approximation $q_M \colon QM \ra M$ and a cofibrant fibrant
  approximation $r_M \colon M \ra RM$.
  Let $\Ho(\mathcal{M})$ and $\gamma \colon \mathcal{M} \ra
  \Ho(\mathcal{M})$ be defined as in 
  \cite[Def.~14.4.5]{may-ponto-more-concise-alg-topo}.
  Then:
  \begin{enumerate}
  \item
    \label{enum:we-iff-gamma-isom}
    A morphism $m$ in $\mathcal{M}$ is a weak equivalence if
    and only if $\gamma(m)$ is an isomorphism.
  \item 
    \label{enum:morphisms-in-HoM}
    Let $f \colon X \ra Y$ be a morphism in $\Ho(\mathcal{M})$.
    Then there is a morphism
    $g \colon RQX \ra RQY$ such that
    \begin{equation}
      f = \gamma(q_Y) \circ \gamma(r_{QY})\inv \circ \gamma(g)
      \circ \gamma(r_{QX}) \circ \gamma(q_X)\inv
    \end{equation}
    in $\Ho(\mathcal{M})$.
  \end{enumerate}
\end{proposition}

\begin{proof}
  \ref{enum:we-iff-gamma-isom} 
  This is the first statement of 
  \cite[Thm.~14.4.6]{may-ponto-more-concise-alg-topo}.

  Before proving \ref{enum:morphisms-in-HoM}
  we start with an observation. 
  Recall first that the definition of $\gamma$ uses the fact that
  any 
  morphism $m 
  \colon M \ra N$ admits morphisms 
  $Qm \colon QM \ra QN$ and $Rm \colon RM \ra RN$ such that $q_N
  \circ Qm = m \circ q_M$ and $r_N \circ m = Rm \circ r_M$ (cf.\
  \cite[14.4]{may-ponto-more-concise-alg-topo}). Note 
  however that we do not assume that $m \mapsto Qm$ or $m
  \mapsto Rm$ respect identities or composition.
  
  For any object
  $M \in \mathcal{M}$ consider the following
  commutative diagram 
  in
  $\mathcal{M}$ which is constructed starting from its top row.
  \begin{equation}
    \label{eq:RQRQM}
    \xymatrix{
      {RQM} &
      {QM} \ar[l]_-{r_{QM}} \ar[r]^-{q_M} &
      {M} \\
      {QRQM} 
      \ar[u]^-{q_{RQM}}
      \ar[d]_-{r_{QRQM}}
      &
      {QQM} \ar[l]_-{Qr_{QM}} \ar[r]^-{Qq_M} 
      \ar[u]^-{q_{QM}}
      \ar[d]_-{r_{QQM}}
      &
      {QM}
      \ar[u]^-{q_{M}}
      \ar[d]_-{r_{QM}}
      \\
      {RQRQM} &
      {RQQM} \ar[l]_-{RQr_{QM}} \ar[r]^-{RQq_M} &
      {RQM} \\
    }
  \end{equation}
  All arrows are weak equivalences by the 2-out-of-3 property.
  Hence $\gamma$ maps all arrows to isomorphisms.
  Since top row and right column 
  of our diagram coincide, this implies that
  \begin{equation}
    \label{eq:bottom=left}
    \gamma(RQq_M) \circ \gamma(RQr_{QM})\inv = \gamma(q_{RQM})
    \circ \gamma(r_{QRQM})\inv.
  \end{equation}
  
  Let $\mathcal{M}_\bifib$ be the full subcategory of
  $\mathcal{M}$ of bifibrant objects. 
  Let $\op{h}(\mathcal{M}_\bifib)$ denote the category
  whose objects are the objects of $\mathcal{M}_\bifib$ and
  whose morphisms are equivalence classes of morphisms up to
  homotopy. There is a canonical functor $\mathcal{M}_\bifib \ra
  \op{h}(\mathcal{M}_\bifib)$ which is the identity on objects
  and maps a morphism $m$ to its equivalence class $[m]$.

  All objects in the bottom row and left column of diagram
  \eqref{eq:RQRQM} are in $\mathcal{M}_\bifib$, and all arrows
  between these objects are isomorphisms in
  $\op{h}(\mathcal{M}_\bifib)$, by
  \cite[14.3.15]{may-ponto-more-concise-alg-topo}.  The functor
  $\gamma$ induces a functor
  $\op{h}(\mathcal{M}_\bifib) \ra \Ho(\mathcal{M})$ (by
  \cite[Lemma~8.3.4]{hirschhorn-model}) which is easily seen to
  be full and faithful (and even an equivalence), cf.\
  \cite[Prop.~1.2.5]{hovey-model-categories}.  (We do not cite
  \cite[Thm.~8.3.6]{hirschhorn-model} or
  \cite[Thm.~1.2.10]{hovey-model-categories} since the result we
  prove is a strengthening of the last sentence in the proof of
  \cite[Thm.~8.3.5]{hirschhorn-model}.)  Therefore
  \eqref{eq:bottom=left} implies
  \begin{equation}
    \label{eq:bottom=left-2}
    [RQq_M] \circ [RQr_{QM}]\inv = [q_{RQM}]
    \circ [r_{QRQM}]\inv
  \end{equation}
  in
  $\op{h}(\mathcal{M}_\bifib)$.

  \ref{enum:morphisms-in-HoM}
  Let $f \colon X \ra Y$ be an
  arbitrary morphism in $\Ho(\mathcal{M})$. Then $f$ is the
  homotopy class $[\hat{f}]$ of a morphism
  $\hat{f} \colon RQX \ra RQY$ in $\mathcal{M}_\bifib$. We have a
  commutative diagram 
  \begin{equation}
    \xymatrix@d{
      {RQX} \ar[r]_-{\hat{f}} &
      {RQY}
      \\
      {QRQX} \ar[r]_-{Q\hat{f}} 
      \ar[u]^-{q_{RQX}} 
      \ar[d]_-{r_{QRQX}}
      &
      {QRQY}
      \ar[u]^-{q_{RQY}} 
      \ar[d]_-{r_{QRQY}}
      \\
      {RQRQX} \ar[r]_-{RQ\hat{f}} &
      {RQRQY}
    }
  \end{equation}
  in $\mathcal{M}_\bifib$.
  If we pass to $\op{h}(\mathcal{M}_\bifib)$, replace the 
  left horizontal arrows by their inverses, and 
  use \eqref{eq:bottom=left-2}, we see that the diagram
  \begin{equation}
    \xymatrix@d{
      {RQX} \ar[r]_-{[\hat{f}]} &
      {RQY}
      \\
      {RQQX}
      \ar[u]^-{[RQq_{X}]} 
      &
      {RQQY}
      \ar[u]^-{[RQq_{Y}]} 
      \\
      {RQRQX} \ar[r]_-{[RQ\hat{f}]} 
      \ar[u]^-{[RQr_{QX}]\inv}
      &
      {RQRQY}
      \ar[u]^-{[RQr_{QY}]\inv}
    }
  \end{equation}
  in $\op{h}(\mathcal{M}_\bifib)$ commutes.
  But this diagram just says that the diagram
  \begin{equation}
    \xymatrix@d{
      {X} \ar[r]_-{f} &
      {Y}
      \\
      {QX}
      \ar[u]^-{\gamma(q_{X})} 
      &
      {QY}
      \ar[u]^-{\gamma(q_{Y})} 
      \\
      {RQX} \ar[r]_-{\gamma(\hat{f})} 
      \ar[u]^-{\gamma(r_{QX})\inv}
      &
      {RQY}
      \ar[u]^-{\gamma(r_{QY})\inv}
    }
  \end{equation}
  in $\Ho(\mathcal{M})$ commutes. This implies the proposition by
  setting $g=\hat{f}$.
\end{proof}

\begin{remark}
  \label{rem:is-homotopy-cat}
  The functor
   $\gamma \colon \mathcal{M} \ra
  \Ho(\mathcal{M})$
  from Proposition~\ref{p:morphisms-in-htpy-cat}
  is of course the localization of $\mathcal{M}$ with
  respect to the set of weak equivalences,
  see \cite[Thm.~8.3.6]{hirschhorn-model}. 
\end{remark}

\begin{theorem}
  \label{t:localization-exists}
  Let $\mathcal{S}$ be a full
  subcategory of a model category $\mathcal{M}$ with respect to
  some universe $\mathscr{U}$ such that each
  object $S$ of 
  $\mathcal{S}$ admits a fibrant cofibrant approximation
  $\tildew{S} \ra S$ with $\tildew{S} \in \mathcal{S}$ and a
  cofibrant fibrant approximation 
  $S \ra \hat{S}$ with $\hat{S} \in \mathcal{S}$. 
  Let $\gamma \colon \mathcal{M} \ra \Ho(\mathcal{M})$ be the
  canonical functor to the
  homotopy category of $\mathcal{M}$. 
  Let 
  $\delta \colon \mathcal{S} \ra \L \mathcal{S}$
  be the localization
  of $\mathcal{S}$ with respect to the set of 
  weak equivalences in $\mathcal{S}$.
  Then 
  \begin{enumerate}
  \item 
    \label{enum:localization-of-S-exists}
    we may and will assume that the category $\L \mathcal{S}$ has
    $\mathscr{U}$-small $\Hom$-sets, that it has the same set of
    objects as $\mathcal{S}$ and that $\delta$ is the
    identity on objects;
  \item 
    \label{enum:localization-subcategory}
    the unique functor $\L \mathcal{S} \ra \Ho(\mathcal{M})$
    whose composition with 
    $\delta$ is the composition $\mathcal{S} \ra \mathcal{M}
    \xra{\gamma} \Ho(\mathcal{M})$ is full and faithful;
  \item 
    \label{enum:weak-equiv-vs-iso}
    a morphism $f$ in $\mathcal{S}$ is a weak equivalence if and
    only if $\delta(f)$ is an isomorphism;
  \item 
    \label{enum:represent-morphisms}
    any morphism $f$ in $\L \mathcal{S}$ has the form
    \begin{equation}
      \label{eq:representation-f}
      f = \delta(q') \circ \delta(r')\inv \circ \delta(f')
      \circ \delta(r) \circ \delta(q)\inv
    \end{equation}
    where $q$, $q'$, $r$, $r'$ are weak equivalences in
    $\mathcal{S}$ and $f'$ is a morphism in $\mathcal{S}$;
    more precisely, 
    $q$ and $q'$ are trivial fibrations from a cofibrant
    object, and $r$ and $r'$ are trivial cofibrations to a
    bifibrant object;
  \item
    \label{enum:f'-we-iff-f-iso}
    if a morphism $f$ in $\L \mathcal{S}$ is represented as in
    \ref{enum:represent-morphisms} then $f$ is an isomorphism if
    and only if $f'$ is a weak equivalence;
  \item   
    \label{enum:simplify-rep-if-S-fibrant}
    if all objects of $\mathcal{S}$ are fibrant, 
    we can assume
    that $r=\id$ and $r'=\id$ in 
    \ref{enum:represent-morphisms},
    i.\,e.\  
    $f = \delta(q' f') \circ \delta(q)\inv$.
  \end{enumerate}
  Assume that $h \colon \mathcal{M} \ra \mathcal{C}$ is a
  functor to some category $\mathcal{C}$ such that a morphism $f$
  in $\mathcal{M}$ is a weak equivalence if and only if
  $h(f)$ is an isomorphism. Let $h_\mathcal{S} \colon
  \mathcal{S} \ra \mathcal{C}$ be the restriction of $h$ to
  $\mathcal{S}$.  
  Then
  \begin{enumerate}[resume]
  \item 
    \label{enum:olhS-reflects-isos}
    the unique functor
    $\ol{h}_\mathcal{S} \colon \L\mathcal{S} \ra \mathcal{C}$ such
    $h_\mathcal{S}=\ol{h}_\mathcal{S} \circ \delta$
    reflects isomorphisms.
  \end{enumerate}
\end{theorem}


\begin{proof}
  We can assume that $\gamma \colon \mathcal{M} \ra
  \Ho(\mathcal{M})$ is
  defined as in Proposition~\ref{p:morphisms-in-htpy-cat}, cf.\ 
  Remark~\ref{rem:is-homotopy-cat}.
  Let $\L \mathcal{S}$ be the full subcategory of
  $\Ho(\mathcal{M})$ consisting of the objects $\gamma(S)=S$, for
  $S \in \mathcal{S}$, and let $\delta \colon \mathcal{S} \ra \L
  \mathcal{S}$ be the functor induced by $\gamma$ (which will
  turn out to be the localization we want).  

  \ref{enum:weak-equiv-vs-iso} and
  \ref{enum:represent-morphisms}: This follows from
  Proposition~\ref{p:morphisms-in-htpy-cat} and our assumption
  that the fibrant cofibrant (resp.\ cofibrant fibrant)
  approximations of objects of $\mathcal{S}$ exist in
  $\mathcal{S}$. 

  \ref{enum:localization-of-S-exists} This and the fact that
  $\delta$ is the localization with respect to the set of weak
  equivalences follows from the
  proof of \cite[Thm.~8.3.6]{hirschhorn-model} 
  because
  \ref{enum:represent-morphisms} is known.

  \ref{enum:localization-subcategory} Obvious by construction.

  \ref{enum:f'-we-iff-f-iso} Clear from 
  \ref{enum:weak-equiv-vs-iso}.

  \ref{enum:simplify-rep-if-S-fibrant} Clear since all cofibrant
  fibrant approximations of objects of $\mathcal{S}$ can be taken
  to be the identities.

  \ref{enum:olhS-reflects-isos} The functor
  $\ol{h}_\mathcal{S}$ exists uniquely by the universal property
  of the localization. Let $f$ be a morphism in
  $\L \mathcal{S}$ such that $\ol{h}_\mathcal{S}(f)$ is an
  isomorphism. Represent $f$ as in 
  \ref{enum:represent-morphisms}. Then 
  $\ol{h}_\mathcal{S}(f)=h(q') \circ h(r')\inv \circ h(f') \circ h(r)
  \circ h(q)\inv$, and $h(q)$, $h(q')$, $h(r)$, $h(r')$ are
  isomorphisms by assumption. Therefore $h(f')$ is an
  isomorphism, so $f'$ is a weak equivalence, and 
  \ref{enum:f'-we-iff-f-iso} shows that $f$ is an isomorphism.
\end{proof}

\def\cprime{$'$} \def\cprime{$'$} \def\cprime{$'$} \def\cprime{$'$}
  \def\Dbar{\leavevmode\lower.6ex\hbox to 0pt{\hskip-.23ex \accent"16\hss}D}
  \def\cprime{$'$} \def\cprime{$'$}

\end{document}